\documentclass[a4paper,reqno]{amsart}

\usepackage{amscd,amsthm,amsfonts,amssymb,amsmath,esint}
\usepackage{amsmath,tikz-cd}

\usepackage[colorlinks=true]{hyperref}
\hypersetup{colorlinks, linkcolor=black, citecolor=black, urlcolor=black}

\usepackage{fullpage}
\usepackage[all]{xy}
\input xy
\xyoption{all}
\usepackage{mathtools}
\usepackage{mathrsfs}
\usepackage{enumerate}
\usepackage{tikz}
\usepackage{tikz-cd}
\usetikzlibrary{%
  matrix,%
  calc,%
  arrows%
}

\usepackage{amsrefs}

\usepackage{stmaryrd}

\newcommand{\maru}[1]{\raise0.2ex\hbox{\textcircled{\scriptsize{#1}}}}
\def\MARU#1{{\rm\ooalign{\hfil\lower.168ex\hbox{#1}\hfil \crcr\mathhexbox20D}}}

    \newcommand{\BC}{{\mathbb {C}}} 
     \newcommand{\BF}{{\mathbb {F}}}

    \newcommand{\BQ}{{\mathbb {Q}}} \newcommand{\BR}{{\mathbb {R}}}
     \newcommand{\BT}{{\mathbb {T}}}

     \newcommand{\BZ}{{\mathbb {Z}}}

    \newcommand{\CK}{{\mathcal {K}}} 
    \newcommand{\CM}{{\mathcal {M}}} 
    \newcommand{\CO}{{\mathcal {O}}} 
     \newcommand{\CR}{{\mathcal {R}}}
     \newcommand{\CT}{{\mathcal {T}}}
     \newcommand{\CV}{{\mathcal {V}}}

     \newcommand{\SF}{{\mathscr {F}}}

    \newcommand{\fa}{{\mathfrak{a}}} 
     
     \newcommand{\ff}{{\mathfrak{f}}}
    \newcommand{\fg}{{\mathfrak{g}}} 
     
     \newcommand{\fl}{{\mathfrak{l}}}
    \newcommand{\fm}{{\mathfrak{m}}} \newcommand{\fn}{{\mathfrak{n}}}
     \newcommand{\fp}{{\mathfrak{p}}}

    \newcommand{\Aut}{{\mathrm{Aut}}}

    \newcommand{\cond}{\mathrm{cond^r}}

    \newcommand{\cyc}{{\mathrm{cyc}}}

    \newcommand{\End}{{\mathrm{End}}}

    \newcommand{\Frob}{{\mathrm{Frob}}}

    \newcommand{\Gal}{{\mathrm{Gal}}} \newcommand{\GL}{{\mathrm{GL}}}
    
    \newcommand{\Hom}{{\mathrm{Hom}}}

    \newcommand{\Ind}{{\mathrm{Ind}}}

    \newcommand{\loc}{{\mathrm{loc}}}

    \newcommand{\ord}{{\mathrm{ord}}} \newcommand{\rank}{{\mathrm{rank}}}

    \renewcommand{\mod}{\ \mathrm{mod}\ }
    \newcommand{\rec}{{\mathrm{rec}}}

    \newcommand{\Res}{{\mathrm{Res}}}
    \newcommand{\Sel}{{\mathrm{Sel}}}

    \newcommand{\sgn}{{\mathrm{sgn}}}

    \newcommand{\ur}{{\mathrm{ur}}}

    \font\cyr=wncyr10

    \newcommand{\Sha}{\hbox{\cyr X}}

    \newcommand{\ov}{\overline}

    \newcommand{\ra}{\rightarrow} 
    
    \newcommand{\nequiv}{\equiv\hspace{-10pt}/\ }

    \theoremstyle{plain}
    \newtheorem{thm}{Theorem}[section] \newtheorem{cor}[thm]{Corollary}
    \newtheorem{lem}[thm]{Lemma}  \newtheorem{prop}[thm]{Proposition}
    \newtheorem {conj}[thm]{Conjecture} \newtheorem{defn}[thm]{Definition}
     \newtheorem{lem-defn}[thm]{Lemma-Definition}
\newtheorem{fact}[thm]{Fact}

\theoremstyle{remark} \newtheorem{remark}[thm]{Remark}
\theoremstyle{remark} 
\theoremstyle{remark} \newtheorem{example}{Example}
\theoremstyle{remark} 

    \numberwithin{equation}{section}

\newcommand{\isom}{\overset{\sim}{\rightarrow}}

\begin{document}
\title{A local sign decomposition for symplectic self-dual Galois representations of rank two}
\author{Ashay Burungale, Shinichi Kobayashi, Kentaro Nakamura, Kazuto Ota}
\address{Ashay A. Burungale:  The University of Texas at Austin, Austin, TX 78712, USA.
} 
\email{ashayburungale@gmail.com}

\address{Shinichi Kobayashi: Faculty of Mathematics,
Kyushu University, 744, Motooka, Nishi-ku, Fukuoka, 819-0395, Japan.}
\email{kobayashi@math.kyushu-u.ac.jp}

\address{Kentaro Nakamura: Faculty of Mathematics,
Kyushu University, 744, Motooka, Nishi-ku, Fukuoka, 819-0395, Japan.}
\email{nakamura.kentaro.858@m.kyushu-u.ac.jp}

\address{Kazuto Ota: \textsc{Department of Mathematics, Graduate School of Science, Osaka University Toyonaka, Osaka 560-0043, Japan}} 
\email{
kazutoota@math.sci.osaka-u.ac.jp}

\begin{abstract}

 We prove the existence of a new structure on the first Galois cohomology of generic families of symplectic self-dual $p$-adic representations of $G_{\BQ_p}$ of rank two (a local sign decomposition): a functorial decomposition into free rank one Lagrangian submodules which 
encodes the $p$-adic variation of Bloch--Kato subgroups via completed epsilon constants, mirroring a    
symplectic structure.

The local sign decomposition has diverse local as well as global arithmetic consequences. This includes compatibility of
the Mazur--Rubin arithmetic local constant and completed epsilon constants,
answering a question of Mazur and Rubin. 
The compatibility
 leads to new cases of the $p$-parity conjecture for Hilbert modular forms at supercuspidal primes $p$. We also formulate and prove an analogue of Rubin's conjecture 
 over ramified quadratic extensions of $\BQ_p$. Using it, we construct an integral $p$-adic $L$-function for anticyclotomic deformation of a CM elliptic curve at primes $p$ ramified in the CM field.

\end{abstract}

\maketitle
\tableofcontents

\section{Introduction}\label{s:Intro}
This paper aims to study the arithmetic of symplectic self-dual $p$-adic representations of $G_{\BQ_p}$ of rank two.
Such representations arise naturally in arithmetic geometry and automorphic representation theory, and deform in $p$-adic families. 
The Bloch--Kato conjecture suggests epsilon constants to be pivotal to their arithmetic. 
 Our main result reveals a new structure on the first Galois cohomology of such representations (a local sign decomposition): a functorial decomposition into free rank one Lagrangian submodules encoding Bloch--Kato subgroups via completed epsilon constants, mirroring a symplectic structure (see Theorem~\ref{thm, intro main}). It has local as well as global arithmetic applications (see \S\ref{ss:applications}).
\subsection{Overview} 
A fundamental conjecture of Deligne, Beilinson, and Bloch--Kato \cite{BK} posits: special values of the complex $L$-function $L(M,s)$ associated to 
a motive $M$ over (a finite extension of) $\BQ$ at integers 
 encode the arithmetic of $M$ in the guise of  
 Bloch--Kato Selmer groups $H^1_{\rm f}(\BQ,M_p)$ associated to  
 the $p$-adic realizations $M_p$. It is a far-reaching generalization of the Birch and Swinnerton-Dyer conjecture for elliptic curves. 

The Bloch--Kato Selmer group $H^1_{\rm f}(\BQ,M_p)$ is a $p$-adic avatar of motivic cohomology. 
 A foundational insight of Bloch and Kato: to introduce the Bloch--Kato subgroup $$H^1_{\rm f}(\BQ_p,V) \subset H^1(\BQ_p,V)$$ 
  for any $p$-adic de Rham  representation $V$ of $G_{\BQ_p}:=\Gal(\ov{\BQ}_p/\BQ_p)$ 
 via $p$-adic Hodge theory, and define $H^1_{\rm f}(\BQ,M_p)$ consisting of global 
Galois 
 cohomology classes whose localization at $p$ lives in $H^1_{\rm f}(\BQ_{p},M_p)$ and which are unramified at other primes. This notion encompasses 
  the Mordell--Weil group of rational points on an elliptic curve  
 as well as 
  the Chow groups of cycles on an algebraic variety  in terms of Galois representations. 
For example, if $V=V_pE$ is the $p$-adic Tate module of an elliptic curve $E_{/\BQ_p}$, 
then  
$H^1_{\rm{f}}(\BQ_p, V)=E(\BQ_p)\hat{\otimes}_{\BZ_p} \BQ_p$ and a discrete analogue of the Bloch--Kato Selmer group 
for $E_{/\BQ}$ is the usual Selmer group $\Sel_{p^\infty}(E)$, 
encoding $E(\BQ)$ and the Tate--Shafarevich group $\Sha(E)[p^\infty]$.

An enriching landscape arises from self-dual motives $M$. 
The associated Galois representations are also self-dual, arising naturally in arithmetic geometry and the realm of automorphic forms. For example, a  Tate twist of the middle dimensional \'etale cohomology of a variety over $\BQ$ (such as the Tate module of an elliptic curve) 
and Galois representations associated to a self-dual cuspidal automorphic representation are self-dual.  
The corresponding Deligne--Langlands epsilon constant $\varepsilon(M)$ (with a suitable normalization) equals $+1$ or $-1$,   conjecturally the sign of the functional equation of the $L$-function $L(M,s)$. It is a product of local epsilon constants $\varepsilon_v(M)\in \{\pm 1\}$. 
The global $\varepsilon$-constant $\varepsilon(M)$ encodes the arithmetic of $M$. 
For example, 
the $p$-parity conjecture posits 
(the Bloch--Kato conjecture mod $2$): $\varepsilon(M)$
determines the parity of the rank of the Bloch--Kato Selmer group $H^1_{\rm f}(\BQ,M_p)$ 
 for any prime $p$. 
Among self-dual representations, an innate dichotomy presents itself: being symplectic or orthogonal. In the latter case  
$\varepsilon(M)$ equals\footnote{More precisely, Saito \cite{Sa} proved $\varepsilon(M_p)=+1$ 
for all but finitely many primes $p$.} $+1$.
 The symplectic case is 
more of an enigma, 
 and akin to phenomena in symplectic geometry.

As for the automorphic realm, 
self-dual representations also form a distinctive landscape. They are enriched by the Gan--Gross--Prasad (GGP) conjectures \cite{GGP}, linking  central $L$-values (resp.~derivatives) of certain self-dual automorphic $L$-functions $L(\pi,s)$ with automorphic periods (resp.~arithmetic cycles) if the automorphic $\varepsilon$-constant 
$\varepsilon(\pi)=+1$ (resp.~$\varepsilon(\pi)=-1$). 
The GGP conjectures have witnessed striking progress over the last decade (cf.~\cites{Zh,BP,YZZ}). Their $p$-adic analogues are also emerging, with applications to the Bloch--Kato conjecture. 
For example, the $p$-adic Waldspurger formula \cites{BDP1,LZZ} is instrumental in the recent progress \cites{Sk,BST,BST1} towards the Birch and Swinnerton-Dyer conjecture.

A salient feature of $p$-adic Galois representations is that they deform in $p$-adic families: global and local deformation rings parametrize self-dual deformations of such a fixed mod $p$ representation. In the automorphic habitat eigenvarieties and completed cohomology parametrize $p$-adic families of automorphic representations. 
For a $p$-adic family $\CM$ of symplectic self-dual Galois representations, 
the variation of $\varepsilon$-constants $\varepsilon(M)$ at geometric 
specializations $M$ of $\CM$  is elemental to its arithmetic. For example, if $\varepsilon(M)=-1$ systematically over $\CM$, then the $p$-parity conjecture predicts 
existence of a $p$-adic family of non-torsion algebraic cycles, a systematic example of which arises from the arithmetic GGP cycles.

In this paper we explore local analogues of such phenomena for symplectic self-dual Galois representations: 
connections between local $\varepsilon$-constants and local arithmetic cycles. 
Specifically, for a $p$-adic de Rham representation $V$ of $G_{\BQ_p}$, an element of $H^1_{\rm{f}}(\BQ_p, V)$ is considered as a local rational point or arithmetic cycle associated to $V$ after Bloch and Kato. 
 For symplectic self-dual de Rham representations $V$ of rank $2$, 
we explore connections between the $\varepsilon$-constant $\varepsilon_p(V)\in\{\pm 1\}$ 
and the Bloch--Kato subgroup $H^1_{\rm{f}}(\BQ_p, V)$. 

Our main result reveals a new structure on the Galois cohomology of rather general symplectic self-dual $\CR$-representations $\CT$ of $G_{\BQ_p}$: the existence of a functorial decomposition 
\begin{equation}\label{eq:lsd}
H^1(\BQ_p,\CT) = H^1_{+}(\BQ_p,\CT)\oplus H^1_{-}(\BQ_{p},\CT)
\end{equation}
into free $\CR$-submodules $H^1_{\pm}(\BQ_p,\CT)$ of rank one which are Lagrangian respect to the symmetric Tate pairing on $H^1(\BQ_p,\CT)$, mirroring an intrinsic symplectic structure  (see Theorem~\ref{thm, intro main}). Here $\CR$ is any commutative Noetherian local complete $\BZ_p$-algebra with finite residue field, $\CT$ a symplectic self-dual $\CR$-representation of $G_{\BQ_{p}}$ of rank two with $H^{0}(\BQ_p,\ov{\CT})=0$ for 
$\ov{\CT}:=\CT\otimes \CR/\mathfrak{m}_{\CR}$, and functoriality refers to compatibility of 
\eqref{eq:lsd} with respect to variation of the pairs $(\CR,\CT)$. 
A key feature of the local sign decomposition \eqref{eq:lsd} is that for de Rham representations $T$ precisely one of the submodules $H^1_{\pm}(\BQ_{p},T)$ equals the Bloch--Kato subgroup $H^{1}_{\rm f}(\BQ_{p},T)$, the sign being determined by its completed $\varepsilon$-constant 
$\hat{\varepsilon}$: 
 an amalgamation of the    
$\varepsilon$-constant and Hodge--Tate weights. 
In light of its functoriality, 
the local sign decomposition 
encodes variation of Bloch--Kato subgroups in a $p$-adic family $\CT$ of symplectic self-dual de Rham representations $T$, such as over a universal local deformation ring. 
It is rooted in the principle of Kato's local $\varepsilon$-conjecture \cite{Ka93} and relies on aspects of the $p$-adic local Langlands correspondence \cite{ColpLL}.

The local sign decomposition has diverse arithmetic consequences. This includes compatibility of
the Mazur--Rubin arithmetic local constant and $\hat{\varepsilon}$-constants for rank two $p$-adic representations of $G_{\BQ_p}$ (see Theorem~\ref{thm, MR, intro}) 
and the $p$-parity conjecture (see Theorem~\ref{thm, intro parity}). 
The compatibility 
links Bloch--Kato subgroups of congruent de Rham representations with their $\hat{\varepsilon}$-constants, answering a question of Mazur and Rubin \cite{MR}. 
In turn we establish new cases of 
the $p$-parity conjecture for Hilbert modular newforms  for supercuspidal primes $p$. 
This gives a new perspective on the work of Mazur--Rubin \cite{MR} and Nekov\'a\v{r}  \cites{NekMRl,NekMRtame}: 
the local sign decomposition \eqref{eq:lsd} is a categorification - an \'etale cohomological incarnation - of the numerical results of \cites{MR,NekMRl,NekMRtame}. 
As a concrete application, it significantly generalizes their results.

A special case of the local sign decomposition for a family of induced Galois representations  
gives a new proof of Rubin's 1987 conjecture \cite{Ru} 
on the structure of anticyclotomic local units over the unramified quadratic extension of $\BQ_p$ (see Theorem~\ref{thm, Rubin's conjecture}), as well as a formulation and proof of its analogue over ramified quadratic extensions of $\BQ_p$ (see Theorem~\ref{thm, Rubin's conjecture ram}). While the former
was resolved \cite{BKO21} in 2021, the latter had not been conjectured. 
 The ramified case exhibits drastic variation of $\varepsilon$-constants and new phenomena abound, for the study of which Rubin-type conjecture lays foundation.

The local sign decomposition is interlaced with Iwasawa theory, which 
seeks to study the arithmetic of a motive $M$ as it deforms in a $p$-adic family $\CM$.  
One aims to formulate an analog of the Bloch--Kato conjecture for the family $\CM$ itself - an Iwasawa main conjecture.
It ought to account  for the variation of Bloch--Kato subgroups over geometric specializations of $\CM$. 
However,  
 these subgroups typically 
 do not interpolate over a $p$-adic family, such as a non-ordinary family. In contrast our signed subgroups $H^1_{\pm}(\BQ_p,T)$ do. 
This leads to an analogue of the Bloch--Kato Selmer group for rank two $p$-adic symplectic self-dual families 
$\CM$ {of $G_\BQ$-representations}
 and a framework for integral Iwasawa theory at primes of non-semistable reduction. 
As an illustration, we introduce an integral $p$-adic $L$-function for the anticyclotomic deformation of a CM elliptic curve for primes $p$ ramified in the CM field, the first example of its kind (see~Theorem~\ref{thm, pL}).

\subsection{Local sign decomposition}
\subsubsection{Set up} Let $p$ be an odd prime. Let $\ov{\BQ}_p$ be an algebraic closure of $\BQ_p$ and $G_{\BQ_p}=\Gal(\ov{\BQ}_p/\BQ_p)$.

Let $\CT$ be a continuous $\CR$-representation of $G_{\BQ_p}$, where $\CT$ is a free $\CR$-module of finite rank and 
 $\CR$ a commutative topological $\BZ_p$-algebra 
satisfying either of the following conditions:

\begin{enumerate} 
 \item[i)] $\CR$ is a $\mathrm{Jac}(\CR)$-adically complete Noetherian semi-local ring such that 
 $\CR/\mathrm{Jac}(\CR)$ is a finite ring.
 \item[ii)] $\CR$ is a finite product of finite extensions of $\BQ_p$. 
\end{enumerate} 
For example $\CR$ can be (the integer ring of) a finite field extension of $\BQ_p$, a finite field of characteristic $p$ or (a quotient of) the power series ring $\BZ_p[\![X_{1},\cdots,X_{r}]\!]$. 
For the first example we use the notation $(R,T)$ instead of $(\CR,\CT)$ to emphasize that it does {not} correspond to a family of representations.

An $\CR$-representation $\CT$ is symplectic self-dual if there exists a skew-symmetric and $G_{\BQ_p}$-equivariant pairing 
 $\langle \;, \; \rangle: \CT \times \CT \rightarrow \CR(1)$ such that  
 $\CT \cong {\rm Hom}_\CR(\CT, \CR(1))$ for $\CR(1):=\CR\otimes_{\BZ_{p}}\BZ_{p}(1)$. 
The arithmetic of the following Galois representations will be the focus of this paper. 
\begin{defn}
A symplectic self-dual pair $(\CR,\CT)$ is generic 
if it satisfies the following.
\begin{itemize}
\item[i)] $H^{0}(\BQ_p,\ov{\CT}_\fm)=0$ if $\CR$ is semi-local, $\fm$ any maximal ideal and $\ov{\CT}_\fm:=\CT\otimes_{\CR}\CR/\fm$. 
\item[ii)] $H^{0}(\BQ_p,T)=0$ if $R:=\CR$ a product of finite field extensions of $\BQ_p$ and $T:=\CT$.
\end{itemize}

\end{defn}
We refer to Examples \ref{ex, ssd} and  \ref{example, ssd two} in \S\ref{ss:ssd}.  
For any generic pair $(\CR, \CT)$, the Galois cohomology $H^1(\BQ_p, \CT)$ is a free $\CR$-module of rank $\rank_{\CR}\CT$. It commutes with base change for $\CR$ and the associated Tate pairing over $\CR$ is perfect.

\subsubsection*{de Rham representations} The seminal works of Fontaine \cite{Fo} and Bloch--Kato \cite{BK} suggest that de Rham representations are geometric and akin to arithmetic phenomena. 
They are defined over (the integer ring of) a finite field extension\footnote{In particular, they correspond to the notation $(R,T)$ as above and not a family of representations $(\CR,\CT)$.} of $\BQ_p$. 
We now describe some arithmetic invariants attached to symplectic self-dual de Rham representations of $G_{\BQ_p}$. 

For such a representation $V$, 
let $$\varepsilon_p(V)\in \{\pm 1\}$$ denote the $\varepsilon$-constant of the associated Weil--Deligne 
 representation (cf.~\S\ref{ss:Epsilon}).

For a de Rham 
representation $V$ of $G_{\BQ_p}$, let $\Gamma(V)\in \BQ^\times$ be the associated $\Gamma$-factor (cf.~\S\ref{ss:Gamma}).

It is a $p$-adic analogue of the archimedean $\Gamma$-factor. 
In the symplectic self-dual case we have $\Gamma(V) \in \{\pm1\}$ and 
 in the rank two case 
 $\Gamma(V)=(-1)^{k-1}$  
 where $(k,1-k)$ are the Hodge--Tate weights of $V$ with $k>0$ (cf.~Lemma~\ref{prop:gamma}). Here and throughout the paper, we normalize the Hodge--Tate weights so that the $p$-adic cyclotomic character has Hodge--Tate weight $1$. 
 \begin{defn}
 For a symplectic self-dual de Rham representation $V$ of $G_{\BQ_p}$, the completed $\varepsilon$-constant 
  is defined by 
 \[\hat{\varepsilon}_p(V):=\Gamma(V)\varepsilon_p(V) \in \{\pm 1\}.\]
 For a $G_{\BQ_p}$-stable lattice $T \subset V$, we define $\hat{\varepsilon}_p(T)$ as $\hat{\varepsilon}_p(V)$. 
  \end{defn}
  \noindent The notion of completed $\varepsilon$-constant will recur through the paper.
  
  A fundamental invariant associated to a de Rham representation $V$ of $G_{\BQ_p}$ is the Bloch--Kato subgroup 
  \[
H^1_{\rm{f}}(\BQ_p, V):=
 \mathrm{Ker}\left(H^1(\BQ_p, V) \rightarrow  H^1(\BQ_p, V\otimes_{\BQ_p} B_{\rm{crys}}) \right). 
\]
For a $G_{\BQ_p}$-stable lattice $T \subset V$, define $H^1_{\rm{f}}(\BQ_p, T)=\iota^{-1}(H^1_{\rm{f}}(\BQ_p, V))$ where  $\iota: H^1(\BQ_p, T) \rightarrow H^1(\BQ_p, V)$ is the natural map.

\subsubsection{Main result}\label{ss,mr}
We consider the arithmetic of symplectic self-dual representations of $G_{\BQ_p}$ of rank two. The central result of this paper is the following. 
\begin{thm}\label{thm, intro main}
Let $p$ be an odd prime. 
For all generic symplectic self-dual pairs  $(\CR, \CT)$ of $G_{\BQ_p}$-representations of rank two,
 there is a functorial decomposition 
\[
H^1(\BQ_p, \CT)=H^1_+(\BQ_p, \CT) \oplus H^1_-(\BQ_p, \CT)
\]
into free $\CR$-submodules of rank one satisfying the following. 
\begin{enumerate}
\item[1)] The $\CR$-submodules $H^1_\pm (\BQ_p, \CT) \subset H^1(\BQ_p,\CT)$  
are  Lagrangian\footnote{That is, they are maximal isotropic submodules (cf.~Definition~\ref{def:Lagrangian}). Note that we use the terminology `Lagrangian' though the Tate pairing is {\it not} skew symmetric.}
with respect to the symmetric Tate pairing on $H^1(\BQ_p,\CT)$. 
 \item[2)] For any morphism $\CT \rightarrow \CT'$ of $\CR$-representations, the canonical morphism induces $\CR$-module homomorphisms 
 \[
 H^1_\pm(\BQ_p, \CT)\rightarrow   H^1_\pm(\BQ_p, \CT').  
 \]
  \item[3)] For any continuous $\BZ_p$-algebra homomorphism $\CR \rightarrow \CR'$, the canonical morphism induces $\CR'$-module isomorphisms 
 \[
 H^1_\pm(\BQ_p, \CT)\otimes_\CR  \CR' \cong  H^1_\pm(\BQ_p, \CT\otimes_\CR  \CR').  
 \]
 In other words, the decomposition is compatible with base change. 
\item[4)] If $\CT\otimes_{\BZ_p}\BQ_p =: T\otimes_{\BZ_p} \BQ_p$ is de Rham, then  
\[
H^1_{-\hat{\varepsilon}_p(T)}(\BQ_p, T)=H^1_{\mathrm{f}}(\BQ_p, T), 
\]
where $H^1_{\pm 1}(\BQ_p,T):=H^1_{\pm}(\BQ_p,T)$. 
\end{enumerate}
Moreover, the properties 3)-4) characterize the decomposition uniquely. 
\end{thm}

The above local sign decomposition is an \'etale manifestation of symplectic self-duality. 
We emphasize its categorical nature: it is functorial and applies over rather general $\CR$, for example over finite fields of characteristic $p$ as well as universal local deformation rings (cf.~Example~\ref{example, ssd two}). 
In view of the de Rham property 4) and the base change property 3)  it encodes the variation of Bloch--Kato subgroups over any generic $p$-adic symplectic self-dual family of rank $2$ in terms of $\hat{\varepsilon}$-constants of its de Rham specializations. 

The uniqueness 
 is determined by
  the variation over all generic symplectic self-dual 
  pairs of rank two, 
  while the de Rham property applies over specific base: (integer ring of) finite extensions of $\BQ_p$.
  We refer to  \S\ref{ss, sgn sub} 
for some examples of signed submodules.

\begin{remark}\noindent
\begin{itemize}
\item[i)] If $\CR$ is local or a finite extension of $\BQ_p$, 
then $H^1(\BQ_p,\CT)$ apparently has zero or two Lagrangians.
 For de Rham representations $T$, there are exactly two Lagrangians since 
 $H^1_{\rm f}(\BQ_p,T)$ is a Lagrangian.
  By Theorem~\ref{thm, intro main}, the same holds for any generic $\CT$, 
i.e. 
 $H^1(\BQ_p,\CT)$ is a metabolic space. 
 \item[ii)] Theorem~\ref{thm, intro main} excludes non-generic representations and the prime $p=2$. 
 A variant of the local sign decomposition also holds for these excluded cases (see Theorem~\ref{thm, main3}).
Moreover, for odd primes $p$, non-generic representations seem to exhibit an integral decomposition mirroring the local sign decomposition (see~\S\ref{subsubsection, anomalous} for the de Rham and mod $p$ cases).  
\end{itemize}
\end{remark}

\subsubsection{The signed submodules and universal norms}
For  any pair $(\CR,\CT)$ as in \S\ref{ss,mr} and $\varepsilon \in \{\pm \}$, 
define 
\[
Z_\varepsilon(\CT):=\{v \in H^1(\BQ_p, \CT)\,|\, s(v) \in H^1_\mathrm{f}(\BQ_p, T_s) \; \text{for any de Rham specialization $T_s$ 
with $\hat{\varepsilon}_p(T_s)=-\varepsilon \cdot 1$} \}. 
\]

A universal norm is an element of $H^1(\BQ_p,\CT)$ whose specialization corresponding to any de Rham specialization $T_s$ lies in the Bloch--Kato subgroup $H^1_{\rm f}(\BQ_p,T_s)$, 
and 
 {the universal norm subgroup} $H^1_{\mathrm{f}}(\BQ_p, \CT)$   
 is the set of all universal norms. 

\begin{prop}\label{prop, decomp via Z}  
For any generic symplectic self-dual $G_{\BQ_p}$-representation $(\CR,\CT)$ of rank two, 
we have the inclusion $$H^1_{\pm}(\BQ_p,\CT)\subseteq Z_{\pm}(\CT).$$ 
It is an equality if $\CR$ contains Zariski dense specializations $s$ with $T_s$ de Rham and  
$\hat{\varepsilon}_p(T_s)=\mp 1$. 
Moreover, the equality for both signs holds if and only if 
$H^1_{\mathrm{f}}(\BQ_p,\CT)=\{0\}.$

\end{prop}

Along with the Zariski density of signed crystalline points on a local deformation ring (cf.~Theorem \ref{thm, density sgn crystalline}), we obtain the following.  
\begin{thm}\label{thm, universal norm zero}
Let $\BT_{\ov{\rho}}^{\Box}$ be the universal framed deformation of a generic mod $p$ representation $\ov{\rho}:G_{\BQ_p}\ra \GL_{2}(\ov{\BF}_{p})$. Then $H^1_{\rm f}(\BQ_p, \BT_{\ov{\rho}}^{\Box})=0$. 
\end{thm}

\subsection{Applications}\label{ss:applications}
\subsubsection{Mazur--Rubin arithmetic local constants, epsilon constants and the parity conjecture}
\subsubsection*{Backdrop}\noindent 
We first introduce the parity conjecture and Mazur--Rubin local constants  (cf.~\cites{MR,NekMRl}).

Let $M$ be a pure motive over a number field $F$ with coefficients in a number field $L_M$. 
For a prime $\mathfrak{p}$ of $L_M$ above $p$, let $V:=M_\mathfrak{p}$ be the $p$-adic representation of $G_{F}:=\Gal(\ov{\BQ}/F)$ associated to $M$ 
with coefficients in $L:=L_{M,\mathfrak{p}}$. 
Suppose that $M$ is symplectic self-dual, i.e. 
there exists an isomorphism 
$c : M\cong M^\vee(1)$ of pure motives satisfying the equality $-c=c^{\vee}(1) : M=(M^{\vee}(1))^{\vee}(1)\cong M^\vee(1)$.

Let $H^1_{\rm{f}}(F, V)$ be the Block--Kato Selmer group associated to $V$ 
 and put\footnote{Note that $H^0(F,V)=0$ since $M$ is pure and self-dual. We still include the general definition since it will be used for geometric Galois representations.} 
$$
\chi_{\rm{f}}(F, V)=\dim_{L} H^1_{\rm{f}}(F, V) - \dim_{L} H^0(F, V).
$$
Let $\varepsilon(V)\in\{\pm 1\}$ be the associated global epsilon constant. 
In fact $\varepsilon(V)$ is defined for any geometric $p$-adic representation $V$ of $G_F$  (cf.~Definition~\ref{def:epsilon}). 
By the Fontaine--Mazur conjecture, such a $V$ arises from a motive over $F$. 
Conjecturally, $\varepsilon(V)$ is the sign of the functional equation of the $L$-function $L(V,s)$.

We have the following mod $2$ analogues of the Bloch--Kato conjecture. 
\begin{conj}\label{conj, p-parity, intro}($p$-parity conjecture) 
For any geometric symplectic self-dual $p$-adic representation $V$ of $G_F$, 
we have
$$
\varepsilon(V)=(-1)^{\chi_{\rm{f}}(F, V)}.
$$
\end{conj}

\begin{conj}(relative $p$-parity conjecture)\label{conj, relative p-parity, intro}
Let $V$ and $V'$ be any two geometric symplectic self-dual $p$-adic representations of $G_F$. 
Then 
\begin{equation}\label{conj, intro relative p-parity II}
(-1)^{\chi_{\rm{f}}(F, V)-\chi_{\rm{f}}(F, V')}=\varepsilon(V)/\varepsilon(V').
\end{equation}
\end{conj}
The relative $p$-parity conjecture implies that the $p$-parity conjecture for $V$ and $V'$ are equivalent, allowing flexibility.
In practice, one seeks to study it for congruent or $p$-adic families of Galois 
representations.
Suppose that $V$ and $V'$ are residually symplectically isomorphic, i.e. there exist $G_F$-stable
symplectic self-dual lattices 
$T \subset V$ and $T'\subset V'$  such that  
$\overline{T}:=T/\mathfrak{p}T$ and $\overline{T}':=T'/\mathfrak{p}T'$ 
are symplectically isomorphic 
over a finite field $\BF$. 
Then Mazur and Rubin \cite{MR} proved that there exist local invariants $\delta_{v}(T,T')$ such that 
\begin{equation}\label{equation, intro relative selmer}
\chi_{\rm{f}}(F, V)-\chi_{\rm{f}}(F, V') \equiv  \sum_{v: \text{finite}} \delta_v(T, T')\quad  \mod 2.
\end{equation}
So the left-hand side of 
\eqref{conj, intro relative p-parity II}
 is a product of local signs called arithmetic local constants, defined as follows. 
For a finite place $v$, define 
Selmer structures $\mathscr{F}_v$ and $\mathscr{F}_v'$
on $H^1(F_v, \overline{T})$ by
$
\mathscr{F}_v:=\mathrm{Im}\left(H^1_{\rm{f}}(F_v, T)  \rightarrow H^1(F_v, \overline{T}) \right) 
$
and 
\[
\mathscr{F}_v':=\mathrm{Im}\left(H^1_{\rm{f}}(F_v, T')  \rightarrow H^1(F_v, \overline{T'}) \cong H^1(F_v, \overline{T}) \right).
\]

Then $\mathscr{F}_v$ and $\mathscr{F}_v'$ are Lagrangian subspaces,  
whose relative position is measured by the following.

\begin{defn}
The Mazur--Rubin arithmetic local constant $\delta_v(T,T')$ 
is defined by 
\[
\delta_v(T, T'):=\dim_{\BF} \left(\mathscr{F}_v/\mathscr{F}_v\cap \mathscr{F}_v' \right) \mod 2\quad \in \BZ/2\BZ.
\]
\end{defn}

 In light of Conjecture~\ref{conj, relative p-parity, intro} and \eqref{equation, intro relative selmer},  
Mazur and Rubin \cite[p.~580]{MR} predicted\footnote{They did not state a precise conjecture.} 
\begin{equation}\label{MR}\tag{MR}
\text{a link between the local constants $\delta_v(T,T')$ and 
$\varepsilon_v(V)/\varepsilon_v(V')$ for any finite place $v$.}
\end{equation} 

They proved some results  towards the question \eqref{MR} for elliptic curves \cite[\S5-6]{MR}, with applications to large Selmer ranks. (Subsequently, the Mazur--Rubin local constants found numerous arithmetic applications, such as \cites{MR1,MR2,KMR}.) 
A few years later, 
 Nekov\'a\v{r} proved that 
\begin{equation}\label{equation, MR vs DL, intro}
 \varepsilon_v(V)/\varepsilon_v(V')=(-1)^{\delta_v(T, T')}
\end{equation}
for $G_{F_v}$-representations $V,V'$ 
 which are residually symplectically isomorphic
if $v \nmid p$ ~\cite{NekMRl}  or
 if $v|p$ and $V, V'$ become Barsotti--Tate over a tamely ramified abelian extension ~\cite{NekMRtame}. 
This compatibility 
was a key to  his 
proof of the $p$-parity conjecture for 
certain abelian varieties 
 and 
Hilbert modular forms.
 
Now we describe our results towards the above problems. 

\subsubsection*{\it A compatibility of Mazur--Rubin and Deligne--Langlands local constants} 
\begin{thm}\label{thm, MR, intro}
Let $p$ be an odd prime. 
Let $T_1$ and $T_2$ be symplectic self-dual de Rham 
representations of $G_{\BQ_p}$ of rank $2$  
which are residually symplectically isomorphic. 
Put $V_i = T_{i} \otimes_{\BZ_p} \BQ_p$ for $i\in\{1,2\}$. 
Then 
\begin{equation}\label{equation, our MR vs DL, intro}
 \frac{\hat{\varepsilon}_p(V_1)}{\hat{\varepsilon}_p(V_2)}
 =(-1)^{\delta_p(T_1, T_2)}. 
\end{equation}
\end{thm}

This completely answers the Mazur--Rubin question \eqref{MR} for rank two representations of $G_{\BQ_p}$. 
It is a $p$-local analogue of the relative parity Conjecture~\ref{conj, relative p-parity, intro}. 
For generic representations, this compatibility is a simple consequence of the functoriality of the local sign decomposition.
So, the latter may be viewed as an absolute version of the Mazur--Rubin constant. {For non-generic representations, the compatibility is approached via this viewpoint (cf.~\S\ref{subsubsection, anomalous})}.

Note that the compatibility \eqref{equation, our MR vs DL, intro} differs from \eqref{equation, MR vs DL, intro}
  if 
the Hodge--Tate weights of $V$ and $V'$ differ. 
This was speculated by Nekov\'a\v{r}, but neither did he state a precise conjecture 
 \cite[p.~2]{NekMRtame} nor did the completed $\varepsilon$-constant appear in the prior results. 
Theorem~\ref{thm, MR, intro} is
 the first general compatibility result for non-semistable Galois representations or general Hodge--Tate weights. The prior results rely on explicit description of the corresponding $\varepsilon$-constants\footnote{which is so far only available for trianguline or mildly ramified representations} and $p$-adic tools such as Breuil--Kisin modules ~\cite{NekMRtame}.

\subsubsection*{Relative $p$-parity conjecture}

\begin{thm}\label{parity-family, intro}
Let $F$ be a number field and $p$ an odd prime totally split in $F$. 
Let $T_1$ and $T_2$ be geometric symplectic self-dual $\CO_L$-representations of $G_{F}$ of rank two 
which are residually symplectically isomorphic, where $L$ is a finite extension of $\BQ_p$. 
 Put $V_i = T_i \otimes_{\BZ_p} \BQ_p$ 
for $i\in\{1,2\}$. 
Then the relative $p$-parity Conjecture \ref{conj, relative p-parity, intro} is true for $V_1$ and $V_2$, i.e. 
\[
(-1)^{\chi_{\rm{f}}(F, V_1)-\chi_{\rm{f}}(F, V_2)}=\varepsilon(V_1)/\varepsilon(V_2). 
\]
\end{thm}

The prior works of Nekov\'a\v{r} \cites{NekMRsome,NekMRl,NekMRtame} and Pottharst--Xiao \cite{PX} primarily assume $V_i$ to be trianguline at primes above $p$, in which case the $\varepsilon$-constants admit an elementary description.

\subsubsection*{$p$-parity conjecture for Hilbert modular forms}

\begin{thm}\label{thm, intro parity}
Let $F$ be a totally real field and $p$ an odd prime totally split in $F$. 
Let $f$ 
be a Hilbert modular newform over $F$ of weight $k=\sum_{\sigma}k_{\sigma}\sigma\in 2\BZ_{>0}[\Hom(F,\ov{\BQ})]$ and trivial central character. 
Put $k_{0}=\max_{\sigma}k_{\sigma}$. 
Let $$\rho_f:G_F\to \Aut_{L}(V_f)$$
be an associated 
$p$-adic Galois representation.  
Suppose that 
the residual representation $\overline{T}_f$ is irreducible. 
Then the $p$-parity Conjecture \ref{conj, p-parity, intro} is true for 
$V_{f}(k_{0}/2)$, i.e. 
\[
\ord_{s=0}L(V_f(k_0/2),s)\equiv \dim_{L}H^1_{\mathrm{f}}(F, V_f(k_0/2)) \mod 2.
\]
\end{thm}

The above result allows $f$ to be supercuspidal at primes above $p$ and $k$ arbitrary, complementing prior results towards the $p$-parity conjecture over totally real fields (see also Proposition~\ref{prop, reduced-parity}).

 For elliptic newforms and non-Eisenstein primes $p$, i.e. primes for which the residual representation is ireducible,  the $p$-parity conjecture is a theorem of   Nekov\'a\v{r} \cite{NekMRsome}, the case of elliptic curves also due to Dockchitser brothers \cite{DD}.
If $F\neq \BQ$, then 
the parallel weight two  and the $p$-ordinary parallel weight cases are  
due to Nekov\'a\v{r}
\cites{NekMRsome,NekMRtame, NekSC}, 
the latter under some hypotheses.
The higher weight finite slope case
is due to Johansson--Newton \cite{JN}, 
which exploits the geometry of Hilbert modular eigenvariety.
Their work excludes supercuspidal primes 
and requires an additional condition if $[F:\BQ]$ is odd 
 \cite[Rem.~1.2.4]{JN}.

\subsubsection{Rubin's conjecture}\label{ss:RuC} While exploring anticyclotomic Iwasawa theory of CM elliptic curves at primes $p$ inert in the CM field, Rubin found a new structure on the module of twisted local units along the anticyclotomic $\BZ_p$-extension of the unramified quadratic extension of $\BQ_p$ and proposed it as a conjecture \cite{Ru}.  
Conditional on it, he envisioned an integral Iwasawa theory for this conjugate symplectic self-dual deformation, which is non-trianguline. 
\subsubsection*{The conjecture}
Let $p$ be an odd prime and 
$K$ the unramified quadratic extension of $\BQ_p$ 
with integer ring $\CO$.
Let $\SF$ be the Lubin--Tate formal group over $\CO$ 
for the uniformizing parameter $\pi:=-p$. 
For $n \geq -1$, write $K_n=K(\SF[\pi^{n+1}])$ 
and $K_\infty=\cup_n K_n$. 
The Galois action on the $\pi$-adic Tate module $T_\pi\SF$
defines a conjugate symplectic self-dual character  
$
\psi: \Gal(K_\infty/K) \xrightarrow{\simeq} \mathrm{Aut}(T_\pi\SF)\cong \CO^\times 
$ (cf.~Definition~\ref{def:csd}).

Let $K_\infty^{\rm ac} \subset K_\infty$ be the anticyclotomic $\BZ_p$-extension of $K$ and put $\Gamma=\Gal(K_\infty^{\rm ac}/K)$.
Define the Iwasawa algebra
$
\Lambda_2=\CO[\![\Gal(K_\infty/K)]\!] \text{ and } 
\Lambda=\CO[\![\Gamma]\!]$. 
 Denote by $\Lambda^{\iota}$ the $\Lambda$-representation of $G_K$ of rank one, on which the $G_K$-action is given by
the character $G_K\to \Gamma \subseteq \Lambda^{\times}$
sending $g\in G_K$ to the image $[g^{-1}]$ of $g^{-1}$ in $\Gamma\subseteq  \Lambda^{\times}$.  
  For a $\Lambda_2$-module $M$, 
put $M^*=M\otimes_\CO T_\pi\SF^{\otimes -1}
=\mathrm{Hom}_\CO(T_\pi\SF, M)$ and let the Galois group act  diagonally on it.

For $n\in \BZ_{\geq 0}$, let $U_n$ be the group of principal units in $K_n^\times$ and 
put ${U}_\infty^*=\varprojlim_n \,(U_n\otimes_{\BZ_p}\CO)^*$.
Let 
$
V_{\infty}^*
$
be the projection of $U_\infty^*$ to the anticicyclotomic $\BZ_p$-line, 
which is a free $\Lambda$-module of rank $2$.
For a finite order character  $\chi: \Gamma \rightarrow \overline{\BQ}_p^\times$, 
let $\delta_\chi: V_{\infty}^{*} \ra \ov{\BQ}_p$ 
be the Coates--Wiles logarithmic derivative as in \cite[\S 2.1.1]{BKO24} and 
let  $\Xi^{\pm}$ denote the set of such 
characters with conductor even or odd power of $p$ respectively.

An insight of Rubin is to introduce 
the $\Lambda$-submodules
\begin{align*}
 V^{*, \pm}_\infty:=\{v \in V_\infty^*\;|\;\delta_\chi(v)=0 \quad \text{for every $\chi \in \Xi^\mp$} \}.
\end{align*}
He showed that $V^{*, \pm}_\infty\cong \Lambda$ and 
$V^{*, +}_\infty\cap V^{*, -}_\infty=\{0\}$.
In 1987, Rubin proposed the following conjecture. 
 \begin{thm}\label{thm, Rubin's conjecture}(Rubin's conjecture) 
 As $\Lambda$-modules, we have
$$V_\infty^*=V^{*, +}_\infty\oplus V^{*, -}_\infty.$$
\end{thm}
The conjecture was resolved  
 in 2021 by us \cite{BKO21} based on global as well as local tools.

\subsubsection*{Rubin's conjecture and local sign decomposition}

\begin{prop}\label{prop, lsd and Rubin's conjecture}
For the Lubin--Tate character $\psi: G_{K} \ra \CO^\times$ as above, 
put $T_\psi=\CO(\psi)$ and 
$\tilde{\mathbb{T}}_\psi =\mathrm{Ind}_{K/\BQ_p}(T_\psi^{\otimes -1}(1) \otimes_{\CO} \Lambda^\iota)$.
Then $\tilde{\mathbb{T}}_\psi$ is a generic symplectic {self-dual 
 $\Lambda$-representation} of $G_{\BQ_p}$ and there is a canonical isomorphism 
 $H^1(\BQ_p, \tilde{\mathbb{T}}_\psi)\cong V_\infty^*$.
 Under this isomorphism, the local sign decomposition as in Theorem~\ref{thm, intro main}
 coincides with Rubin's decomposition as in Theorem~\ref{thm, Rubin's conjecture}.
\end{prop}

We refer to the proof of Corollary~\ref{cor, Rubin's conjecture refined} for the above connection. 

The local sign decomposition gives a new perspective on Rubin's conjecture and 
leads to refinements 
(see Corollary~\ref{cor, Rubin's conjecture refined}). 
In retrospect the following variation of $\varepsilon$-constants initiates the conjecture. 

\begin{fact}\label{lem, unram eps} 
For $\chi\in\Xi^\pm$, we have
$ 
\hat{\varepsilon}_p(\mathrm{Ind}_{K/\BQ_p} (\psi \chi))=\pm 1.
$
\end{fact}

Our approach also gives a proof of the generalization of 
Rubin's conjecture to conjugate symplectic self-dual Galois characters $\psi$ over $K$ which are ramified 
or   
with Hodge--Tate weights 
$(k+1,-k)$ for $k\in\BZ_{\geq 0}$ (cf.~\cite{Kazim}), which remained open 
 (see Theorem~\ref{thm, rubin decomposition} and  
 Corollary~\ref{cor, lsd ind special}). 
\subsubsection{An analogue of Rubin's conjecture over ramified quadratic extensions of $\BQ_p$} 
\subsubsection*{Set-up} Let $p$ be an odd prime. In this subsection $K$ denotes a ramified quadratic extension of $\BQ_p$ and $\omega_{K/\BQ_p}$ the associated quadratic character over $\BQ_p$. 
Let $\CO$ denote the integer ring of $K$.

Let $\psi$ be a conjugate symplectic self-dual Galois character over $K$ with Hodge--Tate weights $(1,0)$ and (additive) conductor one\footnote{There are exactly two such $\psi$ as explicitly constructed in ~\S\ref{subsubsection, ramified p-adic WD}.}.
It arises from a CM elliptic curve with CM by an order of an imaginary quadratic field in which $p$ ramifies. 
Let $L$ be the extension of $K$ generated by the image of $\psi$.
 Put $T_\psi=\CO_L(\psi)$ and let $\Lambda:=\CO_{L}[\![\Gamma]\!]$ denote the anticyclotomic Iwasawa algebra for $\Gamma:=\Gal(K_{\infty}^{\rm ac}/K)$.

We have the following notable anticyclotomic variation of local $\varepsilon$-constants at $p$ 
(cf.~Proposition~\ref{prop, ramified epsilon}). 

 \begin{fact}\label{lem, ram eps}
Let $\chi$ be a finite order anticyclotomic character over $K$ of order $p^n>1$.  
 Then for $a \in \BZ_p^\times$,  
\[
\hat{ \varepsilon}_p({\rm Ind}_{K/\BQ_p}T_\psi(\chi^a))= \left(\frac{a}{p}\right)\hat{\varepsilon}_p({\rm Ind}_{K/\BQ_p}T_\psi(\chi)), \quad 
  \hat{\varepsilon}_p({\rm Ind}_{K/\BQ_p}T_\psi(\chi))=\hat{\varepsilon}_p({\rm Ind}_{K/\BQ_p}T_\psi(\chi^{-p})). 
\]
\end{fact}
\noindent The above variation contrasts the unramified case, where the $\varepsilon$-constant only depends on the order of $\chi$ 
(cf.~Fact~\ref{lem, unram eps}).  Note also that 
${\rm Ind}_{K/\BQ_p}T_\psi(\chi)$ is symplectic self-dual since $\psi\chi$ is conjugate symplectic self-dual, and so the $\varepsilon$-constant ${\varepsilon}_p({\rm Ind}_{K/\BQ_p}T_\psi(\chi))$ is canonically defined, independent of the choice of additive character and embedding of its coefficient field into $\BC$, unlike the case of $\psi\chi$ (cf.~\S\ref{ss:eps-ssd}).

\begin{defn}
For $n\in \BZ_{\geq 1}$, define a partition of the set $\Xi_{n}$ of anticyclotomic characters over $K$ of order $p^n$ by 
 \[
 \Xi_{n}=\Xi_{\psi,n}^+ \cup \Xi_{\psi,n}^-=\{\chi \,|\, \hat{\varepsilon}_p({\rm Ind}_{K/\BQ_p}T_\psi(\chi^{-1}))=+1\}\cup
 \{\chi \,|\, \hat{\varepsilon}_p({\rm Ind}_{K/\BQ_p} T_\psi(\chi^{-1}))=-1\}.
 \] 
 \end{defn}
For a given $\psi$, the above partition is canonically defined in view of the preceding paragraph. 
Note that $|\Xi_{\psi,n}^+| = |\Xi_{\psi,n}^-|$ by Fact~\ref{lem, ram eps}.  For $\cdot\in\{\emptyset, +,-\}$, put  $\Xi^\cdot_\psi=\cup_n \Xi_{\psi,n}^\cdot$.  
 Consider the anticyclotomic deformation $$\mathbb{T}_\psi =T_\psi^{\otimes -1}(1) \otimes_{\CO} \Lambda^\iota $$
 where $\Lambda^\iota$ is as in \S\ref{ss:RuC}.   
  Note that $\BT_\psi$ is a $\Lambda$-conjugate symplectic self-dual $G_{K}$-representation of rank one. 
   Define
 \[
 H^1_{\pm}(K,\mathbb{T}_\psi)=\{ v \in H^1(K,\mathbb{T}_\psi)\;|\; \exp^*_\chi v=0 \;\text{if $\chi \in \Xi^\mp_\psi$}\}
 \]
 and likewise $H^1_{\pm}(K,\BT_\psi \otimes_{\BZ_p}\BQ_p)\subset H^1(K,\BT_\psi) \otimes_{\BZ_p}\BQ_p$, 
 where $\exp^*_\chi$ is the composite of the dual exponential map and 
\[
H^1(K, \mathbb{T}_{\phi}) =\varprojlim_m H^1(K_m, T_{\phi}^{\otimes -1}(1))  \to H^1(K_n, T_{\phi}^{\otimes -1}(1)) \to H^1(K_n, T_{\phi}^{\otimes -1}(1))^{\chi^{-1}} =H^1(K, T_{\phi}^{\otimes -1}(1)(\chi))
\]
where the first arrow is defined using Shapiro's lemma, and  the second is given by $x\mapsto \sum_{\sigma \in \Gal(K_n^{}/K) } x^{\sigma }\chi(\sigma)$.

 \subsubsection*{Result}
 \begin{thm}\label{thm, Rubin's conjecture ram} Let $p$ be an odd prime and $K$ a ramified quadratic extension of $\BQ_p$. Let $\psi$ be a conjugate symplectic self-dual Galois character over $K$ of conductor one with Hodge--Tate weights $(1,0)$ and $L$ the finite extension of $K$ generated by its image. 
 If $p=3$, suppose that either $K=\BQ_3(\sqrt{3})$ or $K=\BQ_3(\sqrt{-3})$ and $\psi \nequiv 1 \mod{\fm_{\CO_L}}$. 
 Then the following holds. 
 \noindent\begin{itemize}
 \item[i)] The $\Lambda$-modules $H^1_{\pm}(K, \mathbb{T}_\psi)$ are free of rank one.
 \item[ii)] The $\Lambda$-submodules $H^1_{\pm}(K,\BT_\psi) \subset H^1(K, \BT_\psi)$ are Lagrangian.  
 \item[iii)] As $\Lambda$-modules, we have 
 $$
 H^1(K, \mathbb{T}_\psi)= {H}^1_{+}(K, \mathbb{T}_\psi) \oplus {H}^1_{-}(K, \mathbb{T}_\psi).
   $$
 \end{itemize}
 Moreover, in the excluded case\footnote{In this case $H^1(K,\BT_\psi)$ is not a free $\Lambda$-module and a different formulation is necessary.} that $p=3$, $K=\BQ_3(\sqrt{-3})$ and $\psi \equiv 1 \mod{\fm_{\CO_L}}$, the assertions as in parts i)-iii) hold for the $\Lambda\otimes_{\BZ_p}\BQ_p$-modules $H^1_{\cdot}(K,\BT_{\psi}\otimes_{\BZ_p}\BQ_p)$.
 \end{thm}

The above result 
can be interpreted in terms of local units 
as in Rubin's conjecture 
(cf.~Theorem~\ref{thm, Rubin's conjecture}), yet  we present the above formulation  
to highlight 
the general viewpoint of this paper. 
It initiates the development of integral anticyclotomic Iwasawa theory of CM elliptic curves at primes $p$ ramified in the CM field. As an illustration, we 
construct an integral $p$-adic $L$-function 
(see~Theorem~\ref{thm, pL}). 

We refer to Theorems~\ref{thm, lsd ind} and
\ref{thm, rubin decomposition ramified} for $\psi$ with arbitrary conductor and Hodge--Tate weights. 

\subsubsection{A new example of a bounded $p$-adic $L$-function 
for a symplectic self-dual {deformation}: anticyclotomic CM $p$-adic $L$-function for ramified (non-semistable) primes
}\label{ss:pL-intro}

\subsubsection*{Backdrop}\label{ss:pL-bck} 
{The Hasse--Weil $L$-function is associated to any motive over a number field, a function of complex variable. An important problem is to find its $p$-adic analogue. 

Despite constructions of $p$-adic $L$-functions in various settings, the conjectural framework remains elusive in general, 
for example non-trianguline deformations. 
The conjectural framework is known for   
cyclotomic deformations of a motive for semistable $p$ (cf. ~\cites{PRbook, PR-semistable}) and for 
Panchishkin deformations (cf. ~\cite{Gr91}), the former being trianguline at $p$.
The framework relies on the existence and properties of a family of isotropic subspaces of an appropriate size 
in $p$-local Galois cohomology groups, and the properties are reflected in the corresponding $p$-adic $L$-functions. 
In both cases, such an isotropic subspace is defined as the Galois cohomology group of a $p$-local subrepresentation, more precisely, a sub 
$(\varphi,\Gamma)$-module over the Robba ring arising from a triangulation of the given family.   
Since the subrepresentation is in general defined only over the rigid analytic space associated to the family, 
the corresponding $p$-adic $L$-function typically has unbounded denominators\footnote{One may seek their bounded counterpart. As such, the signed Iwasawa theory initiated by Pollack \cite{Po} and the second-named author ~\cite{Ko0} sometimes works well.}. As for  Panchishkin deformations, their definition itself assumes the existence of such a subrepresentation over a formal scheme, and 
the conjectural $p$-adic $L$-function is expected to have bounded denominators.

For symplectic self-dual deformations, {the authors} expect that Lagrangian subspaces 
manifest the aforementioned subspaces,  
and the local sign decomposition is the first step towards a new class of $p$-adic $L$-functions. 
Note also that the $\varepsilon$-constant variation, which is mild for semistable or Panchishkin deformations but wild in general, is an obstruction to the existence of a bounded $p$-adic $L$-function: for geometric specializations with global $\varepsilon$-constant $-1$, the central $L$-values 
vanish.

As an application of the local sign decomposition, we construct a {\it bounded}
 $p$-adic $L$-function for 
the anticyclotomic $\BZ_p$-deformation of a CM elliptic curve for primes $p$ ramified in the CM field. 
This conjugate symplectic self-dual deformation is non-trianguline at $p$ and the variation of $\varepsilon$-constants is wild.

\subsubsection*{
The anticyclotomic CM deformation at ramified primes}

{ Iwasawa theory of CM elliptic curves began with the seminal work of Coates and Wiles \cite{CW}. On the analytic side, Katz \cite{Kz} pioneered the construction of pertinent $p$-adic $L$-functions. The theory exhibits peculiar features along the anticyclotomic direction}. 

Let $E$ be a CM  elliptic curve over $\BQ$
and $\CK$ the CM field. 
Suppose that $E$ has CM by $\CO_K$.
Let $p$ be an odd prime and $\CK_\infty^{\rm ac}$ the anticyclotomic $\BZ_p$-extension of $\CK$ 
with the $n$-th layer $\CK_n^{\rm ac}$. The behavior of 
the Mordell--Weil rank of $E(\CK_n^{\rm ac})$ is elemental to Iwasawa theory. 
For $p$ split in $\CK$,
 there exists a constant $c$ such that for $n\gg 0$ we have  
\begin{equation}\label{intro, eq, mw, split}
 \mathrm{rank}_{\CO_K}\,E(\CK_n^{\rm ac})=\frac{1-\varepsilon(E_{/\BQ})}{2}\cdot p^{n}+c.
\end{equation}
 In contrast if $p$ is inert in $K$, 
Greenberg~\cites{Gr83,Gr01} noticed that 
global epsilon constants vary along $\CK_\infty^{\rm ac}$ and   
\[
\mathrm{rank}_{\CO_K}\,E(\CK_n^{\rm ac})-\mathrm{rank}_{\BZ}\,E(\CK_{n-1}^{\rm ac})=
\varepsilon_n p^{n-1}(p-1)\]
 for $n\gg 0$, 
where $\varepsilon_n \in \{0,1\}$ depending on the parity of $n$. Rubin initiated the study of Iwasawa theory \cites{Ru,BKO21} related to this phenomenon. 
Finally, if $p$ is ramified in $\CK$, then global epsilon constants vary drastically 
(see~Fact~\ref{lem, ram eps, bis}) and   
 $$\mathrm{rank}_{\CO_K}\,E(\CK_n^{\rm ac})=\frac{p^{n}-1}{2}+c$$ 
 for $n\gg 0$ (cf.~Theorem~\ref{vGZK}).  
So new points of infinite order appear in basically every layer of $\CK_\infty^{\rm ac}$. 
Note that the variation of Mordell--Weil ranks in the above three cases is markedly different.

A basic problem is to develop anticyclotomic CM Iwasawa theory at ramified primes reflecting the above phenomena. As a first step, we construct an integral $p$-adic $L$-function.

\subsubsection*{Set-up} 
For an elliptic curve $E$ over $\BQ$ with CM by an order of 
an imaginary quadratic field $\mathcal{K}$, 
let $\phi$ be the associated Hecke character over $\CK$ so that $L({\bf \phi},s)=L(E_{/\BQ},s)$.

Let $p$ be an odd prime ramified in $\CK$ and $K$ the completion of $\CK$ at the prime above $p$. 
Fix embeddings $\iota_{\infty}:\overline{\BQ} \hookrightarrow  \BC$ and $\iota_{p}:\overline{\BQ}\hookrightarrow  \overline{\BQ}_p$.
Let ${\phi}$ also denote the associated $p$-adic Galois character over $\CK$. 
Put
$\Lambda=\mathcal{O}_K[\![\mathrm{Gal}(\mathcal{K}_\infty^{\rm ac}/\mathcal{K})]\!]$. 
As a reflection of 
the variation of local epsilon constants at $p$ (cf. Fact \ref{lem, ram eps}), 
we have the following notable variation of the global epsilon constants $\varepsilon(\phi\chi)$
along $\CK_\infty^{\rm ac}$.
 \begin{fact}\label{lem, ram eps, bis}
Let $\chi$ be a finite order anticyclotomic character over $\CK$ of order $p^n>1$.  
 Then for $a \in \BZ_p^\times$,  
 \[
 \varepsilon(\phi\chi^a)= \left(\frac{a}{p}\right)\varepsilon(\phi\chi), \quad 
  \varepsilon(\phi\chi)=\varepsilon(\phi\chi^{-p}). 
 \]
\end{fact}

\begin{defn}
 For $n\in \BZ_{\geq 1}$, define a partition of the set $\Xi_{n}$ of anticyclotomic characters over $\CK$ of order $p^n$ by 
 \[
 \Xi_{n}=\Xi_{\phi, n}^+ \cup \Xi_{\phi, n}^-=\{\chi \,|\, \varepsilon(\phi\chi^{-1})=+1\}\cup
 \{\chi \,|\, \varepsilon(\phi\chi^{-1})=-1\}.
 \] 
 \end{defn}

   Note that $|\Xi_{\phi, n}^+| = |\Xi_{\phi, n}^-|$ by Fact~\ref{lem, ram eps, bis}.   
  For $\cdot\in\{\emptyset, +,-\}$, put  $\Xi^\cdot_{\phi}=\cup_n \Xi_{\phi, n}^\cdot$. 
If $\chi\in\Xi^-_{\phi}$, then $L(\phi\chi^{-1},1)=0$ since $\varepsilon(\phi\chi)=-1$. A natural problem: 
 to construct a $p$-adic $L$-function interpolating algebraic part of the $L$-values $L(\phi\chi^{-1},1)$ for $\chi\in\Xi_{\phi}^+$.

Put 
\[\varepsilon=\varepsilon_\phi
:=
\sgn({\varepsilon}({\phi})/\varepsilon_p({\rm Ind}_{K/\BQ_p}\phi_p)), \qquad \BT_{\phi}=T_{{\phi}}^{\otimes -1}(1)\otimes_{\CO_K} 
\Lambda^{\iota}.
\]

\subsubsection*{Result}

\begin{thm}\label{thm, pL}
Let $E$ be a CM elliptic curve defined over $\BQ$ with CM by an order of an imaginary quadratic field $\CK$. Let $\phi$ be the associated Hecke character. 
Let $p\geq 3$ be a prime ramified in $\CK$ and $K$ the completion of $\CK$ at the prime above $p$.
If either $p\ge 5$ or $p=3$ and ${ \phi} \nequiv 1 \mod{\fm_{\CO_L}}$ on $G_K$ (Case I), then let $v_\varepsilon$ be a basis of the $\Lambda$-module $H^1_{\varepsilon}(K, \mathbb{T}_{ \phi}^{})$. In the excluded case that $p=3$ and ${ \phi} \equiv 1 \mod{\fm_{\CO_L}}$ on $G_K$ (Case II),
let $v_\varepsilon$ be a basis of the $\Lambda\otimes_{\BZ_p}\BQ_p $-module $H^1_{\varepsilon}(K, \mathbb{T}_{\boldsymbol \phi}^{})\otimes_{\BZ_p} \BQ_p$.

Then there exists a $p$-adic $L$-function 
$$\mathscr{L}_{p,v_\varepsilon}(E) \in 
\begin{cases}
 \Lambda \qquad &(\text{Case I})\\
\; \Lambda\otimes_{\BZ_p}\BQ_p  \qquad &(\text{Case II})
\end{cases}
$$ 
 such that  for $\chi \in \Xi^{+}_{\phi}$ we have 
 \[
 \chi^{-1}(\mathscr{L}_{p,v_\varepsilon}(E))=
 \frac{1}{\delta_{v_\varepsilon}^{\omega_E}(\chi)}\cdot \frac{L({\phi}\chi^{-1}, 1)}{\Omega_\infty}
 \]
where $\exp^*_\chi(v_\varepsilon)= \delta_{v_\varepsilon}^{\omega_E}(\chi)\cdot \omega_E$ for $\omega_E$ the N\'eron differential of a minimal Weierstrass model over $\CO_{\CK}$, and 
$\Omega_\infty \in \BC^\times$ is the CM period (cf. \S\ref{ss,erl}).
 \end{thm}

 In the above interpolation formula the non-vanishing of $\delta_{v_\varepsilon}^{\omega_E}(\chi)$ is a 
 consequence of Theorem~\ref{thm, Rubin's conjecture ram} 
 (cf.~Corollary~\ref{cor, nv dual exp}). 

\begin{remark}
In case II the integral elliptic unit main conjecture is subtle to formulate \cite{JLK} and so is an integral refinement of 
$\mathscr{L}_{p,v_\varepsilon}(E)$. 
The $\CO_\CK$-equivariant framework \cite{JLK} may shed some light on the latter.
\end{remark}

 \subsection{Idea of the proof: local sign decomposition} 
 The basic idea is to construct an involution $w_\CT$ on $H^1(\BQ_p, \CT)$ for generic symplectic self-dual pairs $(\CR,\CT)$ of rank two so that 
its eigenspace decomposition yields the local sign decomposition. 

We present two methods for the construction of $w_\CT$:

\begin{itemize}
\item Colmez's $w$-operator in the $p$-adic local Langlands correspondence, 
\item Kato's local epsilon conjecture at $p$. 
\end{itemize}

In the following, we outline the construction via Colmez's $w$-operator and refer to \S\ref{ss:lsd via Kato's eps} for the approach via Kato's local $\varepsilon$-conjecture. Here we only remark that the latter implies the existence of a symplectic structure on $H^1(\BQ_p,\CT)$. Both methods are useful in verifying the properties in Theorem~\ref{thm, intro main}.

For simplicity of notation, suppose that $\CR$ is a complete local Noetherian ring with finite residue field. 
Put 
$
\mathcal{E}_\CR:=\varprojlim_n (\CR/\mathfrak{m}_\CR^n[\![X]\!][\frac{1}{X}])$.  
Let 
$D:=D(\CT)$ be the \'etale $(\varphi, \Gamma)$-module over $\mathcal{E}_\CR$ associated to $(\CR,\CT)$ and 
$\psi$ be the left inverse of $\varphi$.
Then we have
$$
D^{\psi=1}\cong H^1_{\mathrm{Iw}}(\BQ_p,{\CT}), 
$$ where the latter denotes the cyclotomic Iwasawa cohomology \cite{CC}. 

Colmez \cite{ColpLL} defined an operator 
$
w: D^{\psi=0} \rightarrow D^{\psi=0, \iota}, 
$ by 
\[
w(x):=\lim_{n\rightarrow +\infty} \sum_{i \in \BZ_p^\times \!\!\mod p^n}
(1+X)^{1/i}\sigma_{-1/i^2}
(\varphi^n\psi^n(1+X)^{-i}x).
\]
where  $\sigma_a(1+X)=(1+X)^a$ and
$\iota$ is the $\Gamma$-twist arising from $\Gamma \rightarrow \Gamma, \gamma \mapsto \gamma^{-1}$.
The $w$-operator corresponds
to the action of 
$\begin{pmatrix}0 & 1
\\ 1 & 0\end{pmatrix}
\in \mathrm{GL}_2(\BQ_p)$ on the $R[\mathrm{GL}_2(\BQ_p)]$-module $\Pi(D)$ corresponding to $D$ via the $p$-adic local Langlands correspondence (cf.~\cite{ColpLL}). 
Let $\delta_D : \mathbb{Q}_p^{\times}\rightarrow \CR^{\times}$ be the character corresponding to $\chi_{\mathrm{cyc}}^{-1}\mathrm{det}_{\CR}\CT$ by local class field theory. We remark that $\delta_D=\mathbf{1}$ in our self-dual case. Using $D$ and $w$, Colmez defined an $R[\mathrm{GL}_2(\BQ_p)]$-module $D\boxtimes_{\delta_D}\mathbf{P}^1$, a sub $R[\mathrm{GL}_2(\BQ_p)]$-module\footnote{Note that the $\mathrm{GL}_2(\BQ_p)$-stability of $D^{\natural}\boxtimes_{\delta_D}\mathbf{P}^1$ is a deep result of Colmez.} $D^{\natural}\boxtimes_{\delta_D}\mathbf{P}^1$  and 
$$\Pi(D):=D\boxtimes_{\delta_D}\mathbf{P}^1/D^{\natural}\boxtimes_{\delta_D}\mathbf{P}^1.$$
Moreover, he proved that (under some assumptions) there exists a natural $\CR[\mathrm{GL}_2(\BQ_p)]$-linear isomorphism 
$$D^{\natural}\boxtimes_{\delta_D}\mathbf{P}^1\cong \Pi(D)^*\otimes \delta_D\circ \mathrm{det}.$$

By definition of $D^{\natural}\boxtimes_{\delta_D}\mathbf{P}^1$, one has a natural $\begin{pmatrix}\mathbb{Z}_p^{\times}& 0 \\ 
0& 1\end{pmatrix}\cong\Gamma$-linear map 
$$(D^{\natural}\boxtimes_{\delta_D}\mathbf{P}^1)^{\alpha_p=1}\rightarrow D^{\psi=1}\cong H^1_{\mathrm{Iw}}(\BQ_p,{T})$$
for $\alpha_p=\begin{pmatrix}p& 0 \\ 
0& 1\end{pmatrix}$. 
We show that this map induces
$$(D^{\natural}\boxtimes_{\delta_D}\mathbf{P}^1)^{\alpha_p=1}/(\gamma-1)\cong H^1_{\mathrm{Iw}}(\BQ_p,{\CT})/(\gamma-1)\cong H^1(\BQ_p,{\CT})$$
under our generic assumption (see Theorem \ref{c}). Note that the action of $\begin{pmatrix}0 & 1\\ 1 & 0\end{pmatrix}$ on $D^{\natural}\boxtimes_{\delta_D}\mathbf{P}^1$ naturally induces an action on the subquotient $(D^{\natural}\boxtimes_{\delta_D}\mathbf{P}^1)^{\alpha_p=1}/(\gamma-1)$ under our self-dual assumption. This induced action finally yields an involution $w_\CT$ on $H^1_{\mathrm{Iw}}(\BQ_p,{\CT})$ through the above isomorphim. Namely, the involution $w_\CT$ is defined as the Galois theoretic counterpart of the automorphic action of $\begin{pmatrix}0 & 1\\ 1 & 0\end{pmatrix}$ on $\Pi(D)^*$. 

A link between the $w$-operator and the $p$-adic Hodge theoretic invariants like  local $\varepsilon$-constant, 
$H^1_{\mathrm{f}}$ and Hodge--Tate weights, is due to the third-named author \cite{NaKato}.
 It is based on Colmez's $p$-adic Kirillov model and 
Emerton's local-global compatibility. The link leads to a connection between the local sign decomposition, Bloch--Kato subgroups and the completed $\varepsilon$-constant.

{We give two different constructions of the involution $w_\CT$ as both approaches seem to be important for  generalizations of the local sign decomposition. 

}

 \subsection{Vistas}
 \subsubsection*{Iwasawa theory} The proof of Rubin's conjecture \cite{BKO21} led to development of anticyclotomic CM Iwasawa theory at inert primes \cites{BKO24,BKOe,BKOY, BHKO,BKOd}. In the same vein we expect the ramified analogue of Rubin's conjecture as in Theorem~\ref{thm, Rubin's conjecture ram} to be ancillary to the case of ramified primes.
 In a follow-up \cite{BKNOa} we formulate and prove anticyclotomic CM main conjecture for such primes $p$ of additive reduction. 
 Moreover, we begin the study of $p$-adic Waldspurger-type formulas  
 for $\mathscr{L}_{p,v_{\varepsilon}}(E)$ (cf.~\cites{BDP1, LZZ,BKO24}).

 Applications of Theorem~\ref{thm, Rubin's conjecture ram} to non-CM anticyclotomic Iwasawa theory as in  \cites{BBL1,BBL2} will also be explored.   
In a different direction we plan to study the existence of a bounded $p$-adic $L$-function and Iwasawa theory 
for the universal symplectic self-dual deformation over deformation ring of a mod $p$ symplectic self-dual representation of $G_\BQ$. It will exploit the  local sign decomposition and 
 \cites{Nak1,CoWa}}.

 \subsubsection*{Higher rank case} A basic problem is to find analogue of the local sign decomposition for general 
  self-dual Galois representations. 
 In a sequel to this paper 
we will propose a new framework for the arithmetic of symplectic self-dual Galois representations of general rank and an analogue.

 \subsection{Contents} Part \ref{part I} focuses on the local sign decomposition. It is local, and considers $p$-adic representations of $G_{\BQ_p}$ of rank two. Section \ref{s:bck} recalls basic notions. The main result is formulated in section \ref{s:mr-lsd}, followed by its proof in section \ref{s:cst}. Then 
 section \ref{s: sgn submodules} describes  
 a characterization of the signed submodules appearing in the local sign decomposition. It partly relies on the Zariski density of crystalline points with a given sign on universal local deformation rings, which is shown in Appendix \ref{s:appendix}.
 
 Part \ref{part II} presents some local and global applications. Section~\ref{s:mr} establishes the compatibility of the Mazur--Rubin arithmetic local constants and $\hat{\varepsilon}$-constants, and some results towards the $p$-parity conjecture. Then a formulation and proof of Rubin-type conjectures over quadratic extensions of $\BQ_p$ is given in section~\ref{s:Rubin}. Based on it, section~\ref{s:pL} describes a construction of the $p$-adic $L$-function 
 $\mathscr{L}_{p,v_{\varepsilon}}(E)$.

 \subsubsection*{Acknowledgements} We thank 
 Matthias Flach, Wei He, 
  Vytautas Pa\v{s}k\=unas, Christopher Skinner, Ye Tian, Seidai Yasuda, Shou-Wu Zhang and Wei Zhang for  instructive conversations.

The authors would like to express their sincere gratitude to Karl Rubin for encouragement and his inspiring conjecture \cite{Ru}, the search of whose ramified analogue initiated this exploration.  

This work was partially supported by the NSF grant DMS 2302064, 
and the JSPS KAKENHI grants JP22H00096, JP22K03231, JP25K06935, JP21K13774 and JP25K06953. 
The first and the second-named authors also acknowledge support of 
the Max Planck Institute for Mathematics, Bonn. 

\part{Local sign decomposition}\label{part I}
\section{Background}\label{s:bck}

\subsection{Notation}\label{ss:notation}
Throughout, we fix a prime $p$. 

Let $K$ be a finite extension of $\BQ_p$ with residue field $\BF_q$. Let $K_0$ denote the maximal unramified subfield. For a seaprable closure $\ov{K}$ of $K$, let $K_0^{\rm un}$ denote the maximal unramified field extension of $K_0$ in $\ov{K}$.  For a field $F$, define $\ov{F}$ analogously and set $G_F :=\Gal(\ov{F}/F)$. 

Let $W_K$ be the Weil group of $K$, i.e. 
the subgroup of $G_K$ 
consisting of elements whose images in $G_{\BF_q}$  are integral powers of the geometric Frobenius. 
It is a topological group such 
that the inertia group $I_K$ is open, and the relative topology 
on $I_K$ from $W_K$ is the Krull topology. 
Let $v_K: W_K \rightarrow \BZ$ be the valuation that sends the geometric Frobenius to $1$. 
Let $\omega_1$ denote the unramified character of $W_K$ 
taking the value $q^{-1}$ on any geometric Frobenius.

We normalize the Hodge--Tate weights so that the $p$-adic cyclotomic character over $\BQ_p$ has Hodge--Tate weight $1$.

\subsection{Symplectic self-dual Galois representations}\label{ss:ssd}
Let $K$ be a finite extension of $\BQ_p$.

Let $T$ be a continuous $R$-representation\footnote{In the rest of the paper we use the uniform notation $(R,T)$ unlike the introduction.} of $G_{K}$, where $T$ is a free $R$-module and 
 $R$ a commutative topological $\BZ_p$-algebra 
satisfying either of the following: 

\begin{enumerate}
 \item[i)] $R$ is a $\mathrm{Jac}(R)$-adically complete Noetherian semi-local ring such that 
 $R/\mathrm{Jac}(R)$ is a finite ring.
 \item[ii)] $R$ is a finite product of finite extensions of $\BQ_p$. 
\end{enumerate}

\begin{defn}\label{def:ssd}
An $R$-representation $T$ is symplectic self-dual if there exists a skew-symmetric and $G_K$-equivariant pairing 
 $$\langle \;, \; \rangle: T \times T \rightarrow R(1)$$ such that  
 $T \cong {\rm Hom}_R(T, R(1))$ for $R(1):=R\otimes_{\BZ_{p}}\BZ_{p}(1)$. 
\end{defn}

\begin{example}\label{ex, ssd}
\noindent
\begin{itemize}
\item[a)] The $p$-adic Tate module $T=T_{p}E$ of an elliptic curve $E$ defined over $K$ is symplectic self-dual with $R=\BZ_p$. 
\item[b)]  For an $R$-representation $T$ of rank two,  it is symplectic self-dual  if and only if ${\rm det}_R T$ is the $p$-adic cyclotomic character $\chi_{\cyc}$ over $R$. 
\item[c)]  For
any  smooth proper variety $X$ defined over $K$ of dimension $2n+1$, the pair 
($R=\BQ_p$, $T=H^{2n+1}_{\text{ \rm\'et}}(X_{\overline{K}}, \BQ_p(n+1))$) is symplectic self-dual. 
\item[d)] Any cupsidal automorphic representation over a number field which is regular, algebraic and self-dual typically gives rise to  symplectic self-dual Galois representations. 
For example, let $f$ be an elliptic newform of weight $k\geq 2$ and trivial Neben-type, $L$ the completion of its Hecke field at a  prime $\fp$ above $p$ and 
$V_\fp(f)$ the associated  
$G_{\BQ}$-representation with Hodge--Tate weights  
 $(1-k,0)$.
Then the pair ($R=L$, $T=V_\fp(f)(\frac{k}{2}))$ is symplectic self-dual $G_{\BQ_p}$-representation. 
\item[e)] The induction of conjugate symplectic self-dual representations is self-dual. For example, let $K$ be
a quadratic extension of $\BQ_p$ 
 and $T$ a conjugate symplectic self-dual $R$-representation\footnote{This corresponds to replacing $T\times T$ in Definition \ref{def:ssd} with $T\times T^c$ for $T$ a representation of $G_K$, $c\in\Gal(K/\BQ_p)$ the non-trivial element and $T^c$ the inner conjugate.} of $G_K$. 
Let $K_\infty$ be the anticyclotomic $\BZ_p$-extension of $K$ and $\Lambda_{R}=R[\![\Gal(K_{\infty}/K)]\!]$ the Iwasawa algebra. 
Then the pair $(\Lambda_{R},\Ind_{K/\BQ_p}(T\otimes_{R} \Lambda_{R}))$ 
is symplectic self-dual. 
A concrete example arises from anticyclotomic $\BZ_p$-deformation of a conjugate symplectic self-dual $p$-adic Galois character over $K$. 
 \end{itemize}
\end{example}

  \subsubsection*{Rank two case}
This paper considers symplectic self-dual pairs $(R,T)$ of $G_{\BQ_p}$-representations of rank two. 
The above examples a), b), d) and e) are of rank two. 
We explicate the prototypical ones: 
\begin{example}\label{example, ssd two}\noindent
\begin{itemize}
\item[b1)] Any two dimensional Galois representation $\ov{\rho}: G_{\BQ_p} \ra \Aut_{\BF}(\ov{T})$ over a finite field $\BF$ of characteristic $p$ with $\det_{\BF}\ov{T}$ being the mod $p$ cyclotomic character is symplectic self-dual.
\item[b2)] Any rank two $G_{\BQ_p}$-representation $(R,T)$ with $\det_{R}T=\chi_\cyc$  is symplectic self-dual, where $R$ is the integer ring of a $p$-adic local field and $T\otimes_{\BZ_p}\BQ_p$ de Rham.
\item[b3)] For $\ov{\rho}$ as in b1), let $R_{\ov{\rho}}^{\Box}$ be the universal framed deformation ring with determinant $\chi_\cyc$ and $\BT_{\ov{\rho}}^{\Box}$ the universal deformation  (cf.~\cite[Cor.\ in \S 23,  \S 24]{M}). Then the pair $(R_{\ov{\rho}}^{\Box},\BT_{\ov{\rho}}^{\Box})$ is symplectic self-dual.
\end{itemize}
\end{example}

\subsection{Symplectic self-dual Weil--Deligne representations}

Let $K$ be a finite extension of $\BQ_p$.

A Weil--Deligne representation $W$ of $K$ 
is a pair $(\rho, N)$ where 
$\rho: W_K \rightarrow \mathrm{GL}(W)$  
is a finite dimensional continuous $\BC$-representation on $W$ and 
$N$ a nilpotent endomorphism of $W$ satisfying 
\[
\rho(g)N\rho(g)^{-1}=\omega_1(g) N \qquad (g\in W_K). 
\]
\begin{defn}
A Weil--Deligne representation $W$ 
is symplectic self-dual if there exists a non-degenerate skew-symmetric form $B$ on $W$ 
satisfying 
$$B(gm, gn)=B(m,n) \quad (g\in W_K), \qquad B(Nm, n)+B(m, Nn)=0.$$ 
More generally, $W$ is essentially self-dual of weight $t\in\BZ$ if 
$W \otimes \omega_1^{t/2}$ is self-dual. 

\end{defn}
If $\dim_{\BC}W=2$, then it is symplectic self-dual if and only if $\det W=1$.

\subsection{$p$-adic Galois representations and Weil--Deligne representations}\label{ss:Gal-WD}
Let $K$ be a finite extension of $\BQ_p$. 

For any prime $\ell\neq p$, by fixing an isomorphism $\overline{\BQ}_\ell \cong \BC$, 
we can associate a Weil--Deligne representation of $W_{K}$ over $\BC$ 
with an $\ell$-adic representation of $G_{K}$ simply by  Grothendieck's monodromy theorem.
Moreover, there is an equivalence of categories between $\ell$-adic representations of $G_{K}$ and $\ell$-integral Weil--Deligne representations of $W_{K}$ over $\ov{\BQ}_\ell$ (cf.~\cite[\S1]{Tay}). 
The case $\ell=p$ is more complicated.
Fontaine's $D_{\rm{pst}}$-functor as recalled below associates 
a Weil--Deligne representation to a $p$-adic de Rham representation. 
The association is far from being a bijection and in particular, it does not encode the Hodge filtration. 
Though we describe the properties of 
symplectic self-dual Galois and Weil--Deligne representations in parallel,  
we suggest the reader to keep an eye on the type of representation. 
As such, we use the letter $W$ to denote 
Weil--Deligne representations and $V$ for $p$-adic Galois representations.

Let $L$ be a finite extension of $\BQ_p$. 
Let $V$ be a $p$-adic representation of $G_K$ with coefficients in $L$. 
Fix an algebraic closure $\overline{K}$ of $K$. 

For a finite extension $K'/K$ in $\overline{K}$, put $$D_{\mathrm{pst}}^{K'}(V):=(\mathbf{B}_{\rm{st}} \otimes_{\BQ_p}V)^{G_{K'}}$$ 
where $\mathbf{B}_{\rm{st}}$ is the $p$-adic semi-stable period ring of Fontaine.  
This is a $(K'_0 \otimes_{\BQ_p} L)$-vector space 
with operators $\varphi$ and $N$ satisfying $N \circ \varphi=p \varphi \circ N$.
Furthermore, if $K'/K$ is normal, the natural action of $G_K$ on $\mathbf{B}_{\rm{st}} \otimes_{\BQ_p}V$ defines a $K'_0$-semiliner action of 
$\mathrm{Gal}(K'/K)$ on $D_{\mathrm{pst}}^{K'}(V)$ that commutes with $\varphi$ and $N$.  
In other words, $D_{\mathrm{pst}}^{K'}(V)$ is a $(\varphi, N, G_{K'/K})$-module. 
Put
\[ 
D_{\mathrm{pst}}(V):=\varinjlim_{K'/K} D_{\mathrm{pst}}^{K'}(V)
\]
where $K'$ runs through all finite extensions of $K$. 
Then $D_{\mathrm{pst}}(V)$ is $(K_0^{\mathrm{un}} \otimes_{\BQ_p} L)$-module 
with the operator $\varphi$ and $N$ satisfying $N \circ \varphi=p \varphi \circ N$, and 
a natural $K_0^{\mathrm{un}}$-semi-linear  action of the Weil group $W_K$ that commutes with $\varphi$ and $N$. 
If $V$ is de Rham, hence, potentially semi-stable, 
then we have $\dim_{K_0^{\mathrm{un}}} D_{\mathrm{pst}}(V)=\dim_{\BQ_p} V$.

For $\sigma \in W_K$ and $w \in D_{\mathrm{pst}}(V)$, 
define a $K_0^{\mathrm{un}}$-linear  action of $\sigma$ on $w$ by 
\[
\sigma \cdot w:=\sigma \varphi^{fv_K(\sigma)} w 
\]
where  $f=[K_0:\BQ_p]$ and $v_K$ is as in \S\ref{ss:notation}. 
Consider the map 
\[K_0^{\rm{un}} \otimes L \rightarrow \overline{K} \otimes L=\prod_{\tau: L \hookrightarrow \overline{K}} \overline{K}.\] 
Pick an embedding $\tau: L \hookrightarrow \ov{K}$ and an isomorphism 
$\iota: \overline{K} \cong \BC$. 
Then the above $K_0^{\mathrm{un}}$-linear action defines 
a continuous $\BC$-representation 
\[\rho_{\iota \circ \tau}: W_K \rightarrow \mathrm{Aut}_{\BC}(D_{\mathrm{pst}}(V)\otimes_{\iota\circ\tau} \BC).\]
\begin{defn}\label{def:WD}
A Weil--Deligne representation 
associated to a $p$-adic representation $V$ of $G_K$ and the embedding $\iota\circ \tau : L \hookrightarrow \BC$ is given by  
$$\mathrm{WD}(V)_{\iota\circ \tau}=(\rho_{\iota \circ \tau}, N).$$
 \end{defn}
We often simply denote it by $\mathrm{WD}(V)$. 

\begin{lem}
Let $V$ be a symplectic self-dual de Rham representation of $G_K$. 
Then the associated Weil--Deligne representation $\mathrm{WD}(V)_{\iota\circ \tau}$ 
is essentially symplectic of weight $-1$. 
\end{lem}
\begin{proof}
 By construction, we have a canonical isomorphism 
 $$\mathrm{WD}(V^*(1))\cong \mathrm{WD}(V^*) \otimes \omega_1
 \cong \mathrm{WD}(V)^* \otimes \omega_1 .$$  
 The assertion follows from this. 
\end{proof}

We have the following compatibility between the Weil--Deligne and induction functors.

\begin{prop}\label{prop, induced p-adic WD}
Let $K$ and $L$ be $p$-adic local fields, and $\ov{K}$ an algebraic closure. 
Let $K' \subset \ov{K}$ be a finite extension of $K$ and $V$ a $p$-adic representation of $G_{K'}$ over $L$. 
Then we have a canonical isomorphism
\[
\mathrm{WD}(\mathrm{Ind}_{K'/K}V)\cong \mathrm{Ind}_{K'/K} \mathrm{WD}(V). 
\]
\end{prop}
\begin{proof} This follows from definitions, whose details we leave to the interested reader. 
\end{proof}

\subsection{Epsilon constants}\label{ss:Epsilon}

Let $K$ be a local field. 

For a Weil--Deligne representation $W$ of $W_K$, 
a non-trivial additive character $\xi: K \rightarrow \BC^{\times}$ and 
a Haar measure $dx$, let\footnote{Throughout Part \ref{part I}, we use the notation $\varepsilon(\cdot)$ instead of $\varepsilon_p(\cdot)$ as in \S\ref{s:Intro} or part \ref{part II}.} 
\[
\varepsilon(W, \xi, dx) \in \BC^\times
\]
be the associated Deligne--Langlands epsilon constant (cf.~\cites{T,TLC}). 

If $W$ arises from a motive over a number field, 
then $\varepsilon(W,\xi,dx)$ is conjecturally related to 
local component of the conjectural functional equation of the associated $L$-function.

\subsubsection{Self-dual epsilon constants}\label{ss: self-dual epsilon}

A non-trivial additive character $\xi$ on $K$ defines the 
Pontryagin duality $$K \cong K^\vee, y \mapsto (x \mapsto \xi(xy)).$$ 
Let $dx_{\xi}$ be the Haar measure on $K$ which is self-dual with respect to $\xi$.

\begin{lem}\label{lem, epsilon=pm 1}
 If a Weil--Deligne representation $W$ is essentially symplectic self-dual of weight $-1$, then $$\varepsilon(W, \xi, dx_{\xi})\in\{\pm 1\}$$ and 
 it does not depend on the choice of $\xi$. 
\end{lem}
\begin{proof}
This is well-known. We include a proof just to confirm our normalization. 

Another additive character is of the form $\xi_{a}(x):=\xi(ax)$ for some $a \in K^\times$. 
Let $d$ be the dimension of $W$. 
Since $\det W=\omega_1^{\frac{d}{2}}$ and $dx_{\xi_a}=\omega_1(a)^{\frac{1}{2}}dx_{\xi}$, 
by (3.4.3) and (3.4.4) of \cite{T},  we have 
\[
\varepsilon(W, \xi_{a}, dx_{\xi_a})=\det W(a)\, \omega_1(a)^{-d} \varepsilon(W, \xi, dx_{\xi_a})
=\det W(a)\, \omega_1(a)^{-d}\, \omega_1(a)^{\frac{d}{2}}  \varepsilon(W, \xi, dx_\xi)
= \varepsilon(W, \xi, dx_\xi). 
\]
Also by (3.4.7) of \cite{T}, we have 
\[
\varepsilon(W, \xi, dx_\xi)\varepsilon(W^*(\omega_1), \xi_{-1}, dx_\xi)=1. 
\]
Since $W\cong W^*(\omega_1)$, it follows that $\varepsilon(W, \xi, dx_\xi)^2=1$. 
\end{proof}

From now, 
we denote the epsilon constant $\varepsilon(W, \xi, dx_\xi)$ simply by $\varepsilon(W)$. 

\subsubsection{Epsilon constants for 
symplectic self-dual de Rham representations}\label{ss:eps-ssd}
Let $K$ be a finite extension of $\BQ_p$ and 
$V$ a symplectic self-dual de Rham representation of $G_K$ with coefficients in a $p$-adic local field $L$. 
Then, the epsilon constant
\[
\varepsilon(\mathrm{WD}(V)_{\iota\circ \tau}) \in \{\pm 1\}
\]
does not depend on the choice of embeddings $\iota$ and $\tau$ as in Definition~\ref{def:WD} by Rohrlich~\cite{Ro}. 
This is the epsilon constant associated with $V$, which we
simply denote by $\varepsilon(V)$.  

\subsection{Gamma constants and completed epsilon constants}

Since the Weil--Deligne 
representation $\mathrm{WD}(V)$ does not encode  
the Hodge filtration of a de Rham representation $V$, neither does 
the epsilon constant $\varepsilon(V)$. 
On the other hand, if $V$ arises from a motive over a number field, the archimedean epsilon constants are 
determined by the associated Hodge realization. We introduce a modification of the epsilon constant $\varepsilon(V)$ by a $p$-adic counterpart of the archimedean epsilon constant. 

\subsubsection{Gamma constants}\label{ss:Gamma}
Kato and Perrin-Riou independently defined the following $\Gamma$-constant for a de Rham representation $V$ encoding its
$p$-adic Hodge theoretic invariants.  
\begin{defn}
For a de Rham $L$-representation $V$ of $G_{\BQ_p}$ over a $p$-adic local field $L$, the  $\Gamma$-constant is defined by 
\[
\Gamma(V):=\prod_{r \in \BZ} \Gamma^*(r)^{-\dim_L{\rm gr}^{-r}D_{{\rm dR}}(V)}
\]
where 
\[
\Gamma^*(r):=
\begin{cases}
 (r-1)! \quad &(r \geq 1) \\
 \frac{(-1)^r}{(-r)!} \quad &(r \leq 0).  
\end{cases}
\]
\end{defn}

If $V$ arises from a motive over a number field, then the $\Gamma$-constant is related to the associated archimedean epsilon constants.

\begin{lem}\label{prop:gamma}
Let $(k_0,k_1,\cdots,k_{n-1},k_n )$ be Hodge--Tate weights of 
a de Rham $L$-representation $V$ of $G_{\BQ_p}$ without counting multiplicity 
 so that $k_0 < k_1<\cdots<k_{n-1}<k_{n}$. 
If $V$ is symplectic self-dual, then $n$ is odd, 
 $k_i=1-k_{n-i}$ for $0\leq i \leq n$  and 
\begin{align*}
\Gamma(V)&=(-1)^{\sum_{i=0}^{\frac{n-1}{2}} k_{i} \dim_L{\rm gr}^{-k_{i}}D_{{\rm dR}}(V)}
=(-1)^{\sum_{i > 0} i \dim_L{\rm gr}^{i}D_{{\rm dR}}(V)}.
\end{align*}
\end{lem}
\begin{proof}
Note that the Hodge--Tate weights of $V^*(1)$ without counting multiplicity are the following tuple of integers in increasing order  
\[1-k_n<1-k_{n-1}<\cdots <1-k_1<1-k_0. 
\]
Hence $V^*(1)\cong V$ implies that $k_i=1-k_{n-i}$ for $0\leq i \leq n$
and $\dim_L{\rm gr}^{-k_i}D_{{\rm dR}}(V)=\dim_L{\rm gr}^{-k_{n-i}}D_{{\rm dR}}(V)$. 

In particular, $n$ is odd as otherwise $k_{n/2}\notin \BZ$. 
For $i<n/2$,  we have 
\[
\Gamma^*(k_i)\Gamma^*(k_{n-i})=\frac{(-1)^{k_{i}}}{(-k_{i})!} (k_{n-i}-1)! 
=(-1)^{k_{i}} 
\]
since $k_i \le 0 < k_{n-i}$. 
The assertion follows.
\end{proof}

\begin{cor}\label{cor, dR ssd two}
Let $V$ be a two dimensional symplectic self-dual de Rham representation of $G_{\BQ_p}$
such that $H^0(\mathbb{Q}_p, V)=\{0\}$. 
Let $(k,1-k)$ be the Hodge--Tate weights of $V$ for an integer $k \geq 1$.  
Then $\Gamma(V)=(-1)^{k-1}$ and 
$\mathrm{dim}_{L}(H^1_{\rm f}(\mathbb{Q}_p, V))=1$. 
\end{cor}
\begin{proof} 
The assertion on the $\Gamma$-constant is a special case of Lemma~\ref{prop:gamma}. 
As for the latter, note that
\[
\mathrm{dim}_{L}(\mathrm{Fil}^iD_{\mathrm{dR}}(V))=
\begin{cases} 0 & (k \le i),\\ 
1 & (-k <  i <  k), \\
2 & (i \le -k).\end{cases}
 \]
 In particular, we have 
 $\mathrm{dim}_{L}(H^1_{\rm f}(\mathbb{Q}_p, V))=\mathrm{dim}_{L}(D_{\mathrm{dR}}(V)/\mathrm{Fil}^0D_{\mathrm{dR}}(V))=1.$
\end{proof}

\subsubsection{Completed epsilon constants}
The following notion is ancillary to the paper. 
\begin{defn}\label{def:cep}
For a symplectic self-dual de Rham representation $V$ of $G_{\BQ_p}$, 
define the completed epsilon constant $\hat{\varepsilon}(V)$ by 
\[
\hat{\varepsilon}(V):=\Gamma(V)\varepsilon(V) \in \{\pm 1\}. 
\]
\end{defn}
The sign of $\hat{\varepsilon}(V)$ will often be referred to as the sign of $V$.

\subsection{Galois cohomology of generic representations} In this subsection we present basic properties of Galois cohomology of generic $p$-adic representations of $G_{\BQ_p}$.

\begin{defn}
A $G_{\mathbb{Q}_p}$-representation $(R,T)$ is generic if 
\begin{itemize}
\item[i)] $H^{i}(\BQ_p,\ov{T}_\fm)=0$ for $R$ 
a $\mathrm{Jac}(R)$-adically complete Noetherian semi-local ring with  
 $R/\mathrm{Jac}(R)$ finite, $i\in\{0,2\}$, $\fm$ any maximal ideal of $R$ and $\ov{T}_\fm:=T\otimes_{R}R/\fm$,
\item[ii)] $H^{i}(\BQ_p,T)=0$ for $R$ a product of finite field extensions of $\BQ_p$ and $i \in\{0,2\}$.
\end{itemize}
Otherwise, we say that $(R,T)$ is non-generic or anomalous\footnote{Indeed, for an elliptic curve $E$ over $\BQ_p$ 
with good reduction at $p$, the $p$-adic Tate module is anomalous if and only if $p$ is an anomalous prime.}.
\end{defn}

In the case i) the generic assumption implies that 
$\mathrm{rank}_{R/\fm}H^1(\mathbb{Q}_p,  \overline{T}_{\fm})=\rank_{R}T$, and the Tate dual $T^*(1)$ satisfies the same by Tate duality and the Euler--Poincare characteristic formula. 
In particular, for a symplectic self-dual representation, 
the generic condition is equivalent to 
$$
H^{0}(\mathbb{Q}_p,\ov{T}_{\fm})=0. 
$$

The main result of this subsection is the following.
\begin{prop}\label{a}
Let $(R,T)$ be a generic symplectic self-dual $G_{\BQ_p}$-representation of rank $n$.
\begin{itemize}
\item[1)] We have $H^i(\mathbb{Q}_p, T)=0$ for $i=0,2$. 
\item[2)]For any finite $R$-module $M$, the canonical map 
$$H^1(\mathbb{Q}_p, T)\otimes_RM\rightarrow H^1(\mathbb{Q}_p, T\otimes_RM)$$ is an 
isomorphism.
\item[3)] The $R$-module $H^1(\mathbb{Q}_p, T)$ is free of rank $n$. 
\item[4)]For any ideal $I$ of $R$, the canonical map 
$$H^1(\mathbb{Q}_p, T)\otimes_RR/I\rightarrow H^1(\mathbb{Q}_p, T\otimes_RR/I)$$ is 
an isomorphism.
\item[5)] Suppose that $p$ is odd. Let $\gamma$ be a topological generator of $\Gamma:=\mathrm{Gal}(\BQ_p(\zeta_{p^\infty})/\BQ_p)$ 
and put $H^1_{\rm Iw}(\BQ_p,T)=\varprojlim_{n} H^1(\BQ_p(\zeta_{p^n}),T)$. 
Then 
the canonical map 
$$H^1_{\mathrm{Iw}}(\mathbb{Q}_p, T)/(\gamma-1)\rightarrow H^1(\mathbb{Q}_p, T)$$ is 
an isomorphism. 
\item[6)]The Tate pairing, which we denote by 
$$H^1(\mathbb{Q}_p, T^*(1))\times H^1(\mathbb{Q}_p, T)\rightarrow R, \qquad (x, y)\mapsto \{x, y\}_T$$
 is a perfect pairing. It induces a symmetric pairing on $H^1(\BQ_p,T)$.

\end{itemize}
\end{prop}
\begin{proof}

Note that 1) is clear, and 4) is a special case of 2). 

As for 2), it is trivial for $M=R^m$. For a general $M$, take an exact sequence 
$$0\rightarrow N\rightarrow  R^m\rightarrow M\rightarrow 0$$
of $R$-modules. It induces the exact sequence
\begin{equation}\label{Appendix, equation, (i)}
H^1(\mathbb{Q}_p, T)\otimes_RN\rightarrow H^1(\mathbb{Q}_p, T)\otimes_R R^m\rightarrow H^1(\mathbb{Q}_p, T)\otimes_RM\rightarrow 0
\end{equation}
and 
$$0\rightarrow T\otimes_RN\rightarrow  T\otimes_RR^m\rightarrow T\otimes_RM\rightarrow 0.$$
In turn, the latter induces the exact sequence
\begin{equation}\label{Appendix, equation, (ii)}
0\rightarrow H^1(\mathbb{Q}_p, T\otimes_RN)\rightarrow H^1(\mathbb{Q}_p, T\otimes_R R^m)\rightarrow H^1(\mathbb{Q}_p, T\otimes_RM)\rightarrow 0
\end{equation}  
since $H^0(\mathbb{Q}_p, T\otimes_RM)=H^2(\mathbb{Q}_p, T\otimes_RN)=0$. By 
considering the canonical map from (\ref{Appendix, equation, (i)}) to  
$(\ref{Appendix, equation, (ii)})$, 
it follows that the canonical map 
$$H^1(\mathbb{Q}_p, T)\otimes_RM\rightarrow H^1(\mathbb{Q}_p, T\otimes_RM)$$
is surjective. Since $M$ is an arbitrary $R$-module, this surjectivity also holds for $N$. Then, considering the 
canonical map from  (\ref{Appendix, equation, (i)}) to  
$(\ref{Appendix, equation, (ii)})$ again, 
we deduce that 
the canonical map 
$$H^1(\mathbb{Q}_p, T)\otimes_RM\rightarrow H^1(\mathbb{Q}_p, T\otimes_RM)$$
is an isomorphism by a simple diagram chase. 

Henceforth, we suppose that 
$R$ is a $\mathrm{Jac}(R)$-adically complete Noetherian semi-local ring with  
 $R/\mathrm{Jac}(R)$ finite. A slight variant of the following arguments applies if $R$ is a product of finite field extensions of $\BQ_p$, which we leave to the interested reader. 

We prove 3) by induction on the length of $R$ as a $\BZ_p$-module. 
We may assume that $R$ is local. 
When $R=\mathbb{F}$ is a finite field of characteristic $p$, note that $\mathrm{dim}_{\mathbb{F}}H^1(\mathbb{Q}_p, \overline{T})=n$ by the Euler--Poincar\'e formula. 
For a local $R$ with residue field $\BF$, take a non-zero ideal $I$ of $R$ such that $\mathfrak{m}_R I=0$. 
The exact sequence
$$0\rightarrow I\rightarrow R\rightarrow R/I\rightarrow 0,$$ 
induces the following exact sequence 
$$0\rightarrow  H^1(\mathbb{Q}_p, \overline{T})\otimes_{\mathbb{F}}I\rightarrow H^1(\mathbb{Q}_p, T)\rightarrow H^1(\mathbb{Q}_p, T\otimes_RR/I)\rightarrow 0, $$
where we note that $T\otimes_RI=\overline{T}\otimes_{\mathbb{F}}I$ since $\mathfrak{m}_R I=0$. The latter implies that 
$$\mathrm{length}_{\BZ_p}(H^1(\mathbb{Q}_p, T))=n\cdot \mathrm{length}_{\BZ_p}(R)$$
since $\mathrm{length}_{\BZ_p}(H^1(\mathbb{Q}_p, T \otimes_{R} R/I))=n\cdot \mathrm{length}_{\BZ_p}(R/I)$
 by the induction hypothesis. Moreover, by the induction hypothesis, we have an isomorphism 
 $$(R/I)^n\isom H^1(\mathbb{Q}_p, T\otimes_RR/I)=H^1(\mathbb{Q}_p, T)\otimes_RR/I.$$
 Take an $R$-linear lift 
 $$f : R^n\rightarrow H^1(\mathbb{Q}_p, T)$$ of the above isomorphism. It is surjective by Nakayama's lemma.  
 Hence, considering the lengths of both sides, it follows that $f$ is an isomorphism.

We now consider 5). In view of the short exact sequence
$$0\rightarrow H^1_{\mathrm{Iw}}(\mathbb{Q}_p, T)/(\gamma-1)\rightarrow H^1(\mathbb{Q}_p, T)\rightarrow H^2_{\mathrm{Iw}}(\mathbb{Q}_p,T)^{\gamma=1}\rightarrow 0, $$
it suffices to show that $H^2_{\mathrm{Iw}}(\mathbb{Q}_p,T)^{\gamma=1}=0$, which is equivalent to 
$H^2_{\mathrm{Iw}}(\mathbb{Q}_p,T)^{\Delta=1}=0$,  
where $\Delta \subset \Gamma$ denotes the torsion subgroup. 
Let $\gamma_0$ be a topological generator of the free part $\Gamma_{\mathrm{free}}$ of $\Gamma$. 
By 1), we have
$$H^2_{\mathrm{Iw}}(\mathbb{Q}_p,T)^{\Delta=1}/(\gamma_0-1)\isom H^2_{\mathrm{Iw}}(\mathbb{Q}_p,T)/(\gamma-1)
\isom H^2(\mathbb{Q}_p, T)=0.$$ 
 Since $H^2_{\mathrm{Iw}}(\mathbb{Q}_p,T)^{\Delta=1}$ is a finitely generated (torsion) $\mathbb{Z}_p[\![\Gamma_{\mathrm{free}}]\!]\isom \mathbb{Z}_p[\![[\gamma_0]-1]\!]$-module, Nakayama's lemma implies 
 that $H^2_{\mathrm{Iw}}(\mathbb{Q}_p,T)^{\Delta=1}=0$. 

For 6), we may assume that $R$ is local with residue field $\BF$. It suffices to show that the $R$-linear map 
$$H^1(\mathbb{Q}_p, T^*(1))\rightarrow \mathrm{Hom}_R(H^1(\mathbb{Q}_p, T), R), \qquad x\mapsto [y\mapsto \{x, y\}_T]$$
is bijective. Since both the sides are free $R$-modules of the same rank, it suffices to show it is surjective. 
In view of Nakayama's lemma, it suffices to show that the surjectivity holds after tensoring with $\mathbb{F}$ over $R$. 
By 4), the latter map is given by  
$$H^1(\mathbb{Q}_p, \overline{T}^*(1))\rightarrow \mathrm{Hom}_R(H^1(\mathbb{Q}_p, \overline{T}), \mathbb{F}) : x\mapsto [y\mapsto \{x, y\}_{\overline{T}}],$$
which is bijective by Tate duality. 
The induced pairing on $H^1(\BQ_p,T)$ is symmetric by~\cite[Prop.~3.4]{BKOY}. 
\end{proof}

For a symplectic self-dual pair $(R,T)$ we set $$
H^i(\mathbb{Q}_p, T\otimes_{\mathbb{Z}_p}\mathbb{Q}_p):=H^i(\mathbb{Q}_p, T)\otimes_{\mathbb{Z}_p}\mathbb{Q}_p, \qquad R[1/p]:=R\otimes_{\mathbb{Z}_p}\mathbb{Q}_p.$$
In the anomalous case we have the following variant of Proposition~\ref{a}.

\begin{prop}\label{a3}
Let $(R,T)$ be a symplectic self-dual $G_{\BQ_p}$-representation of rank $n$. 
Assume that $$H^2(\mathbb{Q}_p, T\otimes_{\mathbb{Z}_p}\mathbb{Q}_p)=0.$$ 
\begin{itemize} 
\item[1)] We have $H^0(\mathbb{Q}_p, T\otimes_{\mathbb{Z}_p}\mathbb{Q}_p)=0$ and the $R[1/p]$-module $H^1(\mathbb{Q}_p, T\otimes_{\mathbb{Z}_p}\mathbb{Q}_p)$ is locally free of rank $n$. 
\item[2)]For any continuous $\mathbb{Z}_p$-algebra homomorphism $R\rightarrow R'$, we have $$H^i(\mathbb{Q}_p, (T\otimes_R  R')\otimes_{\mathbb{Z}_p}\mathbb{Q}_p)=0$$ for $i\in\{0,2\}$, and
 the canonical map 
$$H^1(\mathbb{Q}_p, T\otimes_{\mathbb{Z}_p}\mathbb{Q}_p)\otimes_{R[1/p]}R'[1/p]\rightarrow H^1(\mathbb{Q}_p, (T\otimes_RR')\otimes_{\mathbb{Z}_p}\mathbb{Q}_p)$$ is 
an isomorphism.
\item[3)]The Tate pairing 
$$H^1(\mathbb{Q}_p, T^*(1)\otimes_{\mathbb{Z}_p}\mathbb{Q}_p)\times H^1(\mathbb{Q}_p, T\otimes_{\mathbb{Z}_p}\mathbb{Q}_p)\rightarrow R[1/p], \qquad (x, y)\mapsto \{x, y\}_T$$
 is perfect, and it induces a symmetric pairing on $H^1(\BQ_p,T\otimes_{\mathbb{Z}_p}\mathbb{Q}_p)$.

\end{itemize}

\end{prop}
\begin{proof}
By the assumption, the perfect complex $C^{\bullet}(G_{\mathbb{Q}_p}, T)\otimes_{\mathbb{Z}_p}\mathbb{Q}_p$ is quasi-isomorphic to 
a bounded complex $P^{\bullet}$ of finite projective $R[1/p]$-modules such that $P^i=0$ for $i\geq 2$. In view of the Tate duality and the self-duality of $(R, T)$, the complex 
$C^{\bullet}(G_{\mathbb{Q}_p}, T)\otimes_{\mathbb{Z}_p}\mathbb{Q}_p$
is also quasi-isomorphic to $$Q^{\bullet}:=\mathrm{Hom}_R(P^{\bullet}, R[-2])$$ whose non-zero terms are in degree $[1, m]$ for some $m\geq1$. 

Since $H^i(Q^{\bullet})\isom H^i(\BQ_p,T\otimes_{\mathbb{Z}_p}\mathbb{Q}_p)=0$ for $i\geq 2$, a standard argument shows that 
$H^1(Q^{\bullet})$ is projective and $Q^{\bullet}$ is quasi-isomorphic to the complex $H^1(Q^{\bullet})[-1]$. Hence, the proposition follows from the base change 
properties, the Tate duality and the Euler--Poincar\'e formula for 
 $C^{\bullet}(G_{\mathbb{Q}_p}, T)\otimes_{\mathbb{Z}_p}\mathbb{Q}_p$. 
\end{proof}
To verify the assumption of the above proposition, a useful criterion is the following. 
\begin{lem}\label{lem, van}
Let $(R,T)$ be a symplectic self-dual $G_{\BQ_p}$-representation. 
 Then  
$H^2(\mathbb{Q}_p, T\otimes_{\mathbb{Z}_p}\mathbb{Q}_p)=0$ if and only if 
$H^0(\mathbb{Q}_p, T_s)=0$ for all the $p$-adic representations $T_s$ obtained from a base change of
$T\otimes \BQ_p$. 
\end{lem}
We leave proof of the above elementary lemma to the interested reader.

\subsection{Lagrangian submodules}

Let $R$ be a commutative Noetherian ring. 

Let $M$ be a finitely generated projective $R$-module with a  pairing $( \;,\;)$ 
that induces an $R$-module isomorphism $M\cong \mathrm{Hom}(M, R)$.

\begin{defn}\label{def:Lagrangian}
An $R$-submodule $N$ of $M$ is Lagrangian\footnote{We emphasize that the pairing $( \;,\; )$ is allowed to be symmetric.} 
if 
\begin{itemize}
\item[i)] $N$ and $M/N$ are projective, 
\item[ii)] $N$ is isotropic, and
\item[iii)]  the pairing $( \;,\;)$ 
induces $R$-module isomorphisms $N \cong \mathrm{Hom}(M/N, R)$. 
\end{itemize}
\end{defn}

We have the following preliminary regarding Langrangian submodules of a rank two module.

\begin{lem}\label{prop, lagrangian}
Let $R$ be 
a commutative Noetherian local ring such that $2 \in R^\times$.
 Let $M$ be a free $R$-module of rank two with a symmetric pairing $( \;,\;)$ 
 which induces an isomorphism $M\cong \mathrm{Hom}(M, R)$ of $R$-modules. 
Then the number of Lagrangian $R$-submodules of $M$ is zero or two. 
\end{lem}
\begin{proof}
Since $M$ is of rank two,  Lagrangian submodules
are of the form $Rv$ such that $( v, v )=0$, and $Rv$ and $M/Rv$ are both free.

Given such an element $v$. 
Put $k:=R/\mathfrak{m}_R$ and $\overline{M}:=M\otimes_R k$. 
Since $M/Rv$ is free, 
 $k\overline{v}$ is also a Lagrangian submodule of $\overline{M}$. 
Take $\overline{w} \in \overline{M}$ such that $$( \overline{v}, \overline{w})=1$$ 
and a lift $w \in M$ of $\overline{w}$. 
Then by Nakayama's lemma, $\{v, w\}$ is a basis of $M$ such that $( {v}, {w}) \in R^\times$. 
If necessary, replacing $w$ by 
$w-\frac{( w, w )}{2 ( v, w )}v$, we may assume that $( w, w)=0$. Then  $Rw$ defines another Lagrangian. 

Now pick a Lagrangian $N$. 
It is of the form $N=Rx$, write $$x=a v +bw$$ for $a, b\in R$. 
Then $( x, x)=0$ implies that $ab=0$. 
Since the pairing induces the isomorphism $N \cong  \mathrm{Hom}(M/N, R)$, at least one of $a$ or $b$ must be a unit of $R$. 
Hence, we have $N=Rv$ or $N=Rw$. 
\end{proof}

As for this paper, a key example of Lagrangians arises from generic symplectic self-dual $G_{\BQ_p}$-representations $T$, whose cohomology $H^1(\BQ_p,T)$ is endowed with a symmetric Tate pairing. 
\begin{cor}\label{cor, de Rham Lagrangian}
 Let $p$ be an odd prime and $T$ a generic symplectic self-dual de Rham representation of $G_{\BQ_p}$ of rank two.  
 Then $H^1(\BQ_p, T)$ has exactly two Lagrangian submodules with respect to the  Tate pairing, one of them being the Bloch--Kato subgroup $H^1_{\rm{f}}(\BQ_p, T)$. 
\end{cor}
\begin{proof}Since $H^1(\BQ_p, T)$ is free, one has $$H^1_{\rm{f}}(\BQ_p, T)=H^1(\BQ_p, T)\cap H^1_{\rm{f}}(\BQ_p, V)$$ and $H^1_{\rm{f}}(\BQ_p, V)
=H^1_{\rm{f}}(\BQ_p, T)[1/p]$. 
Therefore, we obtain 
$$H^1_{\rm{f}}(\BQ_p, T)^{\perp}=H^1(\BQ_p, T)\cap H^1_{\rm{f}}(\BQ_p, V)^{\perp}=H^1(\BQ_p, T)\cap H^1_{\rm{f}}(\BQ_p, V)=H^1_{\rm{f}}(\BQ_p, T).$$
\end{proof}

Our main theorem shows the existence of a Lagrangian for any generic symplectic self-dual representation of rank two, 
and so that of exactly two Lagrangians when $R$ is local. 
 We canonically label these submodules and for de Rham representations, it is linked with Bloch--Kato subgroups 
via completed epsilon constants.

\section{Main result} \label{s:mr-lsd} 
The central result of this paper is the following. 
\begin{thm}\label{thm, main}
Let $p$ be an odd prime. 
For all generic symplectic self-dual pairs  $(R, T)$ of $G_{\BQ_p}$-representations of rank two 
there is a functorial decomposition 
\[
H^1(\BQ_p, T)=H^1_+(\BQ_p, T) \oplus H^1_-(\BQ_p, T)
\]
into free $R$-submodules of rank one with the following properties. 
\begin{enumerate}
\item[1)] The $R$-submodules $H^1_\pm (\BQ_p, T) \subset H^1(\BQ_p,T)$ 
are  Lagrangian 
with respect to the symmetric Tate pairing on $H^1(\BQ_p, T)$. 
 \item[2)] For any morphism $T \rightarrow T'$ of $R$-representations, the canonical morphism induces $R$-module homomorphisms 
 \[
 H^1_\pm(\BQ_p, T)\rightarrow   H^1_\pm(\BQ_p, T').  
 \]
  \item[3)] For any continuous $\BZ_p$-algebra homomorphism $R \rightarrow R'$, the canonical morphism induces $R'$-module isomorphisms 
 \[
 H^1_\pm(\BQ_p, T)\otimes_R  R' \cong  H^1_\pm(\BQ_p, T\otimes_R  R').  
 \]
 In other words the decomposition is compatible with base change. 
\item[4)] If $T\otimes_{\BZ_p} \BQ_p$ is de Rham, then  
\[
H^1_{-\hat{\varepsilon}(T)}(\BQ_p, T)=H^1_{\mathrm{f}}(\BQ_p, T), 
\]
where $H^1_{\pm 1}(\BQ_p,T):=H^1_{\pm}(\BQ_p,T)$. 
\end{enumerate}
Moreover, the {properties} 2)-4) characterize the decomposition uniquely. 
\end{thm}

The following is an immediate application of Theorem~\ref{thm, main}, 
 which will be used in our construction of $p$-adic $L$-function at additive primes (cf.~Theorem~\ref{thm, ram pL}). 
\begin{cor}\label{cor, nv dual exp}
Let $(R, T)$ be a generic symplectic self-dual $G_{\BQ_p}$-representation of rank two. 
Let $v_{\epsilon}$ be a basis of $H^1_{\epsilon}(\BQ_p, T)$ 
for $\epsilon \in \{\pm 1\}$. 
Suppose that $V_s$ is a de Rham specialization of $T$ with sign 
$\hat{\varepsilon}(V_s)=\epsilon$. 
Then we have $$\exp^*_{V_s}(v_\epsilon)\not=0$$ 
for $\exp^*_{V_s}$ the dual exponential map.
\end{cor}

\begin{remark}
The above non-vanishing is a generalization of \cite[Cor.~2.2]{BKO21}, which is a consequence of the proof of Rubin's conjecture. 
\end{remark}

If we use the Lagrangian propery 1) in addition to the properties 3) and 4), 
the uniqueness of the local sign decomposition is relatively easily proven. 

\begin{prop}\label{prop, uniqueness}
 The local sign decomposition as in Theorem \ref{thm, main} is uniquely determined 
by the properties 1), 3) and 4). 
\end{prop}
\begin{proof}
In light of the base change property 3) of the local sign decomposition and completeness of $R$, we may assume that  $R$ is an Artinian semi-local ring. Since $R$ is then a finite product of Artinian rings, we may further assume that $R$ is local. 
In light of the base change property 3) we finally reduce to the case that $T$ is the universal framed deformation
$\mathbb{T}:=\mathbb{T}_{\ov{\rho}}^{\Box}$ of a residual representation $\ov{\rho}$ whose determinant is the cyclotomic character (see b3) of Example~\ref{example, ssd two}).

Suppose that $H^1(\BQ_p,\mathbb{T})$ has two decompositions 
\[
H^1(\BQ_p,\mathbb{T})=Rv_+ \oplus Rv_-=Rw_+ \oplus Rw_-
\]
satisfying the properties 1), 3) and 4) of Theorem \ref{thm, main}.
Since there are only two Lagrangian submodules by~Lemma \ref{prop, lagrangian}, 
it follows that $Rv_+=Rw_+$ or $Rv_+=Rw_-$. 

Suppose that $Rv_+=Rw_-$ and then $Rv_-=Rw_+$. 
Let $s$ be a crystalline specialization of $\BT$ (cf.~Theorem \ref{thm, Paskunas}), and 
$T_s$ denotes the specialization of 
 $\mathbb{T}$ at $s$.  
 Put $\epsilon=\hat{\varepsilon}(T_s)$. Then both 
 $R\,s(v_{-\epsilon})$ and $R\,s(w_{-\epsilon})$ equal  
 $H^1_{\rm{f}}(\BQ_p, T_s)$ by the property 4). 
 Therefore, we obtain $R\,s(v_+)=R\,s(v_{-})$, a contradiction. 
 \end{proof}
 
 For either the anomalous case or the prime $p=2$, we have the following variant of the local sign decomposition.
Recall that by definition, 
$H^i(\mathbb{Q}_p, T\otimes_{\mathbb{Z}_p}\mathbb{Q}_p):=H^i(\mathbb{Q}_p, T)\otimes_{\mathbb{Z}_p}\mathbb{Q}_p$ for $i=0, 1, 2$. 
 
 \begin{thm}\label{thm, main3}
Let $p$ be any prime. 
Let $(R, T)$ be a symplectic self-dual pair of $G_{\BQ_p}$-representations of rank two. Assume that 
 $$H^2(\mathbb{Q}_p, T\otimes_{\mathbb{Z}_p}\mathbb{Q}_p)=0.$$

Then  
there is a functorial decomposition 
\[
H^1(\BQ_p, T\otimes_{\mathbb{Z}_p}\mathbb{Q}_p)=H^1_+(\BQ_p, T\otimes_{\mathbb{Z}_p}\mathbb{Q}_p) \oplus H^1_-(\BQ_p, T\otimes_{\mathbb{Z}_p}\mathbb{Q}_p)
\]
into locally free $R[1/p]$-submodules of rank one with the following properties. 
\begin{enumerate}
\item[1)] The $R[1/p]$-submodules $H^1_\pm (\BQ_p, T\otimes_{\mathbb{Z}_p}\mathbb{Q}_p)\subset H^1(\BQ_p,T\otimes_{\mathbb{Z}_p}\mathbb{Q}_p)$ 
are  Lagrangian 
with respect to the symmetric Tate pairing on $H^1(\BQ_p, T\otimes_{\mathbb{Z}_p}\mathbb{Q}_p)$.
  \item[2)] For any continuous $\BZ_p$-algebra homomorphism $R \rightarrow R'$,  the canonical morphism induces $R'[1/p]$-module isomorphisms 
 \[
 H^1_\pm(\BQ_p, T\otimes_{\mathbb{Z}_p}\mathbb{Q}_p)\otimes_{R[1/p]}  R'[1/p]\cong  H^1_\pm(\BQ_p, (T\otimes_R  R')\otimes_{\mathbb{Z}_p}\mathbb{Q}_p).  
 \]
 In other words the decomposition is compatible with base change. 
\item[3)] If $(L, V)$ is any de Rham specialization of $(R, T)$, then 
\[
H^1_{-\hat{\varepsilon}(V)}(\BQ_p, V)=H^1_{\mathrm{f}}(\BQ_p, V), 
\]
\end{enumerate} 
\end{thm}

\section{Construction of the local sign decomposition}\label{s:cst} 
The main idea is to construct an involution $w_T$ on the Galois cohomology $H^1 (\BQ_p,T)$ whose eigenspace decomposition leads to the local sign decomposition as in Theorem~\ref{thm, main}.

In this section, we introduce two methods for the construction of $w_T$. 
The first is based on Kato's local epsilon conjecture for rank two representations of $G_{\BQ_p}$ as established by the third-named author and Jacinto ~\cites{NaKato,J}. For generic symplectic self-dual representations, 
the conjecture 
leads to a symplectic pairing on $H^1 (\BQ_p, T)$, 
whose comparison with the symmetric Tate pairing yields the involution $w_T$. 
Since it relies on the abstract symplectic pairing, this definition of $w_T$ is not constructive, 
while our second method is: 
we show that $w_T$ is essentially 
Colmez's operator $w_*$ in $(\varphi, \Gamma)$-theory.  
Both constructions are elemental to
properties of the local sign decomposition.

Throughout, we fix a prime $p$.

\subsection{Kato's local epsilon conjecture}
This subsection describes the eponymous conjecture and some results towards it following \cites{Ka93,Narank1, NaKato}.

\subsubsection{Set-up}\label{ss:Kato s-u}
Let $\ov{\BQ}_p$ be an algebraic closure of $\BQ_p$ and $G_{\BQ_p}:=\Gal(\ov{\BQ}_p/\BQ_p)$. 
Let $W_{\BQ_p}\subset G_{\BQ_p}$ be the Weil group of $\BQ_p$ and $$\rec_{\BQ_p}: \BQ_p^\times \simeq W_{\BQ_p}^{\rm ab}$$ be the reciprocity map of local class field theory normalized so that $\rec_{\BQ_p}(p)$ is a lift of the geometric Frobenius ${\rm Fr}_p \in G_{\BF_p}$. Let $I_p \subset W_{\BQ_p}$ be the inertia subgroup. 

Let $T$ be a continuous $R$-representation of $G_{\BQ_p}$ of rank $r_T$, where  
 $R$ is a commutative topological $\BZ_p$-algebra 
satisfying either of the following:

\begin{enumerate}
 \item[i)] $R$ is a $\mathrm{Jac}(R)$-adically complete Noetherian semi-local ring such that 
 $R/\mathrm{Jac}(R)$ is a finite ring.
 \item[ii)] $R$ is a finite product of finite extensions of $\BQ_p$. 
\end{enumerate}
In the second case we often denote $R$ by $L$. Let $T^*(1) := \Hom_{R}(T,R)(1)$ denote the Tate dual of $T$.

For an $R$-representation $T$ of $G_{\mathbb{Q}_p}$, 
denote the complex of continuous cochains of $G_{\mathbb{Q}_p}$ with values in $T$ by $\mathrm{C}_{\mathrm{cont}}^{\bullet}(G_{\mathbb{Q}_p}, T)$, i.e.
 $$\mathrm{C}_{\mathrm{cont}}^i(G_{\mathbb{Q}_p}, T):=\{c:G_{\mathbb{Q}_p}^{i}\rightarrow T | \text{ continuous maps}\}$$ for $i \in \BZ_{\geqq 0}$ with the usual boundary maps. 
Applying the determinant functor of Knudsen--Mumford \cite{KM} to the perfect complex 
$\mathrm{C}_{\mathrm{cont}}^{\bullet}(G_{\mathbb{Q}_p}, T)$
 of $R$-modules 
 leads to the graded invertible $R$-module 
$$\Delta_{R,1}(T):=\mathrm{Det}_R(\mathrm{C}_{\mathrm{cont}}^{\bullet}(G_{\mathbb{Q}_p},T)).$$
It is is of degree $-r_T$ 
by the Euler--Poincar\'e formula. 

For $a\in R^{\times}$ (resp.~$a\in \mathcal{O}_L^{\times}$ if $R=L$), put 
$$R_a:=\{x\in W(\overline{\mathbb{F}}_p)\hat{\otimes}_{\mathbb{Z}_p}R \mid (\varphi\otimes \mathrm{id}_R)(x)=(1\otimes a)\cdot x\},$$ which is an invertible $R$-module. For $T$ as above, we often regard $\mathrm{det}_R(T)$ as a continuous homomorphism $\mathrm{det}_R(T):G_{\mathbb{Q}_p}^{\mathrm{ab}}\rightarrow R^{\times}$. 
Define a constant 
$$a(T):=\mathrm{det}_R(T)(\mathrm{rec}_{\mathbb{Q}_p}(p))\in R^{\times},$$ and 
a graded invertible $R$-module 
$$\Delta_{R,2}(T):=
(\mathrm{det}_R(T)\otimes_RR_{a(T)}, r_T) . $$

For a graded invertible $R$-module $\mathcal{P}=(P, r)$, put $\mathcal{P}^{\vee}:=(P^*, r)$. For a finite projective $R$-module $M$, 
define a trivialization (Zariski locally) by
$$\mathrm{det}_R(M)\otimes_R\mathrm{det}_R(M^*)\isom R : (e_1\wedge \cdots \wedge e_n)\otimes (e_n^*\wedge \cdots \wedge e_1^*)\mapsto 1$$ 
for $\{e_1, \cdots, e_{n}\}$ a (local) basis of $M$ and $\{e_1^*, \cdots, e_n^*\}$ its dual basis, by which we identify $\mathrm{det}_R(M^*)\isom (\mathrm{det}_R(M))^*$. 

\begin{defn}
For an $R$-representation $T$ of $G_{\mathbb{Q}_p}$, 
the local fundamental line $\Delta_{R}(T)$ is defined by 
$$\Delta_{R}(T):=\Delta_{R,1}(T)\boxtimes\Delta_{R,2}(T).$$
\end{defn} 
The local fundamental line is compatible with the following functorial operations. 
\begin{itemize}
\item[1)] For any continuous $\BZ_p$-algebra homomorphism 
$R\rightarrow R'$, the canonical morphism induces an $R'$-module isomorphism 
$$\Delta_{R}(T)\otimes_RR'\isom \Delta_{R'}(T\otimes_RR').$$
\item[2)] For any exact sequence $0\rightarrow T_1\rightarrow T_2\rightarrow T_3\rightarrow 0$ of 
$R$-representations of $G_{\mathbb{Q}_p}$, the canonical morphism induces an $R$-module isomorphism 
$$\Delta_R(T_2)\isom \Delta_R(T_1)\boxtimes\Delta_R(T_3).$$
\item[3)] There exists a canonical $R$-module isomorphism 
$$\Delta_R(T)\isom 
\Delta_R(T^*)^{\vee}\boxtimes(R(r_T), 0) $$
defined as the product of the following two isomorphisms
$$\Delta_{R,1}(T)\isom \Delta_{R,1}(T^*(1))^{\vee},$$ 
 induced by the Tate duality $\mathrm{C}^{\bullet}_{\mathrm{cont}}(G_{\mathbb{Q}_p}, T)\isom 
\mathbf{R}\mathrm{Hom}_R(\mathrm{C}^{\bullet}_{\mathrm{cont}}(G_{\mathbb{Q}_p}, T^*(1)), R[-2]),$ and the isomorphism
$$\Delta_{R,2}(T)\isom 
\Delta_{R,2}(T^*(1))^{\vee}\boxtimes(R(r_T), 0)$$ 

defined by 
 $$x\otimes y\mapsto [z\otimes w\mapsto y\otimes w\otimes z(x)]$$ for 
$x\in \mathrm{det}_R(T), y\in R_{a(T)}, z\in \mathrm{det}_R(T^*(1))\isom (\mathrm{det}_R(T))^*(r_T)$ and  
$w\in R_{a(T^*(1))}$. 
Here we note that $R_{a(T)}\otimes_RR_{a(T^*(1))}=R$ since $a(T)\cdot a(T^*(1))=1$.
\end{itemize}

For all pairs $(R, T)$ as above and $\zeta \in \varprojlim_{n} \mu_{p^n}(\ov{\BQ}_p)$,  
Kato's local epsilon conjecture \cite{Ka93} posits the existence of a compatible family of canonical trivializations of local fundamental lines 
$$\varepsilon_{R,\zeta}(T):\mathbf{1}_R\isom\Delta_{R}(T),$$ referred to as Kato's local epsilon isomorphisms. 
 A characterizing property is that $\varepsilon_{R,\zeta}(T)$ interpolates the de Rham epsilon isomorphisms 
$$\varepsilon^{\mathrm{dR}}_{L,\zeta}(V):\mathbf{1}_L\isom\Delta_{L}(V)$$ 
for all de Rham pairs $(R,T)=(L,V)$, 
as defined by Kato and recalled in \S\ref{ss:dR-eps} below. 

We first recall epsilon constants defined for representations of the Weil--Deligne group $'W_{\mathbb{Q}_p}$ of $\mathbb{Q}_p$.
Let $C$ be a field of characteristic zero which contains all the $p$-power roots of unity. 
For a $\mathbb{Z}_p$-basis $\zeta=\{\zeta_{p^n}\}_{n\geq 0}\in \Gamma(C, \mathbb{Z}_p(1)):=\varprojlim_{n\geqq 0}\mu_{p^n}(C)$, define an additive character 
$$\psi_{\zeta}:\mathbb{Q}_p\rightarrow C^{\times} \text{ by }
\psi_{\zeta}(\frac{1}{p^n})=\zeta_{p^n}.$$

Let $\rho=(M,\rho)$ be a smooth $C$-representation of the Weil group $W_{\mathbb{Q}_p}$, namely $M$ is a finite dimensional $C$-vector space with a $C$-linear smooth (i.e. locally constant) action $\rho$ of $W_{\mathbb{Q}_p}$. 
The theory of local constants associates to it the epsilon constant $$\varepsilon(\rho, \psi, dx)\in C^{\times}$$ 
(cf.~\cite[Thm.~6.5]{Del} and \S\ref{ss:Epsilon}), which depends on the choices of an additive character $\psi:\mathbb{Q}_p^{\times}\rightarrow C^{\times}$ and a $C$-valued 
Haar measure $dx$ on $\mathbb{Q}_p$. 
 In this article we consider this constant only for the pair $(\psi_{\zeta},dx)$ such that $\int_{\mathbb{Z}_p}dx=1$ (note that $dx$ is self-dual with respect to $\psi_{\zeta}$), and denote it by $$\varepsilon(\rho,\zeta)
:=\varepsilon(\rho,\psi_{\zeta},dx)$$ for simplicity. We note that this constant satisfies 
$f(\varepsilon(\rho,\zeta))=\varepsilon(\rho\otimes_{C,f}C',f(\zeta))\in C'$ for any field homomorphism $f : C\rightarrow C'$, and 
$\varepsilon(\rho,\zeta^a)=\mathrm{det}(\rho)(\mathrm{rec}_{\mathbb{Q}_p}(a))\varepsilon(\rho,\zeta)$ for $a\in \mathbb{Z}_p^{\times}$ by Lemma 4.3 and Theorem 6.5 (b) of \cite{Del}.

Let $M=(M,\rho,N)$ be a $C$-representation of the Weil-Deligne group $'W_{\mathbb{Q}_p}$, i.e. $\rho:=(M,\rho)$ is a smooth $C$-representation of $W_{\mathbb{Q}_p}$ with a $C$-linear endomorphism $N:M\rightarrow M$ such that 
$$\widetilde{\mathrm{Fr}_p}\circ N=p^{-1}\cdot N\circ\widetilde{\mathrm{Fr}_p}$$ for any lift $\widetilde{\mathrm{Fr}_p}\in W_{\mathbb{Q}_p}$ of the geometric 
Frobenius $\mathrm{Fr}_p\in G_{\mathbb{F}_p}$. Its epsilon constant is defined by 
$$\varepsilon(M,\zeta):=\varepsilon(\rho,\zeta)\cdot \mathrm{det}_C(-\mathrm{Fr}_p|M^{I_p}/(M^{N=0})^{I_p}).$$

\subsubsection{de Rham epsilon isomorphism}\label{ss:dR-eps}
 For any de Rham
$L$-representation $V$ of $G_{\mathbb{Q}_p}$, 
we recall Kato's definition \cite{Ka93} of de Rham $\varepsilon$-isomorphism 
$\varepsilon^{\mathrm{dR}}_{L,\zeta}(V): \mathbf{1}_L\isom \Delta_{L}(V)$.

Fix a $\mathbb{Z}_p$-basis $\zeta=\{\zeta_{p^n}\}_{n\geq 0}\in \Gamma(\overline{\mathbb{Q}}_p, \mathbb{Z}_p(1))$. 
Let $V$ be a de Rham $L$-representation of $G_{\BQ_p}$. 
By  
the $p$-adic local monodromy theorem and Fontaine's functor $D_{\mathrm{pst}}(-)$, 
we have an associated functorial Weil--Deligne $L\otimes \BQ_p^{\ur}$-representation
 $$W(V)=(W(V), \rho, N)$$ 
 (cf. ~\S\ref{ss:Gal-WD}).

Let $I_L$ denote the set of embedding of $L$ into $\overline{\mathbb{Q}}_p$.

\begin{defn}
We define the epsilon constant $\varepsilon_{L}(W(V), \zeta)$ of $W(V)$ as 
$$\varepsilon_{L}(W(V), \zeta):=(\varepsilon(W(V)_{\tau},\zeta))_{\tau\in I_L}\in (L\otimes_{\mathbb{Q}_p}\overline{\mathbb{Q}}_p)^{\times},$$
where $\varepsilon(W(V)_{\tau},\zeta)\in \overline{\mathbb{Q}}_p^{\times}$ is the epsilon constant of $W(V)_{\tau}
:=W(V)\otimes_{L\otimes \BQ_p^{\mathrm{ur}},\tau\otimes \mathrm{id}}\overline{\mathbb{Q}}_p$ 
for all $\tau$. 
\end{defn}
The following lemma seems to be well-known, but we couldn't find a published reference for it. 
Put $L_{\infty}=L\otimes_{\mathbb{Q}_p}\mathbb{Q}_p(\mu_{p^{\infty}})\subset L\otimes_{\mathbb{Q}_p}\overline{\mathbb{Q}}_p$. 
\begin{lem}\label{p-epsilon}
We have 
$$\varepsilon_{L}(W(V), \zeta)\in (L_{\infty}^{\times})^{I_p=\mathrm{det}W(V)}.$$
In particular, if $\mathrm{det}(W(V))$ is unramified, then $\varepsilon_{L}(W(V), \zeta)\in L^{\times}$. 
\end{lem}
\begin{proof}
We first note that the image of the character $\mathrm{det}(W(V))$ of $\mathbb{Q}_p^{\times}\isom W_{\mathbb{Q}_p}^{\mathrm{ab}}$ is a subgroup of $L^\times \subset (L\otimes_{\mathbb{Q}_p}\mathbb{Q}_p^{\mathrm{ur}})^{\times}$, and so the notation $(L_{\infty}^{\times})^{I_p=\mathrm{det}W(V)}$ in the lemma is well defined. 
Indeed one has $\mathrm{det}(W(V))=W(\mathrm{det}(V))$ and any de Rham character $\chi : G_{\mathbb{Q}_p}\rightarrow L^{\times}$ (e.g. $\mathrm{det}(V)$) is of the form 
$\chi=\chi_{\mathrm{cyc}}^k\mu_{\lambda}\eta$ for $k\in \mathbb{Z}$, $\mu_{\lambda}$ the unramified character such that $\mu_{\lambda}(\mathrm{Frob}_p)=\lambda\in L^{\times}$, $\eta : \Gamma\isom \mathbb{Z}_p^{\times}\rightarrow L^{\times}$ a finite order character. So $W(\chi)=\mu_{\lambda}\eta\cdot |-|_p^k$, and the image is clearly a subgroup of $L^{\times}$.

As for the lemma, it suffices to show 
$$\sigma(\varepsilon_{L}(W(V), \zeta))=\mathrm{det}(W(V))(\chi_{\mathrm{cyc}}(\sigma))\varepsilon_{L}(W(V), \zeta)$$
 for any $\sigma\in W_{\mathbb{Q}_p}$. 
 Note that 
\begin{multline*}
\sigma((\varepsilon(W(V)_{\tau},\zeta))_{\tau\in I_L})=(\sigma(\varepsilon(W(V)_{\sigma^{-1}\circ\tau},\zeta)))_{\tau\in I_L}\\
=(\varepsilon(W(V)_{\sigma\circ\sigma^{-1}\circ\tau},\sigma(\zeta)))_{\tau\in I_L}
=(\varepsilon(W(V)_{\tau},\zeta^{\chi_{\mathrm{cyc}}(\sigma)}))_{\tau\in I_L},
\end{multline*}
where the second equality follows from the base change property $f(\varepsilon(\rho,\zeta))=\varepsilon(\rho\otimes_{C,f}C',f(\zeta))\in C'$ for any field homomorphism $f : C\rightarrow C'$. 
By the formula $\varepsilon(\rho,\zeta^a)=\mathrm{det}(\rho)(\mathrm{rec}_{\mathbb{Q}_p}(a))\varepsilon(\rho,\zeta)$ for $a\in \mathbb{Z}_p^{\times}$, 
we obtain 
$$\varepsilon(W(V)_{\tau},\zeta^{\chi_{\mathrm{cyc}}(\sigma)})=\mathrm{det}(W(V)_{\tau})(\chi_{\mathrm{cyc}}(\sigma))\varepsilon(W(V)_{\tau},\zeta)
=\tau(\mathrm{det}(W(V))(\chi_{\mathrm{cyc}}(\sigma)))\varepsilon(W(V)_{\tau},\zeta).$$
Therefore, it follows that  
$$(\varepsilon(W(V)_{\tau},\zeta^{\chi_{\mathrm{cyc}}(\sigma)}))_{\tau\in I_L}
=\mathrm{det}(W(V))(\chi_{\mathrm{cyc}}(\sigma))(\varepsilon(W(V)_{\tau},\zeta))_{\tau\in I_L}.$$

\end{proof}

Put 
$$
D_{\mathrm{crys}}(V):=(\mathbf{B}_{\mathrm{crys}}\otimes_{\mathbb{Q}_p}V)^{G_{\mathbb{Q}_p}}
$$ on which the Frobenius $\varphi$ naturally acts, 
and $$D_{\mathrm{dR}}(V):=(\mathbf{B}_{\mathrm{dR}}\otimes_{\mathbb{Q}_p}V)^{G_{\mathbb{Q}_p}}, \ \ \mathrm{Fil}^iD_{\mathrm{dR}}(V):=(t^i\mathbf{B}^+_{\mathrm{dR}}\otimes_{\mathbb{Q}_p}V)^{G_{\mathbb{Q}_p}}\ \text{ and }\ t_V:=D_{\mathrm{dR}}(V)/\mathrm{Fil}^0D_{\mathrm{dR}}(V).$$

Based on the above preliminaries, first define an isomorphism 
$$\theta_{L}(V):\mathbf{1}_L\isom \Delta_{L,1}(V)\boxtimes\mathrm{Det}_L(D_{\mathrm{dR}}(V))$$
as the one induced by the following exact sequence of $L$-vector spaces
\begin{multline}\label{exp}
0\rightarrow \mathrm{H}^0(\mathbb{Q}_p, V)\rightarrow D_{\mathrm{crys}}(V)\xrightarrow{(a)}
D_{\mathrm{crys}}(V)\oplus t_V\xrightarrow{(b)}\mathrm{H}^1(\mathbb{Q}_p, V)\\
\xrightarrow{(c)}D_{\mathrm{crys}}(V^*)^{*}\oplus \mathrm{Fil}^0D_{\mathrm{dR}}(V)\xrightarrow{(d)}D_{\mathrm{crys}}(V^{*})^{*}\rightarrow \mathrm{H}^2(\mathbb{Q}_p, V)\rightarrow 0.
\end{multline}
Here the map (a) is the sum of $1-\varphi:D_{\mathrm{crys}}(V)\rightarrow D_{\mathrm{crys}}(V)$ and  the canonical map $D_{\mathrm{crys}}(V)\rightarrow t_V$, and the maps 
(b) and (c) are defined by using the Bloch--Kato's exponential map and its dual, 
and the map (d) is the dual of (a) for $V^*(1)$ (see \cites{Ka93,FK,Narank1} for the precise definition).

We set $t_{\zeta}:=\mathrm{log}[\zeta (\bmod p)]\in \mathbf{B}^+_{\mathrm{dR}}$ and $h_M:=\sum_{r\in \mathbb{Z}}r\cdot \mathrm{dim}_L\mathrm{gr}^{-r}D_{\mathrm{dR}}(V)$. By Lemma \ref{p-epsilon}, the map $$L_{a_p(T)}\otimes_L\mathrm{det}_L(V)\isom  \mathbf{B}_{\mathrm{dR}}\otimes_{\mathbb{Q}_p}\mathrm{det}(V): x\mapsto \frac{1}{\varepsilon_L(W(V),\zeta)}\cdot\frac{1}{t_{\zeta}^{h_V}}\cdot x ,$$ induces an isomorphism $$L_{a_p(T)}\otimes_L\mathrm{det}_L(V)\isom \mathrm{det}_L(D_{\mathrm{dR}}(V)).$$
We denote its inverse by 
$$\theta_{\mathrm{dR},L}(V,\zeta):\mathrm{Det}_L(D_{\mathrm{dR}}(V))\isom \Delta_{L,2}(V).$$ 

\begin{defn}
The de Rham $\varepsilon$-isomorphism 
$$\varepsilon_{L,\zeta}^{\mathrm{dR}}(V):\mathbf{1}_L\isom\Delta_L(V)$$ is defined as the following composite
\begin{multline*}
\varepsilon_{L,\zeta}^{\mathrm{dR}}(V):\mathbf{1}_L\xrightarrow{\Gamma(V)\cdot\theta_L(V)}
\Delta_{L,1}(V)\boxtimes \mathrm{Det}_L(D_{\mathrm{dR}}(V))
\xrightarrow{\mathrm{id}\boxtimes \theta_{\mathrm{dR},L}(V,\zeta)}
\Delta_{L,1}(V)\boxtimes\Delta_{L,2}(V)=\Delta_L(V).
\end{multline*}
\end{defn}
\subsubsection{The conjecture}
In the early 90's Kato proposed the following fundamental (cf.~\cite[Conj.~1.8]{Ka93}, \cite[Conj.~3.4.3]{FK}, and \cite[Conj.~3.8]{Narank1}).

\begin{conj}[Kato's local epsilon conjecture]\label{conj, epsilon} 
For all triples $(R,T,\zeta)$ as in \S\ref{ss:Kato s-u} such that $T$ is an $R$-representation of $G_{\mathbb{Q}_p}$ and $\zeta$ a $\mathbb{Z}_p$-basis of $\Gamma(\overline{\mathbb{Q}}_p,\mathbb{Z}_p(1))$, 
there exists a unique compatible family of isomorphisms
$$\varepsilon_{R,\zeta}(T):\mathbf{1}_R\isom\Delta_R(T), $$ 
 satisfying the following.
\begin{itemize}
\item[1)]For any continuous $\mathbb{Z}_p$-algebra homomorphism $R\rightarrow R'$, we have 
$$\varepsilon_{R,\zeta}(T)\otimes\mathrm{id}_{R'}=\varepsilon_{R',\zeta}(T\otimes_{R}R')$$ under the canonical isomorphism $$\Delta_{R}(T)\otimes_RR'\isom\Delta_{R'}(T\otimes_RR').$$
\item[2)]For any exact sequence $0\rightarrow T_1\rightarrow T_2\rightarrow T_3\rightarrow 0$ of $R$-representations of $G_{\mathbb{Q}_p}$, we have 
$$\varepsilon_{R,\zeta}(T_2)=\varepsilon_{R,\zeta}(T_1)\boxtimes\varepsilon_{R,\zeta}(T_3)$$ under the canonical isomorphism $$\Delta_R(T_2)\isom\Delta_R(T_1)\boxtimes\Delta_R(T_3).$$
\item[3)]For any $a\in \mathbb{Z}_p^{\times}$, we have
$$\varepsilon_{R,\zeta^a}(T)=\mathrm{det}_R(T)(\mathrm{rec}_{\mathbb{Q}_p}(a))\cdot \varepsilon_{R,\zeta}(T).$$
\item[4)] We have the following commutative diagram 
\[
\xymatrix{
\Delta_{R}(T) \ar[rr]^-{\mathrm{can}} &&  \Delta_{R}(T^*(1))^{\vee}\boxtimes(R(r_T),0) \ar[d]^-{\varepsilon_{R,\zeta}(T^*(1))^{\vee}\boxtimes[ \mathbf{e}_{r_T}\mapsto 1]}.\\
\ar[u]^-{\varepsilon_{R,\zeta}(T)}  \mathbf{1}_R \ar[rr]_-{\mathrm{det}_R(T)(\mathrm{rec}_{\mathbb{Q}_p}(-1))\cdot\mathrm{can}} && \mathbf{1}_R\boxtimes {\bf 1}_R  \\
}
\]

\item[5)]For any triple $(L,V,\zeta)$ such that $V$ is de Rham, 
we have
$$\varepsilon_{L,\zeta}(V)=\varepsilon_{L,\zeta}^{\mathrm{dR}}(V).$$
\end{itemize}

\end{conj}
\begin{remark} The above property 4) corrects a typo
in the property (4) of \cite[Conj.~2.1]{NaKato}: 
 on the left vertical map therein, the author wrote 
$\varepsilon_{R,-\zeta}(T)(:=\varepsilon_{R,\zeta^{-1}}(T))$ instead of $\varepsilon_{R,\zeta}(T)$.  
We also note that the commutative diagram in property 4) is the same as the following 
\[
\xymatrix{
\Delta_{R}(T) \ar[rr]^-{\mathrm{can}} &&  \Delta_{R}(T^*(1))^{\vee}\boxtimes(R(r_T),0) \ar[d]^-{\varepsilon_{R,\zeta}(T^*(1))^{\vee}\boxtimes[ \mathbf{e}_{r_T}\mapsto 1]}.\\
\ar[u]^-{\varepsilon_{R,\zeta^{-1}}(T)}  \mathbf{1}_R \ar[rr]_-{\mathrm{can}} && \mathbf{1}_R\boxtimes\mathbf{1}_R  \\
}
\]
in view of the property 3).
\end{remark}
\subsubsection{Results} The following is a main result  of \cites{Ka93, NaKato,J}. 
\begin{thm}\label{thm, rank two epsilon}\noindent
\begin{itemize} 
\item[i)] For any odd prime $p$, Kato's local epsilon Conjecture \ref{conj, epsilon} is true for all pairs $(R, T)$ such that $\mathrm{rank}_R(T)\leq 2$. Namely,  
there exists a unique compatible family of isomorphisms
$$\varepsilon_{R,\zeta}(T):\mathbf{1}_R\isom\Delta_R(T)$$ 
satisfying the properties 1), 3), 4), 5) of Conjecture \ref{conj, epsilon} and also the property 2) in the rank two case 
with $T_1$ and $T_3$ of rank one. 
\item[ii)] If $p=2$, then a similar result holds after inverting $p$, i.e. we have compatible families of isomorphisms 
$$\varepsilon_{R[1/p],\zeta}(T\otimes_{\mathbb{Z}_p}\mathbb{Q}_p):\mathbf{1}_{R[1/p]}\isom\Delta_R(T)\otimes_{\mathbb{Z}_p}\mathbb{Q}_p$$ satisfying 
 the properties 1), 3), 4), 5) of Conjecture~\ref{conj, epsilon}.
 \end{itemize}
\end{thm}
\begin{remark} We describe some of the prior results towards Conjecture~\ref{conj, epsilon} as in Theorem~\ref{thm, rank two epsilon}. 
\begin{itemize}
\item[i)] The rank one case was originally proved by Kato \cite{Ka93} (see also \cites{V,Narank1}). 
\item[ii)] The rank two case was proved by the third-named author \cite{NaKato} for almost all pairs $(R,T)$, and completed by Jacinto \cite{J} building on \cite{NaKato}. 
More precisely, the isomorphism $\varepsilon_{R,\zeta}(T):\mathbf{1}_R\isom\Delta_R(T)$ was defined in \cite{NaKato} for all 
pairs $(R, T)$ such that $\mathrm{rank}_R(T)=2$ under the assumption: the universal (framed or versal) deformation 
ring corresponding to $\overline{T}=T\otimes_RR/\mathfrak{m}_R$ is reduced and $p$-torsion free for $R$ local, and contains a Zariski dense subset consisting of crystalline points in every irreducible component. 
These properties are established for all residual representations $\overline{T}$ by the recent work of B\"{o}ckle--Iyenger--Pa\v{s}k\=unas \cite{BIP}, hence we obtain the isomorphism $\varepsilon_{R,\zeta}(T):\mathbf{1}_R\isom\Delta_R(T)$ for all  pairs $(R, T)$ such that $\mathrm{rank}_R(T)=2$. 
\item[iii)] In \cite{NaKato}, property 5) was proved when the 
Hodge--Tate weights of $V$ are $(k_1, k_2)$ satisfying $k_1\leq 0$ and $k_2\geq 1$. Subsequently, the general case 
was 
established in \cite{J}. For our application to the local sign decomposition,  
 the former suffices. 
\end{itemize}
\end{remark}

\subsection{Definition of  $w_T$ via Kato's local epsilon isomorphism}\label{ss:lsd via Kato's eps}
In this subsection we define an involution $w_T : H^1(\mathbb{Q}_p, T)\rightarrow H^1(\mathbb{Q}_p, T)$ for 
generic symplectic self-dual pairs $(R, T)$ of rank two using 
Kato's epsilon
 isomorphism $\varepsilon_{R,\zeta}(T) : \mathbf{1}_R\isom \Delta_R(T)$ as in Theorem~\ref{thm, rank two epsilon}. 
 In the anomalous case (including the prime $p=2$), we similarly define an involution $w_{T\otimes_{\mathbb{Z}_p}\mathbb{Q}_p} : H^1(\mathbb{Q}_p, T\otimes_{\mathbb{Z}_p}\mathbb{Q}_p)\rightarrow H^1(\mathbb{Q}_p, T\otimes_{\mathbb{Z}_p}\mathbb{Q}_p)$ via the epsilon 
 isomorphism $\varepsilon_{R[1/p],\zeta}(T\otimes_{\mathbb{Z}_p}\mathbb{Q}_p):\mathbf{1}_{R[1/p]}\isom\Delta_R(T)\otimes_{\mathbb{Z}_p}\mathbb{Q}_p.$
\subsubsection{General rank two case}\label{subsubsection, rank 2 kato}
Let $(R, T)$ be a generic rank two representation of $G_{\BQ_p}$, not necessarily symplectic self-dual.

Recall that  $H^i(\mathbb{Q}_p, T)=0$ for $i\in\{0,2\}$ and 
 $H^1(\mathbb{Q}_p, T)$ is a free $R$-module of rank two. Thus, we have
 $\Delta_{R,1}(T)=(\wedge^2 H^1(\mathbb{Q}_p, T))^{-1}$ as an underlying invertible $R$-module, and to give an 
 $R$-module 
 isomorphism 
 $$\mathbf{1}_R\isom \Delta_R(T)=\Delta_{R,1}(T)\otimes_R\Delta_{R,2}(T)$$ is equivalent to give an $R$-module
 isomorphism 
 $$\wedge^2 H^1(\mathbb{Q}_p, T)\isom \Delta_{R,2}(T).$$
 Moreover, the latter is equivalent to giving an $R$-bilinear skew-symmetric perfect pairing 
 $$H^1(\mathbb{Q}_p, T)\times H^1(\mathbb{Q}_p, T)\rightarrow \Delta_{R,2}(T).$$
 
 For a fixed $\mathbb{Z}_p$-base $\zeta\in \mathbb{Z}_p(1)$, 
 we denote by 
 $$[ \ , \ ]_{T, \zeta} : H^1(\mathbb{Q}_p, T)\times H^1(\mathbb{Q}_p, T)\rightarrow \Delta_{R,2}(T)$$ 
 the pairing corresponding to the local epsilon isomorphism 
 $$\varepsilon_{R, \zeta}(T) : \mathbf{1}_R\isom \Delta_{R}(T)$$
 as in Theorem \ref{thm, rank two epsilon}. 
 We also define 
 $$\{\ , \}_T : H^1(\mathbb{Q}_p, T^*(1))\otimes_{R}\Delta_{R,2}(T)\times 
 H^1(\mathbb{Q}_p, T)\rightarrow \Delta_{R,2}(T) : (x\otimes y, z)\mapsto \{x, z\}_T\cdot y.$$
 
\begin{lem}
 There exists an $R$-module morphism $$w_{T,\zeta} : H^1(\mathbb{Q}_p, T)\rightarrow H^1(\mathbb{Q}_p, T^*(1))\otimes_{R}\Delta_{R,2}(T)$$ uniquely characterized by the equality $$\{w_{T,\zeta}(x), y\}_T=[x, y]_{T,\zeta}\in \Delta_{R,2}(T)$$
 for any  $x, y\in H^1(\mathbb{Q}_p, T)$. 
Moreover, for any $a\in \mathbb{Z}_p^{\times}$,  we have 
 $$w_{T, \zeta^a}=\mathrm{det}_R(T)(\sigma_a) w_{T, \zeta}$$ 
\end{lem}
\begin{proof}
The existence is a simple consequence of part 6) of Proposition \ref{a}. 
The rest follows from property 3) of Conjecture \ref{conj, epsilon}
 since $\varepsilon_{R, \zeta^a}(T)=\mathrm{det}_R(T)(\sigma_a)\varepsilon_{R, \zeta}(T)$.   
\end{proof}

 Applying this construction to $T^*(1)$ for $\zeta^{-1}$ instead of $\zeta$, we obtain an $R$-module morphism 
 $$w_{T^*(1),\zeta^{-1}} : H^1(\mathbb{Q}_p, T^*(1))\rightarrow H^1(\mathbb{Q}_p, T)\otimes_R\Delta_{R,2}(T^*(1)).$$
 The following compatibility for $w_{T,\zeta}$ and $w_{T^*(1),\zeta^{-1}}$ involves a canonical isomorphism 
 $$\mathrm{can} : \Delta_{R,2}(T^*(1))\otimes \Delta_{R,2}(T)\isom R(2)$$
 defined by the tensor products of the canonical isomorphism 
 $\mathrm{det}_R(T^*(1))\otimes_R\mathrm{det}_R(T)\isom (\mathrm{det}_R(T))^*(2)\otimes_R\mathrm{det}_R(T)\isom R(2)$ and $R_{a(T^*(1))}\otimes_RR_{a(T)}\isom R : y\otimes z\mapsto yz$.

 \begin{lem}\label{inv}
 The composite
 \begin{multline*}
 H^1(\mathbb{Q}_p, T)\xrightarrow{w_{T,\zeta}}H^1(\mathbb{Q}_p, T^*(1))\otimes_R\Delta_{R,2}(T)\\
 \xrightarrow{-w_{T^*(1),\zeta^{-1}}\otimes \mathrm{id}_{\Delta_{R,2}(T)}}
 H^1(\mathbb{Q}_p, T)\otimes_R\Delta_{R,2}(T^*(1))\otimes_R\Delta_{R,2}(T)\\
 \xrightarrow{\mathrm{id}\otimes \mathrm{can}} H^1(\mathbb{Q}_p, T)(2)
 \xrightarrow{(-)\otimes \zeta^{\otimes -2}}H^1(\mathbb{Q}_p, T)
 \end{multline*}
 is the identity.
 \end{lem}
 \begin{proof}
 This lemma follows from the compatibility of local epsilon isomorphism with Tate duality. 
 
 Fix a basis $\{e_1, e_2\}$ of $H^1(\mathbb{Q}_p, T)$ and $\{e\}$ of $\Delta_{R,2}(T)$ such that 
 $$[e_1, e_2]_{T,\zeta}=e.$$
 Let $\{e_1^*, e_2^*\}$ be the dual basis of $H^1(\mathbb{Q}_p, T^*(1))$ with respect to 
 the Tate pairing $\{\ , \ \}_T$, i.e. 
 $$\{e_i^*, e_j\}_{T}=\begin{cases} 0 & (i\neq j)\\ 1& (i=j)\end{cases},$$ 
 and $e^*\zeta^{\otimes 2}\in \Delta_{R, 2}(T^*(1))$ be the basis such that $$\mathrm{can}(e^*\zeta^{\otimes 2}\otimes e)=\zeta^{\otimes 2}.$$
  
 Then we have 
 $$w_{T,\zeta}(e_1)=e_2^*\otimes e, \ \ w_{T,\zeta}(e_2)=-e_1^*\otimes e$$
 since  
 $$\{e_2^*\otimes e, e_1\}_T=0=[e_1, e_1]_{T,\zeta}, \ \ \{e_2^*\otimes e, e_2\}_T=e=[e_1, e_2]_{T,\zeta}$$
 and 
 $$\{-e_1^*\otimes e, e_1\}_T=-e=[e_2, e_1]_{T,\zeta}, \ \ \{-e_1^*\otimes e , e_2\}_T=0=[e_2, e_2]_{T,\zeta}.$$
 By the property 4) of Conjecture \ref{conj, epsilon} and Theorem \ref{thm, rank two epsilon},  
 we have
 $$[e_1^*, e_2^*]_{T^*(1),\zeta^{-1}}=-e^*\zeta^{\otimes 2}.$$
 Similarly, we obtain 
 $$w_{T^*(1),\zeta^{-1}}(e^*_1)=e_2\otimes e^*\zeta^{\otimes 2}, \ \ w_{T^*(1),\zeta^{-1}}(e^*_2)=-e_1\otimes e^*\zeta^{\otimes 2}$$
 since 
 $$\{e_2\otimes e^*\zeta^{\otimes 2}, e_1^*\}_{T^*(1)}=0=[e^*_1, e^*_1]_{T^*(1),\zeta^{-1}}, \ \ \{e_2\otimes e^*\zeta^{\otimes 2} , e^*_2\}_{T^*(1)}=-e^*\zeta^2=[e^*_1, e^*_2]_{T^*(1),\zeta^{-1}}$$
 and 
 $$\{-e_1\otimes e^*\zeta^{\otimes 2}, e^*_1\}_{T^*(1)}=e^*\zeta^{\otimes 2}=[e^*_2, e^*_1]_{T^*(1),\zeta^{-1}}, \ \ \{-e_1\otimes e^*\zeta^{\otimes 2} , e^*_2\}_{T^*(1)}=0=[e^*_2, e^*_2]_{T^*(1),\zeta^{-1}}.$$

Hence, the composite in the Lemma sends
$$e_1\mapsto e_2^*\otimes e\mapsto (e_1\otimes e^*\zeta^{\otimes 2})\otimes e\mapsto e_1\otimes\zeta^{\otimes 2}\mapsto e_1$$
and
$$e_2\mapsto -e_1^*\otimes e\mapsto -(-e_2\otimes e^*\zeta^{\otimes 2})\otimes e\mapsto e_2\otimes\zeta^{\otimes 2}\mapsto e_2,$$
concluding the proof.
 \end{proof}

\subsubsection{Rank two self-dual case}\label{subsubsection, rank 2 kato self-dual}

Let $(R, T)$ be a generic symplectic self-dual representation of $G_{\BQ_p}$ of rank two.

We have $\mathrm{det}_R(T)\isom R(1)$ 
and in turn, 
$$\Delta_{R, 2}(T)=\mathrm{det}_R(T)$$ since  $a(T)=\mathrm{det}_R(T)(\mathrm{rec}_{\mathbb{Q}_p}(p))=\chi_{\mathrm{cyc}}(\mathrm{rec}_{\mathbb{Q}_p}(p))=1$.
Hence, 
$$\mathrm{Hom}_{R}(T, \mathrm{det}_R(T))\isom T^*(1)\otimes_R\Delta_{R,2}(R)(-1),$$
where $\Delta_{R,2}(R)(-1)$ is regarded as an $R$-module. 
By its composition with the following isomorphism 
$$T\isom \mathrm{Hom}_{R}(T, \mathrm{det}_R(T)) : x\mapsto [y\mapsto x\wedge y],$$ 
we obtain an isomorphism 
\begin{equation}\label{eq:dual}
T\isom T^*(1)\otimes_R\Delta_{R,2}(R)(-1),
\end{equation}
by which we identify both sides. 

Define an $R$-module morphism 
$$w_T : H^1(\mathbb{Q}_p, T)\rightarrow H^1(\mathbb{Q}_p, T)$$  
as the composite
$$H^1(\mathbb{Q}_p, T)\xrightarrow{w_{T,\zeta}\otimes \zeta^{\otimes -1}} H^1(\mathbb{Q}_p, T^*(1))\otimes_R\Delta_{R,2}(T)(-1)
\isom H^1(\mathbb{Q}_p, T),$$
where the last isomorphism is induced by \eqref{eq:dual}. 

\begin{lem}\label{lem:involution}
The morphism $w_T$ is an involution and it is independent of the choice of $\zeta$.
\end{lem}
\begin{proof}
The second assertion 
 holds  since for any $a\in \mathbb{Z}_p^\times$, we have
$$w_{T,\zeta^a}\otimes (\zeta^a)^{\otimes -1}=\mathrm{det}_R(T)(\sigma_a)w_{T,\zeta}\otimes a^{-1}\zeta^{\otimes -1}
=w_{T,\zeta}\otimes \zeta^{\otimes -1}.$$
As for the first, it is a consequence of Lemma \ref{inv} since the isomorphism \eqref{eq:dual} is symplectic. 
\end{proof}

\begin{defn}
If $p$ is an odd prime and $T$ is generic, then the signed submodules $H^1_{\pm}(\BQ_p , T) \subset H^1 (\BQ_p , T)$ are defined by  
$$H^1_{\pm}(\mathbb{Q}_p, T)=H^1(\mathbb{Q}_p, T)^{w_T=\pm 1}.$$
If either $p=2$ or $T$ is anomalous, then we similarly define $H^1_{\pm}(\BQ_p, T\otimes_{\BZ_p}\BQ_p)$.
\end{defn}

Basic properties of the signed submodules are given by 
Proposition~\ref{prop, decomposition and Lagrangian} below, which relies on the following preliminary. 
\begin{lem}\label{2}
Define a pairing
$$\{-, -\} : H^1(\mathbb{Q}_p,T)\times H^1(\mathbb{Q}_p, T)\rightarrow \Delta_{R,2}(T)(-1)$$
as the composite 
\begin{multline*}
H^1(\mathbb{Q}_p,T)\times H^1(\mathbb{Q}_p, T)\isom H^1(\mathbb{Q}_p, T^*(1))\otimes_R\Delta_{R,2}(T)(-1)
\times H^1(\mathbb{Q}_p, T)\\
\xrightarrow{(x\otimes z, y)\mapsto \{x, y\}_T\otimes z} \Delta_{R,2}(T)(-1),
\end{multline*}
where the first morphism is induced by \eqref{eq:dual}. 
This pairing is symmetric and satisfies 
\begin{equation}\label{equation, comparison of pairings}
\{w_T(x), y\}=[x, y]_{T,\zeta}\otimes \zeta^{\otimes -1}
\end{equation}
 for any $x, y \in H^1(\mathbb{Q}_p,T)$. 
\end{lem}
\begin{proof}
The last equality immediately follows from the definition of 
$w_{T,\zeta}$ and $w_T$. 

It remains to show that the pairing $\{-, -\}$ is symmetric. 
By definition, it equals the composite  
\begin{multline*}
H^1(\mathbb{Q}_p,T)\times H^1(\mathbb{Q}_p, T)\xrightarrow{\cup} H^2(\mathbb{Q}_p, T\otimes_RT)
\xrightarrow{H^2([x\otimes y\mapsto x\wedge y])}H^2(\mathbb{Q}_p, \mathrm{det}_R(T))\\
\isom H^2(\mathbb{Q}_p, R(1))\otimes\Delta_{R,2}(T)(-1)
\isom \Delta_{R,2}(T)(-1).
\end{multline*}
Here the second morphism is induced by the morphism $x\otimes y \mapsto x\wedge y$ for $x,y\in T$.
The preceding composite is symmetric since so is the composite
$$H^1(\mathbb{Q}_p,T)\times H^1(\mathbb{Q}_p, T)\xrightarrow{\cup} H^2(\mathbb{Q}_p, T\otimes_RT)
\xrightarrow{H^2([x\otimes y\mapsto x\wedge y])}H^2(\mathbb{Q}_p, \mathrm{det}_R(T)).$$ 
\end{proof}
\begin{prop}\label{prop, decomposition and Lagrangian} 
\noindent 
\begin{itemize}
\item[i)] Let $p$ be an odd prime and $(R, T)$ a generic symplectic self-dual representation of $G_{\BQ_p}$ of rank two. 
Then we have a decomposition 
$$H^1(\mathbb{Q}_p,T)=H^1_+(\mathbb{Q}_p, T)\oplus H^1_-(\mathbb{Q}_p,T)$$ 
into free $R$-submodules of rank one. 
Moreover, the submodules $H^1_\pm(\mathbb{Q}_p, T)$ are Lagrangian with respect to the pairing 
$\{-, -\} : H^1(\mathbb{Q}_p,T)\times H^1(\mathbb{Q}_p, T)\rightarrow \Delta_{R,2}(T)(-1)$ 
as in Lemma~\ref{2}.
\item[ii)] Let $p$ be a prime and $(R,T)$ a symplectic self-dual representation of $G_{\BQ_p}$ of rank two such that 
$$
H^2(\BQ_p,T)\otimes_{\BZ_p}\BQ_p=0.
$$
Then 
we have a decomposition 
$$H^1(\mathbb{Q}_p,T)\otimes_{\BZ_p}\BQ_p=H^1_+(\mathbb{Q}_p, T\otimes_{\BZ_p}\BQ_p)\oplus 
H^1_-(\mathbb{Q}_p,T\otimes_{\BZ_p}\BQ_p)$$ 
into Lagrangian locally free $R[1/p]$-submodules of rank one.
\end{itemize}
\end{prop}

\begin{proof}
In the following we consider part i), and an analogous argument applies for part ii).

Note that  $w_T$ is an involution by Lemma~\ref{lem:involution} and $2\in R^\times$,  
so the existence of the decomposition follows. 
Since $H^1(\mathbb{Q}_p,T)$ is a free $R$-module of rank two, 
to show that $H^1_{\pm}(\BQ_p , T)$ are of rank one, it suffices to show that 
$w_T$ is not a scalar morphism, i.e. $w_T\not=\pm \mathrm{id}$. 

We assume that $w_T=\pm\mathrm{id}$ and deduce a contradiction below. It considers the case $w_T=\mathrm{id}$ and 
an analogous argument applies if $w_T=-\mathrm{id}$.

We have 
$$\{x, y\}=\{w_T(x), y\}=[x, y]_{T,\zeta}\otimes \zeta^{\otimes -1}$$
for arbitrary $x, y\in H^1(\mathbb{Q}_p, T)$. Then it follows that   
$[x, y]_{T,\zeta}=0$ since the pairing $[\, ,\,]_{T,\zeta}$ is skew-symmetric and $\{\, ,\, \}$  symmetric. 
This contradicts the fact that the pairing $[\, , \, ]_{T,\zeta}$ induces an isomorphism 
$\mathrm{det}(H^1(\mathbb{Q}_p, T))\isom \Delta_{R,2}(T)$.

Finally, we show the Lagrangian property. 
For any $x_\pm \in H^1_\pm(\mathbb{Q}_p, T)$, 
 we have 
\[\{x_\pm, x_\pm\}=\pm \{w_T(x_\pm), x_\pm\}=\pm [x_\pm, x_\pm]_{T, \zeta}\otimes \zeta^{\otimes -1}=0\] since $[\,, \,]_{T,\zeta}$ is skew-symmetric.  
On the other hand, if $x_\pm$ is a basis, then $\{x_+, x_-\} \in R^\times$ since the pairing is non-degenerate. 
Hence, the proof concludes.
\end{proof}

\subsection{Local sign decomposition for de Rham representations}
This subsection gives a description of the local sign  decomposition as in \S\ref{ss:lsd via Kato's eps} for de Rham representations. 

Let $V$ be a two dimensional de Rham $L$-representation of $G_{\BQ_p}$. 
Suppose that it is symplectic self-dual with Hodge Tate weights $k_1\leq k_2$ and generic, i.e. $H^0(\mathbb{Q}_p, V)=0$. Then we have
$(k_1, k_2)=(1-k, k)$ 
for some $k\in \BZ_{\geq 1}$. Moreover, 
 $$\mathrm{dim}_{L}(H^1_{\rm f}(\mathbb{Q}_p, V))=\mathrm{dim}_{L}(D_{\mathrm{dR}}(V)/\mathrm{Fil}^0D_{\mathrm{dR}}(V))=1$$ 
and $\Gamma(V)=(-1)^{k-1}$
 by Corollary~\ref{cor, dR ssd two}. 
Recall that $\varepsilon_L(W(V))\in \{\pm 1\}$ and it is independent of the choice of $\zeta$. 
 
Since we assume that $H^0(\mathbb{Q}_p, V)=0$ and so $H^2(\mathbb{Q}_p, V)\simeq H^0(\mathbb{Q}_p, V)^*=0$, 
 the de Rham epsilon isomorphism $\varepsilon_{L,\zeta}^{\mathrm{dR}}(V) : \mathbf{1}_L\isom \Delta_{L}(V)$ 
 is induced by an isomorphism 
 $$\wedge^2 H^1(\mathbb{Q}_p, V)\isom \Delta_{L,2}(V)$$
 whose explicit definition is given in the proof of Proposition \ref{prop, de Rham property} below. To give such an isomorphism is equivalent to giving an 
 $L$-linear skew-symmetric pairing 
 $$H^1(\mathbb{Q}_p, V)\times H^1(\mathbb{Q}_p, V)\rightarrow \Delta_{L, 2}(V).$$
 Proceeding as in \S\ref{subsubsection, rank 2 kato self-dual},  the de Rham epsilon isomorphism leads to such an involution, 
 which we denote by 
 $$w_V^{\mathrm{dR}} : H^1(\mathbb{Q}_p, V)\rightarrow H^1(\mathbb{Q}_p, V).$$
  Here we note that $w_V^{\mathrm{dR}}$ satisfies the same property as in Lemma \ref{inv} since 
 $\varepsilon_{L,\zeta}^{\mathrm{dR}}(V)$ is also compatible with Tate duality by Lemma 3.7 (2) of 
 \cite{Narank1}. 
 
 \begin{defn}
 The signed submodules $H^1_{\pm, {\rm dR}}(\BQ_p , V) \subset H^1 (\BQ_p , V)$ are defined by   
 $$H^1_{\pm, \mathrm{dR}}(\mathbb{Q}_p, V):=H^1(\mathbb{Q}_p, V)^{w_{V}^{\mathrm{dR}}=\pm 1}$$
 \end{defn}
 \begin{prop}\label{1.5}
 Let $V$ be a  two-dimensional de Rham $L$-representation of $G_{\BQ_p}$ which is symplectic self-dual and generic.  
 Then we have a decomposition 
$$H^1(\mathbb{Q}_p,V)=H^1_{+, {\rm dR}}(\mathbb{Q}_p, T)\oplus H^1_{-,{\rm dR}}(\mathbb{Q}_p,T)$$ 
into $L$-vector subspaces of dimension one.  
 \end{prop}
 \begin{proof}
  This follows just as in the proof of Proposition \ref{prop, decomposition and Lagrangian}.  
 \end{proof}

\begin{prop}\label{prop, de Rham property}
For $\epsilon:=-\Gamma(V)\varepsilon(W(V))$, 
we have 
$$H^1_{\epsilon, \mathrm{dR}}(\mathbb{Q}_p, V)=H^1_{\rm f}(\mathbb{Q}_p, V).$$

\end{prop}
\begin{proof}
It suffices to show that $H^1_{\rm f}(\mathbb{Q}_p, V)\subset H^1_{\epsilon, \mathrm{dR}}(\mathbb{Q}_p, V)$ since both are one dimensional. 

The de Rham epsilon isomorphism $\varepsilon_{L,\zeta}^{\mathrm{dR}}(V) : \mathbf{1}_L\isom \Delta_{L}(V)$ is induced by the 
following composite of isomorphisms 
\begin{multline*}
\mathrm{det}_L(H^1(\mathbb{Q}_p, V))\isom H^1_{\rm f}(\mathbb{Q}_p, V)\otimes_L(H^1(\mathbb{Q}_p, V)/H^1_{\rm f}(\mathbb{Q}_p, V))\\
\xrightarrow{\Gamma(V)\cdot \mathrm{exp}^{-1}\otimes \mathrm{exp}^*} (D_{\mathrm{dR}}(V)/\mathrm{Fil}^0D_{\mathrm{dR}}(V))\otimes_L\mathrm{Fil}^0D_{\mathrm{dR}}(V)
\xrightarrow{\overline{x}\otimes y\mapsto -x\wedge y(=y\wedge x)}
\mathrm{det}_L(D_{\mathrm{dR}}(V))\\
\isom D_{\mathrm{dR}}(\mathrm{det}(V))=\frac{1}{t}\Delta_{L,2}(V)
\xrightarrow{y\mapsto \varepsilon_L(W(V))ty} \Delta_{L,2}(V)
\end{multline*}
Here we note that $\overline{x}\otimes y$ is sent to $y\wedge x=-x\wedge y$ since the natural isomorphism is given by 
$$\mathrm{Fil}^0D_{\mathrm{dR}}(V)\otimes_L(D_{\mathrm{dR}}(V)/\mathrm{Fil}^0D_{\mathrm{dR}}(V))\isom \mathrm{det}_L(D_{\mathrm{dR}}(V)) : y\otimes \overline{x}\mapsto y\wedge x.$$
By the canonical isomorphism $V\isom V^*(1)\otimes_L\Delta_{L,2}(V)(-1)$, 
we  identify the above composite 
$$(D_{\mathrm{dR}}(V)/\mathrm{Fil}^0D_{\mathrm{dR}}(V))\otimes_L\mathrm{Fil}^0D_{\mathrm{dR}}(V)\isom 
\mathrm{det}_L(D_{\mathrm{dR}}(V))\\
\isom D_{\mathrm{dR}}(\mathrm{det}_L(V))=\frac{1}{t}\Delta_{L,2}(V)$$ with 
\begin{multline*}
(D_{\mathrm{dR}}(V^*(1))/\mathrm{Fil}^0D_{\mathrm{dR}}(V^*(1))\otimes_L\Delta_{L,2}(V)(-1))\otimes_L\mathrm{Fil}^0D_{\mathrm{dR}}(V)\\
\xrightarrow{(f\otimes z)\otimes y\mapsto -f(y)\otimes z}D_{\mathrm{dR}}(L(1))\otimes_L \Delta_{L,2}(V)(-1)
=\frac{1}{t}L(1)\otimes_L\Delta_{L,2}(V)(-1)=\frac{1}{t}\Delta_{L,2}(V). 
\end{multline*}
Since 
$$\{\mathrm{exp}_V(a), b\}_V=[a, \mathrm{exp}^*(b)]_{\mathrm{dR}}$$
for any $a\in D_{\mathrm{dR}}(V^*(1))$, $b\in H^1(\mathbb{Q}_p, V)$, 
it follows that the above composite $$\mathrm{det}_L(H^1(\mathbb{Q}_p, V))\isom\Delta_{L,2}(V)$$
 maps $x\wedge y$, for $x\in H^1_{\rm f}(\mathbb{Q}_p, V), y\in H^1(\mathbb{Q}_p, V)$, to
 $$-\Gamma(V)\varepsilon_{L}(W(V))\{x, y\}\otimes \zeta.$$

 On the other hand, by the definition of $w_{V}^{\mathrm{dR}}$, this isomorphism is also defined by  
 $$\mathrm{det}(H^1(\mathbb{Q}_p, V))\isom \Delta_{L,2}(V) : x\wedge y\mapsto 
 \{w_V^{\mathrm{dR}}(x), y\}\otimes \zeta$$ for any $x, y\in H^1(\mathbb{Q}_p, V)$. 
 Therefore,  we obtain
 $$\{w_V^{\mathrm{dR}}(x), y\}=-\Gamma(V)\varepsilon_{L}(W(V))\{x, y\}$$
  for $x\in H^1_{\rm f}(\mathbb{Q}_p, V)$ and 
  any  $y\in H^1(\mathbb{Q}_p, V)$. We conclude that 
  $$w_V^{\mathrm{dR}}(x)=-\Gamma(V)\varepsilon_{L}(W(V))x$$ for $x\in H^1_{\rm f}(\mathbb{Q}_p, V)$, 
  which implies that $H^1_{\rm f}(\mathbb{Q}_p, V)\subset H^1_{\epsilon, \mathrm{dR}}(\mathbb{Q}_p, V)$.

\end{proof}

\subsection{{Construction of  $w_T$ via Colmez's operator $w$}} 
In this subsection we give an explicit construction of the involution $w_T$ using Colmez's operator $w_*$ on the $(\varphi, \Gamma)$-module $D=D(T)$ associated to 
a rank two generic symplectic self-dual pair $(R, T)$.

The involution corresponds to the action of $\begin{pmatrix}0 & 1\\ 1 & 0\end{pmatrix}$ on 
the associated representation $\Pi_p(D)$ of $\mathrm{GL}_2(\mathbb{Q}_p)$ by the $p$-adic local Langlands correspondence. 
This construction of $w_T$ is essentially the same as in \S\ref{ss:lsd via Kato's eps} since Kato's epsilon isomorphism $\varepsilon_{R, \zeta}(T) : \mathbf{1}_R\isom \Delta_R(T)$ in Theorem 
\ref{thm, rank two epsilon} is itself defined  \cites{NaKato,J}  using Colmez's operator $w_*$. 

Throughout this subsection, we fix an odd prime $p$. 

\subsubsection{Backdrop}
Let $(R, T)$ be a generic symplectic self dual  $G_{\BQ_p}$-representation.
 Without loss of generality, we assume that $R$ is Artinian local. Indeed, the results in this subsection follow
for a general $R$ from the Artinian local case by taking the limit.

Set $\Gamma := \Gal(\BQ_p (\mu_{p^\infty })/\BQ_p )$, and let $\chi_{\cyc} : \Gamma \simeq 
\BZ_p^\times$ be
the $p$-adic cyclotomic character. 
 For $b \in \BZ_p^\times$, define
$\sigma_b \in \Gamma$ such that $\chi_{\cyc} (\sigma_b) = b$.

Define a topological ring $\mathcal{E}_R=R(\!(X)\!)$ such that $R[\![X]\!]$ is open in $\mathcal{E}_R$ and the induced topology on $R[\![X]\!]$ coincides with the $X$-adic topology.  The ring $\mathcal{E}_R$ is equipped with actions of the Frobenius 
$\varphi$ and $\Gamma:=\mathrm{Gal}(\BQ_p(\zeta_{p^\infty})/\BQ_p)$, which act as $R$-algebra continuous homomorphisms characterized by 
$$\varphi(X)=(1+X)^p-1, \ \ \gamma(X)=(1+X)^{\chi_{\mathrm{cyc}}(\gamma)}-1 \ \ (\gamma \in \Gamma).$$
Let $D=D(T)$ be the \'etale $(\varphi, \Gamma)$-module 
associated to $T$ by the covariant Fontaine functor  \cite{Fo}, namely $D$ is a finite free $\mathcal{E}_R$-module equipped with a Frobenius structure 
$$\varphi : \varphi^*(D)=D\otimes_{\mathcal{E}_R, \varphi}\mathcal{E}_R\isom D$$ and a commuting continuous semilinear $\Gamma$-action. 
Let $\psi$ be the left inverse of $\varphi$ on $D$ defined by  
\[
\psi: D \rightarrow D :  x=\sum_{i=0}^{p-1}(1+X)^i\varphi(x_i) \mapsto x_0. 
\]

We have the following key link between the $(\varphi,\Gamma)$-module $D$ and the Iwasawa cohomology   $H^1_{\rm Iw}(\BQ_p,T):=\varprojlim_{n} H^1(\BQ_p(\zeta_{p^n}),T)$. 
\begin{prop}\label{prop, D-H^1}
There exists a canonical $R[\![\Gamma]\!]$-module isomorphism $$D^{\psi=1}\cong H^1_{\mathrm{Iw}}(\BQ_p, {T}).$$ 
In particular,  we have a canonical $R$-module isomorphism $$D^{\psi=1}/(\gamma-1)\isom H^1(\BQ_p, {T}), $$ 
where $\gamma$ is 
a topological generator of $\Gamma$. 
\end{prop}
This proposition is the content of {\cite{CC} and `in particular' part follows from part 5) of Proposition \ref{a}.

Colmez defined
 \cite[II.1]{ColpLL}
 the $\Gamma$-equivariant $R$-linear operator 
$
w_* : D^{\psi=0} \rightarrow D^{\psi=0, \iota} 
$ by 
\begin{equation}\label{eq:Col-involution}
w_*(x):=\lim_{n\rightarrow +\infty} \sum_{i \in \BZ_p^\times \!\!\mod p^n}
(1+X)^{1/i}\sigma_{-1/i^2}
(\varphi^n\psi^n((1+X)^{-i}x)).
\end{equation}
 Here   
$D^{\psi=0, \iota}$ is $D^{\psi=0}$ but the $\Gamma$-action  is
twisted by the involution $\iota: \Gamma \rightarrow \Gamma, \gamma \mapsto \gamma^{-1}$. 
The operator $w_*$ is related to the action of  $\begin{pmatrix} 0 & 1 
\\ 1 & 0\end{pmatrix} $ on the  $\mathrm{GL}_2(\BQ_p)$-representation 
associated to $T$ under the $p$-adic local Langlands correspondence, and a key to 
Colmez's construction \cite{ColpLL} of $\mathrm{GL}_2(\BQ_p)$-representations from $(\varphi,\Gamma)$-modules. 

The operator $w_{*}$ also underlies the third-named author's proof \cites{NaKato,J} of Kato's local epsilon conjecture for rank two representations of $G_{\BQ_p}$. 
In fact, for de Rham representations, he related the action of $w_*$ to the corresponding epsilon constants.

To relate $D^{\psi=1}$ and $D^{\psi=0}$, consider the map 
\[
1-\varphi: D^{\psi=1} \rightarrow D^{\psi=0}. 
\]
Its kernel is $D^{\varphi=1}$ 
and the image is denoted  by 
$$\mathscr{C}(D):=(1-\varphi)D^{\psi=1}.$$
Our strategy to construct an involution on $H^1 (\BQ_p, T)$: Colmez's operator $w_*$ on $D^{\psi=0}$ indues an involution on $D^{\psi=1}/(\gamma-1)$ (see~Proposition~\ref{prop, D-H^1}).
 
In fact, if $T$ is an irreducible $L$-representation where $L$ is a finite extension of $\BQ_p$, then  
$D^{\varphi=1}=0$ and 
Colmez ~\cite[Prop.~V.2.1]{ColpLL} 
showed\footnote{We have $\delta_D(p)=1$ in the self-dual case.} that $w_*(\mathscr{C}(D)) \subset \mathscr{C}(D)$. 
We generalize Colmez's proof of Proposition V.2.1 in \cite{ColpLL} to families of Galois representations and  reducible cases (see~\S\ref{ss:Col-involution}).

\subsubsection{Preliminaries}
Following Colmez, put
$$D\boxtimes \mathbb{Q}_p:=\{x=(x^{(n)})_{n\in \mathbb{N}}\in D^{\mathbb{N}}\mid \psi(x^{(n+1)})=x^{(n)} (n\geq 0)\},$$
$$D\boxtimes \mathbf{P}^1(=D\boxtimes_{\mathbf{1}} \mathbf{P}^1):=\{(z_1, z_2)\in D^2\mid w_*((1-\varphi\psi)(z_1))=(1-\varphi\psi)(z_2)\}.$$
We also consider 
$$P(\mathbb{Q}_p):=\begin{pmatrix}\mathbb{Q}_p^{\times}& \mathbb{Q}_p 
\\ 0 & 1\end{pmatrix} \subset B(\mathbb{Q}_p):=\begin{pmatrix}\mathbb{Q}_p^{\times}& \mathbb{Q}_p 
\\ 0 & \mathbb{Q}_p^\times\end{pmatrix} \subset \mathrm{GL}_2(\mathbb{Q}_p).$$

Recall that $D\boxtimes \mathbb{Q}_p$ and $D\boxtimes \mathbf{P}^1$ are equipped with an action of $B(\mathbb{Q}_p)$ and $\mathrm{GL}_2(\mathbb{Q}_p)$ 
respectively  with the trivial central character since $T$ is assumed to be symplectic self-dual \cite[II.1]{ColpLL}. For example, we have
$$\begin{pmatrix} a& 0 \\ 0 & 1\end{pmatrix}\cdot(x^{(n)})_{n\in \mathbb{N}}=(\sigma_a(x^{(n)}))_{n\in \mathbb{N}} \ (a\in \mathbb{Z}_p^\times),$$ 
$$\begin{pmatrix} p& 0 \\ 0 & 1\end{pmatrix}\cdot(x^{(n)})_{n\in \mathbb{N}}=(x^{(n+1)})_{n\in \mathbb{N}},$$ 
$$\begin{pmatrix} p^{-1}& 0 \\ 0 & 1\end{pmatrix}\cdot(x^{(n)})_{n\in \mathbb{N}}=(x^{(n-1)})_{n\in \mathbb{N}},$$
 where $x^{(-1)}:=\psi(x^{(0)})$, and   
$$\begin{pmatrix}0 & 1\\ 1 &0\end{pmatrix}\cdot (z_1, z_2)=(z_2, z_1),$$ 
$$\begin{pmatrix} a& 0 \\ 0 & 1\end{pmatrix}\cdot (z_1, z_2)=(\sigma_a(z_1), \sigma_a^{-1}(z_2))\ (a\in \mathbb{Z}_p^\times), $$ 
$$\begin{pmatrix} p& 0 \\ 0 & 1\end{pmatrix}\cdot (z_1, z_2)=(z_1', z_2')$$ with $\psi (z_1')=z_1, z_2'=\psi(z_2)$. 

 The following is the content of \cite[Prop.~II.1.14]{ColpLL}
\begin{prop}
There exists a canonical $B(\mathbb{Q}_p)$-equivariant map 
$$\mathrm{Res}_{\mathbb{Q}_p} : D\boxtimes \mathbf{P}^1\rightarrow D\boxtimes \mathbb{Q}_p : z\mapsto 
\left(\mathrm{Res}_{\mathbb{Z}_p}\left(\begin{pmatrix} p^n & 0 \\ 0 & 1\end{pmatrix}\cdot z\right)\right)_{n\in \mathbb{N}}$$
such that 
$$\mathrm{Ker}(\mathrm{Res}_{\mathbb{Q}_p})=\{(0, z)\mid z\in D^{\mathrm{nr}}\}$$ 
for $D^{\mathrm{nr}}:=\cap_{n=1}^\infty \varphi^n(D)$. 
\end{prop}
We equip $D^{\mathrm{nr}}$ with the structure of a $\mathbb{Q}_p^\times=\begin{pmatrix}\mathbb{Q}_p^{\times}& 0
\\ 0 & 1\end{pmatrix} $-module by\footnote{Note that $\varphi : D^{\mathrm{nr}}\rightarrow D^{\mathrm{nr}}$ is an isomorphism.} 
$$\begin{pmatrix} a& 0 \\ 0 & 1\end{pmatrix}\cdot z=\sigma_a^{-1}(z) \ (a\in \mathbb{Z}_p^\times), \ \ 
\begin{pmatrix} p& 0 \\ 0 & 1\end{pmatrix}\cdot z=\psi(z)(=\varphi^{-1}(z)), \ \ \begin{pmatrix} p^{-1}& 0 \\ 0 & 1\end{pmatrix}\cdot z=\varphi(z).$$
Then the map 
$$D^{\mathrm{nr}}\rightarrow \mathrm{Ker}(\mathrm{Res}_{\mathbb{Q}_p}) : z\mapsto (0,z)$$ is 
a $\mathbb{Q}_p^\times$-equivariant isomorphism. 

Let 
$D^{\sharp}$ be the largest compact $\psi$-stable sub $R[\![X]\!]$-module of $D$ on which $\psi$ is surjective, 
$D^{\natural}$ be the smallest compact sub $R[\![X]\!]$-module of $D$ which is stable by 
$\psi$ and generates $D$ as an $R(\!(X)\!)$-module (see I.3.2 of \cite{ColpLL}). 
\begin{defn}
For $\circ\in\{\sharp, \natural\}$, define 
$$D^{\circ}\boxtimes \mathbb{Q}_p:=\{(x^{(n)})_{n\in \mathbb{N}}\in D\boxtimes \mathbb{Q}_p\mid x^{(n)}\in D^{\circ} \  (n\in \mathbb{N})\}, $$
$$D^{\circ}\boxtimes \mathbf{P}^1:=\{z \in D\boxtimes \mathbf{P}^1\mid  \mathrm{Res}_{\mathbb{Q}_p}(z)\in D^{\circ}\boxtimes \mathbb{Q}_p\}.$$
\end{defn}
Note that the above $R$-submodules are $B(\mathbb{Q}_p)$-stable. In fact, we have the following. 
\begin{thm}[Colmez]
For $\circ\in\{\sharp, \natural\}$, the $R$-submodules
 $D^{\circ}\boxtimes \mathbf{P}^1 \subset D\boxtimes \mathbf{P}^1$ are
stable by $\mathrm{GL}_2(\mathbb{Q}_p)$. 
\end{thm}
\begin{proof} 
It suffices to show that the  
stability holds for a universal (framed) deformation $\BT_{\ov{\rho}}^{\Box}$ over a universal (framed) deformation ring $R_{\ov{\rho}}^{\Box}$ by the proof of  \cite[Prop.~II.2.15]{ColpLL}. 

For the universal case, it is proved in  \cite[Thm.~II.3.3]{ColpLL} under the 
assumption that the corresponding 
universal (framed) deformation 
ring is reduced and $p$-torsion free, and contains 
a Zariski dense subset  consisting of crystalline points in every irreducible component of its generic fiber. 
These properties are established 
for arbitrary universal (framed) deformation 
 rings by the recent work \cite{BIP} of B\"{o}ckle--Iyenger--Pa\v{s}k\=unas. 
\end{proof}

For $\alpha_p=\begin{pmatrix}p^{-1}& 0 \\ 0 & 1\end{pmatrix}$, we have a 
$\Gamma\isom \begin{pmatrix}\mathbb{Z}_p^\times& 0 \\ 0 & 1\end{pmatrix}$-equivariant isomorphism 
\begin{equation}\label{eq, psi natural}
D^{\psi=1}\isom \left(D^{\sharp}\boxtimes \mathbb{Q}_p\right)^{\alpha_p=1} : x\mapsto (x)_{n\in \mathbb{N}}
\end{equation}
since $(D^{\sharp})^{\psi=1}=D^{\psi=1}$ by Proposition II. 4.2 of \cite{Colmira}, via which we identify both sides. 

As composite of $(\mathrm{Res}_{\mathbb{Q}_p})^{\alpha_p=1} : \left(D^{\natural}\boxtimes \mathbf{P}^1\right)^{\alpha_p=1}\rightarrow 
\left(D^{\natural}\boxtimes \mathbb{Q}_p\right)^{\alpha_p=1}$
 and the canonical inclusion $$\left(D^{\natural}\boxtimes \mathbb{Q}_p\right)^{\alpha_p=1}\hookrightarrow 
 \left(D^{\sharp}\boxtimes \mathbb{Q}_p\right)^{\alpha_p=1}=D^{\psi=1}
 \isom H^1_{\mathrm{Iw}}(\mathbb{Q}_p, T),$$ we obtain a $\Gamma\isom \begin{pmatrix}\mathbb{Z}_p^\times& 0 \\ 0 & 1\end{pmatrix}$-equivariant morphism 
 $$\left(D^{\natural}\boxtimes \mathbb{Q}_p\right)^{\alpha_p=1}\rightarrow H^1_{\mathrm{Iw}}(\mathbb{Q}_p, T).$$
 Taking its quotients by $\sigma_a-1$ for a topological generator $a\in \mathbb{Z}_p^\times$ yields an $R$-module morphism 
$$
 \left(D^{\natural}\boxtimes \mathbf{P}^1\right)^{\alpha_p=1}/(\sigma_a-1)\rightarrow H^1_{\mathrm{Iw}}(\mathbb{Q}_p, T)/(\gamma-1)
 \isom H^1(\mathbb{Q}_p,T)
$$
 by part 5) of Proposition~\ref{a}. A key result below is that it is an isomorphism (see~Theorem~\ref{c}). 
 
 We begin with a preliminary. 
 
\begin{lem}\label{b} 
For a topological generator $a\in \mathbb{Z}_p^\times$, 
we have
 $$(D^{\mathrm{nr}})^{\varphi=1, \sigma_a=1}=((D^{\mathrm{nr}})^{\varphi=1})/(\sigma_a-1)=0,$$
 $$ (D^{\mathrm{nr}}/(1-\varphi))/(\sigma_a-1)=(D^{\mathrm{nr}}/(1-\varphi))^{\sigma_a=1}=0.$$
 \end{lem}
 \begin{proof}
 Define a complex $$C^*_{\varphi,\sigma_a}(D^{\mathrm{nr}}) : [D^{\mathrm{nr}}\xrightarrow{(1-\varphi)\oplus(1-\sigma_a)}D^{\mathrm{nr}}\oplus D^{\mathrm{nr}}
 \xrightarrow{(1-\sigma_a, \varphi-1)}D^{\mathrm{nr}}]$$
 of degree $[0,2]$. Denoting its cohomology by $H^i_{\varphi,\sigma_a}(D^{\mathrm{nr}})$, 
   we have 
 $$H^0_{\varphi,\sigma_a}(D^{\mathrm{nr}})=(D^{\mathrm{nr}})^{\varphi=1, \sigma_a=1}, \ \ H^2_{\varphi,\sigma_a}(D^{\mathrm{nr}})
 =(D^{\mathrm{nr}})/(1-\varphi,1-\sigma_a).$$
 Moreover, we have a canonical short exact sequence 
 $$0\rightarrow (D^{\mathrm{nr}})^{\varphi=1}/(\sigma_a-1)\rightarrow H^1_{\varphi,\sigma_a}(D^{\mathrm{nr}})\rightarrow 
 (D^{\mathrm{nr}}/(1-\varphi))^{\sigma_a=1}\rightarrow 0.$$
 
 Hence it suffices to show that $H^i_{\varphi,\sigma_a}(D^{\mathrm{nr}})=0$ for all $i$. Since $D^{\mathrm{nr}}$ is a $\mathbb{Z}_p$-module of finite length\footnote{Recall that $R$ is Artinian local by assumption.} 
 by Proposition I.3.3 of \cite{ColpLL}, we have 
 $$\sum_{i=0}^2(-1)^i\mathrm{length}_{\mathbb{Z}_p}(H^i_{\varphi,\sigma_a}(D^{\mathrm{nr}}))=0.$$ Therefore, it suffices to show that 
 $H^i_{\varphi,\sigma_a}(D^{\mathrm{nr}})=0$ for $i\in\{0,2\}$. 
 
 For $i=0$, we have $$H^i_{\varphi,\sigma_a}(D^{\mathrm{nr}})=(D^{\mathrm{nr}})^{\varphi=1,\sigma_a=1}=D^{\varphi=1,\sigma_a=1}\isom 
 H^0(\mathbb{Q}_p, T)=0.$$ 
 Hence, it remains to show that $H^2_{\varphi,\sigma_a}(D^{\mathrm{nr}})=0$. 
 
 When $\overline{D}=D(\overline{T})$ is absolutely irreducible, 
 we have $D^{\mathrm{nr}}=0$ since $\overline{D}^{\mathrm{nr}}=0$. Since the functor $D\mapsto D^{\mathrm{nr}}$ is left exact by Proposition I.3.3 of \cite{ColpLL}, it follows that $H^2_{\varphi,\sigma_a}(D^{\mathrm{nr}})=0$.

 When $\overline{D}$ is not absolutely irreducible, 
 it sits in the 
 following exact sequence
 $$0\rightarrow \mathcal{E}_{\mathbb{F}}(\delta_1)\rightarrow \overline{D}\rightarrow \mathcal{E}_{\mathbb{F}}(\delta_2)
 \rightarrow 0$$ for some characters 
 $\delta_1, \delta_2 : \mathbb{Q}_p^{\times}\rightarrow \mathbb{F}^{\times}$ such that $\delta_1\delta_2=\overline{\chi}_p$, the latter being mod $p$ cyclotomic character. Here, for any character $\delta : \mathbb{Q}_p^{\times}\rightarrow \mathbb{F}^{\times}$, 
 $\mathcal{E}_{\mathbb{F}}(\delta)=\mathcal{E}_{\mathbb{F}}e_{\delta}$ is a $(\varphi,\Gamma)$-module over $\mathcal{E}_{\mathbb{F}}$
 defined by $$\varphi(e_{\delta})=\delta(p)e_{\delta}, \gamma(e_{\delta})=\delta(\ov{\chi}_p(\gamma))e_{\delta}.$$
 Since $\mathcal{E}_{\mathbb{F}}(\delta)^{\mathrm{nr}}=\mathbb{F}(\delta)$ for any character $\delta : \mathbb{Q}_p^\times \rightarrow \mathbb{F}^\times$, 
 we obtain an exact sequence 
 \begin{equation}\label{eq:ex}
 0\rightarrow \mathbb{F}(\delta_1)\rightarrow \overline{D}^{\mathrm{nr}}\rightarrow \mathbb{F}(\delta_2)
 \end{equation}
 of $\mathbb{F}[\varphi,\Gamma]$-modules. Hence, the set of Jordan--H\"older factors of $D^{\mathrm{nr}}$ is a subset of 
 $\{\mathbb{F}(\delta_1), \mathbb{F}(\delta_2)\}$. 
 
 If $\delta_1, \delta_2\not=\mathbf{1}$, then 
 $H^2_{\varphi,\sigma_a}(D^{\mathrm{nr}})=0$ since $H^2_{\varphi,\sigma_a}(\mathbb{F}(\delta))=0$ for any $\delta\not=\mathbf{1}$. 
 If $\delta_1$ or $\delta_2$ equals $\mathbf{1}$, then the other character is $\overline{\chi}_p$. 
 As $T$ is generic, we may assume that $\delta_1=\overline{\chi}_p$.  
  Then the exact sequence \eqref{eq:ex} becomes 
 $$0\rightarrow  \mathcal{E}_{\mathbb{F}}(1)\rightarrow \overline{D}\rightarrow  \mathcal{E}_{\mathbb{F}}\rightarrow 0,$$
which is non-split since $H^0_{\varphi,\sigma_a}(\overline{D})=H^0(\mathbb{Q}_p, \overline{T})=0$ by the generic assumption. 
We claim that the natural inclusion $\mathbb{F}(1)\rightarrow \overline{D}^{\mathrm{nr}}$ is bijective. If not, then we have 
the short exact sequence 
$$0\rightarrow \mathbb{F}(1)\rightarrow \overline{D}^{\mathrm{nr}}\rightarrow \mathbb{F}\rightarrow 0.$$
Taking the associated long exact sequence for $H^i_{\varphi,\sigma_a}(-)$, 
 we obtain the following exact sequence
$$0\rightarrow H^0_{\varphi,\sigma_a}( \mathbb{F}(1))\rightarrow H^0_{\varphi,\sigma_a}(\overline{D}^{\mathrm{nr}})\rightarrow 
H^0_{\varphi,\sigma_a}( \mathbb{F})\rightarrow 
H^1_{\varphi,\sigma_a}( \mathbb{F}(1))\rightarrow\cdots .$$
Since $H^i_{\varphi,\sigma_a}( \mathbb{F}(1))=0$ for all $i$, it follows that 
$$H^0_{\varphi,\sigma_a}(\overline{D}^{\mathrm{nr}})\isom 
H^0_{\varphi,\sigma_a}( \mathbb{F})=\mathbb{F},$$
which contradicts our assumption that $H^0_{\varphi,\sigma_a}(\overline{D}^{\mathrm{nr}})=H^0_{\varphi,\sigma_a}(\overline{D})=0$. 
Therefore, we have $\mathbb{F}(1)=\overline{D}^{\mathrm{nr}}$. Hence we conclude that $H^2_{\varphi,\sigma_a}(D^{\mathrm{nr}})=0$ since 
$H^2_{\varphi,\sigma_a}(\mathbb{F}(1))=0$. 
 \end{proof}
 \begin{thm}\label{c}
For a topological generator $a\in \mathbb{Z}_p^\times$, 
the map  $$\left(D^{\natural}\boxtimes \mathbf{P}^1\right)^{\alpha_p=1}/(\sigma_a-1)\rightarrow 
 H^1(\mathbb{Q}_p,T)$$ is an isomorphism. 
 \end{thm}
 \begin{proof}
 
 In view of Proposition \ref{prop, D-H^1} and \eqref{eq, psi natural}, 
 it suffices to show that the maps 
 $$\left(D^{\natural}\boxtimes \mathbf{P}^1\right)^{\alpha_p=1}/(\sigma_a-1)\rightarrow 
 \left(D^{\natural}\boxtimes \mathbb{Q}_p\right)^{\alpha_p=1}/(\sigma_a-1)$$ 
 induced by $\mathrm{Res}_{\mathbb{Q}_p}$ and 
 $$\left(D^{\natural}\boxtimes \mathbb{Q}_p\right)^{\alpha_p=1}/(\sigma_a-1)\rightarrow 
 \left(D^{\sharp}\boxtimes \mathbb{Q}_p\right)^{\alpha_p=1}/(\sigma_a-1)$$
 induced by the inclusion $\left(D^{\natural}\boxtimes \mathbb{Q}_p\right)^{\alpha_p=1}\hookrightarrow 
 \left(D^{\sharp}\boxtimes \mathbb{Q}_p\right)^{\alpha_p=1}$ are isomorphisms.

 As for the latter map, we utilize the following $\begin{pmatrix}\mathbb{Q}_p^\times& 0 \\ 0 & 1\end{pmatrix}$-equivariant exact sequence 
 $$0\rightarrow D^{\natural}\boxtimes \mathbb{Q}_p\hookrightarrow 
 D^{\sharp}\boxtimes \mathbb{Q}_p\xrightarrow{(x^{(n)})_{n\in \mathbb{N}}\mapsto x^{(0)}} D^{\sharp}/D^{\natural}\rightarrow 0,$$
as in Proposition I.3.10 of \cite{ColpLL}, where $\begin{pmatrix}\mathbb{Q}_p^\times& 0 \\ 0 & 1\end{pmatrix}$ acts on $D^{\sharp}/D^{\natural}$ by\footnote{Note that $\psi$ is bijective on $D^{\sharp}/D^{\natural}$.}
$$\begin{pmatrix}a& 0 \\ 0 & 1\end{pmatrix}\cdot \overline{x}=\overline{\sigma_a(x)} \  (a\in \mathbb{Z}_p^\times), \ \  
 \begin{pmatrix}p^{-1}& 0 \\ 0 & 1\end{pmatrix}\cdot \overline{x}=\overline{\psi(x)}$$  for 
 $\overline{x}\in D^{\sharp}/D^{\natural}$. 
 Since $\alpha_p-1$ is surjective on $D^{\natural}\boxtimes \mathbb{Q}_p$ by  \cite[Prop.~I.3.16 (i)]{ColpLL}, we obtain the following 
 $\Gamma\isom \begin{pmatrix}\mathbb{Z}_p^\times& 0 \\ 0 & 1\end{pmatrix}$-equivariant 
 short exact sequence
 $$0\rightarrow \left(D^{\natural}\boxtimes \mathbb{Q}_p\right)^{\alpha_p=1}\rightarrow 
 \left(D^{\sharp}\boxtimes \mathbb{Q}_p\right)^{\alpha_p=1}\rightarrow (D^{\sharp}/D^{\natural})^{\psi=1}\rightarrow 0.$$
 Hence, it suffices to show that 
 $((D^{\sharp}/D^{\natural})^{\psi=1})^{\sigma_a=1}=(D^{\sharp}/D^{\natural})^{\psi=1}/(\sigma_a-1)=0$. 
 Since there exists a perfect pairing 
 $$\{\ , \} : D^{\sharp}/D^{\natural}\times D^{\mathrm{nr}}\rightarrow \mathbb{Q}_p/\mathbb{Z}_p$$ such that 
 $$\{\sigma_a(x), y\}=\{x, \sigma_a^{-1}(y)\}  \text{ and } \{\psi(x), y\}=\{x, \varphi(y)\}$$  for $x\in D^{\sharp}/D^{\natural}$ and 
 $y\in D^{\mathrm{nr}}\isom D^*(1)^{\mathrm{nr}}$  (cf.~\cite[Prop.~I.3.4 (ii)]{ColpLL}), we obtain 
 $$(((D^{\sharp}/D^{\natural})^{\psi=1})^{\sigma_a=1})^{\vee}\isom (D^{\mathrm{nr}}/(1-\varphi))/(\sigma_a-1)=0$$
 and 
 $$((D^{\sharp}/D^{\natural})^{\psi=1}/(\sigma_a-1))^{\vee}\isom (D^{\mathrm{nr}}/(1-\varphi))^{\sigma_a=1}=0$$
 by Lemma \ref{b}.
 
 It remains to show that $\left(D^{\natural}\boxtimes \mathbf{P}^1\right)^{\alpha_p=1}/(\sigma_a-1)\rightarrow 
 \left(D^{\natural}\boxtimes \mathbb{Q}_p\right)^{\alpha_p=1}/(\sigma_a-1)$ is an $R$-module isomorphism.  We note that 
 the map $\mathrm{Res}_{\mathbb{Q}_p}$ induces the following $\mathbb{Q}_p^\times$-equivariant exact sequence
$$0\rightarrow D^{\mathrm{nr}}\xrightarrow{z\mapsto (0,z)}D^{\natural}\boxtimes \mathbf{P}^1\xrightarrow{\mathrm{Res}_{\mathbb{Q}_p}} D^{\natural}\boxtimes \mathbb{Q}_p\rightarrow 0$$
by (the same proof as in) Remarque II.2.3 (ii) of \cite{ColpLL}. In turn we obtain the following $\Gamma$-equivariant exact sequence
$$0\rightarrow (D^{\mathrm{nr}})^{\varphi=1}\rightarrow (D^{\natural}\boxtimes \mathbf{P}^1)^{\alpha_p=1}\rightarrow (D^{\natural}\boxtimes \mathbb{Q}_p)^{\alpha_p=1}
\rightarrow D^{\mathrm{nr}}/(1-\varphi).$$
Denote the image of $(D^{\natural}\boxtimes \mathbb{Q}_p)^{\alpha_p=1}
\rightarrow D^{\mathrm{nr}}/(1-\varphi)$ by $M$. Since  
$$M^{\sigma_a=1}\subset (D^{\mathrm{nr}}/(1-\varphi))^{\sigma_a=1}=0$$ by Lemma \ref{b}, 
it follows that $M^{\sigma_a=1}=0$, and so $M/(\sigma_a-1)=0$. 
In view of this vanishing and the fact that $(D^{\mathrm{nr}})^{\varphi=1}/(1-\sigma_a)=0$ (cf.~Lemma \ref{b}), it follows that 
the map \[\left(D^{\natural}\boxtimes \mathbf{P}^1\right)^{\alpha_p=1}/(\sigma_a-1)\rightarrow 
 \left(D^{\natural}\boxtimes \mathbb{Q}_p\right)^{\alpha_p=1}/(\sigma_a-1)\]
  is bijective by a simple diagram chase. 
  \end{proof}
\subsubsection{The involution}\label{ss:Col-involution}
In this subsection we construct the sought after involution on $H^1(\BQ_p,T)$.

Put $w=\begin{pmatrix}0 & 1\\ 1& 0\end{pmatrix}$. 
Since $\alpha_p^{-1}\cdot w=\begin{pmatrix}p & 0 \\ 0 & p\end{pmatrix}\cdot w\cdot\alpha_p$, 
we have\footnote{Note that $D^{\natural}\boxtimes \mathbf{P}^1$ has trivial central character.} 
$$\alpha_p^{-1}(w(x))=
\begin{pmatrix}p & 0 \\ 0 & p\end{pmatrix}(w(\alpha_p(x)))
=w(\alpha_p(x))=w(x)$$
for any $x\in \left(D^{\natural}\boxtimes\mathbf{P}^1\right)^{\alpha_p=1}=
\left(D^{\natural}\boxtimes \mathbf{P}^1\right)^{\alpha_p^{-1}=1}$. 

Hence $w$ induces an involution 
$$w : \left(D^{\natural}\boxtimes \mathbf{P}^1\right)^{\alpha_p=1}\rightarrow \left(D^{\natural}\boxtimes \mathbf{P}^1\right)^{\alpha_p=1}$$
which satisfies $w(\sigma_a(x))=\sigma_a^{-1}(w(x))$. In turn, it induces an $R$-module involution
$$w : \left(D^{\natural}\boxtimes \mathbf{P}^1\right)^{\alpha_p=1}/(\sigma_a-1)\rightarrow \left(D^{\natural}\boxtimes \mathbf{P}^1\right)^{\alpha_p=1}/(\sigma_a-1).$$ 
\begin{thm}\label{thm, involution}
Let $(R, T)$ be a generic symplectic self-dual $G_{\BQ_p}$-representation of rank two. Then 
the above involution $w$ is compatible with Colmez's operator $w_* : D^{\psi=0}\rightarrow D^{\psi=0}$ as in \eqref{eq:Col-involution}. 
Moreover, it induces an $R$-module involution  
$$w'_T : H^1(\mathbb{Q}_p, T)\rightarrow H^1(\mathbb{Q}_p, T)$$
by Theorem \ref{c}, 
which coincides with the involution $w_T$ 
as in \S\ref{subsubsection, rank 2 kato self-dual}.
\end{thm}
\begin{proof} 
The compatibility with $w_*$ follows from the proof of  Proposition V.2.1 of \cite{ColpLL}. 

By \S3.2.2 and pp.~313-314  of \cite{NaKato}, we see that $w'_T$ also satisfies the characterization 
property \eqref{equation, comparison of pairings} of $w_T$. Hence, the two constructions coincide. 
\end{proof}
\subsection{Proofs of Theorems~\ref{thm, main} and \ref{thm, main3}}
\begin{proof} The existence of the local sign decomposition as in Theorem \ref{thm, main} is a consequence of 
  the above two different constructions of the involution $w_T$, which coincide by Theorem~\ref{thm, involution}.

The functorial property 2) is clear from the second construction of the local sign decomposition since Colmez's operator $w_*$ is functorial. 
The property 3) follows from part 4) of Proposition \ref{a} and the functorial property of $w_*$, or,  from the property 1) of Conjecture \ref{conj, epsilon} which is 
proved in Theorem \ref{thm, rank two epsilon}. 
The first property follows from Propositions \ref{prop, decomposition and Lagrangian}. 

 As for the fourth property, note that $H^1_{\hat{\varepsilon}(V)}(\BQ_p ,V)=H^1_{\rm f}(\BQ_p,V)$ by Proposition \ref{prop, de Rham property}. For a $G_{\BQ_p}$-stable lattice $T\subset V$, we have $H^1_{\hat{\varepsilon}(V)}(\BQ_p ,T) \subset H^1_{\rm f} (\BQ_p , T)$. Since both are Lagrangian $R$-submodules of $H^1 (\BQ_p , T)$ of rank one, we conclude that 
 $H^1_{\hat{\varepsilon}(V)}(\BQ_p ,T) = H^1_{\rm f} (\BQ_p , T)$ by  
 Lemma~\ref{prop, lagrangian}.  
 
Similarly, 
the involution $w_{T\otimes_{\mathbb{Z}_p}\mathbb{Q}_p}$ defined in \S 4.2 and Proposition \ref{a3} lead to Theorem~\ref{thm, main3}. 
\end{proof}
\begin{remark}
 The above verification of the de Rham property 4)
relies on Kato's local epsilon conjecture. 
It is based on \cite[Prop.~3.14]{NaKato}, 
which was the key and the deepest part of the third-named author's proof of the conjecture for rank two  representations.
\end{remark}

\section{Characterization of the signed submodules}\label{s: sgn submodules}

In this section, we give a characterization of the signed submodules appearing in the local sign decomposition 
via  de Rham specializations (see Proposition~\ref{prop, decomp by using Z}), and an application: 
 the vanishing of the universal norm subgroup over a universal local deformation ring (see Theorem~\ref{thm, univ}). It relies on  Zariski density of crystalline points with a given sign on the deformation ring as established in appendix \ref{s:appendix}. We also give some examples of the signed submodules for reducible representations (see Proposition \ref{prop, reducible sign}).

The characterization will be used in our proof of Rubin-type conjectures (see~section~\ref{s:Rubin}), and the reducible example will be used in our compatibility of Mazur--Rubin and Deligne--Langlands local constants  for non-generic representations (see~\S\ref{subsubsection, anomalous}). 

In this section we fix an odd prime $p$.

\subsection{de Rham characterization and universal norms}
The characterization of signed submodules 
as in Proposition~\ref{prop, uniqueness}
employs one de Rham specialization and rigidity of Lagrangian submodules of modules of rank two
(cf.~Lemma~\ref{prop, lagrangian}). 
In this subsection  we show that sufficiently many de Rham specializations also characterize 
the decomposition without the Lagrangian property. 

\subsubsection{The characterization}
Let $(R, T)$ be a generic symplectic self-dual $G_{\BQ_p}$-representation. 
For a de Rham specialization $s$ given by $R \rightarrow R_s$, put $T_s:=T\otimes_R R_s$. 
\begin{defn}
For $\varepsilon \in \{\pm\}$, 
define an $R$-submodule $Z_{\varepsilon}(T)\subset H^1(\BQ_p, T)$ by 
\[
Z_\varepsilon(T):=\{v \in H^1(\BQ_p, T)\,|\, s(v) \in H^1_\mathrm{f}(\BQ_p, T_s) \; \text{for any de Rham specialization $s$ 
with  $\hat{\varepsilon}(T_s)=-\varepsilon \cdot 1$} \}. 
\]
The universal norm subgroup $H^1_{\mathrm{f}}(\BQ_p, T)$  of $H^1(\BQ_p, T)$ is defined by\footnote{If $T$ is de Rham, then $H^1_{\rm f}(\BQ_p,T)$ coincides with the Bloch--Kato subgroup.} 
\[H^1_{\mathrm{f}}(\BQ_p, T):=Z_+(T)\cap Z_-(T),\]
i.e. it consists of elements $v$ whose any de Rham specialization $s(v)$ belongs to $H^1_{\rm f}(\BQ_p,T_s)$.
\end{defn}
The main result of this subsection is the following. 
\begin{prop}\label{prop, decomp by using Z}
Let $(R, T)$ be a generic symplectic self-dual $G_{\BQ_p}$-representation of rank two and 
$\varepsilon \in \{\pm \}$. 
 Suppose that the de Rham specializations with sign $-\varepsilon$ are Zariski dense for the pair $(R, T)$. 
 Then $$Z_\varepsilon(T)=H^1_{\varepsilon}(\BQ_p, T).$$ 
 In particular,  if the $\varepsilon$-de Rham specializations are Zariski dense for any $\varepsilon \in\{\pm \}$,  
 then $H^1_{\mathrm{f}}(\BQ_p, T)=\{0\}$. 
\end{prop}

The proof of Proposition \ref{prop, decomp by using Z} in the next subsection yields the following. 

\begin{cor}\label{cor, decomp by using Z}
Let $S$ be a subset of the set of specializations of $R$.  
For $\varepsilon \in\{\pm \}$, 
put 
\[
Z_{S, \varepsilon}(T):=\{v \in H^1(\BQ_p, T)\,|\, s(v) \in H^1_\mathrm{f}(\BQ_p, T_s) \; \text{for any de Rham specialization $s \in S$ 
with $\hat{\varepsilon}(T_s)=-\varepsilon \cdot 1$} \}. 
\]
Suppose that $S$ contains 
Zariski dense de Rham specializations with sign $-\varepsilon$.  
Then 
\[Z_{S, \varepsilon}(T)=Z_{\varepsilon}(T)=H^1_\varepsilon(\BQ_p,T).\] 
\end{cor}

\subsubsection{Proof of Proposition~\ref{prop, decomp by using Z}}

\begin{prop}\noindent
\begin{itemize}
\item[i)] We have  $H^1_{\pm}(\BQ_p, T)\subset Z_{\pm}(T)$. 
\item[ii)]The natural map induces surjective $R$-module morphisms  
\[
H^1_\pm(\BQ_p, T)\rightarrow Z_\pm(T)/H^1_{\mathrm{f}}(\BQ_p, T). 
\]
In particular, if  the universal norm subgroup $H^1_{\mathrm{f}}(\BQ_p, T)$ is trivial, then 
 $H^1_\pm(\BQ_p, T)=Z_\pm(T)$.
\end{itemize}
\end{prop}
\begin{proof}
i) Suppose that  $\hat{\varepsilon}(T_s)=-\varepsilon$. Then we have 
$H^1_\mathrm{f}(\BQ_p, T_s)=
H^1_{\varepsilon}(\BQ_p, T_s)$ by part 4) of Theorem~\ref{thm, main}. Hence $x \in H^1_\varepsilon(\BQ_p, T)$ implies that $s(x) \in H^1_\mathrm{f}(\BQ_p, T_s)$, concluding the proof. 

ii) We only discuss the plus case. 
Take an element $x \in Z_+(T)$, and consider the decomposition $x=x_++x_-$ arising from the local sign decomposition. 
By part i), we have $$x_-=x-x_+\in Z_+(T)\cap H^1_-(\BQ_p, T) \subset  Z_+(T)\cap Z_-(T).$$ 
Hence $x\equiv x_+$ modulo a universal norm. 
 \end{proof}

\begin{proof}[Proof of Proposition~\ref{prop, decomp by using Z}]
Fix an $R$-basis $v_\pm$ of $H^1_{\pm}(\BQ_p, T)$. 

Let $x \in Z_{\varepsilon}(T)$ and consider its local sign decomposition $$x=x_++x_-=r_+v_++r_-v_-$$ 
for $r_{\pm} \in R$. Let $s$ be a de Rham specialization with sign $-\varepsilon$. 
Then we have $s(x) \in H^1_{\mathrm{f}}(\BQ_p, T_s)=H^1_{\varepsilon}(\BQ_p, T_s)$, and so 
\[
s(x_{-\varepsilon})=s(x-x_{\varepsilon})\in H^1_{-\varepsilon}(\BQ_p, T_s)\cap H^1_{\varepsilon}(\BQ_p, T_s)=\{0\}.
\]
Hence, it follows that 
$r_{-\varepsilon}=0$  by the  Zariski density assumption. 
We then have $x=x_\varepsilon \in H^1_\varepsilon(\BQ_p,T)$, concluding the proof.  

The vanishing of $H^1_{\rm f}(\BQ_p,T)$ follows from the property $H^1_+(\BQ_p, T) \cap H^1_-(\BQ_p, T)=\{0\}$ of the local sign decomposition. 
 \end{proof}

\subsubsection{An example: universal deformations}\label{ss: sgn examples}
\begin{thm}\label{thm, univ}
Let $\ov{\rho}: G_{\BQ_p} \ra \Aut_{\BF}(\ov{T})$ be a symplectic self-dual two-dimensional representation over a finite field $\BF$ of characteristic $p$.
 Let $R_{\ov{\rho}}^{\Box}$ be a universal framed deformation ring classifying framed lifts of $\ov{\rho}$ with determinant $\chi_\cyc$ and $\BT_{\ov{\rho}}^{\Box}$ the universal deformation. 
 Then for $\varepsilon \in \{\pm \}$ we have 
 $$
 Z_{\varepsilon}(\BT_{\ov{\rho}}^{\Box})=H^{1}_{\varepsilon}(\BQ_p,\BT_{\ov{\rho}}^{\Box}).
 $$
In particular, $$H^1_{\mathrm{f}}(\BQ_p,\BT_{\ov{\rho}}^{\Box})=0,$$ 
and the local sign decomposition 
is uniquely determined by the properties 2)-4) of Theorem~\ref{thm, main}.  
\end{thm}
\begin{proof}
In light of the Zariski density of crystalline points with a given sign on $R_{\ov{\rho}}^{\Box}$ as established in Theorem~\ref{thm, density sgn crystalline}, the assertion is a consequence of Proposition~\ref{prop, decomp by using Z}. 
\end{proof}
\begin{remark} 
Nekov\'a\v{r} conjectured that the universal norm subgroup is torsion for cyclotomic deformation of a non-ordinary representation, as later shown by Perrin-Riou \cite{PR-norm} and Berger \cite{Ber05}. 
The above setting of universal deformations complements the 
(conjectures in the)
 literature. 
\end{remark}
\subsection{A signed submodule of reducible representations}

\subsubsection{A preliminary}
\begin{lem}\label{lemma, reducible lagrangian}
 Let $(R,T)$ be a generic symplectic self-dual $G_{\BQ_p}$-representation of rank $2n$. 
 Suppose that $T$ has a $G_{\BQ_p}$-subrepresentation $T_1$ that 
is a Lagrangian $R$-submodule 
 with respect to the symplectic pairing associated with $(R, T)$. 
 Put $T_2:=T/T_1$. 
Suppose that either 
$H^0(\BQ_p, \overline{T}_2)=0$, or $R$ is a PID and $H^1(\BQ_p, T_2)$ is free. 
Then the image of  the natural map 
$$H^1(\BQ_p, T_1) \rightarrow H^1(\BQ_p, T)$$
 is a Lagrangian submodule of $H^1(\BQ_p, T)$. 
\end{lem}
\begin{proof}
 First, we show that the image is an isotropic submodule. 
 
In view of  
 the Lagrangian property and self-duality, we have $T_1^*(1) \cong T_2$. 
This leads to the following commutative diagram, 
 which is induced by the cup product and the symplectic self-duality: 
 \[
\xymatrix{
(\;,\;): \; H^1(\mathbb{Q}_p, T_1) \!\!\!\!\!\!\!\!\!\!\!\!\ar@<5.5ex>[d]^{p_1} &\!\!\!\!\!\!\!\! \!\!\!\!\times\!\!\!\!\!\!\!\! \!\!\!\! & \!\!\!\!\!\!\!\!\!\!\!\!H^1(\mathbb{Q}_p, T_2)  \ar[r] & R \\ 
 (\;,\;):\; H^1(\mathbb{Q}_p, T)  \!\!\!\!\!\!\!\!\!\!\!\! & \!\!\!\!\!\!\!\!\!\!\!\!\times\!\!\!\!\!\!\!\!\!\!\!\!   & \!\!\!\!\!\!\!\!\!\!\!\!H^1(\mathbb{Q}_p, T) \ar@<3.0ex>[u]_{p_2}\ar[r] & R. 
}
\]
Then for $x \in H^1(\mathbb{Q}_p, T_1)$ and $y \in H^1(\mathbb{Q}_p, T)$, 
we have $(p_1(x), y)=(x, p_2(y))$. Hence, $(p_1(x), p_1(x))=0$. 

Assume that $H^0(\BQ_p, \overline{T}_2)=0$. Then as in the proof of part 3) of Proposition \ref{a}, 
$H^1(\BQ_p, {T}_2)$ is a free $R$-module of rank $n$. 
We also have $H^2(\BQ_p, {T}_1)=0$.  
Hence, the cokernel of $H^1(\BQ_p, T_1) \rightarrow H^1(\BQ_p, T)$ is free of rank $n$. 
Next, assume that $R$ is a PID and $H^1(\BQ_p, T_2)$ is free. 
Then, the cokernel of $H^1(\BQ_p, T_1) \rightarrow H^1(\BQ_p, T)$ is free. 
Since $H^0(\BQ_p,T_1)\subset H^0(\BQ_p,T)=0$, the Euler--Poincar\'e formula shows that 
the image of $H^1(\BQ_p, T_1) \ra H^1(\BQ_p,T)$ has free rank $n$. 
Hence, it 
   is a direct summand of rank $n$ and the proof concludes. 
\end{proof}

\subsubsection{A signed submodule}\label{ss, sgn sub}
The purpose of this subsection is to prove the following. 

\begin{prop}\label{prop, reducible sign}
Let $T$ be a generic symplectic self-dual de Rham representation of $G_{\BQ_p}$
of rank $2$ with coefficients in $\CO$, the integer ring of a finite extension of $\BQ_p$. 
Suppose that there exists an exact sequence 
$$0\rightarrow T_1 \rightarrow T \rightarrow T_2\rightarrow 0$$
of $G_{\BQ_p}$-representations 
for $T_i$ ($i=1,2$) free of rank $1$ and 
the Hodge--Tate weight of $T_1$  positive. 
\begin{itemize}
\item[i)] Suppose that $T_2$ is the trivial representation $\mathcal{O}$. Then the natural map 
$$H^1(\mathbb{Q}_p, T_1)\rightarrow H^1(\mathbb{Q}_p, T)$$ induces an $\CO$-module isomorphism 
$$ H^1(\mathbb{Q}_p, T_1)/\mathrm{Im}(\delta)\isom H^1_+(\mathbb{Q}_p, T),$$
where $\delta : H^0(\mathbb{Q}_p, T_2)\rightarrow H^1(\mathbb{Q}_p, T_1)$ is the boundary map. In particular, if $T$ is crystalline, then  
 the natural map induces 
$H^1_{\rm{f}}(\mathbb{Q}_p, T)\isom H^1_{\rm{f}}(\mathbb{Q}_p, T_2)$, and 
 if $T$ is non-crystalline, then 
 $H^1_{\rm{f}}(\mathbb{Q}_p, T_1)\isom H^1_{\rm{f}}(\mathbb{Q}_p, T)$. 
 \item[ii)] Suppose that $T_2$ is a non-trivial representation of $G_{\BQ_p}$. 
Write the character associated to $\overline{T}_1$ as $\omega^r\mu$ 
for $\omega$ the Teichm\"uller character, $\mu$ unramified and put $\epsilon=(-1)^r$. Then we have 
 $$H^1_{\rm{f}}(\mathbb{Q}_p, T)=H^1_{\epsilon}(\mathbb{Q}_p, T),$$
and  
 the natural map 
$$H^1(\mathbb{Q}_p, T_1)=H^1_{\rm{f}}(\mathbb{Q}_p, T_1)\rightarrow H^1_{\rm{f}}(\mathbb{Q}_p, T)=H^1_{\epsilon}(\mathbb{Q}_p, T)$$
is injective and of finite index. It is surjective if and only if $H^0(\BQ_p, \overline{T}_2)=0$. 
\end{itemize}
\end{prop}
\begin{proof}

i) Assume that $T_2=\mathcal{O}$, and so $T_1=\mathcal{O}(1)$ by the symplectic self-duality. 

Since $T_1$ is of rank one, it is Lagrangian for the symplectic pairing on $T$. 
Thus, $$X:=H^1(\mathbb{Q}_p, T_1)/\mathrm{Im}(\delta)$$ is a Lagrangian $\mathcal{O}$-submodule of 
$H^1(\mathbb{Q}_p, T)$ by Lemma \ref{lemma, reducible lagrangian}. 
Note that $H^1(\mathbb{Q}_p, \mathcal{O})$ is free.  
Hence, in light of the local sign decomposition for $T$ and Lemma~\ref{prop, lagrangian}, it suffices to show that 
$X \subseteq H^1_+(\mathbb{Q}_p, T),$ 
or equivalently
$$
X \cap H^1_-(\mathbb{Q}_p, T)=0$$ 
in $H^1(\mathbb{Q}_p, T)$.

Let $z\in H^1(\mathbb{Q}_p, \mathcal{O}(1))$ be the extension class 
corresponding to the given one 
$$\alpha:=[0\rightarrow \mathcal{O}(1)\rightarrow T\rightarrow \mathcal{O}\rightarrow 0].$$ 
Note that $\mathrm{Im}(\delta)=\mathcal{O}z.$ 
Then taking the long exact sequences associated to $\alpha$ and $\mathbf{B}_{\mathrm{crys}}\otimes_{\mathbb{Z}_p}\alpha$, and letting $V=T\otimes_{\BZ_p}\BQ_p$, we obtain the following 
commutative diagram with exact horizontal maps: 
\[
\xymatrix{
H^0(\mathbb{Q}_p, T)=0 \ar[d] \ar[r]& H^0(\mathbb{Q}_p, \mathcal{O})=\mathcal{O} \ar[d] \ar[r]^-{\delta}  &H^1(\mathbb{Q}_p, \mathcal{O}(1)) \ar[d]\ar[r]^-{\mathrm{can}} & H^1(\mathbb{Q}_p, T) \ar[d]\\ D_{\mathrm{crys}}(V) \ar[r]^-{\delta_1} & D_{\mathrm{crys}}(L)=L \ar[r]^-{\delta_2}& H^1(\mathbb{Q}_p, \mathbf{B}_{\mathrm{crys}}\otimes_{\mathbb{Q}_p}L(1)) \ar[r]^-{\delta_3} &
H^1(\mathbb{Q}_p, \mathbf{B}_{\mathrm{crys}}\otimes_{\mathbb{Q}_p}V).   \\
}
\]

Note that $V$ is crystalline or non-crystalline semistable with  Hodge--Tate weights $(0,1)$
since it is an extension of semistable representations and 
$H^1_{\rm g}(\BQ_p,\mathbb{Q}_p(1))=H^1(\BQ_p,\mathbb{Q}_p(1))$.
(cf. \cite[\S 6]{Ber}.)

We first assume that $V$ is crystalline. Then 
$$\mathcal{O}z\subseteq H^1_{\rm f}(\mathbb{Q}_p, \mathcal{O}(1))$$
and the map $\delta_3$ is injective since $\delta_1$ is surjective. 
The above inclusion is in fact an equality $$\mathcal{O}z=H^1_{\rm f}(\mathbb{Q}_p, \mathcal{O}(1))$$
as $H^1(\mathbb{Q}_p, \mathcal{O}(1))/\mathcal{O}z$ is torsion-free since so is $H^1(\mathbb{Q}_p, T)$. 
Note that the inverse image of $H^1_{\rm f}(\mathbb{Q}_p, T)$ by the map $H^1(\mathbb{Q}_p, \mathcal{O}(1))\rightarrow H^1(\mathbb{Q}_p, T)$ is 
$H^1_{\rm f}(\mathbb{Q}_p, \mathcal{O}(1))$ since $\delta_3$ is injective. Therefore, we obtain 
$$
X\cap H^1_-(\mathbb{Q}_p, T)
=(H^1(\mathbb{Q}_p, \mathcal{O}(1))/H^1_{\rm f}(\mathbb{Q}_p, \mathcal{O}(1)))\cap H^1_{\rm f}(\mathbb{Q}_p, T)=0$$
 in $H^1(\mathbb{Q}_p, T)$ as $H^1_-(\mathbb{Q}_p, T)=H^1_{\rm f}(\mathbb{Q}_p, T)$ by the property 4) of Theorem \ref{thm, main}. Here we note that 
 $\hat{\varepsilon}(V)=\varepsilon(V)=1$ since $\mathrm{D}_{\mathrm{pst}}(V)=\mathbb{Q}_p^{\mathrm{ur}}\otimes_{\mathbb{Q}_p}\mathrm{D}_{\mathrm{crys}}(V)$ is unramified as a representation of $W_{\mathbb{Q}_p}$ and $N=0$. We now show that 
 $$H^1_{\rm{f}}(\mathbb{Q}_p, T)\isom H^1_{\rm{f}}(\mathbb{Q}_p, T_2).$$ 
 The intersection of the kernel of $H^1(\mathbb{Q}_p, T)\rightarrow H^1(\mathbb{Q}_p, T_2)$ and 
 $H^1_{\rm{f}}(\mathbb{Q}_p, T)=H^1_{-}(\mathbb{Q}_p, T)$ is $0$ by what we have just shown. 
 Hence, it induces an injective map $H^1_{\rm{f}}(\mathbb{Q}_p, T)\rightarrow H^1_{\rm{f}}(\mathbb{Q}_p, T_2)$ 
 of rank $1$ modules. 
 Since the cokernel of $H^1(\mathbb{Q}_p, T)\rightarrow H^1(\mathbb{Q}_p, T_2)$ is contained 
 in the free module $H^2(\mathbb{Q}_p, T_1)$ and 
 $\mathrm{Im}\,H^1(\mathbb{Q}_p, T)=\mathrm{Im}\,H^1_-(\mathbb{Q}_p, T)$, it follows that the map
 $H^1_{\rm{f}}(\mathbb{Q}_p, T)\rightarrow H^1_{\rm{f}}(\mathbb{Q}_p, T_2)$ is surjective. 
 
It remains to consider the case where $V$ is non-crystalline semistable. Then we have
$$\mathcal{O}z\cap H^1_{\rm f}(\mathbb{Q}_p, \mathcal{O}(1)) (=\mathrm{Ker}(\mathrm{can})\cap H^1_{\rm f}(\mathbb{Q}_p, \mathcal{O}(1))) =0,$$
hence we obtain the following inclusions
$$H^1(\mathbb{Q}_p, \mathcal{O}(1))/\mathrm{Im}(\delta)=H^1(\mathbb{Q}_p, \mathcal{O}(1))/\mathcal{O}z\supseteq H^1_{\rm f}(\mathbb{Q}_p, \mathcal{O}(1))
\subseteq H^1_{\rm f}(\mathbb{Q}_p, T)=H^1_+(\mathbb{Q}_p, T),$$
where the last equality holds by the property 4) of Theorem \ref{thm, main}. 
Here we note that $\hat{\varepsilon}(V)=\varepsilon(V)=-1$ since $\mathrm{D}_{\mathrm{pst}}(V)=\mathbb{Q}_p^{\mathrm{ur}}\otimes_{\mathbb{Q}_p}\mathrm{D}_{\mathrm{st}}(V)$ is unramified as a representation of $W_{\mathbb{Q}_p}$ and $$\mathrm{det}(-\mathrm{Fr}_p\mid \mathrm{D}_{\mathrm{pst}}(V)/\mathrm{D}_{\mathrm{pst}}(V)^{N=0}
)=\mathrm{det}(-\mathrm{Fr}_p\mid \mathrm{D}_{\mathrm{pst}}(L))=-1.$$ 
Since $H^1(\mathbb{Q}_p, \mathcal{O}(1))/\mathcal{O}z$ is torsion-free and $V$ is non-crystalline, 
we have $H^1(\mathbb{Q}_p, \mathcal{O}(1))=\mathcal{O}z+H^1_{\rm{f}}(\mathbb{Q}_p, \mathcal{O}(1))$. 
Hence, $H^1(\mathbb{Q}_p, \mathcal{O}(1))/\mathcal{O}z=H^1_{\rm{f}}(\mathbb{Q}_p, \mathcal{O}(1)).$
This implies the equality 
$$H^1(\mathbb{Q}_p, \mathcal{O}(1))/\mathrm{Im}(\delta)=H^1_+(\mathbb{Q}_p, T)$$ 
 since both are 
Lagrangian submodules of $H^1(\BQ_p,T)$.

ii) Write the character associated with $T_1$ as $\chi_{\cyc}^m \omega^n \mu\chi$ for  
$\mu$ unramified and $\chi$ of finite $p$-power order. 
Then\footnote{Note that there is no contribution from the monodoromy part.} we have $\varepsilon(V)=\omega^n(-1)$ and 
$\Gamma(V)=(-1)^{m-1}$, implying $$\hat{\varepsilon}(V)=(-1)^{m+n-1}=(-1)^{r-1}.$$  
Therefore, 
we have 
$H^1_{\rm{f}}(\BQ_p, T)=H^1_{\epsilon}(\BQ_p, T)$ by the property 4) of Theorem~\ref{thm, main}.

Since $T_2$ is non-trivial, the natural map $H^1(\BQ_p, T_1)\rightarrow H^1(\BQ_p, T)$ is injective and 
the image is a submodule of $H^1_{\rm{f}}(\BQ_p, T)$ of finite index by an argument in the above part i).

If $H^0(\BQ_p, \overline{T}_2)=0$, then the quotient is contained in 
the free module $H^1(\BQ_p, T_2)$, hence, $H^1(\BQ_p, T_1)=H^1_{\rm{f}}(\BQ_p, T)$. 
Suppose that $H^0(\BQ_p, \overline{T}_2)\not=0$. 
Then $\overline{T}_2\cong \mathbb{F}$ and $\overline{T}_1\cong \mathbb{F}(1)$, and 
the residual exact sequence is of the form  
\begin{equation}\label{eq,res}
0 \rightarrow \mathbb{F}(1) \rightarrow \overline{T} \rightarrow \mathbb{F}\rightarrow 0.
\end{equation}
We also have $r \equiv m+n \equiv 1\mod 2$, and hence $\hat{\varepsilon}(V)=+1$. 
Since the map 
\[H^1_{\rm{g}}(\BQ_p,  \mathcal{O}(1))=H^1(\mathbb{Q}_p, \mathcal{O}(1))\rightarrow H^1(\mathbb{Q}_p, \mathbb{F}(1))\] is surjective, 
there exists 
an exact sequence 
\[
0 \rightarrow \mathcal{O}(1) \rightarrow T' \rightarrow \mathcal{O}\rightarrow 0
\]
whose extension class is a lift of that of the residual sequence \eqref{eq,res} and $T'$ de Rham.
Hence, by part i) and the base change property of the local sign decomposition, 
the image of 
$H^1(\BQ_p, \overline{T}_1)\rightarrow H^1(\BQ_p, \overline{T})$
is equal to $H^1_+(\BQ_p, \overline{T})$.

Suppose that the image of $H^1(\BQ_p, {T}_1)\rightarrow H^1(\BQ_p, {T})$ is isomorphic to 
$H^1_{\rm{f}}(\BQ_p, {T})=H^1_{-}(\BQ_p, {T})$. Then again by the base change property, 
the image of 
$H^1(\BQ_p, \overline{T}_1)\rightarrow H^1(\BQ_p, \overline{T})$
equals $H^1_-(\BQ_p, \overline{T})$, which is a contradiction. 
\end{proof}
\begin{remark}
In particular, the proposition relates the Bloch--Kato subgroup of $T$ with that of $T_1$ or $T_2$. The various possibilities do not seem to be discussed in the literature. 
\end{remark}

\begin{cor}\label{cor, the reducible sign}
 Let $\overline{T}$ be a generic symplectic self-dual representation of $G_{\BQ_p}$ 
of rank $2$ with coefficients in a finite field $\mathbb{F}$. 
Suppose that $\overline{T}$ has a subrepresentation $\overline{T}_1$ of rank $1$ 
with character $\omega^r\mu$ for  
$\mu$ unramified, and put $\epsilon=(-1)^r$.
\begin{itemize}
\item[i)] If $\ov{T}_1 \cong \mathbb{F}(1)$, then 
$H^1_+(\BQ_p, \overline{T})=\mathrm{Im}\,H^1(\BQ_p, \overline{T}_1)$. 
\item[ii)] If $\ov{T}_1$ is not isomorphic to $\mathbb{F}(1)$, then $H^1_{\epsilon}(\BQ_p,\overline{T})=H^1(\BQ_p,\ov{T}_1)$.
\end{itemize}
\end{cor}
\begin{proof}
This follows from the above proof of Proposition \ref{prop, reducible sign} by choosing an appropriate lift. 
For example, in the case i), we choose $T'$ as in the proof. 
\end{proof}

\begin{cor}\label{cor, lsd decomposable}
Let $(R, T)$ be a generic symplectic self-dual $G_{\BQ_p}$-representation of rank two. 
Suppose that $T$ is decomposable, i.e. $$T=T_1\oplus T_2$$ as a $G_{\BQ_p}$-representation for $T_i$ of rank one.  
Write the character associated to $\overline{T}_1$ as $\omega^{r}\mu$ for 
$\mu$ unramified, 
and put $\varepsilon=(-1)^{r}$. 
Then $$H^1_{\varepsilon}(\BQ_p,T)=H^1(\BQ_p,T_1),\qquad H^1_{-\varepsilon}(\BQ_p,T)=H^1(\BQ_p,T_2).$$ 
\end{cor}
\begin{proof}
Since $H^0(\BQ_p, \overline{T}_i)=0$   for $i \in \{1,2\}$, the submodules $H^1(\BQ_p,  \ov{T}_i) \subset H^1(\BQ_p,\ov{T})$ and $H^1(\BQ_p,T_i) \subset H^1(\BQ_p,T)$ are Lagrangian by
Lemma \ref{lemma, reducible lagrangian}. The sign of $H^1(\BQ_p,  T_i)$ and $H^1(\BQ_p,  \overline{T}_i)$ 
is determined by Cororally \ref{cor, the reducible sign} ii). 
\end{proof}

{\appendix

\section{Zariski density of crystalline points with a sign}\label{s:appendix}
\subsection{Setting}
Let $p$ be an odd prime. Let $\ov{\BQ}_p$ be an algebraic closure of $\BQ_p$, $G_{\BQ_p}=\Gal(\ov{\BQ}_{p}/\BQ_p)$ and $I_{p} \subset G_{\BQ_{p}}$ the inertia subgroup. 

Let $\chi_{\cyc}: G_{\BQ_{p}} \ra \BZ_p^\times$ denote 
the $p$-adic cyclotomic character. We normalise the Hodge--Tate weights so that the one of $\chi_{\cyc}$ is $1$. Let $\omega: G_{\BQ_{p}} \ra \BF_{p}^\times$ denote the mod $p$ reduction of $\chi_{\cyc}$.

\subsection{Existence of a lift} We consider the following existence of a lift of rank two, symplectic self-dual mod $p$ representation of $G_{\BQ_p}$ to characteristic zero. The proof is based on a suggestion of
 V. Pa\v{s}k\=unas, to whom we are grateful. 
\begin{thm}\label{thm, Paskunas}
Let 
$
\ov{\rho}: G_{\BQ_{p}} \ra \GL_{2}(\BF) 
$ 
be a continuous representation with $\det(\ov{\rho})=\omega$ 
 for $\BF$ a finite field of characteristic $p$. 
Then there exists a continuous lift 
$$
\rho: G_{\BQ_{p}} \ra \GL_{2}(W(\BF))
$$
of $\ov{\rho}$ such that 
\begin{itemize}
\item[i)] $\det(\rho)=\chi_{\cyc}$, 
\item[ii)] $\rho$ is crystalline, 
\item[iii)] The associated Hodge--Tate weights are $(k,1-k)$ for some $k\in \BZ_{\geq 1}$ of a {\it given} parity,
\item[iv)] $\rho$ is irreducible. 
\end{itemize}
\end{thm}
\begin{proof} We primarily follow the notation of \cite{Tu} (see also \cites{Ki1,Ki2,Pa}). In particular $R_{\ov{\rho}}^{\Box,1}$ denotes the associated framed deformation ring, classifying framed deformations with determinant 
$\chi_{\cyc}$. 

For an integer $k\in \BZ_{\geq1}$, put $\lambda_{k}=(k,1-k)$. Let $$R_{\ov{\rho}}^{\Box,{\rm cr},k}:= R_{\ov{\rho}}^{\Box,1,{\rm cr}}(\lambda_{k}, 1\oplus 1, \chi_{\cyc})$$
be an associated crystalline deformation ring \cite{Ki1}: 
the quotient of $R_{\ov{\rho}}^{\Box,1}$, classifying crystalline deformations of $p$-adic Hodge type $(\lambda_{k}, 1\oplus 1, \chi_{\cyc})$. In particular the corresponding Hodge--Tate weights are $(k,1-k)$ and the determinant 
$\chi_{\cyc}$. It suffices to show that the corresponding Hilbert--Samuel multiplicity $e(R_{\ov{\rho}}^{\Box,{\rm cr},k}/\varpi)$ is non-zero, where $\varpi$ denotes a uniformiser of $W(\BF)$. Indeed, then 
$R_{\ov{\rho}}^{\Box,{\rm cr},k}$ will be non-zero and being $p$-torsion free, its generic fiber will also be non-zero, leading to a desired lift satisfying the properties i)-iii). 

The positivity of Hilbert--Samuel multiplicity will be shown 
under the large weight hypothesis: 
\begin{equation}\label{w-l}
2k-2 \geq 2(p^2-1),
\end{equation}
which will be assumed henceforth. 

By the Breuil--M\'ezard conjecture \cite[Thm.~1.2]{Tu}, we have
\begin{equation}\label{BM_f}
e(R_{\ov{\rho}}^{\Box,{\rm cr},k}/\varpi)=\sum_{\ov{\sigma}} m_{\ov{\sigma}}^{\rm cr}(\lambda_k)\mu_{\ov{\sigma}}(\ov{\rho}). 
\end{equation}
Here $\ov{\sigma}$ varies over smooth irreducible $\BF$-representations of $\GL_{2}(\BZ_p)$, $m_{\ov{\sigma}}^{\rm cr}(\lambda_k):=m_{\ov{\sigma}}^{\rm cr}(\lambda_{k},1\oplus 1)\in \BZ_{\geq 0}$ denotes the $\GL_{2}(\BZ_p)$-multiplicity as in \cite[p.~2176]{Tu}, 
and 
$\mu_{\ov{\sigma}}(\ov{\rho}) \in \BZ_{\geq 0}$  the Hilbert-Samuel multiplicity of a crystalline deformation ring as in \cite[Rem.~1.3]{Tu}. Note that $$\mu_{\ov{\sigma}}(\ov{\rho})\neq 0 \iff \Hom_{\GL_{2}(\BZ_{p})}(\ov{\sigma},\beta(\ov{\rho}))\neq 0$$ 
(cf.~\cite[Thm.~1.1]{Pa}), 
where $\beta(\ov{\rho})$ denotes the smooth $\BF$-representation associated to $\ov{\rho}$ by Colmez \cite{Co}. A necessary condition for this non-vanishing is that the central character of $\ov{\sigma}$ equals $\mathbf{1}$. We recall that the Breuil--M\'ezard conjecture \cites{BM,Ki2} is unconditionally resolved for odd primes $p$ by the work of Kisin \cite{Ki2}, Pa\v{s}k\=unas \cite{Pa}, Hu--Tan \cite{HT} and Tung \cite{Tu}.

Put $\sigma_{k}={\rm Sym}^{2k-2} E^2\otimes \det^{1-k}$ and let $\ov{\sigma}_k$ denote the associated residual representation. In view of \eqref{w-l} and the proof of \cite[Lem.~A.2]{Pa0}, any smooth irreducible $\BF$-representation $\ov{\sigma}$ of $\GL_{2}(\BZ_p)$ with trivial central character 
appears as a subquoteint of the residual representation $\ov{\sigma}_{k}$ with multiplicity at least two. That is, for such $\ov{\sigma}$, we have $
m_{\ov{\sigma}}^{\rm cr}(\lambda_{k})\geq 2$. 
Therefore, in view of \eqref{BM_f} and the above paragraph, it follows that 
\begin{equation}\label{HM_lb}
e(R_{\ov{\rho}}^{\Box,{\rm cr},k}/\varpi)\geq 2, 
\end{equation}
completing the existence of a lift $\rho$ satisfying the properties i)-iii).

Now we consider the property iv). 
If it does not hold for any lift $\rho$ satisfying i), ii) and iii) for a given parity, then any such lift is reducible. 
On the other hand, the reducible locus has Hilbert--Samuel multiplicity one. Hence, in light of \eqref{HM_lb}, the proof concludes.

\end{proof}

\subsection{Zariski density of crystalline points with a sign}
An application of the existence of crystalline lift with a given sign 
is the following Zariski density of such crystalline points. 
\begin{thm}\label{thm, density sgn crystalline}
Let 
$
\ov{\rho}: G_{\BQ_{p}} \ra \GL_{2}(\BF) 
$ 
be a continuous representation with $\det(\ov{\rho})=\omega$ 
 for $\BF$ a finite field of characteristic $p$ and $R^{\square, \chi_{\mathrm{cyc}}}_{\overline{\rho}}$ the associated universal framed deformation ring with determinant $\chi_\cyc$.
Then for each sign $\varepsilon\in \{\pm \}$, the set $\mathcal{S}_{\varepsilon}^{ \chi_{\mathrm{cyc}}}$ consisting of $x\in \mathrm{Spec}R^{\square, \chi_{\mathrm{cyc}}}_{\overline{\rho}}[1/p]$ whose 
associated $p$-adic Galois representation $\rho_x$ is crystalline with Hodge--Tate weights $(k, 1-k)$ for   
$k\in \BZ_{\geq 1}$ with $(-1)^k=\varepsilon \cdot 1$ is Zariski 
dense in $\mathrm{Spec} R^{\square, \chi_{\mathrm{cyc}}}_{\overline{\rho}}[1/p]$. 
\end{thm}
\begin{proof}
The Zariski density of crystalline points assuming the existence of a lift as in Theorem~\ref{thm, Paskunas}
has already been proved in \cite{BIP2}, albeit without specifying the sign. 
 In the following outline we modify the argument of {\it loc. cit.} to account for the sign. 
We primarily follow the notation of \cite{BIP2}.

First, it suffices to show that $\mathcal{S}_{\varepsilon}^{\chi_{\mathrm{cyc}}}$ is Zariski dense in the rigid generic fiber 
$\mathfrak{X}_{\overline{\rho}}^{\square, \chi_{\mathrm{cyc}}}$ of the formal scheme $\mathrm{Spf}R^{\square, \chi_{\mathrm{cyc}}}_{\overline{\rho}}$ by 
\cite[Lem.~5.1]{BIP2}. To show this density, it is enough to check that the proof of \cite[Cor.~4.3]{BIP2} works also for $\mathcal{S}_{\varepsilon}^{\chi_{\mathrm{cyc}}}$ instead of $\mathcal{S}^{\psi}$ therein. 
 Instead of $\mathcal{S}'$ in that proof, we consider 
the subset $\mathcal{S}'_{\varepsilon}$ of $\mathfrak{X}_{\overline{\rho}}^{\square, \chi_{\mathrm{cyc}}}\times_{\mathrm{Sp}L}\mathcal{X}^{\mathrm{rig}}$ corresponding to pairs 
$(\rho', \theta)$ such that $\rho'\in \mathcal{S}_{\varepsilon}^{\chi_{\mathrm{cyc}}}$ and $\theta$ is crystalline. 
Since $\mathcal{S}_{\varepsilon}^{ \chi_{\mathrm{cyc}}}$ is non-empty by Theorem \ref{thm, Paskunas}, 
so is $\mathcal{S}'_{\varepsilon}$ 
 as it contains a point corresponding to a pair $(\rho, \mathbf{1})$ for 
  $\rho \in \mathcal{S}_{\varepsilon}^{ \chi_{\mathrm{cyc}}}$. 
  
 Then the density of 
$\mathcal{S}_{\varepsilon}^{\chi_{\mathrm{cyc}}}$ follows by the proof of \cite[Cor.~4.3]{BIP2}. Here we note that if we consider 
$\mathcal{S}^{\mathrm{cn}}_{\mathrm{cr}}$ in that proof for a fixed $\rho\in \mathcal{S}_{\varepsilon}^{\chi_{\mathrm{cyc}}}$, then 
any $\rho_x\otimes \theta$ in the statement of \cite[Lem.~4.1]{BIP2} belongs to $\mathcal{S}_{\varepsilon}^{\chi_{\mathrm{cyc}}}$, i.e. the parity of Hodge--Tate weights of $\rho_x\otimes \theta$ are the same as those of $\rho$ since the Hodge--Tate weights of 
$\rho_x$ are congruent to those of $\rho$ $\mod{dt}$, and those of 
$\theta$ are congruent to $0 
\mod{t}$ for $d=2$ and $t=|\mu_{\mathrm{tor}}(\mathbb{Q}_p)|=p-1$. In particular, they are congruent $\mod {2}$ since $p$ is odd. 
\end{proof}

}
\part{Applications}\label{part II}
\setcounter{section}{5}

\section{Mazur--Rubin arithmetic local constants, epsilon constants and the $p$-parity conjecture}\label{s:mr}
In this section we consider some phenomena around the $p$-parity conjecture. The central result is a compatibility of Mazur--Rubin and Deligne--Langlands local constants for rank two symplectic self-dual de Rham representations of $G_{\BQ_p}$ (see Theorem~\ref{thm, MR}). It leads to new cases of the relative $p$-parity conjecture (see Theorem~\ref{thm, parity-family}) as well as the $p$-parity conjecture for Hilbert modular forms over totally real fields (see Theorem~\ref{thm, parity}). Along the way, we study a local sign-like decomposition for non-generic representations of $G_{\BQ_p}$ (see Theorem~\ref{thm, anomalous basis}).

In this section we fix an odd prime $p$. 
\subsection{Backdrop} 
This subsection introduces the $p$-parity conjecture and Mazur--Rubin arithmetic local constants  (cf.~\cites{MR,NekMRl}). 

\subsubsection{$p$-parity conjecture} 
Let $M$ be a pure motive\footnote{in the naive sense \cite{NekMRl}} defined over a number field $F$ with coefficients in a number field $L_M$. 

For a prime $\mathfrak{p}$ of $L_M$ above $p$, let $V:=M_\mathfrak{p}$ be the $p$-adic representation of $G_{F}$ associated to $M$ 
with coefficients in $L:=L_{M,\mathfrak{p}}$. 
Suppose that $M$ is symplectic self-dual, i.e. 
there exists an isomorphism 
$$M\cong M^\vee(1)$$ of pure motives  inducing a skew-symmetric isomorphism  $V \cong V^\vee(1)$
 of $L[G_F]$-modules. 
 
Let $H^1_{\rm{f}}(F, V)$ be the Block--Kato Selmer group associated to $V$ and put 
\[
\chi_{\rm{f}}(F, V)=\dim_{L} H^1_{\rm{f}}(F, V).
\]
Here we note that $\dim_{L} H^0(F, V)=0$ since $M$ is pure and self-dual. 
For a place $v$ of $F$, let $\varepsilon_v(V)\in\{\pm 1\}$ be the local $\varepsilon$-constant, 
defined using a non-trivial additive character $\psi$ and a Haar measure that is 
self-dual with respect to $\psi$ (cf.~Lemma~\ref{lem, epsilon=pm 1}).
\begin{defn}\label{def:epsilon}
The global $\varepsilon$-constant associated to $V$ is defined as $$\varepsilon(V):=\prod_v \varepsilon_v(V).$$ 
\end{defn} 
The $\varepsilon$-constant $\varepsilon_v(V)$ for any finite place $v$
and the archimedean component $\prod_{w |\infty}\varepsilon_w(V)$
 are determined by the local Galois representations $V|_{G_{F_v}}$ and $V|_{G_{F_w'}}$ for all $w'|p$, respectively.   
More precisely, for an integer $m$ and place $w'$ above $p$, put 
\[
h_{m}(V|_{G_{F_{w'}}}):=\dim_{L} \mathrm{gr}^m D_{\rm{dR}}(V|_{G_{F_{w'}}}), \quad g_{w'}^-(V):=\sum_{m<0} m\, 
h_m(V|_{G_{F_{w'}}}), \quad g_{w'}^+(V):=\sum_{m>0} m\, h_m(V|_{G_{F_{w'}}}). 
\]
Then the global $\varepsilon$-constant equals
\[
\varepsilon(V)=\prod_{v \in S_{\mathrm{f}}} \varepsilon_v(V)\cdot
\prod_{w \in S_\infty} \varepsilon_w(V) 
\]
where $S_{\rm f}$ and $S_\infty$ denote the set of finite and infinite places of $F$ and\footnote{Fix an isomorphism  
 $\lambda: \overline{\BQ}_p\cong \BC$. Then for each embedding $\iota_v: F \hookrightarrow \overline{\BQ}_p$, 
we can deduce Hodge theoretic information about a symplectic self-dual motive $M$ defined over $F$
at the archimedean place corresponding to $\lambda \circ \iota_v$  
from its $p$-adic realization with respect to $\iota_v$. 
A subtlety is that there is a one to one correspondence 
between archimedean and $p$-adic embeddings via $\lambda$, 
but not between the archimedean and $p$-adic places of $F$, namely $\# S_\infty \not= \# S_p$ in general. 
Hence, one typically does not expect a finer term-by-term relation than (\ref{equation, def arch epsilon}).} 
\begin{equation}\label{equation, def arch epsilon}
\prod_{w \in S_\infty} \varepsilon_w(V):=(-1)^{\sum_{w' \in S_p} g_{w'}^-(V)}\cdot (-1)^{r_2(F) \dim_{L} \!V/2}
\end{equation} 
for $S_p$ the set of places above $p$ (cf. \cite{NekPIIIe}).

We now let $V$ to be a geometric $p$-adic representation of $G_F$ defined over a $p$-adic local field $L$, 
i.e. unramified outside finitely many places and de Rham at places above $p$. 
Put $$\chi_{\rm f}(F,V)=\dim_{L}H^1_{\rm f}(F,V)-\dim_{L}H^0(F,V).$$
Define $\varepsilon(V)$ as in Definition~\ref{def:epsilon}.

The Bloch--Kato conjecture suggests the following.

\begin{conj}\label{conj, p-parity}($p$-parity conjecture) 
For any symplectic self-dual and geometric $p$-adic representation $V$ of $G_F$, 
we have
$$
\varepsilon(V)=(-1)^{\chi_{\rm{f}}(F, V)}.
$$
\end{conj}

\subsubsection{Relative $p$-parity conjecture}  
\begin{conj}(relative $p$-parity conjecture)\label{conj, relative p-parity}
For any two symplectic self-dual and geometric $p$-adic representations $V$ and $V'$ of $G_F$, 
we have 
\begin{equation}\label{conj, relative p-parity II}
(-1)^{\chi_{\rm{f}}(F, V)-\chi_{\rm{f}}(F, V')}=\varepsilon(V)/\varepsilon(V').
\end{equation}
\end{conj}

The above conjecture is evidently implied by the $p$-parity conjecture for $V$ and $V'$. 
In practice, one first attempts to approach the relative $p$-parity conjecture for a $p$-adic family $\CV$ containing $V$, which is amenable to $p$-adic methods. Then one may aim to approach the $p$-parity conjecture for $V$ via establishing it for a specific specialization of $\CV$ with simple arithmetic features, such as Selmer rank zero or one. 
\subsubsection{Mazur--Rubin arithmetic local constants} 
Let $V$ and $V'$ be symplectic self-dual and geometric $p$-adic representations of $G_F$ 
that are residually symplectically isomorphic, i.e. 
they are deformations of a fixed representation of $G_F$ over a finite field. 

Then Mazur and Rubin \cite[Thm.~1.4]{MR} proved that
\begin{equation}\label{equation, relative selmer}
\chi_{\rm{f}}(F, V)-\chi_{\rm{f}}(F, V') \equiv  \sum_{v: \text{finite}} \delta_v(T, T')\quad  \mod 2
\end{equation}
for local constants $\delta_v (T,T')$ as recalled below. 
So the left hand side of 
\eqref{conj, relative p-parity II} 
 is a product of local signs.

To introduce the invariant $\delta_v (T,T')$, 
let $T \subset V$ and $T'\subset V'$ be symplectic self-dual lattices such that  
$\overline{T}:=T/\mathfrak{p}T$ and $\overline{T}':=T'/\mathfrak{p}T'$ 
are symplectically isomorphic 
over a finite field $\BF$. 
For each finite place $v$, define 
Selmer structures $\mathscr{F}_v$ and $\mathscr{F}_v'$
on $H^1(F_v, \overline{T})$ 
by 
\[
\mathscr{F}_v:=\mathrm{Im}\left(H^1_{\rm{f}}(F_v, T)  \rightarrow H^1(F_v, \overline{T}) \right) 
\] 
and 
\[
\mathscr{F}_v':=\mathrm{Im}\left(H^1_{\rm{f}}(F_v, T')  \rightarrow H^1(F_v, \overline{T'}) \cong H^1(F_v, \overline{T}) \right). 
\]
Here the last isomorphism arises from our assumption $\overline{T}'\cong \overline{T}$, 
$H^1_{\rm{f}}(F_v, T)$ is the Bloch--Kato subgroup if $v|p$, and otherwise 
\[
H^1_{\rm{f}}(F_v, T):=\mathrm{Ker}\left( H^1(F_v, T) \rightarrow H^1(F_v, V)/H^1(G_{F_v}/I_v, V^{I_v}) \right)
\] 
for $I_v$ the inertia subgroup. 
Then $\mathscr{F}_v$ and $\mathscr{F}_v'$ are Lagrangian subspaces of $H^1(F_v, \overline{T})$.

\begin{defn}\label{def:MR} Let $v$ be a finite place of $F$.
\begin{itemize}
\item[i)] For symplectic self-dual representations $T$ and $T'$ of $G_{F_v}$ 
over a complete Noetherian local ring $R$ with residue field $\BF$ 
and an $\BF$-linear symplectic isomorphism $\ov{T}' \cong \ov{T}$ of $G_{F_v}$-representations, 
the Mazur--Rubin arithmetic local constant $\delta_v (T,T')$ is defined by 
\[
\delta_v(T, T')=\dim_{\BF} \left(\mathscr{F}_v/\mathscr{F}_v\cap \mathscr{F}_v' \right) \mod 2\quad \in \BZ/2\BZ.
\] 
\item[ii)] More generally, for 
Lagrangian submodules $L$ and $L'$ of $H^1(F_{v},T)$ and $H^1(F_v , T')$ respectively, and 
an  $\BF$-linear symplectic isomorphism $\ov{f}:\ov{T}' \cong \ov{T}$ of $G_{F_v}$-representations, 
define 
$$
\delta_{v}(L,L')=\dim_{\BF}(\ov{L}/\ov{L}\cap \ov{f}(\ov{L}')) \mod{2} \in \BZ/2\BZ.
$$
\end{itemize}
\end{defn}
Note that the above invariants are defined for local representations, unlike the preceding set-up of the $p$-parity conjecture. Here is an example. 
\begin{example}
Let $A$ and $B$ be elliptic curves defined over $F_v$ such that $A[p]\cong B[p]$ as $\BF_{p}[G_{F_v}]$-modules. 
Then we have
\[
\delta_v(A, B)=\dim_{\mathbb{F}_p} 
\overline{A(F_v)}/(\overline{A(F_v)} \cap \overline{B(F_v)}) \mod 2, 
\]
where $\overline{A(F_v)}:=A(F_v)\otimes_{\BZ} \BZ/p\BZ$ and $\overline{B(F_v)}:=B(F_v)\otimes_{\BZ} \BZ/p\BZ$ denote the Kummer images in $H^1(F_v, A[p])$, the latter via the given isomorphism $B[p]\cong A[p]$.
\end{example}

We now show that the local constants 
do not depend on the choice of 
an isomorphism $\overline{T}'\cong \overline{T}$. 

\begin{prop} 
Let $F_v$ be a non-archimedean local field and $\BF$ a field of odd  characteristic $p>0$.
Let $V$ be a symplectic self-dual representation of $G_{F_v}$ over $\BF$ and $\langle \;, \; \rangle$ the pairing on $V$.
Suppose that there exists a $G_{F_v}$-equivariant $\BF$-linear isomorphism $f: V \rightarrow V$ and  $c\in\BF^\times$ such that 
$$\langle f(x), f(y) \rangle=c \langle x, y \rangle$$ for 
all $x, y \in V$. 
Let $f: H^1(F_{v}, V)\rightarrow H^1(F_{v}, V)$ also denote the induced map. 
Then for any Lagrangian subspace $L$ of $H^1(F_{v}, V)$ with respect to the Tate pairing, we have
$$
\delta_v (L, f(L))=0 \in \mathbb{F}_2. 
$$
In particular,  the  invariant $\delta_v(L,L')$ as in Definition~\ref{def:MR} and so the Mazur--Rubin arithmetic local constant $\delta_v(T, T')$ do not depend on the choice of 
an $\BF$-linear symplectic isomorphism $\overline{T}'\cong \overline{T}$. 
\end{prop}
\begin{proof} 
First, note that the independence of local constants is a consequence of the asserted vanishing. 
Indeed, let $f_1, f_2: \ov{T}'\cong \ov{T}$ 
be two $\BF$-linear symplectic isomorphisms 
of $G_{F_v}$-representations, and $\delta_{v,1}(\cdot), \delta_{v,2}(\cdot)$ the associated invariants. Then $f:=f_{2} \circ f_{1}^{-1}: \ov{T} \cong \ov{T}$ is an $\BF$-linear isomorphism, and we have
$$
\delta_{v,2}(L,L')=\delta_{v,2}(L,f(L))+\delta_{v,2}(f(L),f_{2}(L'))=\delta_{v,1}(L,L').
$$
Here the first equality follows from the transitivity of the local invariant $\delta_{v,2}$ as in \cite[Cor.~2.5]{KMR} and the second from the asserted vanishing. 

To approach the vanishing, we proceed via a number of reductions. First, we reduce to the case that $f \in \mathrm{Aut}_{\BF}(V)$ is diagonalizable. 
Indeed, without loss of generality, we may assume that $\BF$ contains  the eigenvalues of $f$. 
Then the linear map $f^{p^m}$ is diagonalizable 
for a sufficiently large integer $m$: restricting $f$ to a generalized eigenspace of $V$, 
we have 
$$f=\alpha I +N$$ for $\alpha I$ a scalar map and $N$ a nilpotent operator and so 
$f^{p^m}=\alpha^{p^m}I+N^{p^m}=\alpha^{p^m}I$ for an integer $m$ with $N^{p^m}=0$. 
Note that replacing $f$ with $f^{p^m}$ does not affect the assertion since 
for any integer $n$,  we have 
\[
\delta_{v}(L, f^{2n}(L))=\delta_{v}(L, f^{n}(L))+\delta_{v}(f^{n}(L), f^{2n}(L))
=\delta_{v}(L, f^{n}(L))+\delta_{v}(L, f^n(L))=0
\] 
and so
\[
\delta_v (L, f^{2n+1}(L))=\delta_v (L, f(L))+\delta_v (f(L), f^{2n+1}(L))
=\delta_v (L, f(L))+\delta_v (L, f^{2n}(L))=\delta_v (L, f(L)).
\]
Hence, by replacing $f$ with its odd power $f^{p^m}$, we suppose that $f$ is diagonalizable.

Next, we reduce to the existence of an $f$-invariant Lagrangian subspace 
$L_0$ of $H^1(F_{v}, V)$. 
Indeed, then 
\[
\delta_v (L, f(L))=\delta_v (L, L_0)+\delta_v (L_0, f(L))=\delta_v (L, L_0)+\delta_v (f^{-1}(L_0), L)
=\delta_v (L, L_0)+\delta_v (L_0, L)=0, 
\]
as desired. Here the first equality uses the transitive property \cite[Cor.~2.5]{KMR} of the invariant $\delta_v$.

 We may suppose that $f$ is an isometry. 
Let $$V=\oplus_{\lambda \in \BF^\times} V_\lambda$$ be the eigenspace decomposition of $f$. 
Since the $G_{F_v}$-action commutes with $f$, the $\BF$-subspaces $V_\lambda \subset V$ are stable under the $G_{F_v}$-action. 
If $\lambda \mu\not=1$, then 
$V_\lambda$ and $V_\mu$ are orthogonal. 
Let $\epsilon \in \{+1, -1\}$ act on $\BF^\times$ by $\epsilon \cdot x=x^\epsilon$, 
and write $S=\BF^\times /\{\pm 1\}$ for the quotient space of this action. 
 For $\overline{\lambda}\not=\overline{1} \in S$, 
put $V_{\overline{\lambda}}:=V_\lambda \oplus V_{\lambda^{-1}}$. 
Then 
\[V=V_1\oplus V_{-1} \oplus_{\overline{\lambda} \in S\setminus\{\overline{1}\}} V_{\overline{\lambda}}
\]
is an orthogonal decomposition. Hence, the restriction of the pairing $\langle \;, \; \rangle$
 on $V$ to the $\BF$-subspaces
$V_{\pm 1}$ or $V_{\overline{\lambda}}$ is non-degenerate. 
Note that $H^1(F_{v}, V_{\pm 1})$ and $H^1(F_{v}, V_{\overline{\lambda}})$ are orthogonal to each other
with respect to the Tate pairing since  
it is induced by composition of the cup product with the pairing  on $V$. 
Hence, the existence of a Lagrangian is further reduced to the case that $V=V_{\pm1}$ or $V= V_{\overline{\lambda}}$ for $\ov{\lambda} \in S\setminus \{\ov{1}\}$.

If $V= V_{\overline{\lambda}}$ for $\ov{\lambda} \in S\setminus \{\ov{1}\}$, 
then $L_0 = H^{1}(F_{v}, V_{\lambda})$ or $L_0 = H^{1}(F_{v}, V_{\lambda^{-1}})$ is a desired $f$-invariant Lagrangian of $H^1(F_v, V_{\ov{\lambda}})$. 
As for the remaining case $V=V_{\pm 1}$, and $f$ is simply multiplication by $\pm 1$. 
 Since $V$ is symplectic self-dual, it is even dimensional, and then so is $H^1(F_v, V)$ by the Euler--Poicar\'e formula. 
The Tate pairing on $H^1(F_v,V)$ is non-degenerate and symmetric, and the associated matrix can be diagonalized with non-zero diagonal entries. 
By a scalar extension, we may assume that $\BF$ contains  
square roots of $-1$ and 
the diagonal entries. Hence, we can transform the quadratic form associated with the Tate pairing 
to a form with signature $(n,n)$ for $n= \dim_{\BF} H^1(F_v, V)/2$ and 
the diagonal coefficients being $\pm 1$. Then the existence of a desired Lagrangian $L_0$ is apparent.

\end{proof}

\subsubsection{A conjecture on the compatibility of local constants} 
We propose a conjecture on the compatibility of Mazur--Rubin arithmetic local constant and epsilon constants.

In the setting of the relative $p$-parity conjecture as in \S\ref{conj, relative p-parity}, recall that 
$(-1)^{\chi_{\rm{f}}(F, V)-\chi_{\rm{f}}(F, V')}$ 
conjecturally equals 
 $\varepsilon(V)/\varepsilon(V')$. 
Since the latter is a product of ratios of local $\varepsilon$-constants, 
Mazur and Rubin \cite{MR} surmised a connection between their invariant $\delta_v(T,T')$ and 
the ratio $\varepsilon_v(V)/\varepsilon_v(V')$. 
Mazur--Rubin \cite{MR} and Nekov\'a\v{r} \cites{NekMRl, NekMRtame} proved results toward it in various cases. 
However, neither of them 
proposed a conjecture on the compatibility of these local constants.
 
For symplectic self-dual $G_{\BQ_p}$-representations, we propose the following. 

\begin{conj}\label{conj, MR}(compatibility of local constants)
Let $p$ be an odd prime. 
Let $T_1$ and $T_2$ be symplectic self-dual $\CO_L$-representations of $G_{\BQ_p}$ for $L$ a finite extension of $\BQ_p$. 
Suppose that $V_i := T_{i} \otimes_{\BZ_p} \BQ_p$ is de Rham for $i\in\{1,2\}$.
Suppose also that 
 $T_1$ and $T_2$  are residually symplectically isomorphic. 
Then 
\begin{equation}\label{equation, our MR vs DL, conj}
 \frac{\hat{\varepsilon}_p(V_1)}{\hat{\varepsilon}_p(V_2)}:=
 \frac{\Gamma(V_1)\varepsilon_p(V_1)}{\Gamma(V_2)\varepsilon_p(V_2)}=(-1)^{\delta_p(T_1, T_2)}. 
\end{equation}
\end{conj}
In this section we prove the conjecture for rank two representations (see~Theorem~\ref{thm, MR}).

\subsection{Compatibility of Mazur--Rubin and Deligne--Langlands local constants}

The central result of this section is the following. 
\begin{thm}\label{thm, MR}
Let $p$ be an odd prime. 
Let $T_1$ and $T_2$ be symplectic self-dual $\CO_L$-representations of $G_{\BQ_p}$ of rank two for $L$ a finite extension of $\BQ_p$.
Suppose that $V_i = T_{i} \otimes_{\BZ_p} \BQ_p$ is de Rham for $i\in\{1,2\}$, and 
$T_1$ and $T_2$  are residually symplectically isomorphic. 
Then Conjecture~\ref{conj, MR} is true for $T_1$ and $T_2$, i.e. 
$$
 \frac{\hat{\varepsilon}_p(V_1)}{\hat{\varepsilon}_p(V_2)}=
 (-1)^{\delta_p(T_1, T_2)}. 
 $$
\end{thm}
\begin{proof}
We first consider the generic case. That is, we assume that $H^0(\BQ_p, \overline{T}_1)=0$ and and so $H^0(\BQ_p, \overline{T}_2)=0$. 
Let $\BF$ denote the residue field of $\mathcal{O}$. 

By the base change and de Rham properties of the local sign decomposition as in Theorem~\ref{thm, main}, we have a canonical surjection
\[
H^1_{\rm f}(\BQ_p,T_i)=H^1_{-\hat{\varepsilon}_p(V_i)}(\BQ_p,T_i) \;\longrightarrow\; H^1_{-\hat{\varepsilon}_p(V_i)}(\BQ_p,T_i)\otimes_{\CO} \BF=H^1_{-\hat{\varepsilon}_p(V_i)}(\BQ_p,\overline{T}_i) \subset H^1(\BQ_p,\overline{T}_i)\]
for $i\in\{1,2\}$. 
In light of the functoriality of the local sign decomposition for the isomorphism $\overline{T}_1 \cong \overline{T}_2$, it thus follows that  
\[
\delta_p(T, T')=
\begin{cases}
 0  \qquad \text{if \;$\hat{\varepsilon}_p(V_1)=\hat{\varepsilon}_p(V_2)$,}\\
 1  \qquad \text{if \;$\hat{\varepsilon}_p(V_1)\not=\hat{\varepsilon}_p(V_2)$.}
\end{cases}
\]
The assertion follows from this.

The non-generic or ``anomalous" case is a consequence of Theorem~\ref{thm, anomalous MR} in the next subsection by the transitivity of the invariant $\delta_p$ (cf.~\cite[Cor.~2.5]{KMR}).

\end{proof}
\begin{remark}
In light of the above proof, the local sign decomposition may be viewed as an absolute version of the Mazur--Rubin arithmetic local constant. 
\end{remark}

\subsection{The anomalous case}\label{subsubsection, anomalous}
In this subsection we prove a compatibility of Mazur--Rubin and Deligne--Langlands local constants for non-generic representations of $G_{\BQ_p}$, henceforth refereed to as anomalous. It relies on a local sign-like decomposition for anomalous representations, which may be viewed as the main result of the subsection. 

\subsubsection{Compatibility of Mazur--Rubin and Deligne--Langlands local constants bis}
Let $\overline{T}$ be an anomalous symplectic self-dual $G_{\BQ_p}$-representation over a finite field $\BF$ of characteristic $p$, i.e. $H^0(\BQ_p,\ov{T})\neq 0$. 
Let $C$ be the one-dimensional trivial subrepresentation of $\overline{T}$. 

By definition, we have an exact 
 \begin{equation*}
\xymatrix{
0 \ar[r] & C \ar[r]
  &  
\overline{T}  \ar[r]  &  C(1) \ar[r]  &  0
}
 \end{equation*}
 of $G_{\BQ_p}$-representations, 
 and so an exact sequence
 \begin{equation}\label{equation, C dual}
\xymatrix{  
0 \ar[r] & H^1(\BQ_p,C)   \ar[r]  & H^1(\BQ_p,\overline{T}) \ar[r]  &  H^1(\BQ_p,C(1))
\ar[r]  &   0. 
}
 \end{equation}
 Note that $H^1(\BQ_p,\overline{T})$ is a four-dimensional $\BF$-vector space.

Let $\mathcal{O}$ be the integer ring of a finite extension of $\BQ_p$ 
with uniformizer $\pi$ and the residue field $\BF$.  
Suppose that $T$ is a symplectic self-dual lift of $\overline{T}$ 
with coefficients in $\mathcal{O}$ such that $V:=T\otimes_{\BZ_p} \BQ_p$ is de Rham. 
A main result of this subsection is the following.

\begin{thm}\label{thm, anomalous MR}
Let $T$ be a symplectic self-dual $G_{\BQ_p}$-representation over $\mathcal{O}$ of rank two.   
Suppose that $V:=T\otimes_{\BZ_p}\BQ_p$ is de Rham and 
the residual representation $\overline{T}$ has a one-dimensional trivial representation $C$. Then we have
 \[
\delta_p(\overline{H^1_{\rm{f}}(\BQ_p,T)}, H^1(\BQ_p,C))
=-\hat{\varepsilon}_p(V)
 \]
 where $\overline{H^1_{\rm f}(\BQ_p,T)}$ denotes the residual image in $H^1(\BQ_p,\ov{T})$.  
\end{thm}

The above compatibility implies the compatibility of local constants as in Theorem~\ref{thm, MR} for anomalous representations. 

We approach Theorem~\ref{thm, anomalous MR} in the reducible and irreducible cases in the next two subsections. The strategy employs local sign decomposition for auxiliary generic representations, leading to a local sign-like decomposition 
for anomalous de Rham representations (cf.~Theorem~\ref{thm, anomalous basis}). 
  \subsubsection{The reducible case}\label{subsubsection, title}

 Let ${T}$ be a symplectic self-dual de Rham representation of $G_{\BQ_p}$ of rank two over an integer ring $\CO$ of a $p$-adic local field with residue field $\BF$. Put $V=T\otimes_{\BZ_p}\BQ_p$. 
 
 Throughout \S\ref{subsubsection, title} we assume that  
$T$ is anomalous, i.e. $H^0(\BQ_p,  \overline{T})\not=0$, and there is an extension 
\[
0 \rightarrow T_1 \rightarrow T \rightarrow T_2 \rightarrow 0
\]
of $G_{\BQ_p}$-representations for 
$T_1$ of rank one. We may and do assume that the Hodge--Tate weight of $T_1$ is positive. 
Indeed, if the extension is non-split, then this holds, 
and 
if it is decomposable, this may be assumed 
without loss of generality. 

We approach Theorem~\ref{thm, anomalous MR} by analyzing the following three subcases, 
 the key tool being local sign decomposition. 

\begin{prop}
 Suppose that  $\overline{T}_1\cong C$. 
Then $\hat{\varepsilon}_p(V)=-1$ and 
\[H^1_{\rm{f}}(\BQ_p, {T}_1)\cong \mathrm{Im}\,H^1(\BQ_p, {T}_1) =H^1_{\rm{f}}(\BQ_p, {T}),\] the isomorphism being induced by the map 
$H^1(\BQ_p, {T}_1)\rightarrow H^1(\BQ_p, {T})$.  
In particular, $\overline{H^1_{\rm{f}}(\BQ_p,T)}=H^1(\BQ_p,C)$. 
\end{prop}
\begin{proof}
The sign $\hat{\varepsilon}_p(V)$ is computed as in the proof of Proposition \ref{prop, reducible sign} ii). 
(Note that $m+n \equiv 0 \mod 2$ in this case.)

Consider the diagram 
\[
\xymatrix{
 H^0(\mathbb{Q}_p, T_2)=0 \ar[d] \ar[r]^-{\delta}  &H^1(\mathbb{Q}_p, T_1) \ar[d]\ar[r]^-{\mathrm{can}} & H^1(\mathbb{Q}_p, T) \ar[d]\\ 
  D_{\mathrm{crys}}(V_2) \ar[r]^-{\delta_2}& H^1(\mathbb{Q}_p, \mathfrak{B}_{\mathrm{crys}}\otimes_{\mathbb{Q}_p}V_1) \ar[r]^-{\delta_3} &
H^1(\mathbb{Q}_p, \mathfrak{B}_{\mathrm{crys}}\otimes_{\mathbb{Q}_p}V).   \\
}
\]
We show that $\delta_2$ is the zero map. If $V_2$ is not crystalline, this is trivial 
since $D_{\mathrm{crys}}(V_2)=0$. Suppose that $V_2$ is crystalline, and 
it is sufficient to show that $V$ is crystalline. 
Then $V_1$ is crystalline 
by the symplectic self-duality. Since $\overline{T}_1\otimes \overline{T}_2^{\otimes -1} \cong \mathbb{F}(-1)$, 
the crystalline 
representation $T_1\otimes T_2^{\otimes -1}$ is isomorphic to neither $\mathcal{O}$ nor $\mathcal{O}(1)$, 
and has positive Hodge-Tate weight. Hence, 
$H^1_{\mathrm{f}}(\BQ_p, V_1\otimes V_2^{\otimes -1})=H^1(\BQ_p, V_1\otimes V_2^{\otimes -1})$, and 
this implies that the extension $V$ is crystalline. 

The $\CO$-module $H^1(\mathbb{Q}_p, T_2)$ is torsion-free by part 3) of Proposition~\ref{a} since $\overline{T}_2\cong \mathbb{F}(1)$. 
Hence, the assertion follows. 
\end{proof}

\begin{prop}\label{prop, reducible MRC 1}
Suppose that $\overline{T}_1\cong \mathbb{F}(1)$ and $T_1\not\cong \mathcal{O}(1)$. 
Then $\hat{\varepsilon}_p(V)=+1$ and  
\[\delta_p(\overline{H^1_{\rm{f}}(\BQ_p,T)}, H^1(\BQ_p,C))=-1.\]
\end{prop}
\begin{proof} 
The sign $\hat{\varepsilon}_p(V)$ is computed as before. 

The residual exact sequence 
\[
0 \rightarrow \overline{T}_1 \rightarrow \overline{T} \rightarrow \overline{T}_2 \rightarrow 0
\]
splits since $H^0(\BQ_p,  \overline{T})\not=0$. 
In particular, we have the exact sequence
\[
0 \rightarrow 
H^1(\mathbb{Q}_p,\overline{T}_1) \rightarrow 
H^1(\mathbb{Q}_p,\overline{T}) \rightarrow H^1(\mathbb{Q}_p,\overline{T}_2) \rightarrow 0. 
\]
Since $T_1$ is neither isomorphic to 
$\mathcal{O}(1)$ nor $\mathcal{O}$, the $\CO$-module $H^1(\mathbb{Q}_p, T_1)$ is free of rank one and equals 
$H^1_{\rm{f}}(\mathbb{Q}_p, T_1)$. 
So the image of $\overline{H^1(\mathbb{Q}_p, T_1)}$ 
in $H^1(\mathbb{Q}_p,\overline{T})$ is a one-dimensional 
subspace  contained in $\overline{H^1_{\rm{f}}(\mathbb{Q}_p, T)}$. 
In particular, $\overline{H^1_{\rm{f}}(\mathbb{Q}_p, T)}\not={H^1(\mathbb{Q}_p, C)}$. 
Hence, it suffices to show that $\overline{H^1_{\rm{f}}(\mathbb{Q}_p, T)}\cap{H^1(\mathbb{Q}_p, C)}\not=0$. 

Put $C_0=H^0(\BQ_p, V/T)$.  
The pull-back of the sequence 
\[
0 \rightarrow T \rightarrow V \rightarrow V/T \rightarrow 0
\]
by $C_0$ yields the exact sequence 
\[
0 \rightarrow T \rightarrow T' \rightarrow C_0 \rightarrow 0
\]
of $G_{\BQ_p}$-representations. 
This defines the residual exact sequence 
\[
0 \rightarrow C \rightarrow \overline{T}  \rightarrow \overline{T}' \rightarrow \overline{C}_0 \rightarrow 0.
\] 
In the following we simply denote the image of $\overline{T}  \rightarrow \overline{T}'$ by $C(1)$.

First, assume $T'$ is anomalous. Then 
\[0 \rightarrow C(1) \rightarrow \overline{T}' \rightarrow \overline{C}_0 \rightarrow 0\]
splits. Hence, the sequence
\[0 \rightarrow H^1(\BQ_p, C) \rightarrow H^1(\BQ_p, \overline{T}) \rightarrow  H^1(\BQ_p, \overline{T}')\] is exact. 
Therefore, the natural image of $C_0 \isom H^1(\BQ_p, T)_{\mathrm{tor}}$ in $H^1(\BQ_p, \overline{T})$ 
defines a one-dimensional subspace of $H^1(\BQ_p, C) \cap \overline{H^1_{\rm{f}}(\mathbb{Q}_p, T)}$. 

Assume that $T'$ is generic. Hence, it has local sign decomposition as in Theorem~\ref{thm, main}. Consider the commutative diagram 
\[
\xymatrix{
H^1(\mathbb{Q}_p, T) \ar[d]\ar[rr]^{\!\!\!\!\!\!f} & &H^1_+(\mathbb{Q}_p, T')
 \oplus H^1_-(\mathbb{Q}_p, T') \ar[d]\\ 
H^1(\mathbb{Q}_p, \overline{T}) \ar[r]^g & 
  H^1(\mathbb{Q}_p, C(1)) \ar[r]^{\!\!\!\!\!\!\!\!\!\!\!\!\!\!\!\!\!\!h} & H^1_+(\mathbb{Q}_p, \overline{T}')\oplus H^1_-(\mathbb{Q}_p, \overline{T}').   
}
\]
The image of $H^1_{\rm{f}}(\mathbb{Q}_p, T)$ by $f$ 
is contained in $H^1_{\rm{f}}(\mathbb{Q}_p, T')=H^1_{-}(\mathbb{Q}_p, T')$. 
Hence, the image of $\overline{H^1_{\rm{f}}(\mathbb{Q}_p, T)}$ by $h\circ g$ is in 
$H^1_{-}(\mathbb{Q}_p, \overline{T}')$. However,
the image of $h$ is contained in $H^1_{+}(\mathbb{Q}_p, \overline{T}')$ 
 by Corollary \ref{cor, the reducible sign}. 
Therefore, the image of $\overline{H^1_{\rm{f}}(\mathbb{Q}_p, T)}$ by $h\circ g$ 
must be $0$, and so 
$g(\overline{H^1_{\rm{f}}(\mathbb{Q}_p, T)})$ is contained in 
the one-dimensional subspace $\ker(h)$. 
Since $\dim_{\BF}\overline{H^1_{\rm{f}}(\mathbb{Q}_p, T)}=2$, it follows that 
$\overline{H^1_{\rm{f}}(\mathbb{Q}_p, T)}$ 
and $\ker(g)=H^1(\mathbb{Q}_p, C)$ have non-trivial intersection. 
\end{proof}

\begin{prop}\label{prop, reducible MRC 2}
Suppose that ${T}_1\cong \mathcal{O}(1)$. Then we have 
\[\delta_p(\overline{H^1_{\rm{f}}(\BQ_p,T)}, H^1(\BQ_p,C))=-\hat{\varepsilon}_p(V).\]
Moreover,  if 
$T$ is decomposable, then $\overline{H^1_{\rm{f}}(\BQ_p,T)}\cap H^1(\BQ_p,C)=\overline{H^1_{\rm{f}}(\BQ_p, \mathcal{O})}$ and  if $T$ is non-crystalline, then 
$\overline{H^1_{\rm{f}}(\BQ_p,T)}\cap H^1(\BQ_p,C)=0$. 
\end{prop}

\begin{proof}
If $T=\mathcal{O}\oplus \mathcal{O}(1)$, 
then we have $$H^1_{\rm{f}}(\BQ_p, T)=H^1_{\rm{f}}(\BQ_p, \mathcal{O})\oplus H^1_{\rm{f}}(\BQ_p, \mathcal{O}(1))\cong \mathcal{O} \oplus \mathcal{O}.$$ 
Hence, $\overline{H^1_{\rm{f}}(\BQ_p,T)}\cap H^1(\BQ_p,C)=\overline{H^1_{\rm{f}}(\BQ_p, \mathcal{O})}$. 

Suppose that $T$ is a non-split extension of $\mathcal{O}$ by $\mathcal{O}(1)$. 
As in the proof of Proposition \ref{prop, reducible MRC 1}, 
 the image of $\overline{H^1_{\rm{f}}(\mathbb{Q}_p, T_1)}$ 
in $H^1(\mathbb{Q}_p,\overline{T})$ is a one-dimensional 
subspace  contained in $\overline{H^1_{\rm{f}}(\mathbb{Q}_p, T)}$. 
In particular, $\overline{H^1_{\rm{f}}(\mathbb{Q}_p, T)}\not={H^1(\mathbb{Q}_p, C)}$. 
Hence, it suffices to show that 
\[
\overline{H^1_{\rm{f}}(\mathbb{Q}_p, T)}\cap{H^1(\mathbb{Q}_p, C)} =
\begin{cases}
\text{non-zero}  \qquad \text{if \;$T$ crystalline,}\\
 0  \qquad \text{if \;$T$ non-crystalline.}
\end{cases}
\]

We proceed as in the proof of 
Proposition \ref{prop, reducible MRC 1}, and use the notation therein. 
 First, we show that $T'$ is generic. 
 Take an $\CO$-basis $\{e_1, e_2\}$ of $T$ such that 
 \[
 \sigma(e_1)=\kappa_{\rm{cyc}}(\sigma)e_1, \quad 
  \sigma(e_2)=e_2+b(\sigma) e_1
 \]
 for $\sigma \in G_{\BQ_p}$ and $b(\sigma) \in \mathcal{O}$. 
 Let $(\pi^n)$ be the ideal generated by $b(\sigma)$ for all 
 $\sigma \in G_{\BQ_p}$.
 Put $b'(\sigma):=b(\sigma)/\pi^n$ and $e_2':= e_2/\pi^n$. 
 Then we have $$C_0\cong \mathcal{O}/(\pi^n).$$ 
 It follows that $T'=\mathcal{O}e_1+\mathcal{O}e_2'$. 
 Therefore, the extension class is represented by the cocycle $b'(\sigma)$, 
 which is non-zero modulo $\pi$. Hence, $T'$ is generic. 

If $V$ is crystalline, then $\hat{\varepsilon}_p(V)=+1$ and the argument in the proof of 
Proposition \ref{prop, reducible MRC 1} applies. 
Suppose that $V$ is non-crystalline. 
Put $T_2':=\mathcal{O}e_2'$. Then, we have a commutative diagram with 
exact rows 
\[
\xymatrix{
0 \ar[r]  &T_1 \ar[r]\ar@{=}[d] & T \ar[r] \ar[d]&T_2  \ar[r]\ar[d]  & 0\\
0 \ar[r] &T_1 \ar[r] & T' \ar[r] &T_2'  \ar[r]  & 0.
}
\]
Since $T'$ is generic, the image of $H^1_{\rm{f}}(\mathbb{Q}_p, T_1)$ 
in $H^1(\mathbb{Q}_p, T')$ equals $H^1_{\rm{f}}(\mathbb{Q}_p, T')$ by Proposition \ref{prop, reducible sign}. 
Hence, the map 
\[
 \overline{H^1_{\rm{f}}(\mathbb{Q}_p, T)} 
\rightarrow \overline{H^1_{\rm{f}}(\mathbb{Q}_p, T')}=\overline{H^1_{-}(\mathbb{Q}_p, T')}=H^1_{-}(\mathbb{Q}_p, \overline{T}')
\]
is surjective. 
In particular, the image of $\overline{H^1_{\rm{f}}(\mathbb{Q}_p, T)}$ by $h\circ g$ is one-dimensional. 
Let $\overline{v}$ be an element of $\overline{H^1_{\rm{f}}(\mathbb{Q}_p, T)}$ such that 
$$h\circ g(\overline{v})\not=0.$$
On the other hand, $\overline{H^1(\mathbb{Q}_p, T)_{\rm{tor}}}$ is not contained in the kernel of $g$. 
In fact, the element $w \in H^1(\mathbb{Q}_p, T)_{\rm{tor}}$ represented by 
the generator $e_2' \mod T \in C_0$ is given 
by the cocycle $b'(\sigma) e_1$ and so its reduction $\overline{w}$ is  not contained in $C=\BF\overline{e_2}$. Therefore, 
$\overline{w} \notin H^1(\mathbb{Q}_p, C)$, and we also have $h\circ g(\overline{w})=0$ 
since $f(w)=0$. 
Hence, $\{\overline{v}, \overline{w}\}$ is an $\BF$-basis of 
$\overline{H^1_{\rm{f}}(\mathbb{Q}_p, T)}$, and so
$\overline{H^1_{\rm{f}}(\mathbb{Q}_p, T)}\cap{H^1(\mathbb{Q}_p, C)}=0.$ 
\end{proof}

 \subsubsection{The irreducible case} 
  Let ${T}$ be an anomalous symplectic self-dual de Rham representation of $G_{\BQ_p}$ of rank two over an integer ring $\CO$ of a $p$-adic local field.
  Let $\pi \in \CO$ be a uniformizer and $\BF$ the residue field. Put $V=T\otimes_{\BZ_p}\BQ_p$, 
which we assume to be irreducible throughout this subsection.

We approach Theorem~\ref{thm, MR} by relating $T$ to a generic representation by a sequence of isogenies, and employing local sign decomposition for the latter. The main result is Theorem~\ref{thm, anomalous basis} below. 

Let $C' \subset V/T$ be the one-dimensional $\BF$-vector space with trivial $G_{\BQ_p}$-action associated with $C$. 
Then, by the pull-back by $C'$ of 
 \begin{equation}\label{equation, Tp}
\xymatrix{
0 \ar[r] & T \ar[r]
  &  
V  \ar[r]  &  V/T \ar[r]  &  0, 
}
 \end{equation}
  we have\footnote{The notation $T_1$ differs from the one in \S\ref{subsubsection, title}.} an exact sequence 
   \begin{equation}\label{equation, T_1}
\xymatrix{
0 \ar[r] & T \ar[r]
  &  
T_1  \ar[r]  &  C' \ar[r]  &  0. 
}
 \end{equation}
 The Snake lemma then leads to  an exact sequence 
  \begin{equation}
  \xymatrix{
0 \ar[r] & C \ar[r]
  &  
\overline{T}   \ar[r]  & \overline{T}_1 \ar[r]^\rho  &  C'
\ar[r]   & 0. 
}
 \end{equation}
 
 We begin with a preliminary. 
 \begin{lem}\label{lem, torsion places}
 We regard $C'$ as a submodule of $H^1(\BQ_p,T)$ by 
 the connecting map of (\ref{equation, T_1}). 
The natural image of $C' \subset H^1(\BQ_p,T)$ in $H^1(\BQ_p,\overline{T})$ is contained in $H^1(\BQ_p,C)$  
if and only if $H^0(\BQ_p,\overline{T}_1)\not=0$. In particular, 
$\rho$ has a $G_{\BQ_p}$-equivariant splitting if and only if 
the image of $C'$ is contained  in $H^1(\BQ_p,C)$. 
\end{lem}
\begin{proof}
 Fix a generator $c'$ of $C'$ and let $t_1 \in T_1$ be a lift of $c'$. 
 For $\sigma \in G_{\BQ_p}$, put $$t_\sigma:=\sigma(t_1)-t_1 \in T \text{ and } 
 \overline{t}_\sigma:=t_\sigma \mod \pi T.$$ 
 Note that the latter is a cocycle associated with $c'$ in $H^1(\BQ_p,\overline{T})$. 
 Since $C$ is generated by $c:=\pi c'$, the element defined by the cocyle
 $\overline{t}_\sigma$ is in $H^1(\BQ_p,C)$ if and only if $\overline{t}_\sigma=a_\sigma c$ for some $a_\sigma \in \mathcal{O}$. 
In other words, if and only if 
 \[
 \sigma(t_1)-t_1=a_\sigma\pi t_1+\pi s_\sigma=\pi(a_\sigma t_1+s_\sigma)
 \]
 for some $a_\sigma \in \mathcal{O}$ and $s_\sigma \in T$. 
 
 Note that a general element of $T_1$ is of the form $at_1+s$ for some $a\in \mathcal{O}$ and $s \in T$. 
 Hence, the above displayed condition is equivalent to $\sigma(t_1)-t_1 \in \pi T_1$ for all $\sigma$, and in turn to 
$\overline{t_1} \in H^0(\BQ_p,\overline{T}_1)$. 
On the other hand, suppose that there exists a non-zero element $\overline{s}_1 \in H^0(\BQ_p,\overline{T}_1)$.  
Then $\rho$ induces $H^0(\BQ_p,\overline{T}_1)\cong H^0(\BQ_p,C')$, and we may assume that $\rho(\overline{s}_1)=c'$.
Hence, $\overline{s}_1=\overline{t}_1$ and the proof concludes. 
  \end{proof}

If $T_1$ is also anomalous, then we construct $T_2$ by replacing $T$ in the construction \eqref{equation, T_1} with $T_1$. 
Hence, inductively, we obtain a sequence 
$$ 
T_0:=T \subset T_1 \subset T_2 \subset \cdots \subset T_m
$$
of $G_{\BQ_p}$-stable $\CO$-lattices in $V$ such that 
\begin{equation}\label{eq, lattice chain}
H^0(\BQ_p,\overline{T}_i)\not=0 \text{ if $0\leq i<m$ and } H^0(\BQ_p,\overline{T}_m)=0.
\end{equation}
Note that such an $m$ exists by the irreducibility of $V$. 
For $0\leq i<m$, write $C'_i \subset H^0(\BQ_p,V/T_i)$ for the submodule corresponding to the one-dimensional trivial representation in $\overline{T}_i$. 
For simplicity of notation, we often let $C$ denote the one-dimensional trivial representation in $\overline{T}_i$ and
also identify it with $C'_i$ by the multiplication by $\pi$. By construction, we have an exact sequence 
 \begin{equation}\label{equation, T_i}
\xymatrix{
0 \ar[r] & T_i \ar[r]
  &  
T_{i+1}  \ar[r]  &  C'_i \ar[r]  &  0 
}
 \end{equation}
of $G_{\BQ_p}$-representations.

\begin{lem} Pick an integer $0\leq i \leq m$.  Let $C_i:=H^0(\BQ_p,V/T_i)$ be the $\CO$-torsion part of $H^1(\BQ_p,T_i)$. 
\begin{itemize}
\item[i)] The free part $H^1(\BQ_p,T_i)_{\rm{fr}}:=H^1(\BQ_p,T_i)/C_i$ is of rank two, and the image 
 $\overline{H^1(\BQ_p,T_i)} \subset H^1(\BQ_p,\overline{T}_i)$
 of  $H^1(\BQ_p,T_i)$  is three dimensional. 
 Moreover, there exists an $\BF$-basis of $\overline{H^1(\BQ_p,T_i)}$ consisting of the image of an $\CO$-basis of $H^1(\BQ_p,T_i)_{\rm{fr}}$ and a generator of $C_i$. 
\item[ii)] For $0<i<m$, the residual representation $\overline{T}_i$ decomposes as $\overline{T}_i^+\oplus \overline{T}_i^-$ 
where $\overline{T}_i^- \cong C$ and by the canonical projection, we have 
$\overline{T}_i^+\cong \overline{T}_{i-1}^+\cong C(1)$. 
\end{itemize}
\end{lem}
\begin{proof}
i) The first part is clear since $H^1(\BQ_p,V_i)=H^1(\BQ_p,V)$ is two dimensional. 
Hence, we abstractly have 
 \[
 H^1(\BQ_p,T_i) \cong \mathcal{O}\oplus \mathcal{O} \oplus C_i
 \]
and so $\overline{H^1(\BQ_p,T_i)}:=H^1(\BQ_p,T_i)/\pi$ is three dimensional. \\
ii) By the Snake lemma, we have an exact sequence 
  \begin{equation}
  \xymatrix{
0 \ar[r] & C \ar[r]
  &  
\overline{T}_{i-1}   \ar[r]  & \overline{T}_i \ar[r]  &  C'_{i-1}
\ar[r]   & 0. 
}
 \end{equation}
 Hence, $\overline{T}_i$ contains a submodule isomorphic to $\overline{T}_{i-1}^+\cong C(1)$. Since $i<m$, it also contains a submodule isomorphic to $C$. 
\end{proof}
 
 Put $T=T_0$, $\overline{T}_0^-:=C$ and $\overline{T}_0^+:=\overline{T}/C$.

\begin{thm}\label{thm, anomalous basis}
  Let ${T}$ be an anomalous symplectic self-dual de Rham representation of $G_{\BQ_p}$ of rank two over an integer ring $\CO$ of a $p$-adic local field. Suppose that $T \otimes_{\BZ_p}\BQ_p$ is irreducible. 
  Let $T=T_{0} \subset T_1 \subset \dots \subset T_m$ be a chain of $G_{\BQ_p}$-stable $\CO$-lattices as above with $T_m$ being generic.
Let $v^\pm_m$ be an $\CO$-basis of the signed submodule $H^1_\pm(\BQ_p,T_m) \subset H^1(\BQ_p,T_m)$. 
Then for $0\leq i<m$, there exist elements $v^+_i, v^-_i \in H^1(\BQ_p,T_i)$ with the following properties. 

Write $H^1_\pm(\BQ_p,T_i)$ for the $\CO$-submodule of $H^1(\BQ_p,T_i)$ generated by $v_i^\pm$ and 
the torsion part $C_i$. 
\begin{itemize}
\item[1)] The elements $v^+_i, v^-_i$  form a basis of $H^1(\BQ_p,T_{i})_{\rm{fr}}$.  
\item[2)] $\overline{H^1_\pm(\BQ_p,T_i)}$ is a Lagrangian subspace of $H^1(\BQ_p,\overline{T}_i)$. 
 \item[3)] The $\BF$-dimension of $\overline{H^1_+(\BQ_p,T_i)}\cap H^1(\BQ_p,\overline{T}_i^-)$ is even, and 
 that of $\overline{H^1_-(\BQ_p,T_i)}\cap H^1(\BQ_p,\overline{T}_i^-)$ odd and so the latter equals $1$. 
In other words, 
 for $\epsilon \in \{\pm\}$, we have 
 \[
\delta_p(\overline{H^1_\epsilon(\BQ_p,T_i)}, H^1(\BQ_p,\overline{T}_i^-))=(-1)^{\dim_{\BF} \overline{H^1_\epsilon(\BQ_pT_i)}\cap H^1(\BQ_p,\overline{T}_i^-)}=\epsilon \cdot 1.
 \]
 \item[4)] Consider the map $f_i: H^1(\BQ_p,T_i) \rightarrow H^1(\BQ_p,T_{i+1})$ induced by \eqref{equation, T_i}. 
 If $H^1(\BQ_p,\overline{T}_i^-) \cap \overline{C}_i=0$, then 
 \[
 f(v_i^+)-v^+_{i+1} \in C_{i+1}, \quad f(v_i^-)=\pi v^-_{i+1}. 
 \]
 If $\overline{C}_i \subset H^1(\BQ_p,\overline{T}_i^-)$, then 
 \[
 f(v_i^+)=\pi v^+_{i+1}, \quad f(v_i^-)-v^-_{i+1} \in C_{i+1}. 
 \]
\end{itemize}
In particular, we have $H^1_{\rm{f}}(\BQ_p,T_i)=H^1_{\epsilon}(\BQ_p,T_i)$ for 
$\epsilon=-\hat{\varepsilon}_p(V)$, and 
 \[
\delta_p(\overline{H^1_{\rm{f}}(\BQ_p,T)}, H^1(\BQ_p,C))
=-\hat{\varepsilon}_p(V).
 \]
\end{thm}
\begin{proof}
Note that the property 2) is a consequence of the property 4). 
Indeed, the cup product pairing $(x, c) \in \mathcal{O}$ must be $0$ for any $x \in H^1(\BQ_p,T_i)$ and $c \in C_i$, 
and 4) implies that  $v_i^\pm$ is isotropic. 
Hence, the subspace $\overline{H^1_\pm(\BQ_p,T_i)}$ is maximal isotropic. 
In particular, $(\overline{v_i}^+, \overline{v_i}^-)\not=0$.

In the following we construct the desired elements $v_i^\pm$
by induction for $i=m-1, m-2, \cdots, 0$. 

\vskip2mm

{\it The base case}. 
First, suppose that $i=m-1$. 
Then the $\CO$-module module $H^1(\BQ_p,{T}_m)$ is free by property 3) of Proposition~\ref{a}, and so  $C_m=0$. 
The exact sequence (\ref{equation, T_i}) induces the following one
  \begin{equation}
  \xymatrix{
0 \ar[r] & C'_{m-1} \ar[r]
  &  
H^1(\BQ_p,{T}_{m-1})   \ar[r]^{f_{m-1}}  & H^1(\BQ_p,{T}_m) \ar[r]  &  H^1(\BQ_p,C'_{m-1})
\ar[r] &  H^2(\BQ_p,T_{m-1})
\ar[r]   & 0. 
}
 \end{equation}
 Since $C'_{m-1}$ is a one dimensional $\BF$-vector space with trivial $G_{\BQ_p}$-action and 
  $H^2(\BQ_p,T_{m-1})$ is also a one dimensional $\BF$-vector space, so is the cokernel of $f_{m-1}$. 
 Note also that $C'_{m-1}=C_{m-1}$. 
 Then we have the following commutative diagram 
  {\footnotesize
 \begin{equation}
\xymatrix{
      C'_{m-1}  \ar@{^{(}->}[r]        &  H^1(\BQ_p,T_{m-1})       \ar[d] \ar[rr]^{f_{m-1}}   &    & 
      H^1_+(\BQ_p,T_m)\oplus H^1_-(\BQ_p,T_m)   \ar[d]  \\
                   \quad H^1(\BQ_p,\overline{T}^-_{m-1})\ar@{^{(}->}[rd]       & \overline{H^1(\BQ_p,T_{m-1})}   \ar[d]\ar[rr]   & &      \overline{H^1_+(\BQ_p,T_m)}\oplus            \overline{H^1_-(\BQ_p,T_m)}  \ar@{=}[d] \\
                                                 & H^1(\BQ_p,\overline{T}_{m-1}) \ar[rr]  \ar[rd]^{p_{m-1}}        &  &    
               H^1_+(\BQ_p,\overline{T}_{m})\oplus H^1_-(\BQ_p,\overline{T}_{m}) \\
                                                 &                                                         & \qquad H^1(\BQ_p,\overline{T}^+_{m-1})
               \cong H^1(\BQ_p,\overline{T}^+_{m}) \ar[ur]^{\iota_{m}}  \ar[rd] &  \\
                                                             &       
                            H^0(\BQ_p,\overline{T}^-_{m})\cong \BF\;    \ar@{^{(}->}[ur]                                          &     &  0.
}
 \end{equation}
 }
Let $\{w_1, w_2 \} \subset H^1(\BQ_p, T_{m-1})$ be a (lift of) basis of the free part. 
Write 
$$f_{m-1}(w_1)=av^+_m+bv^-_m, \quad f_{m-1}(w_2)=cv^+_m+dv_m^-.$$
Since the image of $\iota_{m-1}$ is the one dimensional $\BF$-vector space $H^1_+(\BQ_p,\overline{T}_{m})$ by 
Corollary \ref{cor, the reducible sign}, 
we have $\pi | b$ and $\pi| d$. Since the cokernel of $f_{m-1}$ is one dimensiontal $\BF$-vector space, 
$a$ or $c$ is a unit.
Hence, by changing the variable, we may assume that $a=1$, $c=0$. 
Then considering the cokernel of $f_{m-1}$ again, $d/\pi$ must be a unit.  
Hence, we can choose $w_1$, $w_2$ so that  $f_{m-1}(w_1)=v_m^+$ and  $f_{m-1}(w_2)=\pi v_m^-$. 
Since $H^0(\BQ_p,\overline{T}_m)=0$, we have $\overline{C}_{m-1} \cap H^1(\BQ_p,\overline{T}_{m-1}^-)=0$ 
by Lemma \ref{lem, torsion places}, and so 
the kernel of $\iota_{m-1}$ is isomorphic to the image of $C'_{m-1}$. 
Hence, we can modify $w_2$ by an element of $C'_{m-1}$ so that $\overline{w}_2 \in H^1(\BQ_p,\overline{T}_{m-1}^-)$.
Then define $v_{m-1}^+=w_1$ and  $v_{m-1}^-=w_2$.  

\vskip2mm
{\it The inductive step}.
We now suppose that $i<m-1$ and that the desired elements $v_{i+1}^+$ and  $v_{i+1}^-$ exist by the induction hypothesis.

The exact sequence (\ref{equation, T_i}) induces the following one
 {\begin{equation}\label{equation, HT_iT_i+1}
  \xymatrix{
  C_i' \ar@{^{(}->} [r]  
  &  
H^1(\BQ_p,{T}_{i})   \ar[r]^{f_i}  & H^1(\BQ_p,{T}_{i+1}) \ar[r]^{g_{i+1}}  &  H^1(\BQ_p,C_i')
\ar[r]^h   \ar[r] &  H^2(\BQ_p,T_{i})
\ar@{->>}[r] &  H^2(\BQ_p,T_{i+1}).
}
 \end{equation}
 }
 Since $H^2(\BQ_p,T_{i})=C_i^\vee$ and $H^2(\BQ_p,T_{i+1})=C_{i+1}^\vee$ are both cyclic, 
 we have 
 \begin{equation}\label{size cases}
 |C_i|=|C_{i+1}| \text{ or } |C_i|=|\BF||C_{i+1}|.
 \end{equation}
 \vskip2mm
 {\it Case I}. Suppose that $|C_i|=|\BF||C_{i+1}|$. 
 
 Then in the exact sequence  (\ref{equation, HT_iT_i+1}), 
 the image of $h$ is one dimensional and hence, so is $\mathrm{Im} \,g_{i+1}$. 
The map $f_i$ induces a map $C_i \rightarrow C_{i+1}$ with kernel $C_i'$. By looking at the order, 
it must be surjective. Hence, $g_{i+1}(C_{i+1})=0$, and 
one of the elements $\{v_{i+1}^+,v_{i+1}^-\}$ generates the one dimensional space $\mathrm{Im}\, g_{i+1}$. 

We have the following commutative diagram
  {
 \begin{equation}\label{equation, big diagram II}
\xymatrix{
 C_i' \ar@{^{(}->} [r]       &  H^1(\BQ_p,T_{i})       \ar[d] \ar[rr]^{f_i}   &   &   H^1(\BQ_p,T_{i+1})   \ar[d]  \ar@{->>}[r] & \mathrm{Im}\,g_{i+1} 
\ar[d] \\
              H^1(\BQ_p,\overline{T}_{i}^-) \ar@{^{(}->}[rd]       & \overline{H^1(\BQ_p,T_{i})}   \ar[d]\ar[rr]   &   &  \overline{H^1(\BQ_p,T_{i+1})}
                 \ar[d] \ar[r] & H^1(\BQ_p,C'_i)\\
                                & H^1(\BQ_p,\overline{T}_{i}) \ar[rr]^{\overline{f}_i}  \ar@{->>}[rd]^{p_i}      &  &   
                H^1(\BQ_p,\overline{T}_{i+1}) \ar@{->>}[ur]_{\overline{g}_{i+1}} \ar@{->>}[rd]^{p_{i+1}}&                 \\
                                          &                                                         &  H^1(\BQ_p,\overline{T}_{i}^+) \ar@{^{(}->}[ur]^{\iota_i}  &  
               &H^1(\BQ_p,\overline{T}_{i+1}^+).
}
 \end{equation}
 }
 By construction, $p_{i+1}\circ \iota_i$ is an isomorphism. 
 In particular, $\mathrm{Im} \,\iota_i$ 
does not intersect non-trivially with $\ker\,p_{i+1}=H^1(\BQ_p,\overline{T}_{i+1}^-)$.
 Since $\overline{f}_i$ induces an isomorphism $\overline{C}_i \cong \overline{C}_{i+1}$ on one dimensional spaces, 
 we have $\overline{C}_{i}\cap H^1(\BQ_p,\overline{T}_{i}^-)=0$.
 
 We now determine the image of $\iota_i$. 
Let $c_{i+1}$ be a generator of $C_{i+1}$. 
 Since $\overline{c}_{i+1} \in \overline{f}_i(\overline{C}_i)$, it is contained in $\mathrm{Im} \,\iota_i$. 
In particular, $H^1_+(\BQ_p,T_{i+1})\not= H^1(\BQ_p,\overline{T}_{i+1}^-)$, and so
$$\dim_{\BF}H^1_+(\BQ_p,T_{i+1})\cap H^1(\BQ_p,\overline{T}_{i+1}^-)=0, \quad \dim_{\BF} H^1_-(\BQ_p,T_{i+1})\cap H^1(\BQ_p,\overline{T}_{i+1}^-)=1$$
by the induction hypothesis. 
 In particular, the images by $p_{i+1}$ of $\overline{v}_{i+1}^+, \overline{c}_{i+1} \in H^1_+(\BQ_p,T_{i+1})$ generate the two dimensional space 
$H^1(\BQ_p,\overline{T}_{i+1}^+)$, and 
$\overline{v}_{i+1}^-+\alpha \overline{c}_{i+1} \in  H^1(\BQ_p,\overline{T}_{i+1}^-)$ for some $\alpha \in \mathcal{O}$. 
Take $d_{i+1} \in H^1(\BQ_p,\overline{T}_{i+1}^-)$ so that 
 $\{ \overline{v}_{i+1}^+, \overline{v}_{i+1}^-, \overline{c}_{i+1}, d_{i+1}\}$ is 
 a basis of $H^1(\BQ_p,\overline{T}_{i+1})$. 
 Note that $\mathrm{Im} \,\iota_i$ is a two dimensional space containing $\overline{c}_{i+1}$ and not 
 intersecting $H^1(\BQ_p,\overline{T}_{i+1}^-)$. So 
 there exist $a, b \in \BF$ such that 
 \begin{equation}\label{Im}
 \overline{v}_{i+1}^++a(\overline{v}_{i+1}^-+\alpha \overline{c}_{i+1})+bd_{i+1} \in \mathrm{Im} \,\iota_i.
 \end{equation}
 Since $\mathrm{Im} \,\iota_i$ is Lagrangian and $\overline{c}_{i+1} \in \mathrm{Im} \,\iota_i$, the above element is  orthogonal to $\overline{c}_{i+1}$. 
 The element $d_{i+1}$ is not orthogonal to $c_{i+1}$ because otherwise, $c_{i+1}$ is orthogonal to all elements 
 of $H^1(\BQ_p,\overline{T}_{i+1})$. Therefore, we have $b=0$. 
 Then the element \eqref{Im} is isotropic only when $a=0$. 
Hence, the space  $\mathrm{Im} \,\iota_i$ is generated by  $\overline{v}_{i+1}^{+}$ and $\overline{c}_{i+1}$.

Since  $C_i \rightarrow C_{i+1}$ is surjective, we may 
choose a (lift of) basis $\{w_1, w_2\}$ of the free part of $H^1(\BQ_p,T_i)$ 
satisfying 
\[f_i(w_1)=mv^+_{i+1}+n(v^-_{i+1}+\alpha c_{i+1}), \quad f_i(w_2)=m'v^+_{i+1}+n'(v^-_{i+1}+\alpha c_{i+1})\]
for some $m, n, m', n' \in \mathcal{O}$.  
Since $\overline{f_i(w_j)} \in  \mathrm{Im} \,\iota_i$, we have $n \equiv n' \equiv 0 \mod \pi$. 
Considering that $\mathrm{Im} \,g_{i+1}$ is a one dimensional $\BF$-vector space, $m$ or $m'$ is a unit.  
Hence, by changing the basis, we may assume that $m=1$ and $m'=0$.  Then again, by  the one dimensionality of $\mathrm{Im} \,g_{i+1}$, 
the ratio $n'/\pi$ must be a unit. Hence, by changing the basis, we may assume that 
\[f_i(w_1)=v^+_{i+1}, \quad f_i(w_2)=\pi(v^-_{i+1}+\alpha c_{i+1}).\]
Since $\pi\alpha c_{i+1} \in f_i(C_i)$, we obtain the desired basis by putting $w_1=v_i^+$ and $w_2=v_i^-$. 
Then the image of $\overline{H^1_+(\BQ_p,T_i)}$ by $\overline{f}_i$ is 
 the two dimensional space  $\mathrm{Im} \,\iota_i$. Hence, $\dim_{\BF}\overline{H^1_+(\BQ_p,T_i)}\cap H^1(\BQ_p,\overline{T}_i^-)=0$. We have $v_i^-  \in \ker\,\overline{f_i}=\ker\,p_i=H^1(\BQ_p,\overline{T}_i^-)$ 
 and $\overline{C_i} \cap H^1(\BQ_p,\overline{T}_i^-)=0$. It follows that 
 $\dim_{\BF}\overline{H^1_-(\BQ_p,T_i)}\cap H^1(\BQ_p,\overline{T}_i^-)=1.$  
 \vskip2mm
 {\it Case II}. 
 Suppose that $|C_i|=|C_{i+1}|$  in \eqref{size cases}. 
 
 Then we have the exact sequence 
  \begin{equation}
  \xymatrix{
0 \ar[r] & C'_i \ar[r]
  &  
H^1(\BQ_p, {T}_{i})   \ar[r]^{f_i}  & H^1(\BQ_p,{T}_{i+1}) \ar[r]  &  H^1(\BQ_p,C_i')
\ar[r]   & 0. 
}
 \end{equation}
 Note that $f_i$ induces a map $C_i \rightarrow C_{i+1}$ whose kernel is $C_i'$. 
By considering the order, the cokernel $C_{i+1}/f_i(C_i)$ is one-dimensional. 
 Hence, the image of a generator $c_{i+1}$ of $C_{i+1}$ in $H^1(\BQ_p,C_i')$ is non-zero. 
The  three dimensional $\BF$-vector space  $\overline{H^1(\BQ_p,{T}_{i+1})}$ has a basis 
$\{\overline{v}_{i+1}^+, \overline{v}_{i+1}^-, \overline{c}_{i+1}\}$, and 
the $\BF$-vector space $H^1(\BQ_p,C_i')$ is generated by the image of 
$c_{i+1}$, and one of the images of $v_{i+1}^+$ and $v_{i+1}^-$. 

As before, we proceed via the commutative diagram (\ref{equation, big diagram II}).  
 We now have two cases depending on the dimension of $\overline{H^1_+(\BQ_p,T_{i+1})}\cap H^1(\BQ_p,\overline{T}_i^-)$, which is $0$ or $2$ 
 by the induction hypothesis. 
 \vskip2mm
 {\it Subcase 1}. First suppose that $$\overline{H^1_+(\BQ_p,T_{i+1})}\cap H^1(\BQ_p,\overline{T}_{i+1}^-)=0.$$
 
 Then the images of $\overline{v}_{i+1}^+$ and $\overline{c}_{i+1}$ by $p_{i+1}$ generate 
 $H^1(\BQ_p,\overline{T}_{i+1}^+)$, and $\overline{v}_{i+1}^-+\alpha \overline{c}_{i+1} \in  H^1(\BQ_p,\overline{T}_{i+1}^-)$ for some $\alpha \in \mathcal{O}$. 
 Take $d_{i+1} \in H^1(\BQ_p,\overline{T}_{i+1}^-)$ so that 
 $\{ \overline{v}_{i+1}^+, \overline{v}_{i+1}^-, \overline{c}_{i+1}, d_{i+1}\}$ is
 a basis of $H^1(\BQ_p, \overline{T}_{i+1})$. 
 Since $p_{i+1}\circ \iota_i$ is an isomorphism, there exist $a, b, c, d \in \BF$ such that 
 \[
 \overline{v}_{i+1}^{++}:=\overline{v}_{i+1}^++a(\overline{v}_{i+1}^-+\alpha \overline{c}_{i+1})+bd_{i+1}, 
 \quad  \overline{c}_{i+1}':=\overline{c}_{i+1}+c(\overline{v}_{i+1}^-+\alpha \overline{c}_{i+1})+dd_{i+1}
 \in \mathrm{Im} \,\iota_i.
 \]
 Since $H^1(\BQ_p,\overline{T}_{i+1}^-)$ is Lagrangian and $\overline{c}_{i+1}$ is not orthogonal to $d_{i+1}$ as before, $d$ must be zero for 
$\overline{c}_{i+1}'$ to be isotropic. Then $c$ is non-zero, for $\overline{g}_{i+1}(\overline{c}_{i+1})
=g_{i+1}({c}_{i+1})\not=0$ and so $\overline{c}_{i+1} \not\in \mathrm{Im} \,\iota_i$. If  $b=0$, then $a=0$ since 
$\overline{v}_{i+1}^{++}$ is isotropic and in turn $\overline{v}_{i+1}^{++}=\overline{v}_{i+1}^{+}$ is not orthogonal to  $\overline{c}_{i+1}'$
because $c\not=0$. This would contradict the fact that $\mathrm{Im} \,\iota_i$ is Lagrangian and hence  $b\not=0$. 
We conclude that $\mathrm{Im} \,\iota_i$ is generated by 
\[
 \overline{v}_{i+1}^{++}:=\overline{v}_{i+1}^++a(\overline{v}_{i+1}^-+\alpha \overline{c}_{i+1})+bd_{i+1}, 
 \quad  \overline{c}_{i+1}':=\overline{c}_{i+1}+c(\overline{v}_{i+1}^-+\alpha \overline{c}_{i+1})
 \]
 with $bc\not=0$. 
 
As before, choose $\{ w_1, w_2\}$ as a (lift of) basis of the free part of $H^1(\BQ_p,T_i)$, and 
write 
\[f_i(w_1)=l v^+_{i+1}+m(v^-_{i+1}+\alpha c_{i+1})+nc_{i+1}, \quad f_i(w_2)=l'v^+_{i+1}+m'(v_{i+1}^-+\alpha c_{i+1})+n'c_{i+1}\]
for $l, m, n, l', m', n'  \in \mathcal{O}$.  
Since $\overline{f_i(w_j)} \in  \mathrm{Im} \,\iota_i$ and it has no $d_{i+1}$-component,  $\overline{f_i(w_j)}$ 
is a scalar multiple of  $\overline{c}_{i+1}'$. Hence, $l\equiv l' \equiv 0 \mod \pi$. 
Since the cokernel of $f_i$ is a two-dimensional $\BF$-vector space, $m$ or $m'$ must be a unit.
Therefore, by changing the variable if necessary, we may assume that $m=1$ and $m'=0$.  Then $\mathrm{Coker} \, f_i$ is 
generated by the images of $v^+_{i+1}$ and $c_{i+1}$. Since it is a two-dimensional $\BF$-vector space, it follows that 
$l'/\pi$ is a unit, and so we may assume that $l'=\pi$ and $l=0$ by changing the basis. 
Thus, we have 
\[f_i(w_1)=(v^-_{i+1}+\alpha c_{i+1})+nc_{i+1}, \quad f_i(w_2)=\pi v^+_{i+1}+n'c_{i+1},\]
and $\pi| n'$ by the two-dimensionality of $\mathrm{Coker} \, f_i$. 
Since $\pi c_{i+1} \in \mathrm{Im}\, f_i$, there exist elements $v^+_i, v^-_i$ such that 
\[
f_i(v^-_i)-v^-_{i+1} \in C_{i+1}, \quad f_i(v^+_i)=\pi v^+_{i+1}.
\] 
In this case we have $\overline{v}^+_{i}, \overline{c}_i \in H^1(\BQ_p, \overline{T}_{i}^-)$ since they are killed by $\overline{f}_i$, and 
$\overline{H^1_-(\BQ_p,T_{i})}\cap H^1(\BQ_p,\overline{T}_{i}^-)$ is generated by $\overline{c}_i$ since $v^-_{i}$ is not killed by $\overline{f}_i$. 
Hence, $\dim_{\BF}\overline{H^1_+(\BQ_p,T_{i})}\cap H^1(\BQ_p,\overline{T}_{i}^-)=2$ and 
\[\dim_{\BF}\overline{H^1_-(\BQ_p,T_{i})}\cap H^1(\BQ_p,\overline{T}_{i}^-)=1.\] 
\vskip2mm
{\it Subcase 2}. 
 Finally, suppose that $$\dim_{\BF} \overline{H^1_+(\BQ_p,T_{i+1})}\cap H^1(\BQ_p,\overline{T}_{i+1}^-)=2.$$
 
 Then $\overline{v}^+_{i+1}, \overline{c}_{i+1} \in H^1(\BQ_p, \overline{T}_{i+1}^-)$, and 
 the space
$\overline{H^1_-(\BQ_p,T_{i+1})}\cap H^1(\BQ_p,\overline{T}_{i+1}^-)$, which is one dimensional by the induction hypothesis, 
 is generated by $\overline{c}_{i+1}$. 
 Since $H^1(\BQ_p,\overline{T}_{i+1})=\mathrm{Im} \,\iota_i+H^1(\BQ_p,\overline{T}_{i+1}^-)$, 
 there exist $a, b \in \BF$ such that 
 \[
 \overline{v}_{i+1}^{--}:=\overline{v}_{i+1}^-+a\overline{v}_{i+1}^++b \overline{c}_{i+1} \in \mathrm{Im} \,\iota_i. 
 \]
Then take $d_{i+1} \in \mathrm{Im} \,\iota_i$ so that 
 $\{\overline{v}_{i+1}^{--}, d_{i+1}\}$ generates $\mathrm{Im} \,\iota_i$. Since  $\overline{v}_{i+1}^{--}$ is isotropic, we have $a=0$. 
 
As before, choose $\{ w_1, w_2\}$ as a (lift of) basis of the free part of $H^1(\BQ_p,T_i)$, and 
write 
\[f_i(w_1)=l v^+_{i+1}+m({v}_{i+1}^{-}+bc_{i+1})+nc_{i+1}, \quad f_i(w_2)=l'v^+_{i+1}+m'(v_{i+1}^-+b c_{i+1})+n'c_{i+1}\]
for $l, m, n, l', m', n'  \in \mathcal{O}$. 
Since $\overline{f_i(w_j)} \in  \mathrm{Im} \,\iota_i$ and it has no $d_i$-component, 
it follows that $l \equiv l' \equiv n \equiv n' \equiv 0 \mod \pi$. 
Then, by changing the basis, we may assume that $m=1$ and $m'=0$, and in turn $l'=\pi$ and $l=n=n'=0$. 
(Note that $\pi c_{i+1} \in f(C_i)$.) Thus,  
\[f_i(w_1)={v}_{i+1}^{-}+bc_{i+1}, \quad f_i(w_2)=\pi v^+_{i+1}.\]
So we put $v_i^-:=w_1$ and $v_i^+:=w_2$. 
In this case note that $\overline{v}^+_{i}, \overline{c}_i \in H^1(\BQ_p,\overline{T}_{i}^-)$, and 
$\overline{H^1_-(\BQ_p,T_{i})}\cap H^1(\BQ_p,\overline{T}_{i}^-)$ is generated by $\overline{c}_i$. 
Hence, it follows that $\dim_{\BF} \overline{H^1_+(\BQ_p,T_{i})}\cap H^1(\BQ_p,\overline{T}_{i}^-)=2$ and 
\[\dim_{\BF} \overline{H^1_-(\BQ_p,T_{i})}\cap H^1(\BQ_p, \overline{T}_{i}^-)=1.\] 

As for the last assertion, note that $H^1_\pm(\BQ_p,V_i)=H^1_\pm(\BQ_p,V_m)$ and 
$H^1_\pm(\BQ_p,T_i)$ contains the torsion part of $H^1(\BQ_p,T_i)$. 
\end{proof}

\subsection{Relative $p$-parity conjecture}

\begin{thm}\label{thm, parity-family}
Let $F$ be a number field and $p$ an odd prime totally split in $F$. 
Let $T_1$ and $T_2$ be geometric symplectic self-dual $\CO_L$-representations of $G_{F}$ of rank two 
which are residually symplectically isomorphic, where $L$ is a finite extension of $\BQ_p$. 
 Put $V_i = T_i \otimes_{\BZ_p} \BQ_p$.
Then the relative $p$-parity Conjecture \ref{conj, relative p-parity} is true for $V_1$ and $V_2$, i.e. 
\[
(-1)^{\chi_{\rm{f}}(F, V_1)-\chi_{\rm{f}}(F, V_2)}=\varepsilon(V_1)/\varepsilon(V_2). 
\]
\end{thm}

\begin{proof} 
Note that $F_v=\BQ_p$ for $v|p$ by our assumption.

In light of of (\ref{equation, relative selmer}), the compatibility of local constants (\ref{equation, MR vs DL, intro}) for $v \nmid p$ and the compatibility for $v|p$ as established in Theorem \ref{thm, MR}, 
it suffices to show that 
\[
\prod_{v \in S_\infty} \frac{\varepsilon_v(V_1)}{\varepsilon_{v}(V_2)} =
\prod_{v \in S_p} \frac{\Gamma(V_1|_{G_{F_v}})}{\Gamma(V_2|_{G_{F_v}})}. 
\]
By definition \eqref{equation, def arch epsilon} of the left hand side, 
this is equivalent to showing that 
\begin{equation}\label{eq, ssd comp}
(-1)^{\sum_{v \in S_p} g_v^-(V_1)-g_v^-(V_2)}=(-1)^{\sum_{v \in S_p} g_v^+(V_1)-g_v^+(V_2)}.
\end{equation}

The self-duality implies that 
$$
h_{m}(V_i|_{G_{F_v}})=h_{m}(V_i^*(1)|_{G_{F_v}})=h_{-1-m}(V_i^*(1))|_{G_{F_v}} 
$$
for $i\in\{1,2\}$. 
Therefore, letting $ h_m(V_i):=h_{m}(V_i|_{G_{F_v}})$, we have
\begin{align*}
g_v^-(V_i)+g_v^+(V_i) 
&=\sum_{m \in \BZ} mh_m(V_i)\\
&=\sum_{m<0} m\, h_m(V_i)+\sum_{m \geq 0} m\, h_{-1-m}(V_i)\\
&\equiv \sum_{m<0} m\, h_m(V_i)-\sum_{m \geq 0} (-1-m)\, h_{-1-m}(V_i)-\sum_{m \geq 0} \, h_{-1-m}(V_i)\\
&= - \dim_L D_v(V_i)/\mathrm{Fil^0}(D_v(V_i))\\
&=-\dim_L V_i\,/2\\
&=-1. 
\end{align*}
The assertion \eqref{eq, ssd comp} follows from this. 
\end{proof}

\subsection{$p$-parity conjecture for Hilbert modular forms}
\subsubsection{Setup}
Let $F$ be a totally real field.

Denote by $I$ the set of embedding $F \hookrightarrow \overline{\BQ}$
and put $t=\sum_{\sigma \in I}\sigma  \in \BZ[I].$ 
 Let $k=\sum_{\sigma \in I}k_{\sigma}\sigma \in 2\BZ_{>0}[I]$ and 
put 
\[k_{0}:=\max\{k_{\sigma} | \sigma \in I\}, \qquad v:=\frac{1}{2}\cdot(k_0t-k) \in \BZ[I].
\] 
 For an integral ideal $\fn$ of $F$,
put 
\[
U_0(\fn)=\left\{\begin{pmatrix}a&b\\
c&d 
\end{pmatrix} \in \GL_2(\widehat{\CO}_F) \;\middle | \; c \in \fn \widehat{\CO}_F\right\},
\]
where  $\widehat{\CO}_F:=\CO_F\otimes_{\BZ}\widehat{\BZ}$
for $\widehat{\BZ}=\varprojlim_n \BZ/n\BZ$. 
Let $S_{k}(U_{0}(\fn))$ denote the space $S_{k,v+k-t, I}(U_0(\fn); \mathrm{M}_2(F);\BC )$ of cuspidal Hilbert modular forms  of weight $k$ and level $U_0(\fn)$ as in \cite[\S 2]{H} .

Let $f \in S_{k}(U_{0}(\fn))$ be a normalized newform. It is an eigenform for all Hecke operators $T(\fa)$ for ideals $\fa \subseteq \CO_F$ (cf.\ \cite[\S 2]{H}). Fix an embedding $\iota_p: \ov{\BQ}\hookrightarrow \ov{\BQ}_p$.
Let $$\rho_f:G_F\to \Aut_{L}(V_f)$$ be an associated two dimensional irreducible Galois representation over a $p$-adic local field $L$.  
For any prime $\mathfrak{l}\nmid p\fn$ of $F$, the Galois representation $V_f$ is unramified at $\mathfrak{l}$ and 
\[
\mathrm{Tr}(\rho_f(\Frob_\mathfrak{l}))=\lambda_f(\mathfrak{l}),\quad \det(\rho_f(\Frob_\mathfrak{l}))=N_{F/\BQ}(\mathfrak{l})^{k_0-1} 
\]
 (cf.\ \cite[Thm.\ 2.43]{H06}). 
Here $\Frob_\mathfrak{l}$ is the geometric Frobenius at $\fl$, $N_{F/\BQ}$ denotes the norm map, 
$L$ is a sufficiently large finite extension of $\BQ_p$ containing the valuation ring $\widehat{\CO}(v)\subseteq \overline{\BQ}_p$ as in \cite[p.\ 312]{H} and 
the Hecke eigenvalues $\lambda_f(\fa)$. 
Note that $\det_L(V_f) \cong L(1-k_0)$ and so $V_f(k_{0}/2)$ is symplectic self-dual.  

Let $L(V_{f}(k_{0}/2),s)$ denote the associated complex $L$-function, normalized to have center at $s=0$.
\subsubsection{Results}
\begin{prop}\label{prop, reduced-parity}
Let $F$ be a totally real field and $p$ an odd prime totally split in $F$. 
Let $f\in S_{k}(U_{0}(\fn))$ be a Hilbert modular newform over $F$ of weight $k\in 2\BZ_{>0}[I]$ and $V_f$ an associated $p$-adic $L$-representation of $G_F$ as above. 
 Let $\BF$ denote the residue field of $\CO_L$. 
Suppose that 
 there exists $g\in S_{k^{\prime}}(U_0(\fn^{\prime}))$
 for some $k^{\prime}\in \BZ_{>0}[I]$ and an ideal  $\fn^{\prime} \subseteq \CO_F$ 
 satisfying the following. 
\begin{itemize}
\item[a)]
There exists an isomorphism of $\BF[G_F]$-modules 
$$\overline{T}_f(k_0/2)\cong \overline{T}_g(k^{\prime}_0/2)$$ 
for an $\CO_L$-lattice $T_g$ of $V_g$,
where $L$ is possibly enlarged so that $V_g$ is also defined over $L$. 
\item[b)]
The $p$-parity Conjecture \ref{conj, p-parity} is true for the symplectic self-dual $G_F$-representation $V_{g}(k^{\prime}_{0}/2)$. 
\end{itemize}
Then the $p$-parity Conjecture \ref{conj, p-parity} is true for 
$V_{f}(k_{0}/2)$, i.e. 
\[
\ord_{s=0}L(V_f(k_0/2),s)\equiv \dim_{L}H^1_{\mathrm{f}}(F, V_f(k_0/2)) \mod 2.
\]
\end{prop}
\begin{proof}
This is a direct consequence of Theorem \ref{thm, parity-family}.
\end{proof}
\begin{thm}\label{thm, parity}
Let $f\in S_{k}(U_{0}(\fn))$ be a Hilbert modular newform over a totally real field $F$ as in Proposition \ref{prop, reduced-parity}. 
Suppose that 
the residual representation $\ov{T}_f$ is irreducible.
Then the $p$-parity Conjecture \ref{conj, p-parity} is true for 
$V_{f}(k_{0}/2)$, i.e. 
\[
\ord_{s=0}L(V_f(k_0/2),s)\equiv \dim_{L}H^1_{\mathrm{f}}(F, V_f(k_0/2)) \mod 2.
\]
\end{thm}

\begin{proof}
It suffices to  show the existence of a Hilbert modular newform $g$ as in Proposition \ref{prop, reduced-parity}.

By Hida, $f$ is congruent to a parallel weight two normalized 
Hecke eigenform $h \in S_{2t,t}(U_{0}(p^{\alpha}\fn))$  for some $\alpha \ge 0$ (see \cite[Thm.\ 3.2]{H} and \cite[Thm.\ 2.3]{H89}), i.e.
$\lambda_h(\mathfrak{l}) \in \CO_L$ and 
\begin{equation}\label{eq, cong}
\lambda_f(\mathfrak{l}) \equiv \lambda_h(\mathfrak{l}) \bmod  \fm, \qquad (\mathfrak{l} \nmid p\fn)
\end{equation}
(see also~\cites{Sh,Oh}). 
Here $S_{2t,t}(U_{0}(p^{\alpha}\fn))$ denotes the space $S_{2t,t,I}(S(p^{\alpha});\BC)$ in \cite[p. 143]{H89} with $S=U_0(\fn)$. 
Note that $h$ has Neben character $\epsilon^{2-k_0}$
where $\epsilon:(\CO_F\otimes \BZ_p)^{\times} \to \BZ_p^{\times} $ denotes  composition of the norm map $(\CO_F\otimes \BZ_p)^{\times}\to \BZ_p^{\times}$ and the Teichm\"uller character (cf.~\cite[p.\ 143]{H89}).

Put $$g=h\otimes \omega^{\frac{k_0}{2}-1} \in S_{2t,t}(U_{0}(p^{\alpha}\fn)).$$
 Since $\overline{T}_f$ is irreducible and totally odd,
 it is absolutely irreducible. 
Then in view of \eqref{eq, cong}, it follows that 
 $$\overline{T}_f(k_0/2) \simeq \overline{T}_g(1)$$ 
 as $\BF[G_F]$-modules. 
As for the property b), the $p$-parity conjecture for the weight two Hilbert modular form $h$ is due to  
Nekov\'a\v{r} \cite[Thm.\ C]{NekMRtame}, concluding the proof. 
\end{proof}

\section{Rubin-type conjectures}\label{s:Rubin}

 In this section we formulate and prove an analogoue of Rubin's conjecture over ramified quadratic extensions of $\BQ_p$, and a refinement of the original conjecture over the unramified quadratic extension of $\BQ_p$ 
(see~ Theorems~\ref{thm, lsd ind}, \ref{thm, rubin decomposition} and \ref{thm, rubin decomposition ramified}). 
The approach is based on the  local sign decomposition for a  
symplectic self-dual family of Galois representations of $G_{\BQ_p}$ 
 induced from conjugate symplectic self-dual characters over a quadratic extension of $\BQ_p$.

Throughout, we fix an odd prime $p$. 

\subsection{Epsilon constants}
The aim of this subsection is to calculate epsilon constants of Weil--Deligne representations of $W_{\BQ_p}$ of rank two induced by characters\footnote{For primes $p>2$, any irreducible Weil--Deligne representation of rank two is induced by a character.}. Their variation in $p$-adic anticyclotomic families leads to a backdrop for Rubin-type conjectures.

Throughout this subsection, a representation means a Weil--Deligne representation, and in particular, with complex coefficients.

\subsubsection{Set-up}
Let $F$ be a finite field extension of $\BQ_p$ and 
$K/F$ a quadratic field extension. 
Denote  the associated quadratic character of $F^\times$ 
by $\omega_{K/F}$.

As before, let $\omega_1$ be the unramified character of the Weil group $W_K$  
mapping geometric Frobenius to $q^{-1}$ for $q:=|\CO_K/\fm_{\CO_K}|$.
We identify that  $K^\times \cong W_{K}^{\rm ab}$
by Artin reciprocity sending a uniformizer to a geometric Frobenius. 
\begin{defn}
A continuous character 
 $\psi: K^\times \rightarrow \BC^\times$ is 
 conjugate symplectic self-dual with respect to $F$ if 
\begin{equation}\label{equation,csd}
\psi |_{F^{\times}}=\omega_{K/F}\cdot \omega_1.
\end{equation}
\end{defn}
In the case $F=\BQ_p$ we omit the expression `with respect to $F$'.

\begin{lem}\label{prop, conj self-dual}
 Let $\psi: K^\times \rightarrow \BC^\times$ be a continuous character. 
 Then $\mathrm{Ind}_{K/F} \psi$ is symplectic self-dual of weight $-1$ 
 if and only if $\psi$ is conjugate symplectic self-dual with respect to $F$. 
\end{lem}
\begin{proof}
This is a special case of \cite[Lem.~3.5 (i)]{GGP}. 
In our case, it 
simply follows from the determinant formula for induced representations: 
 $$\det (\mathrm{Ind}_{K/F} \psi)=\omega_{K/F}\cdot\psi |_{F^\times}.$$
 \end{proof}
 
For a character $\psi$ of $K^\times$, let $a(\psi)$ denote the exponent of its conductor, namely
$a(\psi)=0$ if $\psi$ is unramified, and otherwise, it is the smallest positive
integer $m$ such that $\psi$ is trivial on units $\equiv 1 \mod{\varpi^{m}}$ 
for $\varpi$ a uniformizer. We often refer to $a(\psi)$ as the conductor of $\psi$.

In the rest of the subsection we suppose that 
$F=\BQ_p$ 
and consider self-dual epsilon constants as in \S\ref{ss:Epsilon}.

\subsubsection{The unramified case.}\label{subsubsection, epsilon, unramified}
The main result of this subsection is Proposition~\ref{prop, root number, inert}. 

Let $K$ denote the unramified quadratic extension of $\BQ_p$. Let $\delta \in K$ be a $p$-adic unit such that $K=\BQ_p(\delta)$ and $\delta^2 \in \BQ_p$. 

\begin{lem}\label{lem, wdc unram}\noindent 
\begin{itemize}
\item[i)]There is a unique unramified conjugate symplectic self-dual character $\phi_K$ of $K^\times$. 
\item[ii)] Any conjugate symplectic self-dual character of $K^\times$ is of the form $\phi_K \chi$, where 
$\chi$  is a character of $K^\times$ such that $\chi|_{\BQ_p^\times}=1$. 
\end{itemize}
\end{lem}
\begin{proof} 
Let $\psi$ be a conjugate symplectic self-dual character of $K^\times$. 
By (\ref{equation,csd}), we have $\psi(p)=-p^{-1}$ and so part i) follows. 
In turn, it implies part ii).
\end{proof}

\begin{prop}\label{prop, root number, inert}
 Let $\phi_K$ and $\chi$ be characters of $K^\times$ as in Lemma~\ref{lem, wdc unram}. 
 Then we have 
\[
\varepsilon(\mathrm{Ind}_{K/\BQ_p} (\phi_K \chi))=\chi(\delta)(-1)^{a(\chi)}.
\] 
In particular, if the order of $\chi$ is odd, then $\varepsilon(\mathrm{Ind}_{K/\BQ_p} (\phi_K \chi))=(-1)^{a(\chi)}$. 
\end{prop}
\begin{proof}
 In the following the notation is as in \S\ref{ss: self-dual epsilon}. 
 
 Let $\xi$ be a non-trivial additive character of $\BQ_p$ with $n(\xi)=0$ and put 
 $\xi_{K}:=\xi \circ {\rm Tr}_{K/\BQ_p}$. Let $dx$ be a Haar measure on $K$ 
 which is self-dual with respect to $\xi_{K}$. 
 Then, $n(\xi_K)=0$ and $\int_{\mathcal{O}_K}dx=1$. 
Since $K/\BQ_p$ and $\phi_K$ are both unramified, we have 
$\varepsilon(\phi_K, \xi_K, dx)=\varepsilon(\mathrm{Ind}_{K/\BQ_p} \phi_K)=1$. Hence, we have
\begin{equation}\label{equation, virtual rep}
\varepsilon(\mathrm{Ind}_{K/\BQ_p} (\phi_K \chi))=
\frac{\varepsilon(\mathrm{Ind}_{K/\BQ_p} \phi_K \chi)}{\varepsilon(\mathrm{Ind}_{K/\BQ_p} \phi_K)}=
\frac{\varepsilon(\phi_K\chi, \xi_K, dx)}{\varepsilon( \phi_K, \xi_K, dx)}=\varepsilon(\phi_K\chi, \xi_K, dx).  
\end{equation}
Note that 
$[\mathrm{Ind}_{K/\BQ_p} \phi_K\chi]-[\mathrm{Ind}_{K/\BQ_p} \phi_K]$ is 
a virtual representation of degree $0$ (cf.~\cite[(3.4.8)]{T}). 

For Weil--Deligne representations $V$ and $W$ with $W$ unramified, recall that 
\begin{equation}\label{equation, epsilon unramified twist}
\varepsilon(V\otimes W, \xi, dx)=\varepsilon(V, \xi, dx)^{\mathrm{dim}\,W}\cdot \det W(\varpi)^{a(V)+n(\xi) \dim V}
\end{equation}
(cf.~\cite[(3.4.6)]{T}).  
Together with Fr\"ohlich--Queyrut \cite[Thm.~3]{FQ}, we have 
\[
\varepsilon(\phi_K \chi, \xi_K, dx )=\varepsilon(\chi, \xi_K, dx )\cdot \phi_K(p)^{a(\chi)}
\equiv \chi(\delta)(-1)^{a(\chi)} \;\mod^{\!\!\!\times} \,\mathbb{R}_{>0}. 
\]

Hence, the assertion follows from (\ref{equation, virtual rep}).  
\end{proof}

\subsubsection{The ramified case.}
The main result of this subsection is Proposition~\ref{prop, ramified epsilon}.

Let $K$ denote a ramified quadratic extension of $\BQ_p$. 
Since $p$ is odd, $K$ has two possibilities. 
Let $\delta \in K$ be a uniformizer such that $K=\BQ_p(\delta)$ and $\delta^2 \in \BQ_p$. 
Let $\xi$ be a non-trivial additive character 
of $\BQ_p$ such that $n(\xi)=0$. 
Then $\xi_{K}:=\xi\circ \mathrm{Tr}_{K/\BQ_p}$ is an additive character of $K$ of conductor $1$. 

 Any conjugate self-dual character of $K^\times$ is ramified since so is 
$\omega_{K/\BQ_p}$.

\begin{lem}\label{lem, WD csd} 
\noindent 
\begin{itemize}
\item[i)] There are exactly two 
conjugate symplectic self-dual characters $\psi$ of $K^\times$ of conductor one. More precisely, they  depend on the choice of a square root 
$\sqrt{p^{*}}$of
$p^{*}:=\left(\frac{-1}{p}\right)p$, and are characterized by 
$$\psi(\delta)=\sqrt{p^{*}}^{-1}.$$ 
In particular, their ratio equals the unramified quadratic character of $K^\times$. 
\item[ii)] Any conjugate symplectic self-dual character of $K^\times$ is of the form $\phi_K \chi$ for  
$\chi$  a character of $K^\times$ such that $\chi|_{\BQ_p^\times}=1$. 
\item[iii)] For any $\chi$ as above, 
the conductor $a(\chi)$ is even. 
\end{itemize}
\end{lem} 
\begin{proof}
In view of (\ref{equation,csd}), we have $$\psi(\delta^2)=\left(\frac{-1}{p}\right)p^{-1}$$ and 
$\psi |_{\mu_{p-1}}$ is uniquely determined. Thus, part i) follows.  Part ii) is immediate. 

As for iii), note that $\chi |_{\BQ_p^\times}=1$, and $\chi(1+\delta^{a(\chi)-1})$ is a primitive $p$-th root of unity. 
\end{proof}

Our formula for epsilon constants will involve the following Gauss-like sum. 
\begin{defn}
For a ramified character of $K^\times$ with $\chi|_{\BQ_p^\times}=1$, 
define 
\[
G_{\chi, \delta}:=\sum_{a \in \BF_p^\times}\left(\frac{a}{p}\right)\chi(1+\delta^{a(\chi)-1})^a. 
\]
\end{defn}  
\begin{prop}\label{prop, ramified epsilon}
Let $p$ be an odd prime and $K$ a ramified quadratic extension of $\BQ_p$.
 Let $\phi_K$ be a conjugate symplectic self-dual character of $K^\times$ of conductor one. 
 \begin{itemize}
\item[ i)] We have 
 \[
 \varepsilon(\mathrm{Ind}_{K/\BQ_p}\phi_K)=\omega_{K/\BQ_p}(-2).
 \]
 \item[ii)] Let $\chi$ be a ramified character of $K^\times$ such that $\chi|_{\BQ_p^\times}=1$. Then
 \begin{align*}
 \varepsilon(\mathrm{Ind}_{K/\BQ_p}\phi_K\chi)
 &=\omega_{K/\BQ_p}(-2)\cdot\phi_K\chi(\delta)\cdot\left(\frac{-1}{p}\right)^{\frac{a(\chi)}{2}} G_{\chi, \delta}\\
 &=\omega_{K/\BQ_p}(p)\phi_K\chi(\delta) 
 \left(\frac{-1}{p}\right)^{\frac{a(\chi)}{2}}\left(\frac{\overline{v}_\chi}{p}\right)\sum_{a \in \BF_p^\times}\left(\frac{a}{p}\right)\xi(\frac{a}{p}),
\end{align*}
where $v_\chi \in \mathcal{O}_K^\times$ is an element such that  \[
\chi(1+\delta^{a(\chi)-1}a)=\xi_{K}({v_\chi^{-1} a}/{\delta^2}) \quad \text{for $a \in \BZ_p$}
\]
and $\ov{v}_\chi$ denotes its image in the residue field.  
\end{itemize}
\end{prop}
\begin{proof}
In the following, the notation for epsilon constants is as in \S\ref{ss:Epsilon}. 

Let $\xi$ be a non-trivial additive character of $\BQ_p$ with $n(\xi)=0$ as above, and $dx_\xi$ the self-dual Haar measure on $\BQ_p$ with respect to $\xi$.
Then $\xi_{K}:=\xi\circ \mathrm{Tr}_{K/\BQ_p}$ is an additive character of $K$ with $n(\xi_K)=1$ and let $dx$ be a Haar measure on $K$.

For any character $\eta$ of $K^\times$ with $\eta|_{\BQ_p^\times}=1$, we have 
\begin{align}\label{equation, epsilon induction} 
  \varepsilon(\mathrm{Ind}_{K/\BQ_p}\phi_K\eta, \xi, dx_\xi)&= 
  \varepsilon(\mathrm{Ind}_{K/\BQ_p} 1, \xi, dx_\xi) \cdot 
 \frac{\varepsilon(\phi_K\eta, \xi_K, dx)}{\varepsilon(1, \xi_K, dx)}\notag\\
& =  \frac{\varepsilon(1, \xi, dx_\xi)}{\varepsilon(1, \xi_K, dx)} \cdot 
\varepsilon(\omega_{K/\BQ_p}, \xi, dx_\xi)\varepsilon(\phi_K\eta, \xi_K, dx), 
\end{align}
where 
the second equality follows from the decomposition $\mathrm{Ind}_{K/\BQ_p} 1= 1 \oplus \omega_{K/\BQ_p}$.
Note that $\varepsilon(\omega_{K/\BQ_p}, \xi, dx_\xi)$ is a quadratic Gauss sum by definition. 

We have
\begin{align*}
 \varepsilon(\phi_K, \xi_K, dx)&=\int_{p^{-1}\mathcal{O}_K^\times}\phi_K^{-1}(x)\xi_K(x)dx\\
 &=p\phi_K(p)\sum_{a \in \BF_p^\times} \phi_K(a)^{-1}\xi_K(a/p) \int_{1+\fm_{\CO_K}} dx\\
 &=\omega_{K/\BQ_p}(2p)\sum_{a \in \BF_p^\times} \omega_{K/\BQ_p}(a)\xi(a/p) \int_{1+\fm_{\CO_K}} dx. 
\end{align*}
The latter equals $\omega_{K/\BQ_p}(2)\varepsilon(\omega_{K/\BQ_p}, \xi, dx_\xi)$ up to a positive real number. 
Hence,  substituting $\eta=1$ in 
\eqref{equation, epsilon induction}, we have
\[
\varepsilon(\mathrm{Ind}_{K/\BQ_p}\phi_K)=\omega_{K/\BQ_p}(2) \frac{\varepsilon(\omega_{K/\BQ_p}, \xi, dx_\xi)^2}{|\varepsilon(\omega_{K/\BQ_p}, \xi, dx_\xi)|^2}
=\omega_{K/\BQ_p}(-2).
\]

Now we consider part ii). 
In the following 
for a Weil--Deligne representation $W$ of $W_K$ we let 
 $W(\rho, \xi_K)$ denote the root number $W(\rho,\xi_K):=\varepsilon(\rho, \xi_K, dx)/|\varepsilon(\rho, \xi_K, dx)|$, which is independent of the choice of the Haar measure $dx$. 
Moreover for $a, b \in \BC^\times$ we write $a \sim b$ if $ab^{-1} \in \BR_{>0}\cdot\mu_{p^\infty}$. 

By Lamprecht's formula \cite[Prop.~1]{TLC}, we have 
\[
W(\phi_K\chi, \xi_K)\sim (\phi_K\chi)(c_\chi)\xi_K(c^{-1}_\chi).
\]
Here $c_\chi\in K^\times$  is an element with $v_K(c_\chi)=a(\chi)+n(\xi_K)=a(\chi)+1$ so that 
\begin{equation}\label{eq: chi xi}
\chi(1+y)=\xi_K(c^{-1}_\chi y)
\end{equation}
 for any $y$ satisfying $v_K(y)\geq a(\chi)/2$, where $v_K$ is the normalized valuation of $K$ with $v_K(\delta)=1$.

Since $c_\chi$ has odd valuation, $\chi(c_\chi\delta^{-1})$ is a $p$-power root of unity with $p \not=2$. 
Hence, in view of \eqref{eq: chi xi}, it follows that  
$$W(\phi_K\chi, \xi_K)\sim \phi_K(c_\chi)\chi(\delta).$$ 
Write $c_\chi=\delta^{a(\chi)+1}v_\chi$ for $v_\chi \in \mathcal{O}_K^\times$. 
Then for $a\in\BZ_p$, we have 
\[
\chi(1+\delta^{a(\chi)-1}a)=\xi_{K}({v_\chi^{-1} a}/{\delta^2})
=\xi(u{v_\chi^{-1} a}/{p}) \quad \text{}
\]
for $u:=2p/\delta^2$ and 
\[
W(\phi_K\chi, \xi_K)\sim \phi_K(c_\chi)\chi(\delta)\sim \phi_K\chi(\delta)\left(\frac{-1}{p}\right)^{\frac{a(\chi)}{2}}\left(\frac{\overline{v}_\chi}{p}\right) .
\] 
As for the latter relation, we have used the fact $\phi_K(-\delta^2)\sim 1$ since 
$-\delta^2 \in \BQ_p$ is a norm from $K$. 
 
Hence, by (\ref{equation, epsilon induction}), it follows that 
\begin{align*}
 \varepsilon(\mathrm{Ind}_{K/\BQ_p}\phi_K\chi)&=\omega_{K/\BQ_p}(p)\phi_K\chi(\delta) 
 \left(\frac{-1}{p}\right)^{\frac{a(\chi)}{2}}\left(\frac{\overline{v}_\chi}{p}\right)\sum_{a \in \BF_p^\times}\left(\frac{a}{p}\right)\xi(a/p).\\
 &=\omega_{K/\BQ_p}(-2)\phi_K\chi(\delta)\left(\frac{-1}{p}\right)^{\frac{a(\chi)}{2}}\sum_{a \in \BF_p^\times}\left(\frac{a}{p}\right)\chi(1+\delta^{a(\chi)-1}a).
\end{align*}
\end{proof}

\begin{cor}\label{cor, ramified root number}
In the setting of Proposition \ref{prop, ramified epsilon},  
suppose that the order of $\chi$ is a power of $p$.  
\begin{itemize}
\item[i)] For $b \in \mathbb{Z}_p^\times$, we have 
\[
\varepsilon(\mathrm{Ind}_{K/\BQ_p}\phi_K\chi^b)=\left(\frac{b}{p}\right)\varepsilon(\mathrm{Ind}_{K/\BQ_p}\phi_K\chi)
\]
\item[ii)] If $\chi^p\not=1$, then 
\[
\varepsilon(\mathrm{Ind}_{K/\BQ_p}\phi_K\chi^{-\delta^2})=\varepsilon(\mathrm{Ind}_{K/\BQ_p}\phi_K\chi)
\]
\end{itemize}
\end{cor}
\begin{proof}
This is a consequence of  Proposition \ref{prop, ramified epsilon}. 

For illustration, we consider part ii). Put $\chi':=\chi^{-\delta^2}$. 
Then $a(\chi)=a(\chi')+2$ and $$\chi(1+\delta^{a(\chi)-1})=\chi(1-\delta^{a(\chi')-1})^{-\delta^2}=\chi'(1-\delta^{a(\chi')-1}).$$  
Hence $G_{\chi, \delta}=\left(\frac{-1}{p}\right)G_{\chi', \delta}$ and the assertion follows 
from Proposition \ref{prop, ramified epsilon} since $\chi(\delta)=1$. 
\end{proof}

\subsection{Conjugate symplectic self-dual characters}
In this subsection we explicitly describe conjugate symplectic self-dual $p$-adic Galois characters over quadratic extensions of $\BQ_p$.

\subsubsection{Notions}
Fix an algebraic closure $\ov{\BQ}_p$ of $\BQ_p$.

Let $F$ be a finite extension of $\BQ_p$. Recall that $\chi_\cyc : G_{F} \ra \BZ_p^\times$ denotes the cyclotomic character, which we view as a character of $F^\times$ via Artin reciprocity map of local class field theory. 
Let $K/F$ be a quadratic field extension and 
denote  the associated quadratic character of $F^\times$ by $\omega_{K/F}$. 
\begin{defn}\label{def:csd}
A continuous character 
 $\psi: G_K \rightarrow \overline{\BQ}_p^\times$ is 
 conjugate symplectic self-dual with respect to $F$ if 
\begin{equation}\label{equation,p-csd}
\psi |_{F^{\times}}=\omega_{K/F}\cdot \chi_{\cyc}
\end{equation}
where we view $\psi$ as a character of $K^\times$ via Artin reciprocity map. 
\end{defn}

Note that a $p$-adic character of $K^\times$ satisfying  (\ref{equation,p-csd}) extends uniquely to 
a continuous character of $G_K$ via the Artin reciprocity map. 
In this way we regard  
such characters of $K^\times$ as Galois characters.

\begin{lem}
 Let $\psi: G_K \rightarrow \overline{\BQ}_p^\times$ be a continuous character. 
 Then $\mathrm{Ind}_{K/F} \psi$ is symplectic self-dual  
 if and only if $\psi$ is conjugate symplectic self-dual with respect to $F$. 
\end{lem}
\begin{proof}
We proceed as in the proof of Lemma \ref{prop, conj self-dual}.
\end{proof}

In the rest of this subsection let $R$ be a commutative profinite ring.
\begin{defn}\label{def:ac} 
A continuous character $\psi: G_K \rightarrow R^\times$ is 
anticyclotomic with trivial 
central character with respect to $F$ if 
$$\psi(cgc^{-1})=\psi(g^{-1}) \text{ and }  \psi(c^2)=1$$ for any $g \in G_K$ and 
an extension $c \in G_{F}$ of the  non-trivial element of $\mathrm{Gal}(K/F)$. 
\end{defn}

When $F=\BQ_p$, we omit the expression `with respect to $F$'.

\begin{lem}\label{lemma, trivial central character}\noindent
\begin{itemize}
\item[i)] 
Definition~\ref{def:ac} does not depend on the choice of an extension $c\in G_F$ of the non-trivial element of $\Gal(K/F)$.
\item[ii)] A character $\psi: G_K \rightarrow R^\times$ is anticyclotomic with trivial 
central character with respect to $F$ if and only if 
$\psi|_{F^\times}=1$. 
\end{itemize}
\end{lem}
\begin{proof}
Let $c_1$ be another extension and put $d=cc_1^{-1}$. Then for $g\in G_K$,  
\[
\psi(c_1gc_1^{-1})=\psi(d) \psi(c_1gc_1^{-1}) \psi(d)^{-1}=\psi(cgc^{-1})=\psi(g^{-1}). 
\] 
Since $\psi(cdc^{-1})=\psi(d^{-1})$ and $\phi(c^2)=1$, we have $\psi(c_1^{-1}c^{-1})=\psi(c_1c^{-1})$.
This implies that 
\[\psi(c_1^2)= \psi(c_1c^{-1})\psi(c c_1)=\psi(c c_1)\psi(c_1c^{-1})=1.
\]

The restriction to $F^\times$ corresponds to the transfer map $t: G_{F}^{\rm{ab}} \rightarrow G_{K}^{\rm{ab}}$ 
under the Artin reciprocity. Recall that $t(g)=g \cdot (c^{-1} g c)$ if $g \in G_{K}^{\rm{ab}}$ and  
$t(g)=c^{-1}g^2c=(c^{-1}gc)^2$ if $g \in G_{F}^{\rm{ab}} \setminus G_{K}^{\rm{ab}}$. 
The assertion ii) follows from this and i). 
\end{proof}

\begin{defn} 
Let $K/F$ be a quadratic field extension as above. 
An abelian field extension $L$ of $K$ is a generalized dihedral extension of $F$ if 
\begin{itemize}
\item[i)] $L$ is Galois over $F$ and,
\item[ii)] There exists an element $c \in \mathrm{Gal}(L/F)$ 
of order two extending the non-trivial element of $\mathrm{Gal}(K/F)$ such that 
\[
\mathrm{Gal}(L/F)= \mathrm{Gal}(L/K) \rtimes \langle c \rangle 
\]
satisfying $cgc^{-1}=g^{-1}$ for any $g \in \mathrm{Gal}(L/K)$. 
\end{itemize}
\end{defn}
By definition, any continuous character $\psi: G_K \rightarrow R^\times$ factoring
through a generalized dihedral extension of $F$ 
is anticyclotomic with trivial 
central character with respect to $F$. 
An example is provided by characters factoring through the anticyclotomic $\BZ_p$-extension of a quadratic extension of $\BQ_p$.

The following preliminary will be used in some of our later arguments. 

 \begin{lem}\label{lem, ac char} Let $p$ be an odd prime and let $K$ be a quadratic field extension of $\BQ_p$.
 Let $\Gamma$ be the Galois group of the anticyclotomic $\BZ_p$-extension of $K$. 
 \begin{itemize}
\item[i)] Let $\pi$ be a uniformizer of $\mathcal{O}_K$. 
 Then  the image of 
 \[(1+\pi \mathcal{O}_K)/(1+p\BZ_p) \rightarrow  (1+\pi \CO_K)_{\mathrm{free}}:= (1+\pi \CO_K)/(1+\pi \CO_K)_{{\mathrm{tor}}}, \quad \alpha \mapsto \alpha /c(\alpha)
 \]
is isomorphic to $\BZ_p$. 
 Moreover, $1+\pi O_K$ is free as a group if $K/\BQ_p$ is unramified or $K/\BQ_p$ is ramified and $p\not=3$. 
 In the excluded case, the torsion part of $1+\pi\CO_K$ is isomorphic to $\mu_3 \cap K^\times$. 
\item[ ii)] If $K/\BQ_p$ is unramified,
then  a character $\psi: G_K  \rightarrow \overline{\BQ}_p^\times$ factors 
through $\Gamma$ 
 if and only 
  if $\psi|_{\BQ_p^\times}=1$ and $\psi|_{K^\times_{\rm{tor}}}=1$. 
\item[ iii)] If $K/\BQ_p$ is ramified, then a character $\psi: G_K  \rightarrow \overline{\BQ}_p^\times$ factors 
through $\Gamma$ 
 if and only 
  if $\psi|_{\BQ_p^\times}=\psi|_{\mu_p \cap K^{\times}}=\psi(\delta)=1$,
  where $\delta$ is a uniformizer of $K^{\times}$ such that $\delta^2 \in \BQ_p^{\times}.$
 \end{itemize}
\end{lem}
\begin{proof}
If $K/\BQ_p$ is unramified or $K/\BQ_p$ is ramified and $p\not=3$, then $1+\pi \mathcal{O}_K$ is torsion-free
since the exponential map is defined on $\pi \mathcal{O}_K$.  
In the excluded case note that 
$$1+\pi \mathcal{O}_K=(\mu_3\cap K) \times (1+\pi^2\mathcal{O}_K),$$ and $1+\pi^2\mathcal{O}_K$ is free by the same argument. 

As for part iii), it is a consequence of the fact that the natural exact sequence
\[
0\to (1+\pi\CO_K)/(1+p\BZ_p) \to K^{\times}/\BQ_p^{\times} \to \BZ/2\BZ\to 0
\]
splits uniquely such that $\delta\in K^{\times}$ is the lift of $1 \in \BZ/2\BZ$. 
\end{proof}

\subsubsection{The unramified case}

In this subsection $K$ denotes the unramified quadratic extension of $\BQ_p$.

\begin{lem}\label{lem, csd I}
Let $\mathscr{F}_\pi$ be the Lubin--Tate formal group over $\mathcal{O}_K$ for the uniformizing parameter $\pi$ of $K$. Then 
the associated Lubin--Tate character $$\phi_\pi : G_K \ra \Aut_{\CO_K} T_\pi \mathscr{F}_\pi$$ 
is conjugate symplectic self-dual 
if and only if $\pi=-p$. 
\end{lem}
\begin{proof}
Note that $\phi_\pi(\pi)=1$ and $\phi_\pi(u)=u$ for $u \in \mathcal{O}_K^\times$. 
Then $\phi_\pi$ is conjugate symplectic self-dual if and only if 
$\phi_\pi(p)=-1$. 
The assertion follows from this. 
\end{proof}
\begin{remark}
The lemma recovers the well-known fact that 
the formal group of a CM elliptic curve defined over $\BQ_{p}$ with good supersingular reduction at $p$ 
is the Lubin--Tate formal group of height for the uniformizing parameter $-p$. 
\end{remark}

In view of the lemma, we let $\pi=-p$ henceforth. 
Let $\phi^*_\pi$ denote the character  $\phi_\pi \circ c$ for $c\in\mathrm{Gal}(K/\BQ_p)$ the non-trivial element and 
$$\psi_\pi:=\frac{\phi_\pi}{\phi^*_\pi}.$$

\begin{lem}\label{lem, shape of conj self-dual I}
A $p$-adic character of $K^\times$  is conjugate symplectic self-dual and de Rham if and only if  
it is of the form $\phi_\pi \psi_\pi^k\chi$ 
where  $k \in \BZ$ and $\chi$ a finite order character of $K^\times$ with $\chi |_{\BQ_p^\times}=1$. 
\end{lem}
\begin{proof}
Any de Rham character $\psi$ of $G_K$ is of the form $\phi_{\pi}^a \psi_{\pi}^b \mu \chi$,  
where $a, b \in \BZ$, $\mu$ is an unramified character and $\chi$ a finite order character 
(cf. \cite[App.~B]{BCon}). 

Considering the value of such a conjugate symplectic self-dual $\psi$ at $p$, it follows that $\mu$ is of finite order. So we may assume that $\mu=1$. Then, 
\[
\omega_{K/\BQ_p}\chi_{\mathrm{cyc}}=\psi|_{\BQ_p^\times}=\omega_{K/\BQ_p}^a\chi_{\mathrm{cyc}}^a\chi|_{\BQ_p^\times}.
\]
Hence, $a=1$ and $\chi|_{\BQ_p^\times}=1$.
\end{proof}

\begin{prop}\label{prop, unramified p-adic WD char}
Let  $\psi$ be a conjugate symplectic self-dual $p$-adic character $$\psi=\phi_\pi \psi_\pi^k\chi: K^\times \rightarrow L^\times$$ 
for  $L$ a finite extension of $K$, $k \in \BZ$ and $\chi$ a finite order character with $\chi |_{\BQ_p^\times}=1$. 
Pick an embedding $\tau : L \hookrightarrow \overline{\BQ}_p$ and an isomorphism $\iota: \overline{\BQ}_p \cong \BC$. 
Then the Weil--Deligne representation $\mathrm{WD}(\psi)_{\iota \circ \tau}$ associated to $\psi$ is the character $\phi_K \chi$,  where 
$\phi_K$ is the unique unramified conjugate symplectic self-dual Weil-Deligne character of $K^\times$ 
and  $\chi$ is regarded as a $\BC$-valued character via the embedding $\iota\circ \tau: L \hookrightarrow \BC$. 
\end{prop}
\begin{proof}
Since $\phi_\pi \psi_\pi^k$ is crystalline, it corresponds to $\phi_K$, the unique unramified 
conjugate symplectic self-dual character. Hence, the proposition follows by noting that 
 the functor associating a $p$-adic potentially semistable Galois representation to 
a Weil--Deligne representation is compatible with tensor product.  

In the following we describe an alternate explicit argument. 

We may assume that $\psi\in\{\phi_\pi, \phi_\pi^*, \psi_\pi\}$. 
First we consider the case $\psi \in \{\phi_\pi , \phi^*_\pi\}$. 
Let $D(\psi)=D_{\mathrm{pst}}(\psi)$ be the associated filtered $(\varphi, N, G_K)$-module.  
It is a $K^{\rm{un}}\otimes_{\BQ_p} L$-module of rank one. 
Since $\psi$ is crystalline, note that 
$D(\psi)=K^{\rm{un}}\otimes_{K} D$ where $D:=D_{\rm{crys}}(\psi)$
 is a  $K\otimes_{\BQ_p} L$-module of rank one with $N=0$. For $\psi=\phi_\pi$, the operator 
$\varphi^2$  acts by $\pi^{-1}$ on $D$, and  
 for $\psi=\phi^*_\pi$, it also 
acts by $\pi/p^2=\pi^{-1}$ 
since $\phi_\pi \cdot \phi_\pi^*$ is the cyclotomic character over  $K$. 
In particular,  for $\psi=\psi_\pi$, the operator $\varphi^2$  acts trivially.

Now we consider the case $\psi=\phi_\pi \psi_\pi^k$.
Since the $G_K$-action on $D$ is trivial, 
the Weil group acts on the associated Weil--Deligne representation by 
\[
g\cdot v:=g\circ \varphi^{fv(g)}v=g\circ \varphi^{2v(g)}v=\pi^{-v(g)}v, 
\] 
which coincides with $\phi_K$. 
\end{proof}

\subsubsection{The ramified case}\label{subsubsection, ramified p-adic WD}

In this subsection $K$ denotes a ramified quadratic extension of $\BQ_p$. 
Let $\delta \in K$ be a uniformizer such that $\delta^2 \in \BQ_p$.

Let $\phi_\delta$ be the Lubin--Tate character with parameter $\delta$. 
Being crystalline, it is not conjugate symplectic self-dual, however its certain twist is as we now introduce. 
Let $\theta$ be a tamely ramified $p$-adic character of $K^\times$ such that
\[
\theta(\delta^2)=\left( \frac{-1}{p}\right) \cdot \frac{\delta^2} {p}, \quad \theta|_{\BZ_p^\times}=\omega_{K/\BQ_p}|_{\BZ_p^\times}. 
\]
Depending on the choice of a square root of $p^* :=\left(\frac{-1}{p}\right)p$, 
there are exactly two such characters $\theta$ characterized by $\theta(\delta)=\delta/\sqrt{p^*}$.
Put 
$$\phi_{\theta}=\phi_{\delta}\cdot\theta, \quad
\psi_\theta=\frac{\phi_\theta}{\phi_\theta\circ c}$$
for $c\in\Gal(K/\BQ_p)$ the non-trivial element.

\begin{lem}\label{lem, shape of conj self-dual II}
\noindent\begin{itemize}
\item[i)] The $p$-adic character $\phi_\theta$ is conjugate symplectic self-dual. 
\item[ii)] 
A $p$-adic character of $G_K$  is conjugate symplectic self-dual and de Rham if and only if  
it is of the form $\phi_\theta \psi^k_\theta\chi$ 
where  $k \in \BZ$ and $\chi$ a finite order character with $\chi |_{\BQ_p^\times}=1$. 
\end{itemize}
\end{lem}
\begin{proof}
Part i) follows directly from the definition. 

As for ii), 
any de Rham character $\psi$ of $G_K$ is of the form $\phi_\theta^a \psi_\theta^b \mu \chi$,  
where $a, b \in \BZ$, $\mu$ an unramified character and $\chi$ a finite order character. 
Using it, one may proceed as in the proof of Lemma~\ref{lem, shape of conj self-dual I}. 
\end{proof}

\begin{prop}\label{prop, ramified p-adic WD char}
Let  $\psi$ be a conjugate symplectic self-dual $p$-adic character $$\psi=\phi_\theta \psi_\theta^k\chi: G_K \rightarrow L^\times$$ 
where  $L$ is a finite extension of $K$, $k \in \BZ$ and $\chi$ a finite character with $\chi |_{\BQ_p^\times}=1$. 
Pick an embedding $\tau : L \hookrightarrow \overline{\BQ}_p$ and an isomorphism $\iota: \overline{\BQ}_p \cong \BC$. 
Then the Weil--Deligne representation $\mathrm{WD}(\psi)_{\iota \circ \tau}$ associated to $\psi$ is the character $\psi_{K} \chi$,  where 
$\psi_K$ is the unique conjugate symplectic self-dual character of $K^\times$ of conductor $1$
 with 
 $$\psi_K(\delta)=\iota \circ \tau(\psi(\delta)/\delta)=(-1)^k\left(\frac{-1}{p}\right)^k\iota \circ \tau(\theta(\delta)/\delta)$$ 
 as in~Lemma~\ref{lem, WD csd} i) 
 and  $\eta$ is regarded as a $\BC$-valued character via the embedding $\iota\circ \tau$.  
\end{prop}
\begin{proof}
It suffices to 
calculate $WD(\rho)_{\iota\circ \tau}$ for $\rho$ being $\phi_\delta$, an unramified character or a finite character, for which one may 
proceed as in the proof of Lemma \ref{prop, unramified p-adic WD char}. 
\end{proof}

\subsection{Rubin-type conjectures}  
\subsubsection{Set-up}  As before, fix an algebraic closure $\ov{\BQ}_p$ of $\BQ_p$ 
and finite extensions of $\BQ_p$ are taken in $\ov{\BQ}_p$.

Let $K$ be a quadratic extension of $\BQ_p$ and $L$ a finite extension of $K$.  
Let $$\phi: G_K \rightarrow L^\times$$
be a de Rham conjugate symplectic self-dual character
with Hodge--Tate\footnote{precisely, the Hodge--Tate weights of $\mathrm{Ind}_{K/\BQ_p}\phi$} weights $(k+1,-k)$ for $k\in\BZ_{\geq 0}$. 
For example, we may take $\phi=\phi_{\pi}\psi_{\pi}^k$ as in~Lemma~\ref{lem, csd I} if $K$ is unramified, 
and $\phi=\phi_{\theta}\psi_{\theta}^k$ as in Lemma \ref{lem, shape of conj self-dual II} if $K$ is ramified. 
Let $T_{\phi}$ be the free $\mathcal{O}_L$-module of rank one on which $G_K$ acts by $\phi$. 

 Let $K_\infty$ denote the anticyclotomic $\BZ_p$-extension of $K$ with Galois group $\Gamma$, 
 namely the unique $\BZ_p$-extension of $K$ which is generalized dihedral over $\BQ_p$.   Let $\Lambda$ denote the Iwasawa algebra $\mathcal{O}_L[\![\Gamma]\!]$.
We denote by $K_n$ the $n$-th layer. 

We regard a  character $\chi:  \Gamma \rightarrow \overline{\BQ}^\times_p$
as a character of $G_K$ via the natural map
 $G_K \rightarrow \mathrm{Gal}(K_\infty/K)=\Gamma$.  
 \begin{defn}\noindent 
 \begin{itemize}
\item[i)] For an integer $n\geq 0$, let $\Xi_n$ denote the set of finite order characters $\chi:  \Gamma \rightarrow \overline{\BQ}^\times_p$ of conductor $a(\chi)=n$. Put $\Xi=\cup_n \Xi_n$. 
\item[ii)] Let $\Xi_{\rm dR}$ denote the subset of $\BC_p^\times$-valued characters of $\Gamma$ which are de Rham. 
Define a partition of $\Xi_{\rm dR}$  by
$$
\Xi_{\rm dR}=\Xi_{\phi,\rm dR}^+ \cup \Xi_{\phi,\rm dR}^{-}, \qquad
{\Xi}^\pm_{\phi,\rm dR}=\{\chi\in\Xi_{\rm dR}  \;|\;\ \hat{\varepsilon}({\rm \Ind}_{K/\BQ_p}\phi\chi^{-1})=\pm 1\}.
$$  
\end{itemize}
 \end{defn}

We emphasize that the definition of the partition $\Xi_{\phi,\rm dR}^\pm$ is based on completed epsilon constants of self-dual induced representations, which do not depend on choices of additive character and embedding in view of self-duality, as opposed to 
conjugate symplectic self-dual characters.

A primary Iwasawa-theoretic object is the anticyclotomic deformation 
$$ \BT_\phi= T_\phi^{\otimes -1}(1) \otimes_{\CO_L}\Lambda^{\iota},$$
where $\Lambda^\iota$ is the $\Lambda$-module free of rank one equipped with
 $G_K$ action defined by $g a = [g]^{-1}a$ for $g\in G_K$ and $a\in \Lambda$, and 
 $[g]\in \Gamma (\subseteq \Lambda^{\times})$ denotes the image of $ g\in G_K$ under the natural surjection 
 $G_K\to \Gamma.$  
Put 
$$
\tilde{\mathbb{T}}_\phi=\mathrm{Ind}_{K/\BQ_p} (T_{\phi}^{\otimes -1}(1)\otimes_{\mathcal{O}_L} \Lambda^{\iota}). 
$$

\begin{lem}\label{lem, induced rep satisfy the condition}\noindent
\begin{itemize}
\item[i)] The $\Lambda$-module $\tilde{\mathbb{T}}_\phi$ is symplectic self-dual of rank two. 
\item[ii)] Suppose that $p>3$ or $K=\BQ_3(\sqrt{3})$ if $p=3$. 
Then $\tilde{\BT}_\phi$ is generic, i.e.  $H^0(\BQ_p, \tilde{\mathbb{T}}_\phi\otimes_{\Lambda} (\Lambda/\mathfrak{m}_{\Lambda}))=0$.
\end{itemize}
\end{lem}
\begin{proof}Note that $\Lambda^{\iota}|_{\BQ_p^\times}$ is the trivial representation 
by Lemma \ref{lemma, trivial central character} ii), and so 
\begin{align*}
 \det \tilde{\mathbb{T}}_\phi 
&=\omega_{K/\BQ_p} \cdot (\phi^{-1}\chi_{\rm{cyc}} |_{\BQ_p^\times }\otimes \Lambda^{\iota}|_{\BQ_p^\times})=\chi_{\rm{cyc}} 
\end{align*}
by the determinant formula for induced representations (note that 
$\chi_{\rm{cyc}}|_{\BQ_p^{\times}}=\chi_{\rm{cyc}}^2$), 
yielding part i). 

As for part ii), first note that $H^0(\BQ_p, \tilde{\mathbb{T}}_\phi\otimes (\Lambda/\mathfrak{m}_{\Lambda}))$ is isomorphic to
\[ 
H^0(\BQ_p, \mathrm{Ind}_{K/\BQ_p}(T_{\phi}^{\otimes -1}(1)\otimes_{\CO_L} \CO_L/\fm_{\CO_L} ) )
\cong H^0(\BQ_p, \mathrm{Ind}_{K/\BQ_p}(T_{\phi}\otimes_{\CO_L} \CO_L/\fm_{\CO_L} ) )
\cong 
H^0(K, \overline{T}_\phi)\]
for $\overline{T}_{\phi}:=T_\phi \otimes_{\CO_L}\CO_L/\fm_{\CO_L}$,
where the second isomorphism is due to the self-duality of $\phi$ and the definition of $\mathrm{Ind}_{K/\BQ_p}$. 
We claim that $H^0(K, \overline{T}_\phi)=0$, i.e. $G_K$ acts non-trivially on $\ov{T}_\phi$. 
Indeed, 
if $K/\BQ_p$ is unramified, 
 then 
$p \in K^\times$ acts on $\ov{T}_\phi$ 
 as multiplication by $-1$. 
 As for the ramified case, first suppose that $p>3$. 
 Then $\ov{\phi}|_{\BZ_p^\times}$ is non-trivial. 
If $p=3$ and $K=\BQ_3(\sqrt{3})$, then $\omega_{K/\BQ_3}(3)=-1$. In turn $\ov{\phi}(\delta)$ is a primitive fourth root of unity for $\delta\in K^\times$ as in \S\ref{subsubsection, ramified p-adic WD} and so $\delta$
acts non-trivially on $\ov{T}_\phi$. 

\end{proof}

\begin{remark}
 The mod $p$ reduction of $\mathrm{Ind}_{K/\BQ_p}\tilde{\BT}_\phi$ is irreducible
  if and only if $K/\BQ_p$ is 
 unramified. 
\end{remark}

A fundamental invariant of $\BT_\phi$ is the Galois cohomology $H^1(K,\BT_\phi)$. We introduce the following submodules. 
\begin{defn}
For a subset $S$ of de Rham specializations of $\BT_\phi$, 
define
$$
H^1_{S, \pm}(K,\BT_\phi)=\{v \in H^1(K, \BT_\phi)\,|\, s(v) \in H^1_\mathrm{f}(K, \BT_{\phi,s}) \; \text{for any $s\in S$ 
with $\hat{\varepsilon}(\tilde{\BT}_s)=\mp 1$} \},
$$
where $\tilde{\BT}_s$ denotes the specialization of $\BT_\phi$ at $s$. Similarly, define 
$$
H^1_{S, \pm}(K,\BT_\phi \otimes_{\BZ_p}\BQ_p) =
\{v \in H^1(K, \BT_\phi) \otimes_{\BZ_p}\BQ_p \,|\, s(v) \in H^1_\mathrm{f}(K, \BT_{\phi,s}\otimes_{\BZ_p}\BQ_p) \; \text{for any $s\in S$ 
with $\hat{\varepsilon}(\tilde{\BT}_s)=\mp 1$} \}.
$$
\end{defn}
We emphasize that the above definition is based on completed epsilon constants of induced representations. 
One seeks to study the $\Lambda$-module structure of the dialectic modules $H^1_{S,\pm}(K,\BT_\phi)$ and their relation with 
the ambient module 
$H^1(K,\BT_\phi)$. This is the content of 
 the next subsection.

\subsubsection{A general Rubin-type conjecture} 
\begin{thm}\label{thm, lsd ind}
Let $p$ be an odd prime and $K$ a quadratic extension of $\BQ_p$. 
Let $\phi$ be a de Rham conjugate symplectic self-dual character of $G_K$ 
valued in a $p$-adic local field $L$. Put $T_\phi=\CO_L(\phi)$ and $\BT_\phi=T_{\phi}^{\otimes -1}(1)\otimes_{\mathcal{O}_L} \Lambda^\iota$ for $\Lambda$ the anticyclotomic Iwasawa algebra over $K$ with $\CO_L$-coefficients. Let $S$ be a subset of de Rham specializations of $\BT_\phi$ such that $S\cap \Xi_{\phi,\rm dR}^+$ and $S\cap \Xi_{\phi,\rm dR}^-$ are infinite. If $p=3$, suppose that either $K=\BQ_3(\sqrt{3})$ or $K=\BQ_3(\sqrt{-3})$ and $\phi \nequiv 1 \mod{\fm_{\CO_L}}$.
\begin{itemize}
\item[i)] The $\Lambda$-modules $H^1_{S,\pm}(K,\BT_\phi)$ are free of rank one and independent of the choice of $S$. 
\item[ii)] The $\Lambda$-submodules $H^1_{\pm}(K,\BT_\phi):=H^1_{S,\pm}(K,\BT_\phi) \subset H^1(K, \BT_\phi)$ are Lagrangian. 
\item[iii)] As $\Lambda$-modules, we have 
$$
H^1(K,\BT_\phi)=H^1_+(K,\BT_\phi) \oplus H^1_{-} (K,\BT_\phi). 
$$
\end{itemize}
In the excluded case 
that $p=3$, $K=\BQ_3(\sqrt{-3})$ and $\phi \equiv 1 \mod{\fm_{\CO_L}}$, the assertions as in parts i)-iii) hold\footnote{Recall that $H^i(K,\BT_\phi\otimes_{\BZ_p}\BQ_p):=H^i(K,\BT_\phi)\otimes_{\BZ_p}\BQ_p.$ 
We can show that $H^1(K,\BT_\phi)$ is $\Lambda$-torsion free.} 
 for the $\Lambda \otimes_{\BZ_p}\BQ_p$-modules $H^1_{\cdot}(K,\BT_\phi \otimes_{\BZ_p}\BQ_p)$.
In particular as $\Lambda\otimes_{\BZ_p}\BQ_p$-modules we have 
$$
H^1(K,\BT_\phi \otimes_{\BZ_p}\BQ_p) =H^1_+(K,\BT_\phi \otimes_{\BZ_p}\BQ_p) \oplus 
H^1_{-} (K,\BT_\phi \otimes_{\BZ_p}\BQ_p). 
$$
\end{thm}
\begin{proof}
By Shapiro's lemma, note that 
$$
H^1(K,\BT_\phi)=H^1(\BQ_p , \tilde{\BT}_\phi). 
$$
Hence, in view of Lemma \ref{lem, induced rep satisfy the condition}, the assertions i)-iii) are consequences of the local sign decomposition as in Theorem~\ref{thm, main} and Corollary \ref{cor, decomp by using Z}. Moreover, we have
\begin{equation}\label{eq, K-pm}
H^1_\pm(K,\BT_\phi)=H^1_\pm(\BQ_p , \tilde{\BT}_\phi). 
\end{equation}

As for part iv) note that 
$$
H^0(K,T_s)=0
$$
for all characteristic zero specializations $T_s$ of $\BT_\phi$ since $T_s|_{\BQ_p^\times}$ is non-trivial, and so $H^2(\BQ_p,\BT_\phi \otimes \BQ_p)=0$ by Lemma~\ref{lem, van}. Hence, Theorem~\ref{thm, main3} yields the local sign decomposition
$$
H^1(K,\BT_\phi \otimes_{\BZ_p}\BQ_p) =H^1_+(K,\BT_\phi \otimes_{\BZ_p}\BQ_p) \oplus 
H^1_{-} (K,\BT_\phi \otimes_{\BZ_p}\BQ_p). 
$$
The signed submodules satisfy the property i) by the proof of Corollary \ref{cor, decomp by using Z}. 
\end{proof}

\subsubsection{A Rubin-type conjecture: over the unramified quadratic extension}

In this subsection we express Theorem \ref{thm, lsd ind}  in a form closer to the original Rubin conjecture \cite{Ru}.

Throughout  
$K$ denotes the unramified quadratic extension of $\BQ_p$.
Let $\delta\in K$ be a $p$-adic unit such that $K=\BQ_p(\delta)$ and $\delta^2 \in \BZ_p^\times$. 
Let $\phi$ be a $p$-adic conjugate symplectic self-dual de Rham character of $G_K$. 
By Lemma \ref{lem, shape of conj self-dual I}, 
it is of the form $\phi=\phi_\pi\psi_\pi^k\eta$ where $\eta$ is a finite character with $\eta|_{\BQ_p^\times}=1$. 
For simplicity, we say $\eta$ is the finite character associated to $\phi$.

\begin{defn} Let $\Xi=\Xi^+ \cup \Xi^-$ be a partition of finite order non-trivial characters of the anticyclotomic Galois group $\Gamma$ defined by 
\[
\Xi^+=\cup_{n=1}^\infty\;\Xi_{2n}, \quad  \Xi^-=\cup_{n=1}^\infty\;\Xi_{2n+1}. 
\]
\end{defn}
\noindent Note that the elements of $\Xi^\pm$ are wildly ramified characters.

Towards Rubin's original conjecture \cite{Ru}, we establish the following. 
\begin{thm}\label{cor, Rubin's conjecture refined} Let the setting\footnote{Rubin assumed that $p\geq 5$, but we also allow $p=3$ (cf.~\cite{YZ}).}
 be as in \S\ref{ss:RuC}. 
\begin{itemize}
\item[i)] For the conjugate symplectic self-dual character $\phi=\phi_\pi$, 
there is a canonical 
 isomorphism $$V_{\infty}^{*} \cong H^1(K,{\mathbb{T}}_\phi),$$ and 
the decomposition in 
Theorem~\ref{thm, lsd ind} iii) coincides with the one conjectured by Rubin \cite[Conj.~2.2]{Ru}. 
\item[ii)] We have 
\begin{align*}
 V^{*, \pm}_\infty=\{v \in V_\infty^* \;|\;\ \delta_\chi(v)=0 \quad \text{for every $\chi \in {\Xi}^\mp_{\phi,\rm dR}$} \}. 
\end{align*}
\end{itemize}
\end{thm}

\begin{proof}
For the canonical isomorphism $V_{\infty}^{*} \cong H^1(K,{\mathbb{T}}_\phi)$, we refer to 
\cite[\S 2.2]{BKO24}, where the identification of 
the map $\delta_\chi$ as in \S\ref{ss:RuC} with the dual exponential map is also explained. 

For $v \in H^1(K,{\mathbb{T}}_\phi)$ and $\chi\in\Xi^\mp$, we denote by $\chi(v)$ the image of $v$ 
under the specialization map 
\begin{equation*}
H^1(K, \mathbb{T}_{\phi})   \to H^1(K, \mathbb{T}_{\phi})\otimes_{\Lambda,\chi^{-1}}K_{\chi} \to H^1(K, T^{\otimes -1}(1)\otimes  {\Lambda}\otimes_{\Lambda,\chi^{-1}}K_{\chi})=  H^1(K, T_{\phi}^{\otimes -1}(1)(\chi))
\end{equation*}
induced by the $\CO_K$-homomorphism $\Lambda\to K_{\chi}:=K(\mathrm{Im}(\chi))$ sending
 $[\gamma] \in \Gamma$ to $\chi^{-1}(\gamma)$. 
Let $\exp_{\chi}^*$ denotes the composition of the above and the dual exponential map for $V_{\phi}^{\otimes -1}(1)(\chi)$. 
We put $\varepsilon_{\chi}:= \hat{\varepsilon}(\mathrm{Ind}_{K/\BQ_p}(T_\phi^{\otimes -1}(1)(\chi))=\hat{\varepsilon}(\mathrm{Ind}_{K/\BQ_p}(\phi\chi^{-1}))$. 

Then, by \cite[(2.9)]{BKO24}, 
we have 
 \begin{align*}
 V_{\infty}^{*,\pm} &=  \{ v \in H^1(K,{\mathbb{T}}_\phi)\;|\; \exp^*_\chi v=0 \;\text{for all $\chi\in\Xi^\mp$}\}\\
 &= \{ v \in H^1(K,{\mathbb{T}}_\phi)\;|\; \chi(v) \in H^1_{\rm f}(K, T_\phi^{\otimes -1}(1)(\chi))\;
 \text{for all $\chi\in\Xi^\mp$} \}\\
  &= \{ v \in H^1(K,{\mathbb{T}}_\phi)\;|\; \chi(v) \in H^1_{-\varepsilon_{\chi}}(K, T_\phi^{\otimes -1}(1)(\chi))\;
  \text{for all $\chi\in\Xi^\mp$}\}\\ 
   &= \{ v \in H^1(K,{\mathbb{T}}_\phi)\;|\; \chi(v) \in H^1_{\pm}(K, T_\phi^{\otimes -1}(1)(\chi))\;
   \text{for all $\chi\in\Xi^\mp$}\} \\
   &= H^1_\pm(K,{\mathbb{T}}_\phi).
\end{align*} 
Here, the second last equality relies on Proposition~\ref{prop, root number, inert} and 
the last  on Corollary~\ref{cor, decomp by using Z}. Then the assertion follows from  
Theorem \ref{thm, lsd ind}. 
\end{proof}
\begin{remark}
Part i) has been proved in \cite{BKO21} by a different method, which led to a new proof \cite{BKOY} of Kato's local epsilon conjecture in this case. 
Moreover, part ii) has been proved except for finitely many $\chi$
in \cite[\S5]{BKOY} by using 
elliptic units and Kato's explicit reciprocity law. 
\end{remark}

Now we explicate Theorem \ref{thm, lsd ind} for general conjugate symplectic self-dual characters.

\begin{thm}\label{thm, rubin decomposition}
Let $p$ be an odd prime and $K$ the unramified quadratic extension of $\BQ_p$. 
Let $\phi$ be a de Rham conjugate symplectic self-dual character of $G_K$ with Hodge--Tate weights $(k+1,-k)$ for $k\geq 0$ and of finite character $\eta$.  
Put $T_\phi=\CO_K(\phi)$ 
and $\BT_\phi=T_{\phi}^{\otimes -1}(1)\otimes_{\mathcal{O}_K} \Lambda^\iota$ for $\Lambda$ the anticyclotomic Iwasawa algebra over $K$ with $\CO_K$-coefficients. 
\begin{itemize}
\item[i)] The $\Lambda$-modules $H^1_{\pm}(K,\BT_\phi)$ as in Theorem~\ref{thm, lsd ind} are given by  
\[
H^1_\pm(K,\mathbb{T}_\phi)=\{x \in H^1(K, \mathbb{T}_\phi)\;|\; \exp^*_{\chi}(x)=0 \; \text{for almost all $\chi \in S^\mp$}\} 
\]
where 
\[
S^\mp=
\begin{cases}
  \Xi^\mp \quad (k: \text{even}), \\
   \Xi^\pm \quad (k: \text{odd})
\end{cases}
\text{if $\eta(\delta)=1$}, 
\quad S^\mp=
\begin{cases}
  \Xi^\pm \quad (k: \text{even}), \\
   \Xi^\mp \quad (k: \text{odd})
\end{cases}
\text{if $\eta(\delta)=-1$}.
\]
Moreover, if $\eta$ is at most tamely ramified, this remains true for all characters  $\chi \in S^\mp$ and 
$\Xi^+$ can be replaced with $\Xi^+ \cup \{\rm{1}\}$. 
\item[ii)] The $\Lambda$-submodules $H^1_{\pm}(K,\BT_\phi) \subset H^1(K, \BT_\phi)$ are Lagrangian. 
\item[iii)]As $\Lambda$-modules, we have 
$$
H^1(K,\BT_\phi)=H^1_+(K,\BT_\phi) \oplus H^1_{-} (K,\BT_\phi). 
$$ 
\end{itemize}

\end{thm}

\begin{proof} 
By Propositions \ref{prop, unramified p-adic WD char}, \ref{prop, root number, inert} and \ref{prop, induced p-adic WD}, the corresponding $\varepsilon$-constant is given by 
\[
\varepsilon(\mathrm{Ind}_{K/\BQ_p}\phi\chi)=
\varepsilon(\mathrm{Ind}_{K/\BQ_p}\phi_K\eta\chi)={\eta(\delta)}(-1)^{a(\eta\chi)}. 
\]
Then the completed $\varepsilon$-constant equals ${\eta(\delta)}(-1)^{a(\eta\chi)+k}$. 
\end{proof}

\begin{cor}
\label{cor, lsd ind special}
 For any crystalline conjugate symplectic self-dual character $\phi$ of $G_K$ with Hodge--Tate weights $(k+1,-k)$ and $k\geq 1$, for example 
$\phi=\phi_\pi \psi_\pi^k$ as in Lemma~\ref{lem, shape of conj self-dual I}, the decomposition as in Theorem~\ref{thm, rubin decomposition} ii) 
coincides with the one conjectured by B\"uy\"ukboduk \cite{Kazim}, 
resolving  Conjectures~2.5 and 2.13 of {\it loc. cit}. 
\end{cor}

\begin{remark}\label{rmk:Rubin}\noindent 
Theorem \ref{thm, rubin decomposition} and Corollary \ref{cor, lsd ind special} have applications 
to CM Iwasawa theory at inert primes $p$. 
For example, it leads to an integral $p$-adic $L$-function and a main conjecture, generalizing our results in the good reduction case \cite{BKO21} to the non-semistable reduction case, and it also 
resolves the signed Iwasawa main conjecture in ~\cite[Thm.~B]{Kazim}. 
\end{remark}

In the following we consider higher weight specializations of the signed submodules. 

Let $\omega_2$ be the character of $G_K$ such that $\omega_2(-p)=1$ and 
$\omega_2|_{\mathcal{O}_K^\times}$ is the Teichm\"uller character. 
The anticyclotomic character $\psi_\pi=\phi_\pi/\phi_\pi^*$ has trivial central character and 
$\psi_\pi|_{K^\times_{\mathrm{tor}}}=\omega_2^{1-p}$. 
Hence, $\psi_\Gamma:=\psi_\pi \omega_2^{p-1}$ factors through the anticyclotomic Galois group $\Gamma$ by Lemma~\ref{lem, ac char}.

\begin{cor}\label{higher-weight-unram}
Let $\eta$ be a finite character of $G_K$ such that $\eta|_{\BQ_p^\times}=1$. 
For an integer $l \geq 0$, let $\epsilon_l$ denote the sign of 
$(-1)^{l} \left(\frac{-1}{p}\right)^{l}$. 
Then we have 
\[
H^1_\pm(K, \mathbb{T}_{\phi_\pi\eta})=\{x \in H^1(K, \mathbb{T}_{\phi_\pi\eta})\;|\; \exp^*_{\psi_\Gamma^l\chi}(x)=0 \; 
\text{for $\chi \in S^{\mp \epsilon_l}$ with $a(\chi)>a(\eta)$}\},
\]
where $\exp^*_{\psi_\Gamma^l\chi}$ is the composition of
\[
\begin{split}
H^1(K, \mathbb{T}_{\phi_\pi\eta})&= \varprojlim_{m}H^1(K_m, T_{\phi_\pi\eta}^{\otimes -1}(1))
\cong   \varprojlim_{m}H^1(K_m, T_{\phi_\pi\eta}^{\otimes -1}(1))\otimes \psi_{\Gamma}^{-l}
=\varprojlim_{m}H^1(K_m, T_{\phi_\pi\eta\psi_{\Gamma}^l}^{\otimes -1}(1))\\
&  \rightarrow 
H^1(K, T_{\phi_\pi\eta\psi_{\Gamma}^l}^{\otimes -1}(1)(\chi))
\end{split}
\]
and the dual exponential map.
In particular, we have 
\[
H^1_\pm(K, \mathbb{T}_{\phi_\pi\eta})\otimes \psi_\Gamma^{-l}
=H^1_{\pm\epsilon_l}(K, \mathbb{T}_{\phi_\pi\psi_\Gamma^l\eta}).
\] 
\end{cor}
\begin{proof}
Put  $\xi:=\omega_2^{(p-1)l}$ for simplicity. 
Then the Weil--Deligne character associated with 
$\phi_\pi\psi_\Gamma^l\eta\chi$ is $\phi_K\xi\eta\chi$ for any $\chi \in \Xi$. 

So we have
$$
\hat{\varepsilon}({\rm Ind}_{K/\BQ_p}\phi_\pi\psi_\Gamma^l\eta\chi)=(-1)^{l}\varepsilon({\rm Ind}_{K/\BQ_p}\phi_K\xi\eta\chi)=(-1)^{l}\eta\xi\chi(\delta) (-1)^{a(\xi\eta\chi)}=(-1)^l \left(\frac{-1}{p}\right)^{l} (-1)^{a(\xi\eta\chi)}\eta\chi(\delta)$$ 
by Proposition~\ref{prop, root number, inert}. 
If $a(\chi)>a(\xi\eta)$, then we in turn have  
$$
\hat{\varepsilon}({\rm Ind}_{K/\BQ_p} \phi_\pi\psi_\Gamma^l\chi)= \epsilon_l \cdot 
\hat{\varepsilon}({\rm Ind}_{K/\BQ_p} \phi_\pi\chi).  
$$ 
Hence, the assertion follows from 
Theorem~\ref{thm, lsd ind} i).  
\end{proof}

\subsubsection{A Rubin-type conjecture: over ramified quadratic extensions} The main result of this subsection is 
Theorem~\ref{thm, rubin decomposition ramified}. 

Throughout $K$ denotes a ramified quadratic extension of $\BQ_p$.
Let $\delta \in K$ be a uniformizer such that $K=\BQ_p(\delta)$ and $\delta^2 \in \BQ_p$. 
For a more concrete formulation of the Rubin-type conjecture, 
let $\phi_1$ be  a conjugate symplectic self-dual $p$-adic de Rham character of $G_K$  
with Hodge--Tate weights $(k+1, -k)$ so that $k \geq 0$. We assume that 
the Weil--Deligne character associated to $\phi_1$ has 
conductor $1$. 

For a  finite order character $\chi: \Gamma \rightarrow \ov{\BQ}_p^\times$, put 
\[
\varepsilon_{\phi_1}(\chi):=
\begin{cases}
 \phi_1(\delta)\delta^{-1}  \sum_{a \in \BF_p^\times }\left(\frac{a}{p}\right) \chi(1+\delta^{a(\chi)-1})^a \qquad &(\chi\not=1)\\
\; 1  \qquad &(\chi=1).
\end{cases}
\] 
Note that $\varepsilon_{\phi_1}(\chi)\in\{\pm 1\}$ and it is independent of the choice of $\delta$. 
\begin{lem}\label{lem, ramified completed epsilon} 
Let $\phi_1$ be  a conjugate symplectic self-dual $p$-adic de Rham character of $G_K$ of conductor one  
with Hodge--Tate weights $(k+1, -k)$ for $k \geq 0$, and $\chi: \Gamma \rightarrow \ov{\BQ}_p^\times$ a finite order character. 
\begin{itemize}
\item[i)] We have 
\[
\hat{\varepsilon}(\mathrm{Ind}_{K/\BQ_p}\phi_1\chi)=(-1)^k\omega_{K/\BQ_p}(-2)\left(\frac{-1}{p}\right)^{\frac{a(\chi)}{2}}\varepsilon_{\phi_1}(\chi).  
\] 
\item[ii)] For $b \in \mathbb{Z}_p^\times$, we have 
\[
\hat{\varepsilon}(\mathrm{Ind}_{K/\BQ_p}\phi_1\chi^b)=\left(\frac{b}{p}\right)\hat{\varepsilon}(\mathrm{Ind}_{K/\BQ_p}\phi_1\chi)
\]
Moreover, if $\chi^p\not=1$, then 
\[
\hat{\varepsilon}(\mathrm{Ind}_{K/\BQ_p}\phi_1\chi^{-\delta^2})=\hat{\varepsilon}(\mathrm{Ind}_{K/\BQ_p}\phi_1\chi).
\]
\end{itemize}
\end{lem}
\begin{proof} 
Part i) follows 
from  Propositions \ref{prop, ramified epsilon} and \ref{prop, ramified p-adic WD char} since $\chi(\delta)=1$, and ii) from 
Corollary~\ref{cor, ramified root number}. 
Note that $\phi_1(\delta)\delta^{-1}$ is the value at $\delta$ of 
the Weil--Deligne character of conductor $1$ associated to $\phi_1$, and 
that $a(\chi)$ is even by Lemma \ref{lem, WD csd} iii). 
\end{proof}

\begin{defn}
Let $\Xi_{n}=\Xi^{+}_{\phi_1,n} \cup \Xi^{-}_{\phi_1,n}$
be a partition of finite order anticyclotomic characters of $\Gamma$ of (additive) conductor $n \in \BZ_{\geq 1}$ defined by 
\[
\Xi^\pm_{\phi_1, n}=\{\chi \in \Xi_n\;|\: \hat{\varepsilon}(\mathrm{Ind}_{K/\BQ_p}\phi_1\chi^{-1})=\pm 1\}.
\]
Put $\Xi^\pm_{\phi_1}=\cup_{n=0}^\infty \Xi^\pm_{\phi_1, n}$. 
\end{defn}
Note that $\Xi_n=\emptyset$ for $n$ odd by Lemma \ref{lem, WD csd} iii).

\begin{lem}\noindent
\begin{itemize} 
\item[i)] For $n\geq 2$, we have 
\[
\# \Xi^+_{\phi_1, n}=\# \Xi^-_{\phi_1, n}. 
\]
\item[ii)] The partition $\Xi_{n}=\Xi^{+}_{\phi_1,n} \cup \Xi^{-}_{\phi_1,n}$ is canonical. Moreover, its labeling depends on the choice of $\phi_1$ and it is determined by the labeling of $\Xi_{2}$.
\end{itemize}
\end{lem}
\begin{proof} 
This follows from Lemma~\ref{lem, ramified completed epsilon} ii). 
In fact, pick an element $u \in \BZ_p^\times$ for which 
$\left(\frac{u}{p}\right)=-1$. Then, note that 
$\chi \in \Xi^+_{\phi_1, n}$ if and only if $\chi^u \in \Xi^-_{\phi_1, n}$, and 
$\chi \in \Xi^\pm_{\phi_1, n+2}$ if and only if $\chi^{-\delta^2} \in \Xi^\pm_{\phi_1, n}$.

\end{proof}

\begin{thm}\label{thm, rubin decomposition ramified}
Let $p$ be an odd prime and $K$ a ramified quadratic extension of $\BQ_p$. 
Let $\phi_1$ be a de Rham conjugate symplectic self-dual character of $K^\times$ with Hodge--Tate weights $(k+1,-k)$ for $k\geq 0$ and 
valued in a $p$-adic local field $L$.  
Suppose that the Weil--Deligne character associated to $\phi_1$ has conductor $1$. 
Put $T_{\phi_1}=\CO_L(\phi_1)$ 
and $\BT_{\phi_1}=T_{\phi_1}^{\otimes -1}(1)\otimes_{\mathcal{O}_L} \Lambda^\iota$ for $\Lambda$ the anticyclotomic Iwasawa algebra over $K$ with $\CO_L$-coefficients. 
If $p=3$, suppose that either $K=\BQ_3(\sqrt{3})$, or $K=\BQ_3(\sqrt{-3})$ and $\phi_1 \nequiv 1 \mod{\fm_{\CO_L}}$. 
\begin{itemize}
\item[i)] The $\Lambda$-modules $H^1_{\pm}(K,\BT_{\phi_1})$ as in Theorem~\ref{thm, lsd ind} are given by 
\[
H^1_\pm(K,\mathbb{T}_{\phi_1})=\{x \in H^1(K, \mathbb{T}_{\phi_1})\;|\; \exp^*_{\chi}(x)=0 \; \text{for all $\chi \in \Xi_{\phi_1,n}^\mp$}\} . 
\] 
\item[ii)] The $\Lambda$-submodules $H^1_{\pm}(K,\BT_{\phi_1}) \subset H^1(K, \BT_{\phi_1})$ are Lagrangian. 
\item[iii)]As $\Lambda$-modules, we have 
$$
H^1(K,\BT_{\phi_1})=H^1_+(K,\BT_{\phi_1}) \oplus H^1_{-} (K,\BT_{\phi_1}). 
$$
\end{itemize}
In the excluded case that $p=3$, $K=\BQ_3(\sqrt{-3})$ and $\phi_1 \equiv 1 \mod{\fm_{\CO_L}}$, the assertions as in parts i)-iii) hold for the $\Lambda \otimes_{\BZ_p}\BQ_p$-modules $H^1_{\cdot}(K,\BT_{\phi_1} \otimes_{\BZ_p}\BQ_p)$. 
\end{thm}
\begin{proof}
By Theorem~\ref{thm, lsd ind} i), 
we have 
\[
H^1_\pm(K,\BT_{\phi_1})=\{x \in H^1(K,\mathbb{T}_{\phi_1})\;|\; \exp^*_{\chi}(x)=0 \quad \text{for all $\chi \in \Xi^\mp_{\phi_1}$}\}.
\]
Hence, Theorem~\ref{thm, lsd ind} concludes the proof. 
\end{proof}

Note that the anticyclotomic character $\psi_{1}=\phi_{1}/\phi_{1}^*$ has trivial central character and it factors through the anticyclotomic Galois group $\Gamma$ by Lemma~\ref{lem, ac char}. 
We have the following description of signed submodules under higher weight specializations. 

\begin{cor}\label{twist-ramified}
For an integer $l \geq 0$, let $\epsilon_l$ denote the sign of 
$(-1)^{l} \left(\frac{-1}{p}\right)^{l}.$
Then we have
\begin{align*}
H^1_\pm(K, \mathbb{T}_{\phi_1})= 
 \{x \in H^1(K, \mathbb{T}_{\phi_1})\;|\; \exp^*_{ \psi_{1}^l \chi}(x)=0 \; \text{for all $\chi \in \Xi^{\mp \epsilon_l}_{\phi_{1}}$}\},
\end{align*} 
where $\exp^*_{ \psi_{1}^l \chi}$ is defined in the same way as in
 Corollary \ref{higher-weight-unram}. 
In particular, 
\begin{align*}
H^1_\pm(K, \mathbb{T}_{\phi_1})\otimes \psi_1^{-l} &=H^1_{\pm\epsilon_l}(K, \mathbb{T}_{\phi_1\psi_1^l}).
\end{align*} 
\end{cor}
\begin{proof}
The Weil--Deligne characters associated to $\phi_1\psi_1^l$ is and $\phi_1$ differ by $\epsilon_l$ by Proposition \ref{prop, ramified p-adic WD char}. So we have 
$$
\hat{\varepsilon}({\rm Ind}_{K/\BQ_p}\phi_1\psi_1^l\chi)=(-1)^l \left(\frac{-1}{p}\right)^l 
\hat{\varepsilon}({\rm Ind}_{K/\BQ_p}\phi_1\chi).
$$
Therefore, 
Theorem~\ref{thm, lsd ind} i) concludes the proof. 
\end{proof}

\section{$p$-adic $L$-functions for CM elliptic curves at additive primes}\label{s:pL}
In this section we consider an elliptic curve defined over $\BQ$ with complex multiplication by 
an order of an imaginary quadratic field $\mathcal{K}$. For primes $p$ ramified in $\CK$, we construct a bounded  $p$-adic $L$-function for 
the deformation of the $p$-adic Tate module of $E$ 
arising from the anticyclotomic $\BZ_p$-extension $\CK_\infty^{\rm ac}$ of $\CK$, which is integral except the case $p=3$ and 
$\CK = \BQ(\sqrt{-3})$ (see~Theorem~\ref{thm, ram pL}). 
We also establish a formula for the variation of Mordell--Weil rank of $E$ along the subextensions of $\CK_\infty^{\rm ac} /\CK$ (see Theorem~\ref{vGZK}).

In the sequel \cite{BKNOa} we develop anticyclotomic CM Iwasawa theory at primes $p$ ramified in the CM field
more generally, including CM abelian varieties over $\BQ$ and CM elliptic curves defined over abelian extensions of $\CK$. 
To highlight the phenomena, we only consider CM curves over $\BQ$ in this section.

\subsection{Set-up}\label{ss: s-u}
\subsubsection{} Let $A$ be a CM elliptic curve defined over $\BQ$ and $\CK$ the CM field. 
Let $\omega_{\CK/\BQ}$ denote the quadratic character associated to the extension $\CK/\BQ$.

Note that the class number of $\mathcal{K}$ equals $1$ and $\mathcal{K}=\BQ(\sqrt{-q})$ 
for 
$$q\in\{1, 2, 3, 7, 11, 19, 43, 67, 163\}.$$  
We suppose that $q\geq 3$ and put $p=q$.
Then $p$ is the unique ramified prime in $\mathcal{K}$, and we write $p\CO_\CK=\fp^2$. 
Put $K:=\mathcal{K} \otimes_{\BQ} \BQ_p=\BQ_p(\sqrt{-p})$.

Fix embeddings $\iota_{\infty}:\overline{\BQ} \hookrightarrow  \BC$ and $\iota_p:\overline{\BQ}\hookrightarrow  \overline{\BQ}_p$. Let $$[\;]: \mathcal{K} \isom \mathrm{End}(A) \otimes_{\BZ} \BQ$$ be the isomorphism so that 
$[a]^*\omega_A=a\omega_A$ for a N\'eron differential $\omega_A$ of $A$ and any $a \in \mathcal{K}$. 
Let $\boldsymbol{\phi}$ be the global $p$-adic Hecke character associated to $A$. 
Then  the $p$-adic Tate module $T_pA$ naturally has a structure of $\CO_K$-module free of rank one on which $G_{\CK}$ acts by $\boldsymbol{\phi}$,
and denote by $T_{\boldsymbol{\phi}}$ the $\CO_K[G_{\CK}]$-module. 
Put $V_{\boldsymbol{\phi}}=T_{\boldsymbol{\phi}}\otimes_{\CO_K} K$. 

\begin{lem} We have a canonical isomorphism 
$$V_pA\otimes_{\BQ_p} K \cong \mathrm{Ind}_{\mathcal{K}/\BQ}\,V_{\boldsymbol{\phi}}$$
of $K[G_{\BQ}]$-modules. 
\end{lem}
\begin{proof}
There exists a $K[G_\CK]$-equivariant canonical projection 
\[V_pA\otimes_{\BQ_p} K \cong 
V_{\boldsymbol{\phi}} \oplus V_{\boldsymbol{\phi}^c}
 \rightarrow V_{\boldsymbol{\phi}},\] 
where ${\boldsymbol \phi}^{c}:={\boldsymbol \phi} \circ c$ for $c\in \Gal(\CK/\BQ)$ the non-trivial element. 
Note that $V_pA\otimes_{\BQ_p} K$  has 
structure of a $K[G_{\BQ}]$-representation,  
extending the $K[G_\mathcal{K}]$-action. 
   Therefore, by the adjoint property of induced representations and 
 the irreducibility of $V_pA\otimes_{\BQ_p}K$, the assertion follows. 
 \end{proof}

For $\cdot \in \{\emptyset, \otimes^{-1} \}$, 
put $V_{\boldsymbol{\phi}}^{\cdot}=T_{\boldsymbol{\phi}}^{\cdot}\otimes_{\CO_K}K$. 

\subsubsection{}
Let $\CK_\infty^{\rm ac}$ be the anticyclotomic $\BZ_p$-extension of $\CK$ and $\CK_n^{\rm ac}$ its subextension of degree $p^n$ over $\CK$ for an integer $n\geq 0$. Put $\Lambda:=\mathcal{O}_K[\![\mathrm{Gal}(\mathcal{K}^{\rm ac}_\infty/\mathcal{K})]\!]$. 

We have the following variation of global epsilon constants along the anticyclotomic tower. 
\begin{prop}\label{lem, eps global} 
Let $A$ be a CM elliptic curve defined over $\BQ$ and $\CK$ the CM field. 
Let $p$ be an odd prime ramified in $\CK$, ${\boldsymbol \phi}$ the associated $p$-adic character of $G_\CK$ and $K$ the localization of $\CK$ at the prime above $p$.
Let $\chi: \Gal(\CK_{\infty}^{\rm ac}/\CK) \ra \ov{\BQ}^\times$ be a finite order character. 
\begin{itemize}
\item[i)] We have 
$$
\frac{\varepsilon(\boldsymbol{\phi} \chi)}{\varepsilon(\boldsymbol{\phi})}
=\frac{\varepsilon_p({\rm Ind}_{K/\BQ_p}\boldsymbol{\phi}\chi)}{\varepsilon_p({\rm Ind}_{K/\BQ_p}\boldsymbol{\phi})}.  
$$
\item[ii)] For $b \in \mathbb{Z}_p^\times$, 
\[
\varepsilon(\boldsymbol{\phi}\chi^b)=\left(\frac{b}{p}\right)\varepsilon(\boldsymbol{\phi}\chi). 
\]
\item[iii)] If $\chi^p\not=1$, then 
\[
\varepsilon(\boldsymbol{\phi}\chi^{-\delta^2})=\varepsilon(\boldsymbol{\phi}\chi) 
\]
for $\delta \in K$ a uniformizer\footnote{One may choose $\delta$ so that $\delta^2=-p$.} such that $\delta^2 \in \BQ_p$. 
\item[iv)] For an integer $n\geq 1$, let $\Xi_{\CK_n^{\rm ac}}$ denote the set of $\ov{\BQ}^\times$-valued characters of $\Gal(\CK_{n}^{\rm ac}/\CK)$
not factoring through $\Gal(\CK_{n-1}^{\rm ac}/\CK)$. Then we have 
$$
\#\{\chi \in \Xi_{\CK_n^{\rm ac}}| \varepsilon(\boldsymbol{\phi}\chi)=+1\}=
\#\{\chi \in \Xi_{\CK_n^{\rm ac}}| \varepsilon(\boldsymbol{\phi}\chi)=-1\}. 
$$
\end{itemize}
\end{prop} 
\begin{proof}
i) Note that 
\[
\frac{\varepsilon(\boldsymbol{\phi} \chi)}{\varepsilon(\boldsymbol{\phi})}
=\frac{\varepsilon({\rm Ind}_{\mathcal{K}/\BQ}\boldsymbol{\phi}\chi)}{\varepsilon({\rm Ind}_{\mathcal{K}/\BQ}\boldsymbol{\phi})}
=\frac{\hat{\varepsilon}_p({\rm Ind}_{\mathcal{K}/\BQ}\boldsymbol{\phi}\chi)}{\hat{\varepsilon}_p({\rm Ind}_{\mathcal{K}/\BQ}\boldsymbol{\phi})}.  
\prod_{l\not= p} \frac{\varepsilon_l({\rm Ind}_{\mathcal{K}/\BQ}\boldsymbol{\phi}\chi)}{\varepsilon_l({\rm Ind}_{\mathcal{K}/\BQ}\boldsymbol{\phi})}.\]
Hence, for a place $v|l$ of $K$, it suffices to show that 
the local epsilon constants 
$\varepsilon_v({\boldsymbol \phi}_v\chi,\psi_v, dx_v)$ and $\varepsilon_v({\boldsymbol \phi}_v, \psi_v,dx_v)$ 
are equal up to a $p$-power root of unity, where ${\boldsymbol \phi}_v$ is the Weil--Deligne character associated to 
${\boldsymbol \phi}|_{G_{K_v}}$, $\psi_v$ is a non-trivial additive character of $K_v$, 
and $dx_v$ is a Haar measure on $K_v$. 
This equality is a direct consequence of (\ref{equation, epsilon unramified twist}) since $\chi$ is unramified at $v$. 

Parts ii)-iv) follow  from Lemma~\ref{lem, ramified completed epsilon} iii). 
\end{proof}
\begin{defn} 
For $n\geq 1$, let $\Xi_{\CK_{n}^{\rm ac}}=\Xi_{\CK_{n}^{\rm ac},\boldsymbol{\phi}}^{+} \cup \Xi_{\CK_{n}^{\rm ac},\boldsymbol{\phi}}^{-}$
be a partition of  $\ov{\BQ}^\times$-valued 
characters of $\Gal(\CK_{n}^{\rm ac}/\CK)$
 defined by 
$$
\Xi_{\CK_{n}^{\rm ac},\boldsymbol{\phi}}^{\pm} =\{\chi \in \Xi_{\CK_{n}^{\rm ac}} | \varepsilon(\boldsymbol{\phi}\chi^{-1})=\pm 1 \}.
$$
For $\cdot \in\{\emptyset,+,-\}$, 
put $\Xi_{\CK_{\infty}^{\rm ac},\boldsymbol{\phi}}^{\cdot}=\cup_{n} \Xi_{\CK_{n}^{\rm ac},\boldsymbol{\phi}}^{\cdot}$. 
\end{defn} 
If $\chi \in \Xi_{\CK_{n}^{\rm ac},{\boldsymbol \phi}}^-$, then $\varepsilon({\boldsymbol \phi}\chi^{-1})=-1$ and so
$$
L(\boldsymbol{\phi}\chi^{-1},1)=0 
$$
for $L(\boldsymbol{\phi}\chi,s)$ the associated Hecke $L$-function with center at the point $s=1$.
In view of this vanishing one may seek $p$-adic interpolation of a normalization of algebraic part of the Hecke $L$-values $L({\boldsymbol \phi}\chi^{-1},1)$ for $\chi \in \Xi_{\CK_n^{\rm ac},{\boldsymbol \phi}}^+$. A main result of this section is the existence of such a bounded $p$-adic $L$-function (see~Theorem~\ref{thm, ram pL}). 

\subsection{Variation of Mordell--Weil ranks}
This subsection establishes a formula for the Mordell--Weil rank of the CM elliptic curve $A$ along subextensions of the anticyclotomic $\BZ_p$-extension $\CK_\infty^{\rm ac}/\CK$. 

 The main result of this subsection is the following. 
 \begin{thm}\label{vGZK}
Let $A$ be a CM elliptic curve defined over $\BQ$ and $\CK$ the CM field. 
Let $p$ be an odd prime ramified in $\CK$ and $\CK_n^{\rm ac}$ the subextension of the anticyclotomic $\BZ_p$-extension $\CK_\infty^{\rm ac} /\CK$ of degree $p^n$ over $\CK$.
Then there exists a constant $c\in\BZ$ such that  for any sufficiently large integer $n$, we have 
 $$\rank_{\BZ}\,A(\CK_n^{\rm ac})=p^{n}+c. $$ 
\end{thm} 

We begin with some preliminaries.

For a finite order character $\chi$  of $\Gal(\CK_{\infty}^{\rm ac}/\CK)$,  
let $f_{\chi} \in S_2(\Gamma_{0}(D_{K}N_{\CK/\BQ}(\cond({\boldsymbol \phi}\chi)))$ denote the theta series  associated to the conjugate symplectic self-dual\footnote{the natural global analogue of Definition~\ref{def:csd}} Hecke character ${\boldsymbol \phi}\chi^{}$ over $\CK$, where  $\cond(\cdot)$ denotes the conductor.
In particular,  we have $$L(f_{\chi},s)=L({\boldsymbol \phi}\chi,s).$$ 

Denote by $A_{\chi}$ the CM abelian variety defined over $\BQ$ associated to $f_\chi$. Note that 
$$L(A_{\chi}/\BQ,s)=\prod_{\sigma: F_{\chi}\hookrightarrow \overline{\BQ}}L(f_{\chi}^{\sigma},s)$$
where $F_\chi$ is the Hecke field of $f_\chi$. 
Suppose that the order of $\chi$ is $p^n$. 
Then $F_\chi$ is the maximal totally real field in $\BQ(\zeta_{p^n})$. 
Note  that $\mathcal{K}=\BQ(\sqrt{-p})\subseteq \BQ(\zeta_{p^n})$, and so 
$\End_{\BQ}^0 (A_\chi) \cong F_\chi$ and 
 $\mathrm{End}_{\mathcal{K}}^0(A_{\chi})\cong \mathcal{K}\otimes_\BQ F_\chi \cong \BQ(\zeta_{p^n})$ 
 where $\mathrm{End}^0:=\mathrm{End} \otimes_\BZ \BQ$.  
As $G_{\BQ}$-representations,  we have 
\begin{equation}\label{equation, A_chi tate}
V_pA_\chi\otimes_{\BQ_p} \overline{\BQ}_p \cong \prod_{\sigma} \mathrm{Ind}_{\mathcal{K}/\BQ}
\;\overline{\BQ}_p ({\boldsymbol \phi}\chi^{\sigma})=\prod_{a \in (\BZ/p^n\BZ)^\times, \left(\frac{a}{p}\right)=1} 
\mathrm{Ind}_{\mathcal{K}/\BQ}\;\overline{\BQ}_p({\boldsymbol \phi}\chi^{a})
\end{equation}
where $\chi^\sigma$ runs through characters of $\Gal(\BQ(\zeta_{p^n})/\CK)\cong \Gal(F_\chi/\BQ)$.

 We have the following fundamental result towards the Birch and Swinnerton-Dyer conjecture due to Gross--Zagier \cite{GZ} and Kolyvagin (cf.~\cite[Prop.~3.5]{BKO24}).
  \begin{prop}\label{GZK}
 Suppose that $\mathrm{ord}_{s=1} L({\boldsymbol \phi}\chi^{},s)\in\{0,1\}$. 
 Then $\mathrm{dim}_{F_{\chi}} A_{\chi}(\BQ)\otimes_{\BZ} \BQ=\mathrm{ord}_{s=1} L({\boldsymbol \phi}\chi^{},s)$. 
 \end{prop} 

\begin{proof}[Proof of Theorem \ref{vGZK}] 
Let $B_n$ denote 
the Weil restriction $\Res_{\CK_n^{\rm ac}/\CK}(A_{/\CK_n^{\rm ac}})$ of $A$ over $\CK_{n}^{\rm ac}$. 
By \cite[Thm.~3]{DN}, we have 
$$\mathrm{End}^0(B_n)=\mathrm{End}^0(A) [\Gal(\CK_{n}^{\rm ac}/\CK)]=\mathcal{K}[\Gal(\CK_{n}^{\rm ac}/\CK)].$$ 
Here, the action of $\Gal(\CK_{n}^{\rm ac}/\CK)$ on $B_n(\mathcal{K})$ by the above equality is compatible with 
the Galois action of $A(\CK_{n}^{\rm ac})$ arising from the canonical identification $B_n(\mathcal{K})=A(\CK_{n}^{\rm ac})$. 

Let $\Phi_{p^n}(X)$ be the $p^n$-th cyclotomic polynomial and 
let $\Phi_{p^n}(X)=\Phi_{p^n}^+(X)\Phi_{p^n}^-(X)$ be the decomposition into monic polynomials in $\mathcal{K}[X]$. 
Fix a generator $\gamma$ of $\Gal(\CK_{n}^{\rm ac}/\CK)$ and put 
\[R_n^\cdot:=\mathcal{K}[\Gal(\CK_{n}^{\rm ac}/\CK)]/(\Phi_{p^n}^\cdot(\gamma))\]
for $\cdot \in \{\emptyset, \pm\}$. 
Then a character 
$\chi: \Gal(\CK_{n}^{\rm ac}/\CK) \rightarrow \overline{\BQ}^\times$ of order $p^n$ 
determines a 
specialization 
\[
R_n=R_n^+ \times R_n^-\rightarrow \overline{\BQ}. 
\] 
According to Proposition~\ref{lem, eps global} ii), we can give the sign labeling of $\Phi_n^\pm(X)$ so that 
$\chi$ factors through $R_n^+$ if and only if $\varepsilon({\boldsymbol \phi}\chi)=+1$. 

We have a decomposition of  $\mathcal{K}$-algebras
 \[\mathrm{End}^0(B_n)\cong R_n^+\times R_n^-\times R',\] 
and hence the isogeny decomposition 
\[
B_n  \sim A_n^+  \times A_n^-  \times A'
\]
for abelian varieties over $\mathcal{K}$. By construction, note that 
$A(\CK_{n-1}^{\rm ac})\otimes_{\BZ}\BQ=A'(\mathcal{K})\otimes_{\BZ}\BQ$ and 
\begin{equation*}
(A(\CK_n^{\rm ac})/A(\CK_{n-1}^{\rm ac}))\otimes_{\BZ} \BQ\cong 
 \left(A_{n}^+(\CK) \oplus A_{n}^-(\CK)\right) \otimes_{\BZ} \BQ. 
\end{equation*}
Note also that
\[
V_pA_n^\pm\otimes_{{K}} \overline{\BQ}_p \cong V_pA \otimes_{{K}} \hat{R}_n^\pm \otimes_{{K}} \overline{\BQ}_p \cong \prod_{\chi:R_n^\pm \rightarrow \overline{\BQ}} \overline{\BQ}_p ({\boldsymbol \phi}\chi)
\]
as $G_{\mathcal{K}}$-representations where $\hat{R}_n^\pm:={R}_n^\pm\otimes_{\mathcal{K}} K$. 
Hence, by (\ref{equation, A_chi tate}), we have 
\[V_pA_n^\pm\otimes_{\BQ_p} \overline{\BQ}_p\cong V_pA_\chi \otimes_{\BQ_p} \overline{\BQ}_p|_{G_{\mathcal{K}}}\]
 for any $\chi$ factoring through $R_n^\pm$. 
Then we have 
$A_n^\pm \sim A_\chi$ over $\mathcal{K}$ by Faltings. 
In particular, we may assume that $A_n^\pm$ is of the form $A_\chi$ and defined over $\BQ$. 

Put $V= A^{\pm}_n(\mathcal{K})\otimes_{\BZ}\BQ$. 
For the non-trivial element $c\in \mathrm{Gal}(\mathcal{K}/\BQ)$, consider the eigen decomposition 
$V=V^{c=1}\oplus V^{c=-1}.$
Then, $V^{c=1}=A^\pm_n(\BQ)\otimes_{\BZ}\BQ$, and the multiplication  by a purely totally imaginary element 
in $\mathrm{End}_\mathcal{K}^0(A_n^\pm)\cong \BQ(\zeta_{p^n})$ gives an isomorphism 
$V^{c=1}\cong V^{c=-1}$. Hence, 
\[\dim_\BQ (A^\pm_n(\mathcal{K})\otimes_{\BZ}\BQ)
=2\dim_{\BQ}(A^\pm_n(\BQ)\otimes_{\BZ}\BQ).\] 
By Rohrlich \cite{Ro'},
for sufficiently large $n$, 
we have 
$$
\ord_{s=1}L({\boldsymbol \phi}\chi,s)=\frac{1-\varepsilon({\boldsymbol \phi}\chi)}{2}. 
$$
Hence, by Proposition~\ref{GZK}, 
we have 
\[
\rank_{\BZ} A_{n}^+(\BQ)=0, \quad \rank_{\BZ} A_{n}^-(\BQ)=(p^n-p^{n-1})/2
\]
for almost all $n$. 
Therefore, $\rank_{\BZ} A(\CK_n^{\rm ac})-\rank_{\BZ}A(\CK_{n-1}^{\rm ac})=p^n-p^{n-1}$ for any such $n$, 
and the assertion follows. 
 \end{proof}

 \begin{remark}
 A generalization to CM elliptic curves not necessarily defined over $\BQ$  will appear in \cite{BKNOa}. 
 \end{remark}

  \subsection{Elliptic units}\label{ss:ell}
This subsection describes preliminaries regarding elliptic units following Kato \cite[\S 15.3]{Kato}. 

\subsubsection{Set-up}
We consider a slightly more general set-up. Let $\CK$ be any imaginary quadratic field and $p$ an odd prime. 

For a non-zero ideal $\fg$ of $\CO_{\mathcal{K}}$,
let $\mathcal{K}(\fg)$ denote the ray 
class field 
of conductor $\fg$. If $\fg = c\CO_\mathcal{K}$ for some $c \in \BZ\setminus\{0\},$
then let $\mathcal{K}[\fg]$ denote the ring class field associated to the order $\BZ+ c\CO_{\CK}$ of $\CO_K$. 
Put  
 $\mathcal{K}[p^{\infty}\fg] = \cup_{m\ge 0}\mathcal{K}[ p^m \fg]$, 
 $\Lambda_{\fg}=\CO_K[\![\Gal(\mathcal{K}[p^{\infty}\fg]/\mathcal{K})]\!]$ 
 and $Q(\Lambda_{\fg})$ its total quotient ring. 
 For a finitely generated $\BZ_p$-module $T$ with continuous action of $G_\mathcal{K}$ and $i\in \BZ$,
 define Iwasawa cohomology groups 
\begin{alignat*}{2}
 \mathbf{H}^i_{\fg}(T)&=\varprojlim_n H^i(\mathcal{K}(p^n \fg), T),  &\mathbf{H}^i_{\fg}(T\otimes \BQ_p)&= \mathbf{H}^i_{\fg}(T)\otimes_{\BZ_p}\BQ_p,\\
 \mathbf{H}^i_{\fg,  \mathrm{ac}}(T)&=\varprojlim_n H^i(\mathcal{K}[p^n \fg], T),\quad   &\mathbf{H}^i_{\fg, \mathrm{ac}}(T\otimes \BQ_p)&= \mathbf{H}^i_{\fg,  \mathrm{ac}}(T) \otimes_{\BZ_p} \BQ_p .
\end{alignat*}
If $\fg=(1)$, then we omit $\fg$ in the notation above.

Let $\mathfrak{f}$ be a non-zero ideal of $\CO_\mathcal{K}$ 
such that the natural map $\CO_{\CK}^\times \rightarrow (\CO_\mathcal{K}/\mathfrak{f})^\times$ 
is injective\footnote{If $\mathcal{K}$ is neither $\BQ(i)$ nor $\BQ(\sqrt{-3})$, then  
this holds if $\mathfrak{f}\not=(2)$.}.
To define elliptic units, we introduce another CM elliptic curve $E$ 
defined over $K(\ff)$ with 
an isomorphism $\iota: \mathcal{O}_\mathcal{K} \cong \mathrm{End}(E)$ such that the composite 
\[
\mathcal{O}_\mathcal{K} \xrightarrow{\iota} \mathrm{End}(E) \rightarrow \mathrm{End}_{\mathcal{K}'}(\mathrm{Lie} (E))=\mathcal{K}'
\]
coincides with the natural inclusion. 
We may assume that 
there exists a generator $\alpha$ of $E[\ff]$ and an isomorphism $E\cong \BC/\ff$ over $\BC$  
sending $\alpha$ to $1\bmod \ff$. 
Such a CM pair $(E, \alpha)$ (also with $\iota$) exists uniquely up to a canonical isomorphism (cf.\ \cite[(15.3.1)]{Kato}). 
Let $\mathfrak{a}$ be an ideal of $\mathcal{O}_{\mathcal{K}}$ coprime to $6p\mathfrak{f}$.

 Let 
\begin{equation}\label{ell unit}
{}_\mathfrak{a}z_{\mathfrak{f}}:={}_\mathfrak{a}\theta_{E}(\alpha)^{-1} \in \mathcal{K}(\mathfrak{f})^\times
\end{equation}
be the standard elliptic unit of modulus $\mathfrak{f}$ 
as in~\cite[\S15.4]{Kato}. 
For an ideal $\fg \subset \CO_\CK$, let 
$$\!_{\fa}z_{p^n \fg} \in H^1(\mathcal{K}(p^n \fg), \BZ_p(1))$$ also denote  the image of 
 $\!_{\fa}z_{p^n \fg}$ as in \eqref{ell unit} 
under the Kummer map. (Here, we assume $p^n \fg$ satisfies the condition on $\ff$ above.)
Then $\!_{\fa}z_{p^{\infty}\fg } = (\,\!_{\fa}z_{p^n \fg})_n$ is an element of ${\bf H}^1_{\fg}(\BZ_p(1)).$

 \subsubsection{Zeta morphism}
Let $A$ be a CM elliptic curve defined over $\BQ$ as in \S\ref{ss: s-u} and 
$\ff_A$ denotes the conductor of the associated Hecke character over $\CK$. 

 Fix an element   $\gamma \in V_{\boldsymbol{\phi}}^{\otimes-1}$ and 
 put $\fg =N_{\mathcal{K}/\BQ}\ff_{A}$.  
 Let 
 \[
\!_{\fa} z_{p^{\infty}\fg}^{\gamma} \in {\bf H}^1_{\fg}(T_{\boldsymbol{\phi}}^{\otimes -1}(1))\otimes_{\BZ_p}\BQ_p
 \]
 denotes the image of the elliptic unit $\!_{\fa}z_{p^{\infty}\fg }$   
under the composite
\[
 {\bf H}^1_{\fg}(\BZ_p(1)) \xrightarrow{\otimes \gamma}  {\bf H}^1_{\fg}(\BZ_p(1))\otimes V_{\boldsymbol{\phi}}^{\otimes -1} \xrightarrow{\sim} {\bf H}^1_{\fg}(V_{\boldsymbol{\phi}}^{\otimes -1}(1))
\]
(cf.~\cite[Prop.~15.9]{Kato}). 
If $\gamma \in T_{\boldsymbol{\phi}}^{\otimes-1}$, then 
 $
\!_{\fa} z_{p^{\infty}\fg}^{\gamma} \in {\bf H}^1_{\fg}(T_{\boldsymbol{\phi}}^{\otimes -1}(1)).
 $

Denote by $\!_{\fa} z^{\gamma,  \mathrm{ac}}_{p^{\infty}\fg} \in {\bf H}^1_{\fg,  \mathrm{ac}}(T_{\boldsymbol{\phi}}^{\otimes -1}(1))$
the anticyclotomic projection of $\!_{\fa} z^{\gamma}_{p^{\infty}\fg}$ 
and put
\[
z_{p^{\infty}\fg , \fa}^{\gamma,  \mathrm{ac}}=\left(N_{\mathcal{K}/\BQ}(\fa)- \boldsymbol{\phi}(\fa)^{-1}(\fa, \mathcal{K}[p^{\infty}\fg]/\mathcal{K})\right)^{-1}\!_{\fa}z_{p^{\infty} \fg }^{\gamma,  \mathrm{ac}} \in  {\bf H}^1_{\fg, \mathrm{ac}}(T_{\boldsymbol{\phi}}^{\otimes -1}(1)) \otimes_{\Lambda_{\fg, {\rm ac}}} Q(\Lambda_{\fg,  \mathrm{ac}}),  
\]
where $(\fa,\CK[p^\infty \fg]/\CK)\in\Gal(\CK[p^\infty \fg]/\CK)$ denotes the associated Artin symbol. 
\begin{lem}
For   $\gamma \in V_{\boldsymbol{\phi}}^{\otimes-1}$, 
the element $z_{p^{\infty}\fg,\fa}^{\gamma,  \mathrm{ac}}=:z_{p^{\infty}\fg}^{\gamma,  \mathrm{ac}}$ is independent of the choice of $\fa$ and 
$$z_{p^{\infty}\fg,\fa}^{\gamma,  \mathrm{ac}} \in {\bf H}^1_{\fg, \mathrm{ac}}(T_{\boldsymbol{\phi}}^{\otimes -1}(1)) \otimes_{\BZ_p}\BQ_{p}.$$
Moreover, if $\gamma \in T_{\boldsymbol{\phi}}^{\otimes-1}$, then
$z_{p^{\infty}\fg,\fa}^{\gamma,  \mathrm{ac}} \in {\bf H}^1_{\fg, \mathrm{ac}}(T_{\boldsymbol{\phi}}^{\otimes -1}(1))$. 
\end{lem}
\begin{proof}
The independence is shown in \cite[(15.5.4)]{Kato}.

Pick a prime $l\nmid pN_{\mathcal{K}/\BQ}(\fg)$ 
 inert in $\mathcal{K}$ such that $l^3\not\equiv -1 \bmod p$ and 
put $\fa=(l)$. Then we have $N_{\mathcal{K}/\BQ}(\fa)=l^2$ and $\boldsymbol{\phi}(\fa)=-l$ (cf.\ \cite[Lem.\ 3.1]{Ru}).
Since $(\fa, \mathcal{K}[p^\infty \fg]/\mathcal{K})=1$, it follows 
that $$N_{\mathcal{K}/\BQ}(\fa)- \boldsymbol{\phi}(\fa)^{-1}(\fa, \mathcal{K}[p^{\infty}\fg]/\mathcal{K}) \in \Lambda_{\fg}^{\times}$$ 
and the assertion follows. 
\end{proof} 
   \begin{defn} 
   The anticyclotomic zeta morphism for ${\boldsymbol \phi}$ is defined by 
\[
\mathfrak{z}_{{\boldsymbol{\phi}},\fg}: \;T_{\boldsymbol{\phi}}^{\otimes-1} \rightarrow {\bf H}^1_{\fg, \mathrm{ac}}(T_{\boldsymbol{\phi}}^{\otimes -1}(1)), \quad \gamma \mapsto z_{p^{\infty}\fg }^{\gamma, \mathrm{ac}}.
\]
and its extension to $V_{\boldsymbol{\phi}}^{\otimes-1}$. 
Define $\mathfrak{z}_{{\boldsymbol{\phi}}}$ to be 
the composition of  $\mathfrak{z}_{{\boldsymbol{\phi}},\fg}$ and 
the correstriction  \[{\bf H}^1_{\fg, \mathrm{ac}}(T_{\boldsymbol{\phi}}^{\otimes -1}(1)) 
\rightarrow {\bf H}^1_{\mathrm{ac}}(T_{\boldsymbol{\phi}}^{\otimes -1}(1)) .\] 
\end{defn}

   \subsubsection{Explicit reciprocity law}\label{ss,erl}
The main result of this subsection is Theorem~\ref{thm, erl}.

As before, we fix embeddings $\iota_\infty: \overline{\BQ} \hookrightarrow \BC$ and 
$\iota_p: \overline{\BQ} \hookrightarrow \BC_p$. 
Put $\CO_{\CK,(p)}:=\CO_{\CK}\otimes_{\BZ} \BZ_{(p)}$ and 
fix an $\CO_{\CK,(p)}$-basis $\gamma_A$ of $\Hom_{\CO_{\CK, (p)}}(H_1(A(\BC), \BZ_{(p)}), \CO_{\CK,(p)})$
 and a minimal Weierstrass model of the elliptic curve $A$ over $\CO_{\mathcal{K},(p)}$ with N\'eron differential $\omega_A\in \mathrm{coLie}(A)$.
We regard $\gamma_A$ 
as an $\CO_K$-basis of $T_{\boldsymbol{\phi}}^{\otimes-1}$
via the isomorphism 
$H_1(A(\BC), \BZ)\otimes_{\BZ}\BZ_{p} \cong T_p(A)$ arising from the embedding $\iota_\infty$. 

Consider the period map 
$$\mathrm{per}: \mathrm{coLie}(A)\to H^1(A(\BC), \BC)= \mathrm{Hom}_{\BQ}(H_1(A(\BC), \BQ), \BC) : \omega\mapsto 
\left[\;\gamma\mapsto \int_{\gamma}\omega \;\right].$$
It factors through 
 the $\BC$-vector subspace $$\mathrm{Hom}_{\CK}(H_1(A(\BC), \BQ), \BC)
 =\mathrm{Hom}_{\mathcal{K}}(H_1(A(\BC), \BQ), \CK)\otimes_{\CK}\BC,$$
 hence obtains an $\CO_{\CK}$-linear map which we also denote by 
 $$\mathrm{per}: \mathrm{coLie}(A)\rightarrow \mathrm{Hom}_{\CO_{\CK, (p)}}(H_1(A(\BC), \BZ_{(p)}), \CO_{\CK,(p)})\otimes_{\CO_{\CK,(p)}}\BC.$$
 (cf.\ \cite[p.\ 257]{Kato}).
Define  $\Omega_{\infty} \in \BC^{\times}$
by $$\mathrm{per}(\omega_A)=\Omega_{\infty}\gamma_A.$$ 
Consider the complex uniformization $\BC/\Lambda_A \isom A(\BC)$ such that the pull-back of $\omega_A$ 
is $dz$. Then we have 
\[\Lambda_A \otimes_\BZ \BZ_{(p)} =\CO_{\CK,(p)}\Omega_\infty.\] 

By definition, we have the following.
\begin{lem}\label{lem, int} 
We have 
$
\mathfrak{z}_{{\boldsymbol{\phi}},\fg}({\gamma_A}) \in {\bf H}^1_{\fg, \mathrm{ac}}(T_{\boldsymbol{\phi}}^{\otimes -1}(1))
$
and 
$\mathfrak{z}_{{\boldsymbol{\phi}}}({\gamma_A}) \in {\bf H}^1_{\mathrm{ac}}(T_{\boldsymbol{\phi}}^{\otimes -1}(1))$. 
\end{lem}

Let $K_n$ denote the completion of $\mathcal{K}_n^{\rm{ac}}$ under the embedding $\iota_{p}:\ov{\BQ}\hookrightarrow \ov{\BQ}_p$,  and 
identify $\mathrm{Gal}(\mathcal{K}_n^{\rm{ac}}/\mathcal{K})$ with 
$\mathrm{Gal}(K_n/K)$ since $ h_\CK=1$. 
For a character $\chi: \mathrm{Gal}(\mathcal{K}_n^{\rm{ac}}/\mathcal{K}) \rightarrow \ov{\BQ}^\times$, consider 
the composite
\begin{align*}
\exp^*_{\omega_A,n}: H^1(K, \BT_{\boldsymbol \phi}^{})
\rightarrow H^1(K_n, V_{\boldsymbol{\phi}}^{\otimes -1}(1)) 
\stackrel{\;\exp^*\;}{\longrightarrow} D_{\rm{dR}}^0(K_n, V_{\boldsymbol{\phi}}^{\otimes -1}(1))
=K_n \omega_A \stackrel{\;\omega_A^\vee\;}{\longrightarrow}  K_n, 
\end{align*}
where $\exp^*$ denotes the dual exponential map for $T_{\boldsymbol{\phi}}^{\otimes -1}(1)$.

We have the following explicit reciprocity law for elliptic units due to Kato (cf.~\cite[Thm.~15.9]{Kato}). 
\begin{thm}\label{thm, erl} 
Let $A$ be a CM elliptic curve defined over $\BQ$ and $\CK$ the CM field. 
Let $p$ be an odd ramified prime and ${\boldsymbol \phi}$ the associated $p$-adic character of $G_\CK$. 
Then for any character $\chi:\Gal(\CK_{n}^{\rm ac}/\CK) \ra \ov{\BQ}^\times$, we have\footnote{
The left and right hand sides are elements of $\overline{\BQ}$ via the embeddings $\iota_p$ and $\iota_\infty$. 
} 
\begin{align*}
\sum_{\sigma \in \mathrm{Gal}(K_n/K)}\chi(\sigma)\exp^*_{\omega_A,n}(\mathfrak{z}_{{\boldsymbol{\phi}}}(\gamma_A)^{\sigma})=\frac{L(\overline{\boldsymbol{\phi}}\chi, 1)}{\Omega_\infty}=\frac{L({\boldsymbol{\phi}}\chi^{-1}, 1)}{\Omega_\infty}. 
\end{align*} 
\end{thm}
\begin{remark}
The above also holds for an inert $p$ if we replace $L({\boldsymbol{\phi}}\chi^{-1}, 1)$ by the imprimitive 
$L$-value $L_p({\boldsymbol{\phi}}\chi^{-1}, 1)$.
\end{remark}

\subsection{$p$-adic $L$-function}\label{ss:pL}
\subsubsection{Elliptic units and local sign decomposition} 
Let the setting be as in \S\ref{ss:ell}.

For the $p$-adic Hecke character $\boldsymbol{\phi}$, 
let $\mathbb{T}_{\boldsymbol \phi}$ denote the $G_{\CK}$-representation 
$T_{\boldsymbol{\phi}}^{\otimes -1}(1)\otimes_{\CO_K} 
\Lambda^{\iota}$ over the anticyclotomic Iwasawa algebra 
$\Lambda$,
where $\Lambda^{\iota}$ denotes the $\Lambda$-free module of rank one with basis $1\in \Lambda$
 on which $G_{\CK}$ acts by $ga=[g^{-1}]a$ for $g\in G_{\CK}$ and $a \in \Lambda.$
Then we have a canonical $\Lambda$-module isomorphism
$\varprojlim_m H^1(K_{m}, T_{\boldsymbol{\phi}}^{\otimes -1}(1))\cong H^1(K, \mathbb{T}_{\boldsymbol \phi}),$
by which we identify these modules. 

Note that $\mathbb{T}_{\boldsymbol \phi}$ is a conjugate symplectic self-dual $\Lambda$-representation of $G_K$.

By the analogue of Rubin's conjecture  as in~Theorems~\ref{thm, rubin decomposition} and \ref{thm, rubin decomposition ramified}, if either $p\ge 5$ or $p=3$ and ${\boldsymbol \phi} \nequiv 1 \mod{\fm_{\CO_L}}$ on $G_K$, then we have
$$
H^1(K,\BT_{\boldsymbol \phi})=H^1_+(K,\BT_{\boldsymbol \phi}) \oplus H^1_{-}(K,\BT_{\boldsymbol \phi}),
$$
and  if $p=3$ and ${\boldsymbol \phi} \equiv 1 \mod{\fm_{\CO_L}}$ on $G_K$, similarly defining
 \[
H^1_\pm(K,\mathbb{T}_{\boldsymbol \phi})=\{x \in H^1(K, \mathbb{T}_{\boldsymbol \phi})\;|\; \exp^*_{\chi}(x)=0 \; \text{for all $\chi \in \Xi_{{\boldsymbol\phi }|_{G_{K}},n}^\mp$}\},
\]
we have an analogous decomposition after tensoring with $\BQ_p$.

\begin{prop}\label{prop, zeta sign}
We have
$$\loc_{\fp}(\mathfrak{z}_{{\boldsymbol{\phi}}}(\gamma_A)) \in H^1_{\boldsymbol{\varepsilon}^{(p)}(\boldsymbol{\phi})}(K, \BT_{\boldsymbol \phi}^{}) \subset H^1(K, \BT_{\boldsymbol \phi}^{}),$$ 
where  
\[
{\varepsilon}^{(p)}(\boldsymbol{\phi}):=
\frac{{\varepsilon}(\boldsymbol{\phi})}{\varepsilon_p(\mathrm{Ind}_{K/\BQ_p}(\boldsymbol{\phi}_p))}.
\]
\end{prop}
\begin{proof}
 If 
$\boldsymbol{\varepsilon}(\boldsymbol{\phi}\chi^{-1})=-1$, then 
$$\exp^*_\chi(\mathfrak{z}_{{\boldsymbol{\phi}}}(\gamma_A))=0$$ by Theorem~\ref{thm, erl}. In the following we relate this global epsilon constant to the local epsilon constant at $p$. 

For a finite order anticyclotomic character $\chi$ over $\CK$, we have 
\[
\frac{{\varepsilon}(\boldsymbol{\phi}\chi^{-1})}{{\varepsilon}(\boldsymbol{\phi})}=
\frac{\varepsilon_p(\mathrm{Ind}_{K/\BQ_p}(\boldsymbol{\phi}\chi^{-1})}{\varepsilon_p(\mathrm{Ind}_{K/\BQ_p}({\boldsymbol{\phi}}))}
\]
by Lemma \ref{lem, eps global} i). 
So it follows that 
 \[ 
 {\varepsilon}_p(\mathrm{Ind}_{K/\BQ_p}(\boldsymbol{\phi}\chi^{-1}))=
{\varepsilon}_p(\mathrm{Ind}_{K/\BQ_p}(\boldsymbol{\phi}))
\cdot
\frac{{\varepsilon}(\boldsymbol{\phi}\chi^{-1})}{{\varepsilon}(\boldsymbol{\phi})}.\] 
 Therefore, the condition $\boldsymbol{\varepsilon}(\boldsymbol{\phi}\chi^{-1})=-1$ is equivalent to 
the condition that \[ 
-{\varepsilon}_p(\mathrm{Ind}_{K/\BQ_p}(\boldsymbol{\phi}\chi^{-1}))=\frac{{\varepsilon}_p(\mathrm{Ind}_{K/\BQ_p}(\boldsymbol{\phi}))}{\boldsymbol{\varepsilon}(\boldsymbol{\phi})}=:\boldsymbol{\varepsilon}^{(p)}(\boldsymbol{\phi}). 
\]
So the assertion follows from Theorem~\ref{thm, rubin decomposition ramified} i) and Lemma~\ref{lem, int}. 
\end{proof}

\begin{remark}\noindent
\begin{itemize}
 \item[i)] If the conductor of ${\boldsymbol \phi}$ at $\fp$ equals one, then 
 $$
\varepsilon^{(p)}({\boldsymbol \phi})=\omega_{K/\BQ_p}(-2){\varepsilon}(\boldsymbol{\phi})
=\left(\frac{-2}{p}\right){\varepsilon}(\boldsymbol{\phi}) 
$$
by Proposition \ref{prop, ramified epsilon} i).
\item[ii)]Note that 
$\boldsymbol{\varepsilon}^{(p)}(\boldsymbol{\phi})=-\prod_{\ell\nmid p}\varepsilon_\ell(\mathrm{Ind}_{K/\BQ_p}(\boldsymbol{\phi}))$. 
 Hence, Proposition \ref{prop, zeta sign} is compatible with
 Kato's global epsilon conjecture for the anticyclotomic $\BZ_p$-deformation of ${\boldsymbol \phi}$, which predicts the functional equation  
 \cite[Conj.~4.4]{NaKato} for zeta elements. 
\end{itemize}
\end{remark}

\subsubsection{$p$-adic $L$-function}
Let $\varepsilon$ denote the sign of $\boldsymbol{\varepsilon}^{(p)}(\boldsymbol{\phi})$. 

In light of Proposition \ref{prop, zeta sign} we are led to the following bounded $p$-adic $L$-function at primes $p$ of additive reduction. 
\begin{defn}\label{def, pL}
If either $p\ge 5$ or $p=3$ and ${\boldsymbol \phi} \nequiv 1 \mod{\fm_{\CO_L}}$ on $G_K$ (Case I), then let  $v_\varepsilon$ be a basis of the $\Lambda$-module $H^1_{\varepsilon}(K, \mathbb{T}_{\boldsymbol \phi}^{})$. 
 In the excluded case that $p=3$ and ${\boldsymbol \phi} \equiv 1 \mod{\fm_{\CO_L}}$ on $G_K$ (Case II),
let $v_\varepsilon$ be a basis of the $\Lambda\otimes_{\BZ_p}\BQ_p $-module $H^1_{\varepsilon}(K, \mathbb{T}_{\boldsymbol \phi}^{})\otimes_{\BZ_p} \BQ_p$. 
Define a $p$-adic $L$-function 
$$\mathscr{L}_{p,v_\varepsilon}(A) \in 
\begin{cases}
 \Lambda \qquad &(\text{Case I})\\
\; \Lambda\otimes_{\BZ_p}\BQ_p  \qquad &(\text{Case II}).
\end{cases}
$$ 
by 
\[
\loc_{\fp}(\mathfrak{z}_{{\boldsymbol{\phi}}}(\gamma_A))=\mathscr{L}_{p,v_\varepsilon}(A)\cdot v_{\varepsilon}. 
\]
\end{defn}
Note that the $p$-adic $L$-function $\mathscr{L}_{p,v_\varepsilon}(A)$ also depends on the choice of an ideal $\mathfrak{f}$ as in Theorem~\ref{thm, erl}, which we suppress in the notation for simplicity. 
For a character $\chi \in \Xi_{\CK_{\infty}^{\rm ac}, \boldsymbol{\phi}}^{+}$, we have 
 \[\exp^*_\chi(v_\varepsilon) \in D^0_{\mathrm{dR}}(K, V_{\boldsymbol \phi}^{\otimes -1}(1)(\chi^{}))\subseteq K_{\infty}\otimes_KK_{\chi}\omega_A,\]
  and denote by $\delta^{\omega_A}_{v_\varepsilon}(\chi)$
   its image by the projection 
  \[K_{\infty}\otimes_K K_{\chi}\omega_A \rightarrow \overline{\BQ}_p,\quad  
  a\otimes b \omega_A\mapsto ab.\]
 By Corollary~\ref{cor, nv dual exp}, we have $\delta^{\omega_A}_{v_\varepsilon}(\chi)\not=0$. 
 
 An interpolation property of $\mathscr{L}_{p,v_\varepsilon}(A)$ is given by the following. 
\begin{thm}\label{thm, ram pL}
Let $A$ be a CM elliptic curve defined over $\BQ$ and $\CK$ the CM field. 
Let $p$ be an odd prime ramified in $\CK$ and ${\boldsymbol \phi}$ the associated $p$-adic character of $G_\CK$. For a basis $v_\varepsilon$ of the signed submodule $H^1_{\varepsilon}(K, \mathbb{T}_{\boldsymbol \phi})$ or  $H^1_{\varepsilon}(K, \mathbb{T}_{\boldsymbol \phi})\otimes_{\BZ_p}\BQ_p$  as in Definition~\ref{def, pL}, let $\mathscr{L}_{p,v_\varepsilon}(A)$ be the associated anticyclotomic $p$-adic $L$-function. 
Then for any $\chi \in \Xi_{\CK_{\infty}^{\rm ac}, \boldsymbol{\phi}}^{+}$ 
we have
 \[
 \chi^{-1}(\mathscr{L}_{p,v_\varepsilon}(A))=
 \frac{1}{\delta_{v_\varepsilon}^{\omega_A}(\chi)}\cdot \frac{L({\phi}\chi^{-1}, 1)}{\Omega_\infty}.  
 \]
  \end{thm}
 \begin{proof}
 This is a simple consequence of Theorem~\ref{thm, erl}, 
 noting the commutative diagram
 \[
 \xymatrix{
 H^1(K, \mathbb{T}_{\boldsymbol \phi}) \ar[r]^-{\chi^{-1}} \ar[dr]_-{\exp_{\chi^{}}^*}& H^1(K, T_{\boldsymbol \phi}^{\otimes -1}(1)(\chi^{})) = H^1(K, T_{{\boldsymbol \phi}\chi^{-1}}^{\otimes -1}(1)))\ar[d]^{\exp^*} \\
 & D_{\mathrm{dR}}(K, V_{\boldsymbol \phi}^{\otimes -1}(1)(\chi^{})), 
 }
 \]
 where the horizontal arrow is due to $ H^1(K, \mathbb{T}_{\boldsymbol \phi})\otimes_{\Lambda,\chi^{-1}}
 K_\chi=H^1(K, T_{\varphi}^{\otimes -1}(1)(\chi^{}))$. 
 \end{proof}
 \begin{remark} 
 If the conductor of the Weil--Deligne representation associated to ${\boldsymbol \phi}|_{G_{K}}$ equals one, then the interpolation region for 
$\mathscr{L}_{p,v_{\varepsilon}}(A)$ may be explicitly stated as consisting of characters 
 $\chi \in \Xi^{\varepsilon}_{\phi|G_{K}}$, 
 where the local and global $p$-power order anticyclotomic characters are identified. 
  \end{remark}

\end{document}